\documentclass[11pt,letterpaper,reqno,draft]{amsart}
\usepackage{amsfonts}

\usepackage{cite}

\usepackage{amsmath,amssymb,amsthm,amsfonts}
\usepackage{mathrsfs}
\usepackage{bbm}
\usepackage[final]{hyperref}

\usepackage{orcidlink}

\usepackage{todonotes}

\usepackage{color}

\newtheorem{lemma}{Lemma}[section]
\newtheorem{theorem}{Theorem}[section]

\newtheorem{remark}{Remark}[section]
\newtheorem{corollary}{Corollary}[section]

\numberwithin{equation}{section}

\newcommand{\dis}{\displaystyle}
\allowdisplaybreaks

\newcommand{\C}{\mathbb{C}}

\newcommand{\R}{\mathbb{R}}

\newcommand{\Z}{\mathbb{Z}}

\renewcommand{\S}{\mathbb{S}}
\newcommand{\T}{\mathbb{T}}

\newcommand{\FP}{\mathbf{P}}

\newcommand{\FI}{\mathbf{I}}

\newcommand{\Fu}{\mathbf{u}}
\newcommand{\Fb}{\mathbf{b}}
\newcommand{\Fk}{\mathbf{k}}
\newcommand{\Fv}{\mathbf{v}}
\newcommand{\Fx}{\mathbf{x}}
\newcommand{\FV}{\mathbf{V}}
\newcommand{\Fe}{\mathbf{e}}

\newcommand{\FTv}{\tilde{\Fv}}
\newcommand{\tv}{\tilde{v}}

\newcommand{\CA}{\mathcal{A}}

\newcommand{\CD}{\mathcal {D}}
\newcommand{\CE}{\mathcal{E}}
\newcommand{\CF}{\mathcal{F}}

\newcommand{\CH}{\mathcal{H}}
\newcommand{\CI}{\mathcal{I}}
\newcommand{\CJ}{\mathcal{J}}
\newcommand{\CK}{\mathcal{K}}
\newcommand{\CL}{\mathcal{L}}

\newcommand{\CZ}{\mathcal{Z}}

\newcommand{\CW}{\mathcal{W}}
\newcommand{\CP}{\mathcal{P}}
\newcommand{\CQ}{\mathcal{Q}}

\newcommand{\CV}{\mathcal{V}}

\newcommand{\SH}{\mathscr{H}}

\newcommand{\SV}{\mathscr{V}}

\newcommand{\SR}{\mathscr{R}}
\newcommand{\SU}{\mathscr{U}}
\newcommand{\SP}{\mathscr{P}}

\newcommand{\RL}{\mathfrak{L}}

\newcommand{\na}{\nabla}

\newcommand{\al}{\alpha}

\newcommand{\ga}{\gamma}
\newcommand{\om}{\omega}
\newcommand{\Om}{\Omega}
\newcommand{\la}{\lambda}
\newcommand{\de}{\delta}
\newcommand{\si}{\sigma}
\newcommand{\pa}{\partial}
\newcommand{\ka}{\kappa}
\newcommand{\eps}{\epsilon}

\newcommand{\ta}{\theta}
\newcommand{\vth}{\vartheta}

\newcommand{\vps}{\varepsilon}

\newcommand{\Ga}{\Gamma}

\newcommand{\lag}{\langle}
\newcommand{\rag}{\rangle}

\DeclareSymbolFont{yhlargesymbols}{OMX}{yhex}{m}{n} \DeclareMathAccent{\yhwidehat}{\mathord}{yhlargesymbols}{"62}

\makeatletter
\@namedef{subjclassname@2020}{%
  \textup{2020} Mathematics Subject Classification}
\makeatother

\begin{document}

\title[3D Couette flow via Boltzmann equation in diffusive limit]{The 3D kinetic Couette flow via the Boltzmann equation in the diffusive limit}

\author[R.-J. Duan]{Renjun Duan}
\address[RJD]{Department of Mathematics, The Chinese University of Hong Kong,
Shatin, Hong Kong, P.R.~China}
\email{rjduan@math.cuhk.edu.hk}

\author[S.-Q. Liu]{Shuangqian Liu}
\address[SQL]{School of Mathematics and Statistics, and Key Lab NAA-MOE, Central China Normal University, Wuhan 430079, P.R.~China}
\email{sqliu@ccnu.edu.cn}

\author[R. M. Strain]{Robert M. Strain}
\address[RMS]{University of Pennsylvania, Department of Mathematics, David Rittenhouse Lab, 209 South 33rd Street, Philadelphia, PA 19104-6395, USA}
\email{strain@math.upenn.edu}

\author[A. Yang]{Anita Yang}
\address[AY]{Department of Mathematics, The Chinese University of Hong Kong,
Shatin, Hong Kong, P.R.~China}
\email{ayang@math.cuhk.edu.hk}

\begin{abstract} 
In this paper we study the Boltzmann equation in the diffusive limit in a channel domain $\T^2\times (-1,1)$ for the 3D kinetic Couette flow. Our results demonstrate that the first-order approximation of the solution is governed by the perturbed incompressible Navier-Stokes-Fourier system around the fluid Couette flow.  Moreover, in the absence of external forces, the 3D kinetic Couette flow asymptotically converges over time to the 1D steady planar kinetic Couette flow.  Our proof relies on (i) the Fourier transform on $\T^2$ to essentially reduce the 3D problem to a one-dimensional one, (ii) anisotropic Chemin-Lerner type function spaces, incorporating the Wiener algebra, to control nonlinear terms and address the singularity associated with a small Knudsen number in the diffusive limit, and (iii) Caflisch’s decomposition, combined with the $L^2\cap L^\infty$ interplay technique, to manage the growth of large velocities.
\end{abstract}

\date{February 19, 2025}
\subjclass[2020]{35Q20, 35B40}


\keywords{Boltzmann equation, kinetic Couette flow, shear force, stationary solution, asymptotic stability}
\maketitle

\thispagestyle{empty}
	
\tableofcontents

\section{Introduction}


The Boltzmann equation is a fundamental physical model in collisional kinetic theory which describes the motion of a rarefied gas when the Knudsen number is finite, cf.~\cite{CerBook, Ko, Sone07, TM}.  The stability of the standard Maxwellian equilibrium along the Boltzmann dynamics is a result of the celebrated $H$-theorem and the intrinsic structure of the Boltzmann collision operator.  However, a very important physical motivation for introducing the Boltzmann equation is to study the non-equilibrium dynamics of a rarefied gas which can admit stationary states that are non-Maxwellian.  The physical importance of these stationary non-equilibrium states has recently been discussed in the review article \cite{EM-JSP} and the references therein.  In most situations, those non-equilibrium effects are caused by inhomogeneous data at physical boundaries or at far fields with variable temperature and velocity, or possibly caused by a given external force.  For mathematical analysis based on the perturbation theory of solutions \cite{Guo-2010,U86, U74}, these non-equilibrium steady states are not Maxwellian but they are still close to a global Maxwellian and they still decay exponentially for large velocities; for instance, we refer readers to the early work \cite{Guir1,Guir2} and important progress \cite{EGKM-13, EGKM-18, EGM, UA1, UA2}. 

However, there also exists a class of non-equilibria achieved by the Boltzmann dynamics that exhibit a polynomial tail for large velocities. Among them, we mention several recent works \cite{BNV-2019, DL-2020, DLY-2021, JNV-ARMA, Kep} in the context of kinetic shear flow governed by the Boltzmann equation.  On this topic, we refer to the monograph \cite{GaSa} and the survey \cite{NV}. We will discuss in more detail the following two cases:
\begin{itemize}
  \item If there is no physical boundary, that's in the spatially homogeneous case, the uniform kinetic shear flow is a class of homoenergetic solutions determined by the time-evolutionary Boltzmann equation with a deformation force \cite{Gal-R, Tru}. In such case, the large time behavior of solutions is self-similar with a time-growing temperature that depends on the Boltzmann collision kernel \cite{Bv75, Bv76, BCS, Cer89, Cer00, Cer02}. In the Maxwell molecules collision kernel, for example, the temperature or equivalently the total energy grows exponentially to infinity at infinite time and the self-similar steady profile only has a polynomial tail \cite{BNV-2019, DL-2020, JNV-ARMA, Kep}.     
  \item If there is a physical boundary in space, a typical example is the famous Couette flow for the rarefied gas confined by two parallel infinite plates moving relative to each other with opposite velocities. The Couette problem is one of the simplest problems of gas dynamics but raises many challenging difficulties in the mathematical study.  Recently, for the diffusive reflection boundary condition, under the smallness assumption on the relative velocity the first and second authors of this paper together with their collaborator \cite{DLY-2021} gave a rigorous justification of the existence and positivity of the 1D Couette flow that depends on one spatial component normal to the boundary plates; see also \cite{MW} for the asymptotic stability of the 1D Couette flow under the 3D perturbation. In addition, we refer to the series of works \cite{AEMN, AN06, AN05} on the Boltzmann equation in a Couette setting for rarefied gas between two coaxial rotating cylinders. 
\end{itemize}

Although there has been much progress on the kinetic shear flow either in the spatially homogeneous setting (i.e., homoenergetic solutions) or in the spatially one-dimensional setting (i.e, planar Couette flow), it has remained largely open to obtain the long time asymptotic stability of those kinetic shear flows in the multi-dimensional spatially inhomogeneous regime. In contrast, such stability or even instability theory has been extensively developed in recent years in the context of classical fluid dynamic equations such as the Euler and Navier-Stokes systems; see \cite{BGM-BAMS,BM-15, BMV-16, BGM-17, IJ, ScHe}, for instance.  Then it becomes an important task to develop analogous mathematical results at the kinetic level in case of the finite Knudsen number, and additionally to construct Boltzmann solutions in the hydrodynamic limit when the Knudsen number is vanishing.

In this paper we study the 3D Couette problem for the Boltzmann equation in the diffusive limit.  Our first main goal is to construct the stationary 3D kinetic Couette flow solution.  Our second main goal is to prove the exponential asymptotic stability of this stationary state for solutions to the full time-dependent problem. To the best of our knowledge, this is the first result on the existence and stability of the multi-dimensional kinetic Couette flow governed by the nonlinear Boltzmann equation with a small Knudsen number. Additionally, in the absence of external forces, our results show the exponential convergence of the 3D time-dependent solutions toward the steady 1D planar kinetic Couette flow uniformly with respect to a small Knudsen number.  We expect that the techniques that are introduced herein will be useful to understand other physically important stationary non-equilibrium states and their stability.

\subsection{Problem} 
Let the rarefied gas  be confined in a three-dimensional periodic channel domain
$$
\Omega:=\{\Fx=(x,y,z)\in \T\times \T\times (-1,1)\}
$$
where $\T:=\R/(2\pi \Z)$. For simplicity, we write $\bar{x}=(x,y)\in \T^2$. Let the two boundary plates at $z=\pm 1$ be moving relative to each other with opposite velocities along the horizontal $x$-direction
$$
U_\pm=(\pm \al \eps,0,0),
$$
respectively, where $\al>0$ denotes the shear strength and $\eps>0$ is the Knudsen number defined as the ratio of the molecular mean free path to a representative physical length. For later use we define the macroscopic shear velocity for the 3D fluid Couette flow by
\begin{equation}
\label{def.msv}
U(z)=(\al\eps z,0,0), \quad -1\leq z\leq 1,
\end{equation}
that connects the boundary velocities $U_\pm$ linearly. We further suppose that such motion of the rarefied gas flow is governed by the following initial boundary value problem of the Boltzmann equation under the diffusive scaling:
\begin{align}\label{rbe}
\eps\pa_t\widetilde{F}^\eps+\Fv\cdot\na_\Fx \widetilde{F}^\eps+\eps^2\Phi\cdot\na_\Fv \widetilde{F}^\eps=\frac{1}{\eps}Q(\widetilde{F}^\eps,\widetilde{F}^\eps),\ t>0,\ \Fx\in\Om,\ \Fv\in\R^3.
\end{align}
We supplement \eqref{rbe} with the diffusive reflection boundary condition at $z=\pm 1$ as
\begin{align}\label{obd-1}
\widetilde{F}^\eps(t,x,y,\pm1,\Fv)|_{v_{z}\lessgtr0}
=
\sqrt{2\pi}M_{w\pm}\int_{\tv_{z}\gtrless0}\widetilde{F}^\eps(t,x,y,\pm1,\FTv)|\tv_{z}|d\FTv,
\end{align}
where we use the integration variable $\FTv=(\tv_x,\tv_y,\tv_{z})\in\R^3$ over $\tv_{z}\gtrless0$ respectively.  The initial condition is given by
\begin{align}\label{s-id-o}
\widetilde{F}^\eps(0,\Fx,\Fv)=\widetilde{F}_0^\eps(\Fx,\Fv).
\end{align}
Here $\widetilde{F}^\eps=\widetilde{F}^\eps(t,\Fx,\Fv)\geq0$ stands for the density distribution function of gas particles with velocity $\Fv=(v_x,v_y,v_{z})\in\R^3$ at time $t\geq0$ and position $\Fx=(x,y,z)\in \Om$. 
The function $\Phi=(\Phi_x(z),\Phi_y(z),\Phi_z(z))$ is a given external force field and for a technical reason we have assumed that $\Phi$ depends only on $z$. Moreover, $M_{w\pm}$ are the boundary Maxwellians at the plates $z=\pm 1$ of the form
\begin{equation}
\label{def.bcM}
M_{w\pm}=(2\pi)^{-\frac{3}{2}}\exp\left(-\frac{|\Fv-U_\pm |^2}{2}\right).
\end{equation}
These Maxwellians have the same macroscopic velocities as those of the moving boundaries. 
In addition, $Q(\cdot,\cdot)$ is the bilinear Boltzmann collision operator which takes the form of
\begin{align}
Q(F_1,F_2)&=\int_{\R^3\times\S^2_+}B_0(\Fv-\Fv_\ast,\om)[F_1(\Fv_\ast')F_2(\Fv')-F_1(\Fv_\ast)F_2(\Fv)]d\Fv_\ast d\om \notag \\
&:= Q_{\rm{gain}}-Q_{\rm{loss}}, \label{Q-op}
\end{align}
where 
$\om\in \S^2_+:=\{\om~|~\om\cdot (\Fv_\ast-\Fv)\geq0\}$.  Additionally $(\Fv,\Fv_\ast)$ and  $(\Fv',\Fv'_\ast)$ denote the pre-post collisional velocity pairs of particles, which satisfy
 \begin{align}\label{pp-ve}
\Fv'=\Fv+(\Fv_\ast-\Fv)\cdot\om \om,\ \Fv_\ast'=\Fv_\ast-(\Fv_\ast-\Fv)\cdot\om \om.
\end{align}
Note that the $\omega$-representation in \eqref{pp-ve} is a consequence of the following molecular conservation laws of momentum and energy for elastic collisions:
 \begin{align}
\Fv+\Fv_\ast=\Fv'+\Fv_\ast',\quad |\Fv|^2+|\Fv_\ast|^2=|\Fv'|^2+|\Fv_\ast'|^2.\notag
\end{align}
The collisional kernel $B_0(\Fv-\Fv_\ast,\om)$ is assumed to satisfy the angular cutoff assumption and take the form:
\begin{align}
B_0(\Fv-\Fv_\ast,\om)=|\Fv-\Fv_\ast|^\ga b_0(\cos\ta),\notag
\end{align}
where $\ga\in[0,1]$, $\cos\ta=\om\cdot(\Fv_\ast-\Fv)\geq 0$ with $\om\in \S^2_+$.  We suppose that there exists a constant $C>0$ such that
$$
0\leq b_0(\cos\ta)\leq C\cos\ta.
$$
Thus our assumption covers the cases of Maxwell molecules ($\gamma=0$) and hard potentials ($0<\gamma\leq 1$) including the hard sphere model (which is $\gamma=1$ and $b_0(\cos\ta)=C\cos\theta$).

\subsection{Reformulation}

We seek solutions of \eqref{rbe} and \eqref{obd-1}  in the form of
$$
\widetilde{F}^\eps(t,\Fx,\Fv)=F^\eps(t,\Fx,\Fv-U(z))=F^\eps(t,\Fx,v_x-\al\eps z,v_y,v_{z}),
$$
where $\Fv-U(z)=(v_x-\al\eps z,v_y,v_{z})$ is the peculiar velocity of gas particles, cf.~\cite{GaSa}.   
Then \eqref{rbe}, \eqref{obd-1} and \eqref{s-id-o}  can be converted to
\begin{align}\label{s-rbe}
\eps\pa_tF^\eps+\Fv\cdot\na_\Fx F^\eps+\eps^2\Phi\cdot\na_\Fv F^\eps+\al\eps z\pa_{x}F^\eps-\al\eps v_{z}\pa_{v_x}F^\eps=\frac{1}{\eps}Q(F^\eps,F^\eps),
\end{align}
where $t>0$, $\Fx\in\Om$ and $\Fv\in\R^3$.  The diffuse reflection boundary condition is given by
\begin{align}\label{obd}
F^\eps(t,x,y,\pm1,\Fv)|_{v_{z}\lessgtr0}=\sqrt{2\pi}\mu\int_{\tv_{z}\gtrless0}F^\eps(t,x,y,\pm1,\FTv)|\tv_{z}|d\FTv. 
\end{align}
Then the initial condition is 
\begin{align}\label{s-id}
{F}^\eps(0,\Fx,\Fv)={F}_0^\eps(\Fx,\Fv):=\widetilde{F}_0^\eps(\Fx,\Fv+U(z))
=
\widetilde{F}_0^\eps(\Fx,v_x+\al\eps z,v_y,v_{z}).
\end{align}
 Above $\mu$ is the normalized global Maxwellian which is defined by
\begin{equation}
\label{def.gmax}
\mu=(2\pi)^{-3/2}e^{-|\Fv|^2/2}.
\end{equation}
Note that under such a reformulation the boundary Maxwellians become spatially homogeneous with zero bulk velocity while an extra forcing term $-\al\eps v_{z}\pa_{v_x}F^\eps$ is present in the equation.

\subsubsection{Steady problem}
It is expected that the long time behavior of \eqref{s-rbe}, \eqref{obd} and \eqref{s-id} will be governed by the following steady problem
\begin{align}\label{st-rbe}
\Fv\cdot\na_\Fx F_{st}^\eps+\eps^2\Phi\cdot\na_\Fv F_{st}^\eps+\al\eps z\pa_{x}F_{st}^\eps-\al\eps v_{z}\pa_{v_x}F_{st}^\eps=\frac{1}{\eps}Q(F_{st}^\eps,F_{st}^\eps),
\end{align}
where $\Fx\in\Omega$, $\Fv\in \R^3$, and
\begin{align}\label{st-obd}
F_{st}^\eps(x,y,\pm1,\Fv)|_{v_{z}\lessgtr0}=\sqrt{2\pi}\mu\int_{\tv_{z}\gtrless0}F_{st}^\eps(x,y,\pm1,\FTv)|\tv_{z}|d\FTv.
\end{align}
Moreover, we assume that the total mass of gas particles in the full domain $\Omega$ is fixed to be
\begin{align}\label{mass}
\int_{\Om\times\R^3}F^\eps_{st}(\Fx,\Fv)d\Fx d\Fv=|\Om|.
\end{align}
Furthermore, at a formal level, in the diffusive limit $\eps\to 0$, $F_{st}^\eps$ tends to $\mu$  in terms of \eqref{st-rbe}, \eqref{st-obd} and \eqref{mass}. Then, to construct the solution, we set the Hilbert expansion as
\begin{align}\label{st-exp}
F_{st}^\eps=\mu+\eps\sqrt{\mu}\{f_1+\eps f_2+\eps^{\frac{1}{2}}f_R\},
\end{align}
where $f_1$ and $f_2$ are respectively the first-order and second-order correction terms, $f_R$ is the remainder and we have ignored the dependence of $f_R$ on $\eps$ for brevity. We now define the linearized collision operator by
\begin{align}
Lf=-\mu^{-1/2}\{Q(\mu,\sqrt{\mu}f)+Q(\sqrt{\mu}f,\mu)\},\notag
\end{align}
and we define the non-linear collision operator as
\begin{align}
\Ga(f,g)=\mu^{-1/2}Q(\sqrt{\mu}f,\sqrt{\mu}g).\notag
\end{align}
Next, we plug \eqref{st-exp} into \eqref{st-rbe} and compare the different orders of $\eps$ in the resulting equation, to obtain
\begin{align}\label{f0}
Lf_1=0,
\end{align}
\begin{align}\label{f1}
\Fv\cdot\na_\Fx f_1+\al v_{z}v_x\mu^{\frac{1}{2}}+Lf_2=\Ga(f_1,f_1),
\end{align}
and
\begin{align}\label{f2}
\Fv\cdot\na_\Fx f_2+\al z\pa_{x}f_1-\al \mu^{-\frac{1}{2}}v_{z}\pa_{v_x}\{\sqrt{\mu}f_1\}-\Phi\cdot \Fv\mu^{\frac{1}{2}}=\Ga(f_1,f_2)+\Ga(f_2,f_1).
\end{align}
Hence,  the remainder $f_R$ satisfies
\begin{align}\label{f-rbe}
\Fv\cdot\na_\Fx f_{R}+&\eps^2\Phi\cdot\na_\Fv f_{R}+\al\eps z\pa_{x}f_{R}-\al\eps v_{z}\pa_{v_x}f_{R}
+\frac{\al\eps v_xv_{z}}{2}f_{R}-\frac{\eps^2}{2}\Phi\cdot \Fv f_{R}
\notag\\
+\frac{1}{\eps} Lf_{R}
=
&\eps^{\frac{1}{2}}\Ga(f_{R},f_{R})+\Ga(f_{R},f_1+\eps f_2)+\Ga(f_1+\eps f_2,f_{R})+\eps^{\frac{3}{2}}\Ga(f_2,f_2)
\notag\\&-\eps^{\frac{3}{2}}\mu^{-1/2}\Phi\cdot \na_\Fv(\sqrt{\mu}f_{1})-\eps^{\frac{5}{2}}\mu^{-1/2}\Phi\cdot \na_\Fv(\sqrt{\mu}f_{2}) 
\notag\\&-\al\eps^{\frac{3}{2}}z\pa_xf_2-\al\eps^{\frac{3}{2}} \mu^{-1/2}v_z\pa_{v_x}\{\sqrt{\mu}f_2\}.
\end{align}
Moreover, by \eqref{mass} it holds that 
\begin{align}
\int_{\Om\times\R^3}f_R(\Fx,\Fv)d\Fx d\Fv=0.\notag
\end{align}
Next, we determine $f_1$ and $f_2$. In fact, from \eqref{f0}, one has
\begin{align}\label{f1-def}
f_1=\FP f_1=\left\{\rho_s+\Fv\cdot \Fu_s+\frac{1}{2}(|\Fv|^2-3)\ta_s\right\}\sqrt{\mu},
\end{align}
where $\FP$ is an $L^2$ projection from $L^2_\Fv$ to the null space of $L$,  denoted by 
$$
\ker (L)={\rm span}\left\{1,\Fv,\frac{1}{2}(|\Fv|^2-3)\right\}\sqrt{\mu}.
$$
Moreover, \eqref{f1} gives
\begin{align}\label{ic-bq-re}
\na_{\Fx}\cdot \Fu_s=0,\ \ \na_{\Fx}(\rho_s+\ta_s)=0.
\end{align}
Due to this, without loss of generality we assume $\rho_s=-\ta_s-\frac{1}{\Om}\int_{\Om}\rho_s(\Fx)d\Fx$,
 and we further define
\begin{align}\label{f2-def}
f_2=&L^{-1}\left\{-\Fv\cdot\na_\Fx f_1-\al v_{z}v_x\mu^{\frac{1}{2}}+\Ga(f_1,f_1)\right\}
\notag\\
&+
\left\{\Fu_{s,2}\cdot \Fv+\frac{|\Fv|^2-3}{2}\ta_{s,2}\right\}\sqrt{\mu}
\notag\\
=&-\sum\limits_{i,j=1}^3\bar{A}_{ij}\pa_iu_{s,j}-\sum\limits_{i=1}^3\bar{B}_i\pa_i\ta_s
+\{\FI-\FP\}\left\{\frac{(\Fv\cdot\Fu_s)^2}{2}\sqrt{\mu}\right\}\notag
\\
&
+\{\FI-\FP\}\left\{\frac{(|\Fv|^2-5)^2\ta_s^2}{8}\sqrt{\mu}\right\}
-\al L^{-1}(v_{z}v_x\mu^{\frac{1}{2}})
\notag
\\
&
+(|\Fv|^2-5)(\Fv\cdot \Fu_s)\ta_s\sqrt{\mu}+\left\{\Fu_{s,2}\cdot \Fv+\frac{|\Fv|^2-3}{2}\ta_{s,2}\right\}\sqrt{\mu},
\end{align}
where  $\FI$ is the identity operator.  We further define
$$
\bar{A}_{ij}:= L^{-1}A_{ij}=L^{-1}\left\{\left(v_iv_j-\frac{\de_{ij}|\Fv|^2}{3}\right)\sqrt{\mu}\right\},
$$
$$
\bar{B}_i:= L^{-1}B_{i}=L^{-1}\left\{\frac{(|\Fv|^2-5)v_i}{2}\sqrt{\mu}\right\}.
$$
Additionally for $i\in\{1,2,3\}$ we have used the notations
$$
\pa_i\in\{\pa_{x}, \pa_{y},\pa_{z}\},\ \ v_i\in\{v_x,v_y,v_z\}, \ u_{s,i}\in\{u_{s,x},u_{s,y},u_{s,z}\}.
$$
Above we also have used the following known identity
$$
2\Ga(\FP f,\FP f)=L\left\{\frac{(\FP f)^2}{\sqrt{\mu}}\right\}.
$$
Moreover, the macroscopic velocity of $f_2$, namely $\Fu_{s,2}:=(u_{s,2,x},u_{s,2,y},u_{s,2,z})$, is defined as
\begin{align}\label{u2}
\Fu_{s,2} &=-\al\na_\Fx\Delta^{-1}\{z\pa_x\rho_s\},\quad u_{s,2,z}(x,y,\pm1)=0,
\end{align}
which follows from the inner product $\lag \eqref{f2},\sqrt{\mu}\rag.$ The macroscopic temperature of $f_2$, denoted by $\ta_{s,2}$, is given by \eqref{P}.
It should be pointed out that we have assumed the macroscopic mass density of $f_2$, represented by $\rho_{s,2}$, to be zero. 

By considering velocity moments $[\Fv,\frac{1}{2}(|\Fv|^2-5)]\sqrt{\mu}$ for equation \eqref{f2},
and utilizing \eqref{f2-def} as well as \eqref{ic-bq-re}, one can deduce that
$[\rho_s,\Fu_s,\ta_s]$ satisfies the perturbed incompressible Navier-Stokes-Fourier system around the fluid Couette flow $(\al z,0,0)^T$:
\begin{eqnarray}\label{ins-s}
\left\{\begin{array}{rll}
&\na_\Fx\cdot \Fu_s=0,\ \ \rho_s=-\ta_s,\\[2mm]
&\Fu_s\cdot\na_\Fx \Fu_s+\na_\Fx P+\al z\pa_{x}\Fu_s+\al(u_{s,z},0,0)^T-\Phi=\eta\Delta_\Fx \Fu_s,\\[2mm]
&\al z\pa_{x}\ta_s+\na_\Fx\ta_s\cdot \Fu_s=\frac{2}{5}\ka\Delta_\Fx \ta_s,\\[2mm]
&\Fu_s(x,y,\pm1)=0,\ \ta_s(x,y,\pm1)=0,
\end{array}\right.
\end{eqnarray}
where
\begin{align}\label{P}
\na_\Fx P=\na_\Fx\left\{\ta_{s,2}-\frac{1}{3}|\Fu_s|^2\right\},\quad \int_{\Om}Pd\Fx=0,
\end{align}
and the positive constants $\eta$ and $\ka$ denote the diffusive coefficient and heat conductivity coefficient, respectively.

\begin{remark}
Let $U_s=(\alpha z,0,0)+\Fu_s$, then $U_s$ satisfies the boundary-value problem on the incompressible Navier-Stokes system in the finite channel domain $\T^2\times (-1,1)$:
\begin{eqnarray}\label{ins-soo}
\left\{\begin{array}{rl}
&\na_\Fx\cdot U_s=0, \\[2mm]
&U_s\cdot\na_\Fx U_s+\na_\Fx P-\Phi=\eta\Delta_\Fx U_s,\\[2mm]
&U_s(x,y,\pm1)=\pm \alpha,
\end{array}\right.
\end{eqnarray}
which is consistent with that derived from the diffusive limit of the original steady problem corresponding to \eqref{rbe} and \eqref{obd-1}.
If $\Phi\equiv 0$ is further assumed, the only solution to \eqref{ins-soo} is given by the planar Couette flow $(\alpha z,0,0)$.  This is exactly why we have introduced the bulk velocity in \eqref{def.msv}; to facilitate the reformulation of the problem in the moving frame.
\end{remark}

Furthermore, in view of \eqref{f1-def}, $\eqref{ins-s}_1$, $\eqref{ins-s}_4$, \eqref{st-exp} and \eqref{st-obd}, the boundary condition for
\eqref{f-rbe} can be written as
\begin{align}\label{st-bd}
f_{R}(x,y,\pm1,\Fv)|_{v_z\lessgtr0}=P_\ga f_R+\eps^{\frac{1}{2}} \{-f_2+P_\ga f_2\}
:= P_\ga f_R+\eps^{\frac{1}{2}}r,
\end{align}
where
\begin{align}
P_\ga f :=\sqrt{2\pi\mu}\int_{v_z\gtrless0}f(x,y,\pm1,\Fv)\sqrt{\mu}|v_z|d\Fv,\ r :=-f_2+P_\ga f_2.\notag
\end{align}
Note that
\begin{align}\label{r-bd}
\int_{v_z\lessgtr 0}r(\pm 1)\sqrt{\mu} |v_z| d\Fv=0,\ \int_{v_{z}\lessgtr 0}P_\ga f_R(\pm1)r(\pm 1) |v_z| d\Fv=0,
\end{align}
due to the definitions \eqref{f2-def} and \eqref{u2}.

{\bf The purpose of this paper is twofold: 
\begin{itemize}
  \item to rigorously derive the perturbed incompressible Navier-Stokes-Fourier system \eqref{ins-s}, and 
\item to prove the stability of \eqref{s-rbe}, \eqref{obd} and \eqref{s-id} around the stationary solution \eqref{st-rbe} satisfying \eqref{st-obd} and \eqref{mass}.
\end{itemize}}

In order to solve \eqref{f-rbe}, as in \cite{DL-2020, DLY-2021}, we make the following ansatz:
\begin{align}\label{fst-dec}
\sqrt{\mu}f_R=f_{R,1}+\sqrt{\mu}f_{R,2},
\end{align}
where $f_{R,1}$ and $f_{R,2}$ satisfy the coupled boundary-value problems:
\begin{align}\label{f1-eq}
\Fv\cdot\na_\Fx f_{R,1}&+\eps^2\Phi\cdot\na_\Fv f_{R,1}+\al\eps z\pa_{x}f_{R,1}-\al\eps v_z\pa_{v_x}f_{R,1}+\frac{1}{\eps} \nu f_{R,1}
\notag\\
=&
\frac{1}{\eps}\chi_M\CK f_{R,1}+\frac{\eps^2}{2}\Phi\cdot \Fv\sqrt{\mu}f_{R,2}-\frac{\al\eps v_xv_z}{2}\sqrt{\mu}\{\FI-\FP\}f_{R,2}
\notag
\\
&
-\eps^{\frac{3}{2}}\Phi\cdot \na_\Fv(\sqrt{\mu}f_{1})
-\eps^{\frac{5}{2}}\Phi\cdot \na_\Fv(\sqrt{\mu}f_{2})
\notag
\\
&
-
\al\eps^{\frac{3}{2}} z\pa_{x}(\sqrt{\mu}f_{2})+\al\eps^{\frac{3}{2}} v_z\pa_{v_x}(\sqrt{\mu}f_{2})
\notag
\\
&
+\eps^{\frac{1}{2}}Q(f_{R,1},f_{R,1})
+\eps^{\frac{1}{2}}Q(f_{R,1},\sqrt{\mu}f_{R,2})+\eps^{\frac{1}{2}}Q(\sqrt{\mu}f_{R,2},f_{R,1})\notag\\
&+\eps^{\frac{3}{2}}Q(\sqrt{\mu}f_2,\sqrt{\mu}f_2)+Q(f_{R,1},\sqrt{\mu}\{f_1+\eps f_2\})
\notag
\\
&
+Q(\sqrt{\mu}\{f_1+\eps f_2\},f_{R,1}), 
\end{align}
with 
\begin{align}\label{f1-bd}
f_{R,1}(x,y,\pm1,\Fv)|_{v_z\lessgtr0}=0,
\end{align}
and
\begin{align}\label{f2-eq}
\Fv\cdot\na_\Fx f_{R,2}&+\eps^2\Phi\cdot\na_\Fv f_{R,2}+\al\eps z\pa_{x}f_{R,2}-\al\eps v_z\pa_{v_x}f_{R,2}
+\frac{1}{\eps} Lf_{R,2}\notag\\
=&\frac{1}{\eps}(1-\chi_M)\mu^{-1/2}\CK f_{R,1}-\frac{\al\eps  v_xv_z}{2}\FP f_{R,2}+\eps^{\frac{1}{2}}\Ga(f_{R,2},f_{R,2})\notag\\
&+\Ga(f_{R,2},f_1+\eps f_2)+\Ga(f_1+\eps f_2,f_{R,2}),
\end{align}
\begin{align}\label{f2-bd}
f_{R,2}(x,y,\pm1,\Fv)|_{v_z\lessgtr0}=&\sqrt{2\pi\mu}\int_{v_z\gtrless0}\{f_{R,1}(\pm1)+f_{R,2}(\pm1)\sqrt{\mu}\}|v_z|d\Fv
\\
&
+\eps^{\frac{1}{2}}r.
\notag 
\end{align}
Here
$$
\nu=\int_{\R^3\times\S^2_+}B_0(\Fv-\Fv_\ast,\om)\mu(\Fv_\ast)d\Fv_\ast d\om\sim(1+|\Fv|)^\ga,
$$
\begin{align}\label{CK-def}
\CK f=Q(f,\mu)+Q_{\rm{gain}}(\mu,f),
\end{align}
and $\chi_{M}(\Fv)$ is a non-negative smooth cutoff function such that
\begin{align}
\chi_{M}(\Fv)=\left\{\begin{array}{rll}
1,&\ |\Fv|\geq M+1,\\[2mm]
0,&\ |\Fv|\leq M,
\end{array}\right.\notag
\end{align}
for some large $M>0$.  
Note that
$$
Lf=\nu f-Kf
$$
with
\begin{align}\label{sp.L}
Kf=\mu^{-\frac{1}{2}}\left\{Q(\mu^{\frac{1}{2}}f,\mu)+Q_{\textrm{gain}}(\mu,\mu^{\frac{1}{2}}f)\right\},
\end{align}
where $Q_{\textrm{gain}}$ denotes the positive part of $Q$ in \eqref{Q-op}. Moreover, under the angular cutoff assumption,  it holds that
\begin{align}
Kf=\int_{\R^3}\Fk(\Fv,\Fv_\ast)f(\Fv_\ast)\,d\Fv_\ast=\int_{\R^3}(\Fk_2-\Fk_1)(\Fv,\Fv_\ast)f(\Fv_\ast)\,d\Fv_\ast,\notag
\end{align}
with
\begin{equation}\notag 
0\leq \Fk_1(\Fv,\Fv_\ast)\leq \tilde{c}_1|\Fv-\Fv_\ast|^\ga e^{-\frac{1}{4}(|\Fv|^2+|\Fv_\ast|^2)},
\end{equation}
and 
\begin{equation}
 0\leq \Fk_2(\Fv,\Fv_\ast)\leq \tilde{c}_2|\Fv-\Fv_\ast|^{-2+\ga}e^{-
\frac{1}{8}|\Fv-\Fv_\ast|^{2}-\frac{1}{8}\frac{\left||\Fv|^{2}-|\Fv_\ast|^{2}\right|^{2}}{|\Fv-\Fv_\ast|^{2}}},\notag
\end{equation}
where both $\tilde{c}_1$ and $\tilde{c}_2$ are positive constants.

\subsubsection{Time-dependent problem}
The non-negativity of the steady solution $F^\eps_{st}$ will be ascertained by the asymptotic stability for the initial boundary value problem \eqref{s-rbe}, \eqref{obd} and
\eqref{s-id}.
To solve this initial boundary value problem in the diffusive limit $\eps\to 0$, we look for a solution that takes the form
\begin{align}
F^\eps=F^\eps_{st}+\eps\sqrt{\mu}\{g_1+\eps g_2+\eps^{\frac{1}{2}}g_R\},\notag
\end{align}
with initial data
\begin{align}
F^\eps_0(\Fx,\Fv)=F^\eps_{st}+\eps\sqrt{\mu}\{g_1(0,\Fx,\Fv)+\eps g_2(0,\Fx,\Fv)+\eps^{\frac{1}{2}}g_{R,0}(\Fx,\Fv)\},\notag
\end{align}
where $g_1$ is given by
\begin{align}\label{g1-it}
g_1=\left\{\rho(t,\Fx)+\Fu(t,\Fx)\cdot \Fv+\frac{|\Fv|^2-3}{2}\ta(t,\Fx)\right\}\sqrt{\mu},
\end{align}
with $[\rho(t,\Fx),\Fu(t,\Fx),\ta(t,\Fx)]$ satisfying
\begin{eqnarray}\label{ns-ust}
\left\{\begin{array}{rll}
&\na_\Fx\cdot \Fu=0,\ \ \rho=-\ta,
\\[2mm]
&\pa_t\Fu+\Fu\cdot\na_\Fx \Fu+\na_\Fx \tilde{P}+\al z\pa_{x}\Fu+\Fu\cdot\na_\Fx \Fu_s+\Fu_s\cdot\na_\Fx \Fu
\\ &
\hspace{40mm}+\al(u_{z},0,0)^T=\eta\Delta_\Fx \Fu,
\\[2mm]
&\pa_t\ta+\al z\pa_{x}\ta+\na_\Fx\ta\cdot \Fu+\na_\Fx\ta\cdot \Fu_s+\na_\Fx\ta_s\cdot \Fu=\frac{2}{5}\ka\Delta_\Fx \ta,
\\[2mm]
&\Fu(x,y,\pm1)=0,\ \ta(x,y,\pm1)=0,\\
&\Fu(0,\Fx)=\Fu_0(\Fx),\ \ta(0,\Fx)=\ta_0(\Fx),
\end{array}\right.
\end{eqnarray}
and $g_2$ is given by
\begin{align}\label{g2-it}
g_2=&L^{-1}\left\{-\Fv\cdot\na_\Fx g_1+\Ga(g_1,g_1)+\Ga(f_1,g_1)+\Ga(g_1,f_1)\right\}
\notag\\&
+\left\{\Fu_2\cdot \Fv+\frac{|\Fv|^2-3}{2}\ta_2\right\}\sqrt{\mu}
\notag\\
=&-\sum\limits_{i,j=1}^3\bar{A}_{ij}\pa_iu_{j}-\sum\limits_{i=1}^3\bar{B}_i\pa_i\ta
\notag\\&
+\{\FI-\FP\}\left\{\frac{(\Fv\cdot\Fu)^2}{2}\sqrt{\mu}\right\}
+\{\FI-\FP\}\left\{\frac{(|\Fv|^2-5)^2\ta^2}{8}\sqrt{\mu}\right\}
\notag\\&
+(|\Fv|^2-5)(\Fv\cdot \Fu)\ta\sqrt{\mu}+\{\FI-\FP\}\left\{(\Fv\cdot\Fu)(\Fv\cdot\Fu_s)\sqrt{\mu}\right\}
\notag\\&
+\{\FI-\FP\}\left\{\frac{(|\Fv|^2-5)^2\ta \ta_s}{4}\sqrt{\mu}\right\}+\left\{\Fu_2\cdot \Fv+\frac{|\Fv|^2-3}{2}\ta_2\right\}\sqrt{\mu}.
\end{align}
Note that 
$$
\FP g_2=\left\{\Fu_2\cdot \Fv+\frac{|\Fv|^2-3}{2}\ta_2\right\}\sqrt{\mu},
$$
where both $g_1(0,\Fx,\Fv)$ and $ g_2(0,\Fx,\Fv)$ are determined by the initial data $[\Fu_0(\Fx),\ta_0(\Fx)]=[\Fu(0,\Fx,\Fv),\ta(0,\Fx,\Fv)]$. For more details, we refer to the discussions in Section \ref{sec-usp}.

Furthermore, the remainder $g_R$ satisfies
\begin{align}
\eps\pa_tg_R&+\Fv\cdot\na_\Fx g_R+\eps^2\Phi\cdot\na_\Fv g_R+\al\eps z\pa_{x}g_R-\al\eps v_z\pa_{v_x}g_R
+\frac{1}{\eps}Lg_R
\notag\\
=&\eps^{\frac{3}{2}}\pa_tg_2-\eps^{\frac{3}{2}}\al z\pa_xg_2+\eps^{\frac{3}{2}}\al\mu^{-\frac{1}{2}}v_z\pa_{v_x}\{\sqrt{\mu}g_2\}
\notag\\&
-\eps^{\frac{3}{2}}\mu^{-\frac{1}{2}}\Phi\cdot\na_\Fv[\sqrt{\mu}(g_1+\eps g_2)]
+\frac{\eps^2}{2}\Phi\cdot\Fv g_R-\frac{\al\eps v_xv_z}{2}g_R+\eps^{\frac{1}{2}}\Ga(g_R,g_R)
\notag\\&
+\Ga(g_R,g_1+\eps g_2)
+\Ga(g_1+\eps g_2,g_R)+\eps^{\frac{3}{2}}\Ga(g_2,g_2)
\notag\\&
+\Ga(g_R,f_1+\eps f_2)+\Ga(f_1+\eps f_2,g_R)
\notag\\&
+\Ga(f_R,g_1+\eps g_2)
+\Ga(g_1+\eps g_2,f_R)+\eps^{\frac{1}{2}}\{\Ga(g_R,f_R)+\Ga(f_R,g_R)\},\notag
\end{align}
with
\begin{align}
\sqrt{\mu}g_R(0,\Fx,\Fv)=\sqrt{\mu}g_{R,0}(\Fx,\Fv),\notag
\end{align}
and
\begin{align}
g_R(t,\pm1,\Fv)|_{v_z\lessgtr0}=P_\ga g_R+\eps^{\frac{1}{2}}\tilde{r},\notag
\end{align}
where $\tilde{r}=-g_2+P_\ga g_2.$

\subsection{Main results}
To state our main results, we introduce some notations for norms right away.  Further notations will be clarified later on. We first define the  Fourier transform with respect to the variable $\bar{x}= (x,y)\in\T^2$ as
\begin{equation*}
\hat{f}(t,\bar{k},z,\Fv)=\CF_{\bar{x}} f(t,\bar{k},z,\Fv)=\int_{\T^2} e^{-  i \bar{k}\cdot \bar{x}}f(t,\bar{x},z,\Fv)\,d\bar{x},\quad \bar{k}=(k_x,k_y)\in \Z^2.
\end{equation*}
Denoting a velocity weight by $w^l(\Fv)=(1+|\Fv|^2)^l$, then for $T>0$ we introduce the weighted norms
\begin{align}
\|f\|_{X_{T}}:=\sum\limits_{\bar{k}\in\Z^2}\sup_{0\leq t\leq T} \|w^l\hat{f}(t,\bar{k},\Fv)\|_{L^\infty_{z,\Fv}}.
\notag
\end{align}
 Moreover, we set $\|f\|_\nu=\|f\|_{L^2_\nu}=\|\nu^{\frac{1}{2}}f\|_{L^2}=\|\nu^{\frac{1}{2}}f\|_2$.  Furthermore, for fixed $\bar{k}\in \Z^2$ and $T>0$, we define the following mixture norms
\begin{align*}
\|w^{l}\hat{f}(t,\bar{k},z,\Fv)\|_{L^\infty_{T,z,\Fv}}&=\sup\limits_{0\leq t\leq T}
\|w^{l}\hat{f}(t,\bar{k},z,\Fv)\|_{L^\infty_{z,\Fv}},\\
\|w^{l}\hat{f}(t,\bar{k},z,\Fv)\|_{L^\infty_{T}L^2_{z,\Fv}}&=\sup\limits_{0\leq t\leq T}
\|w^{l}\hat{f}(t,\bar{k},z,\Fv)\|_{L^2_{z,\Fv}},\\
\|w^{l}\hat{f}(t,\bar{k},z,\Fv)\|^2_{L^2_{T,z,\Fv}}
&=\int_0^T\|w^{l}\hat{f}(t,\bar{k},z,\Fv)\|^2_{L^2_{z,\Fv}}dt,\\
\|w^{l}\hat{f}(t,\bar{k},z,\Fv)\|^2_{L^2_{T,z,\nu}}
&=\int_0^T\|w^{l}\hat{f}(t,\bar{k},z,\Fv)\|^2_{L^2_{z}L^2_{\nu}}dt. 
\end{align*}


We now state our main results in the paper. Our first main theorem is concerned with the solvability of the steady problem \eqref{st-rbe} and \eqref{st-obd}.

\begin{theorem}[Existence of the steady solution]\label{st-sol-th}
Let $\ga\in[0,1]$ and $l_\infty\gg l_2\gg4$.
There exists a constant $\delta_0>0$ such that for any $\al,\eps\in (0,\delta_0)$ and for any $\Phi(z)\in C([-1,1];\R^3)$ with $\|\Phi(z)\|_{2}\leq \delta_0$, 
the steady boundary value problem \eqref{st-rbe} and \eqref{st-obd} admits a unique strong solution $F_{st}^{\eps}(\Fx,\Fv)$ satisfying \eqref{mass} and
\begin{align}\label{fst-mth}
F_{st}^\eps=\mu+\eps\sqrt{\mu}\{f_1+\eps f_2+\eps^{\frac{1}{2}}f_R\}\geq0,
\end{align}
as well as the following things:
\begin{itemize}
\item
For the coefficients $f_1$ and $f_2$, it holds that
\begin{align}\notag
f_1=\left\{\rho_s+\Fv\cdot \Fu_s+\frac{1}{2}(|\Fv|^2-3)\ta_s\right\}\sqrt{\mu},
\end{align}
where $[\rho_s,\Fu_s,\ta_s]$ satisfies the Navier-Stokes equations \eqref{ins-s} which is a perturbation around the fluid Couette flow $[\al z,0,0]$.
And $f_2$ is given by \eqref{f2-def}. In addition, for any integer $m\geq0$, 
it holds
\begin{align}\label{f12-es}
\sum\limits_{\bar{k}\in\Z^2} \|w^{l_\infty}(1+|\bar{k}|^m)[\hat{f}_1,\hat{f}_2](\bar{k},z,\Fv)\|_{\infty}\lesssim \vps_{\Phi,\al}.
\end{align}
Here and in the rest of the paper we denote
\begin{equation}
\vps_{\Phi,\al}:=\|\Phi(z)\|_{2}+\al.\notag
\end{equation}
\item
For the remainder $f_R$, it holds that $\sqrt{\mu}f_R=f_{R,1}+\sqrt{\mu}f_{R,2}$, where
$f_{R,1}$ and $f_{R,2}$ satisfy
\begin{align}\label{fr-es-lem}
\eps^{-\frac{3}{2}}\sum\limits_{\bar{k}\in\Z^2}&\|w^{l_\infty}\hat{f}_{R,1}\|_{\infty}+\eps^{-\frac{5}{2}}\sum\limits_{\bar{k}\in\Z^2}\|w^{l_2}\hat{f}_{R,1}\|_\nu
+\eps^{\frac{1}{2}}\sum\limits_{\bar{k}\in\Z^2}\|w^{l_\infty}\hat{f}_{R,2}\|_{\infty}
\notag\\&
+\sum\limits_{\bar{k}\in\Z^2}\|\FP\hat{f}_{R,2}\|_2
+\eps^{-1}\sum\limits_{\bar{k}\in\Z^2}\|\{\FI-\FP\}\hat{f}_{R,2}\|_\nu
\lesssim \vps_{\Phi,\al}.
\end{align}
\end{itemize}
\end{theorem}

The second result is devoted to the stability of the steady solution established in Theorem \ref{st-sol-th} under initial small perturbations.

\begin{theorem}[Stability of the steady solution]\label{sta-th}
Let $\ga\in[0,1]$, $l_\infty\gg l_2\gg4$, and $\ta_0(\Fx)=0$. Suppose 
\begin{align}\label{id-pos}
F^\eps_0(\Fx,\Fv)=F^\eps_{st}+\eps\sqrt{\mu}\{g_{1}(0,\Fx,\Fv)+\eps g_2(0,\Fx,\Fv)+\eps^{\frac{1}{2}}g_{R,0}(\Fx,\Fv)\}\geq0,
\end{align}
with
$$
\sqrt{\mu }g_{R,0}(\Fx,\Fv)=\eps^{\frac{3}{2}} g^{(1)}_{R,0}(\Fx,\Fv)+\sqrt{\mu}g^{(2)}_{R,0}(\Fx,\Fv),
$$
where $F^\eps_{st}$ is constructed as in \eqref{fst-mth} where for $g_1$ and $g_2$ we use \eqref{g1-it} and \eqref{g2-it}, respectively. There exists a constant $\vps_0>0$ such that if   
\begin{align}\label{gr0-sm}
\sum\limits_{\bar{k}\in\Z^2}\left\|w^{l_\infty}\left[\hat{g}^{(1)}_{R,0},\hat{g}^{(2)}_{R,0}\right]\right\|_{\infty}
+\sum\limits_{m\leq8}\|\pa_i^m\Fu_0(\Fx)\|_{H^4_z}\leq\vps_0,
\end{align}
where $\pa_i=\pa_x$ or $\pa_y$,
then the initial boundary value problem \eqref{s-rbe}, \eqref{obd} and
\eqref{s-id} admits a unique global in time   solution $F^\eps(t,\Fx,\Fv)$
such that
$$
F^\eps=F^\eps_{st}+\eps\sqrt{\mu}\{g_1+\eps g_2+\eps^{\frac{1}{2}}g_R\}\geq0,
$$
satisfying the following properties:
 \begin{itemize}
 \item
There exists $\la_0>0$ such that for any $0\leq t<+\infty$ we uniformly have
 \begin{align}\label{g12-es}
\sum\limits_{\bar{k}\in\Z^2}
\|w^{l_\infty}\yhwidehat{\pa^{m_0}_t\pa^{m_1}_x[g_1,g_2]}(t)\|_{\infty}
&
+\sum\limits_{\bar{k}\in\Z^2}
\|w^{l_\infty}\yhwidehat{\pa^{m_0}_t\pa^{m_1}_x[g_1,g_2](t)}\|_{L^2_{z,\Fv}}
\notag
\\
&
\lesssim \vps_0e^{-\la_0t}.
\end{align}

 \item For the remainder $g_R$, it holds that $\sqrt{\mu}g_R=g_{R,1}+\sqrt{\mu}g_{R,2}$.  Here  $g_{R,1}$ and $g_{R,2}$ satisfy
for any $0\leq t<+\infty$ the uniform estimate
 
\begin{align}\label{rm-de}
\eps^{-\frac{3}{2}}\sum\limits_{\bar{k}\in\Z^2}&\|w^{l_\infty}\hat{g}_{R,1}(t)\|_{\infty}+
\eps^{-\frac{3}{2}}\sum\limits_{\bar{k}\in\Z^2}\|w^{l_2}\hat{g}_{R,1}(t)\|_{2}+
\eps^{\frac{1}{2}}\sum\limits_{\bar{k}\in\Z^2}\|w^{l_\infty}\hat{g}_{R,2}(t)\|_{\infty}
\notag
\\
&
+
\sum\limits_{\bar{k}\in\Z^2}\|\hat{g}_{R,2}(t)\|_{2}
\leq Ce^{-\la_0t}\left\{\sum\limits_{\bar{k}\in\Z^2}\left\|w^{l_\infty}\left[\hat{g}^{(1)}_{R,0},\hat{g}^{(2)}_{R,0}\right]\right\|_{\infty}+\vps_0\right\}.
\end{align}
\end{itemize}
\end{theorem}

\begin{remark} We provide an example to demonstrate
the validity of the non-negativity condition specified in \eqref{id-pos}. For this purpose, we set
\begin{align}
\tilde{g}^{(1)}_{R,0}=\frac{\vps_0}{4}w^{-l_\infty}(\Fv)+\frac{\vps_0}{4}\rho^{(1)}_R\mu,\quad 
\tilde{g}^{(2)}_{R,0}=\vps_0^{\frac{3}{2}}\mu^{q_0}+\vps^{\frac{3}{2}}_0\rho^{(2)}_R\sqrt{\mu},
\label{gr0-ex}
\end{align}
where $q_0\in(0,\frac{1}{2})$ is a fixed constant and $\rho^{(i)}_R$ $(i=1,2)$ are chosen as
\begin{align}
\rho^{(1)}_R=-\int_{\R^3}w^{-l_\infty}(\Fv)d\Fv,\quad  
\rho^{(2)}_R=-\int_{\R^3}\mu^{q_0+\frac{1}{2}}(\Fv)d\Fv.\notag
\end{align}
Note that for sufficiently small value of $\vps_0$, the above $\tilde{g}^{(i)}_{R,0}$ $(i=1,2)$ satisfies the condition \eqref{gr0-sm}.

Moreover, by \eqref{f12-es}, \eqref{fr-es-lem}, \eqref{g12-es} and \eqref{gr0-ex}, it can be established that there exist constants
$C_0>0$  and $q_1\in(0,\frac{1}{2})$, dependent on $l_\infty$, such that
\begin{align}
\eps\sqrt{\mu}&\left\{f_1+\eps f_2+\eps^{\frac{1}{2}}f_R\right\}+\eps\sqrt{\mu}\left\{g_{1}(0,\Fx,\Fv)+\eps g_2(0,\Fx,\Fv)+\eps^{\frac{1}{2}}\vps^{\frac{3}{2}}_0\rho^{(2)}_R\sqrt{\mu}\right\}
\notag
\\
&
+\eps^{3}\frac{\vps_0}{4}\rho^{(1)}_R\mu
\leq C_0\eps^3 \vps_{\Phi,\al}w^{-l_\infty}(\Fv)+C_0\eps\vps_0\mu^{\frac{1}{2}+q_1}.\notag
\end{align}
Here, it is assumed that $0<\vps_{\Phi,\al}\ll \vps_0.$ Note that the first term on the right hand side of the above inequality can be bounded by
the first part of $\tilde{g}^{(1)}_{R,0}$ given by \eqref{gr0-ex}, owing to
$0<\vps_{\Phi,\al}\ll \vps_0.$
For the second term, it can be observed that
$\mu(\Fv)\geq C_0\eps\vps_0\mu^{\frac{1}{2}+q_1}(\Fv)$ for bounded $\Fv$. As for the complement set of such $\Fv$, 
an additional choice of $q_1>0$ is made such that
$$
\frac{1}{2}>q_1\geq \frac{2}{5}q_0+\frac{3}{10}.
$$
Consequently, it follows that
\begin{align}
(\eps\vps_0)^{\frac{3}{2}}\mu^{\frac{1}{2}+q_0}(\Fv)\gtrsim \mu^{\frac{3}{2}(\frac{1}{2}-q_1)+\frac{1}{2}+q_0}(\Fv)\gtrsim \mu^{\frac{1}{2}+q_1}(\Fv).\notag
\end{align}
Therefore, for such a pair $[\tilde{g}^{(1)}_{R,0},\tilde{g}^{(2)}_{R,0}]$, \eqref{id-pos} is valid.
\end{remark}

We should point out that when $\Phi\equiv 0$, Theorem \ref{st-sol-th} and Theorem \ref{sta-th} imply the exponential convergence of the time-dependent 3D kinetic Couette flow toward the steady 1D planar kinetic Couette flow for the Boltzmann equation in the diffusive limit. In fact, when  $\Phi\equiv 0$, by uniqueness the 3D steady Couette flow governed by \eqref{st-rbe} and \eqref{st-obd} can be reduced to equivalently solving the 1D planar Couette flow $F^\epsilon_{st}(z,\Fv)$ depending only on $z$ and $\Fv$ that satisfies
\begin{align*}
v_z\pa_z F_{st}^\eps-\al\eps v_{z}\pa_{v_x}F_{st}^\eps=\frac{1}{\eps}Q(F_{st}^\eps,F_{st}^\eps),
\end{align*}
with the boundary condition
\begin{align*}
F_{st}^\eps(\pm1,\Fv)|_{v_{z}\lessgtr0}=\sqrt{2\pi}\mu\int_{\tv_{z}\gtrless0}F_{st}^\eps(\pm1,\FTv)|\tv_{z}|d\FTv.
\end{align*}
This boundary-value problem for the planar Couette flow has been studied in \cite{DLY-2021} for fixed $\epsilon=1$.  Our work not only extends the result in \cite{DLY-2021} to the regime of the fluid dynamic limit, but we also obtain the convergence of multi-dimensional time-dependent solutions to one-dimensional steady solutions.  We prove convergence that is uniform in small $\epsilon>0$ and for all times $t\geq 0$.  This is analogous to the result that $U(z)=(\alpha\epsilon z, 0,0)$ is asymptotically stable in the context of the 3D incompressible Navier-Stokes-Fourier system with constant temperature.

\subsection{Literature review}
The fluid dynamic limit is an important subject in the Boltzmann theory. The formal asymptotic expansion which reveals the fluid dynamic structure inside the Boltzmann equation is first due to Hilbert \cite{Hil}, and also to Chapman and Enskog \cite{CC-90}; we also refer to   Grad \cite{Gr-63} for the moment method. For a detailed review of the rigorous justifications of fluid dynamic limits of the Boltzmann equation, we refer the readers to Bouchut-Golse-Pulvirenti \cite{BGP} and Saint-Raymond \cite{SR}. Here we mention a few results on classical solutions \cite{BU91, Caf80, DeMEL, Guo-2006}, weak solutions \cite{BGL91,BGL93,GSR04,GSR09, LiMa1, LiMa2}, and solutions with slow velocity decay \cite{BMAM, GeLo}. 

Let's formulate the problem under consideration in a little more general setting. Associated with a parameter $\eps>0$, the Mach and Knudsen numbers are taken respectively as $M\!a=\eps$ and $K\!n=\eps^q$ with $q\geq 1$. In terms of the von Karman relation $K\!n\sim M\!a/Re$ (cf.~\cite{Sone07}), the Reynolds number is given as $Re\sim \eps^{1-q}$.  We set the Boltzmann solution to have the form $F=\mu(1+\eps g)$ around the normalized global Maxwellian $\mu$ as in \eqref{def.gmax}.  Then it is well known that the perturbation $g$ tends, as $\eps\to 0$, to the solution of the fluid system described by the incompressible Navier-Stokes equations if $q>1$ and to the incompressible Euler equations if $q=1$, cf.~\cite{BGP,SR}. Specifically, the limiting solution is of the form $g=\rho+\Fv\cdot \Fu+\frac{1}{2} (|\Fv|^2-3)\theta$ with the divergence-free velocity field $\Fu$, the Boussinesq  relation $\na (\rho+\theta)=0$, and 
\begin{equation*}
\pa_t \Fu+\Fu\cdot \na \Fu +\na p=\mu_q \Delta \Fu,\quad \pa_t\theta+\Fu\cdot \na\theta=\kappa_q \Delta\theta,
\end{equation*}  
where $\na=\na_{\Fx}$, $\Delta=\Delta_\Fx$.  Additionally $\mu_1=\kappa_1=0$ for $q=1$, and $\mu_q>0$, $\kappa_q>0$ for $q>1$ are the viscosity and heat-conductivity coefficients, respectively. It is obvious to see that in both cases $q>1$ and $q=1$, the fluid Couette flow $u_\al^{sh}:=(\al z,0,0)$, namely, a linear shear flow, is a non-trivial space-dependent special solution to the equation of the divergence-free velocity field $\Fu$. Here, $\al>0$ denotes the shear rate. A natural but challenging problem to ask is whether it is possible to construct the Boltzmann kinetic solution around the non-trivial Couette flow $u_\al^{sh}$ in the hydrodynamic regime. The problem was investigated by Esposito-Lebowitz-Marra \cite{ELM-94, ELM-95}  in the one-dimensional stationary case via the Navier-Stokes approximation; see also \cite{DE-96}. 

On the other hand, it is also a very interesting problem to search for a density distribution function at the kinetic level that can characterize the effect of the macroscopic shearing motion. In fact, to do so, one can formally make a change of variables by 
$$
v_x\mapsto v_x-\al\eps z
$$ 
meaning that the molecular velocity of each particle in the $x$-direction is sheared by $\al z$ that is linear in the $z$-direction with rate $\al>0$, so the dependence on $z$ also occurs in the first component of the velocity variables. Then, in such a Lagrangian frame with a shearing velocity the Boltzmann equation can be reformulated as 
\begin{equation}\notag
\eps \pa_t F + (\Fv+\eps u_\al^{sh})\cdot \na_{\Fx} F-\al\eps v_z\pa_{v_x} F=\frac{1}{\eps^q}Q(F,F).
\end{equation}  
The most interesting case is when the distribution function becomes spatially homogeneous as follows
\begin{equation}
\label{equsf}
\eps \pa_t F -\al\eps v_z\pa_{v_x} F=\frac{1}{\eps^q}Q(F,F).
\end{equation}
The above equation is called the Boltzmann equation for uniform shear flow with the shear rate $\al>0$. In kinetic theory, it is also called a simple homoenergetic flow. A basic problem is therefore to study the global existence and large time asymptotic behavior of solutions for uniform shear flow with $\al>0$ and $\eps>0$. It certainly depends on the choice of initial data as well as the inter-molecular interaction type, for instance, whether or not the initial data admit finite energy, and whether or not it is a gas of Maxwell molecules that determines how large the large-velocity magnitude of the nonlinear collision term around equilibrium is in comparison with the shearing term. 

In what follows comment on the literature related to those research topics mentioned above. In particular, we will first focus on the problem \eqref{equsf} for the Maxwell molecule gas, since few mathematical results on other topics have been known thus far.
\begin{itemize}
  \item First of all, regarding the uniform shear flow described by \eqref{equsf} or a homoenergetic affine flow with a general deformation matrix, work dates back to Truesdell \cite{Tru} and Galkin \cite{Gal,Gal-R} who independently studied the existence of solutions for a Maxwell molecule gas by solving the infinite moment system of the Boltzmann equation. Specifically, for uniform shear flow in the Maxwell molecule case, although the collisions of particles conserve energy, the shearing motion induces extra heat so that the energy and hence temperature of the system increases in time. Thus, the system becomes much farther from equilibrium as times goes on. It turns out the large time behavior of solutions is determined by a self-similar profile with a prescribed finite second order moment. Existence of solutions in large time is basically assured by the interplay between the dissipative relaxation, shearing force and collisions. One very interesting point, but also an intrinsic property of the self-similar profile, is that it is a non-Maxwellian distribution having a polynomial tail in large velocity that has been confirmed by Monte Carlo simulations. We refer readers to the monograph Garz\'o-Santos \cite{GaSa} for systematic research on the subject; see also other monographs by Cercignani \cite[Chapter 8.8]{CerBook} and Truesdell-Muncaster \cite[Chapter 26]{TM}.
     
  \item Second, for the homogeneous Boltzmann equation with $\al=0$, there is a large number of papers in the literature. Here, we only mention the series of works by Bobylev-Cercignani \cite{BC02b,BC02a,BC03} for investigations of exact eternal or self-similar solutions. The self-similar profile has infinite energy when $\al=0$. In the angular non-cutoff case, the issue was further systematically studied by Cannone-Karch \cite{CK10,CK13} as well as Morimoto \cite{Mor}, and the large time asymptotics toward the self-similar profile was also established by Morimoto-Yang-Zhao \cite{MYZ}.
  
  \item Third, when $\al\neq 0$, the global solution to \eqref{equsf} was first rigorously constructed by Cercignani \cite{Cer89,Cer00,Cer02}. The properties of solutions were further studied by Bobylev-Caraffini-Spiga \cite{BCS}. The high-velocity tail was derived from the Boltzmann-Fokker-Planck model by Acedo-Santos-Bobylev \cite{ASB}. Great progress has been recently made by James-Nota-Vel\'azquez \cite{JNV-ARMA,JNV-JNS,JNV19}, Matthies-Theil \cite{MT} and Bobylev-Nota-Vel\'azquez \cite{BNV-2019}. Specifically, \cite{JNV-ARMA} studied the existence of self-similar solutions in a weak topology sense, and  \cite{JNV-JNS,JNV19}  derived the explicit long-time asymptotics of solutions in either the collision-dominated or hyperbolic-dominated cases. The large-time asymptotic stability of self-similar profiles was recently established in \cite{BNV-2019} on the basis of the Bobylev's Fourier transform approach \cite{Bo75,Bo88}. The analysis in \cite{MT} confirmed the non-existence of self-similar solutions with the exponential high-velocity tail.
  
\item Last, as we mentioned before, although the hydrodynamic limit of the spatially inhomogeneous Boltzmann equation has been extensively studied in general settings, as far as we know there are very few results are known for the reformulated Boltzmann equation \eqref{rbe} as $\eps\to 0$; readers may refer to \cite{CerCo, Ko,LL,OSA,Ro,Sone07,STO} and reference therein for some discussions or numeric results on this problem.  In particular, on one hand, the stability/instability of kinetic shear flow was formally studied in \cite{GaSa} in terms of the corresponding fluid dynamical equations, but rigorous mathematical analysis has remained very challenging. On the other hand, the stability problem for the incompressible Euler or Navier-Stokes equations for the velocity field $\Fu$ around the Couette flow $u_\al^{sh}$ has been investigated in great depth by Bedrossian together with his collaborators Masmoudi, Germain and Vicol \cite{BM-15,BMV-16,BGM-17} and by many other people; we refer to the survey \cite{BGM-BAMS} and references therein. Therefore, it is an interesting problem to develop new techniques to understand the stability of Couette flows at the kinetic level either in large time or in the hydrodynamic regime.       

\end{itemize}

We point out that if we directly work on \eqref{rbe} instead of \eqref{s-rbe} by taking the approximation of \eqref{def.bcM} around $\mu$ then one could adopt the approach by Wu-Ouyang \cite{WuOu1,WuOu2} to study the existence and asymptotic stability of stationary solutions. However our goal using in the Lagrangian formulation  \eqref{s-rbe} is to capture the macroscopic shear velocity as in \eqref{def.msv} for the 3D Couette flow so that the zero-order approximation
$$
\mu(v_x-\alpha\epsilon z, v_y,v_z)
$$ 
exactly matches the boundary condition \eqref{obd-1}.

\subsection{Strategies of the proof}
In the proof we employ several key ingredients and ideas, outlined as follows.
\begin{itemize}
\item
Initially, inspired by the prior works of the first two authors and their collaborator \cite{DL-2020, DLY-2021}, we have introduced a refined Caflisch decomposition \cite{Caf80} to address the velocity growth induced by the shear force:
\begin{align}
\sqrt{\mu}f_R = f_{R,1} + \sqrt{\mu}f_{R,2}.\notag
\end{align}
However, in contrast to \cite{DL-2020, DLY-2021}, in this work we are required to more carefully select the equations satisfied by the two components, $f_{R,1}$ and $f_{R,2}$. To obtain higher-order $\eps$-estimates for the first component $f_{R,1}$,  the macroscopic and microscopic parts of the second component $f_{R,2}$ are separated into distinct equations, leveraging the higher-order estimates enjoyed by the microscopic part. Furthermore, the assignment of nonlinear terms, boundary conditions, and initial data is carefully crafted to derive the critical estimates:
$$
\sum\limits_{\bar{k}\in\Z^2}\|w^{l_2}\hat{f}_{R,1}(\bar{k})\|_\nu\lesssim \vps_{\Phi,\al}\eps^{\frac{5}{2}},\qquad
\sum\limits_{\bar{k}\in\Z^2}\|\FP\hat{f}_{R,2}(\bar{k})\|_2\lesssim\vps_{\Phi,\al}.
$$
These estimates are critical in the power of $\eps$ because of to the following inequalities:
\begin{align}
Re\left(\eps^{-2}(1-\chi_M)\mu^{-1/2}\CK \hat{f}_{R,1},\hat{f}_{R,2}\right)
\leq \frac{C_\eta}{\eps^4}|\hat{f}_{R,1}|_2^2+\eta|\hat{f}_{R,2}|_2^2,\notag
\end{align}
and
$$
\sum\limits_{\bar{k}\in\Z^2}\|\FP\hat{f}_{R,2}(\bar{k})\|_2\lesssim \eps^{-1}\sum\limits_{\bar{k}\in\Z^2}\|\{\FI-\FP\}\hat{f}_{R,2}(\bar{k})\|_\nu.
$$
The latter inequality remains valid even in the Euler scaling \cite{DL-t}.  This illustrates the robustness of our estimates across different scaling conditions.

\item Building upon the work of the first three authors and their collaborator in \cite{DLSS}, we introduce anisotropic Chemin-Lerner type function spaces.  These spaces are
exemplified by expressions such as:
$$
\sum\limits_{\bar{k}\in\Z^2}\|w^{l_\infty}\hat{f}(\bar{k})\|_{L^\infty_{T,z,\Fv}}=
\sum\limits_{\bar{k}\in\Z^2}\sup\limits_{0\leq t\leq T}\|w^{l_\infty}\hat{f}(t,\bar{k})\|_{L^\infty_{z,\Fv}},
$$
and
$$
\sum\limits_{\bar{k}\in\Z^2}\|w^{l_\infty}\hat{f}(\bar{k})\|_{L^2_{T,z,\Fv}}
=\sum\limits_{\bar{k}\in\Z^2}\left(\int_0^T\|w^{l_\infty}\hat{f}(t,\bar{k})\|^2_{L^2_{z,\Fv}}dt\right)^{\frac{1}{2}}.
$$
The advantages of employing those function spaces are twofold: 

\begin{enumerate}
    \item Dimension reduction: The 3D problem for a finite channel can be transformed into a 1D problem. When converting from $L^1_{\Fv,v_z}$ to $L^2_{\Fv,z}$, the Jacobian takes the form:
\begin{align}
{\bf 1}_{|t-s|>\ka_0 \eps}\left|\frac{v_{z}}{\pa Z(s)}\right|\sim{\bf 1}_{|t-s|>\ka \eps}\left|\frac{\pa(z-(t-s)v_{z})}{\pa v_{z}}\right|^{-1}\lesssim \eps^{-1},\notag
\end{align}
which is an improvement over the previous form \cite{DL-t}:
\begin{align}
\left|{\bf 1}_{|t-s|>\ka_0 \eps}\frac{\pa \Fv}{\pa(\Fx-\Fv(t-s))}\right|\lesssim \eps^{-3},\notag
\end{align}
thus alleviating a second-order singularity in $\eps$.

\item Banach algebraic properties: These function spaces exhibit Banach algebraic characteristics.  This makes it possible to establish the following crucial nonlinear estimate
\begin{align}
\sum\limits_{\bar{k}\in\Z^2}&\left\{\int_0^T\int_{-1}^1|\langle \hat{Q}(\hat{f},\hat{g}),w^{2l_2}  \hat{h}\rangle|dzdt\right\}^{\frac{1}{2}}\notag\\
\lesssim &\sum\limits_{\bar{k}\in\Z^2}\left\{\int_0^T\int_{-1}^1\sum\limits_{\bar{l}\in\Z^2}
\|w^{l_2} \hat{f}(\bar{k}-\bar{l})\|_{2}\|\sqrt{\nu}w^{l_2} \hat{g}(\bar{l})\|_{2}\|\sqrt{\nu}w^{l_2} \hat{h}(\bar{k})\|_{2}dzdt\right\}^{\frac{1}{2}}\notag\\
\leq& \eta\|\sqrt{\nu}w^{l_2} \hat{h}\|_{L^1_{\bar{k}}L^2_{T,z,\Fv}}+
C_\eta\|w^{l_\infty} \hat{f}\|_{L^1_{\bar{k}}L^\infty_{T,z,\Fv}}\|\sqrt{\nu}w^{l_2} \hat{g}\|_{L^1_{\bar{k}}L^2_{T,z,\Fv}}.\notag
\end{align}
\end{enumerate}

\item The linearized equation for the first component $f_{R,1}$ in \eqref{f1-eq} lacks self-adjointness, resulting in the absence of a basic $L^2$ structure. Different from \cite{DL-2022}, where time growth appears, weighted $L^2$ estimates for $f_{R,1}$ are then essential to derive $L^2$ higher-order estimates in $\eps$. The key distinction here is the involvement of diffusive boundary condition.  Unlike the second component, there is no cancellation between the incoming and outgoing boundaries. Specifically, for $h_{R,1}$ from \eqref{hr1}, the term
$
\int_{v_z>0}v_zw^{2l_2}|\hat{h}_{R,1}(1)|^2d\Fv\notag
$
cannot dominate
$
\int_{v_z<0}|v_z|w^{2l_2}|\hat{h}_{R,1}(1)|^2d\Fv.\notag
$
To overcome this difficulty, we artificially attribute the total diffusive reflection at the boundary to the second component. Then we are abel to control the inhomogeneous diffusive boundary for the second component with a a Ukai-type trace inequality, cf.~Lemma \ref{ukai}, in our anisotropic function space.

\item The coefficients $f_1$, $f_2$, $h_1$, and $h_2$ are determined by solving the Navier-Stokes-Fourier system \eqref{ins-s} and \eqref{ns-ust}, respectively. However, obtaining direct anisotropic estimates the space $L^1_{\bar{k}}H^2_z$ and achieving elliptic estimates for $\Fu_s$ and $\Fu$ is extremely challenging due to the mixed boundary conditions.
To overcome this dificulty, we prove estimates in the energy space $H^m_{\bar{x}}H^2_z$ and then we utilize the following interpolation inequality:  
\begin{align}
\sum\limits_{\bar{k}\in\Z^2}\|(1+|\bar{k}|^m)[\hat{\Fu}_s,\hat{\ta}_s]\|_{H^2_z}\lesssim \sum\limits_{\bar{k}\in\Z^2}\left(\lag\bar{k}\rag^{2m+4}\|\hat{\Fu}_s\|^2_{H^2_z}\right)^{\frac{1}{2}}.\notag
\end{align}

\end{itemize}

\subsection{Notations} 
In addition to the norms given previously, we now define several additional notations and norms that will be used in the rest of the paper.
\begin{itemize}
\item
Throughout this paper,  $C$ denotes some generic positive (generally large) constant. An estimate $D\lesssim E$ means that  there is a generic constant $C>0$
such that $D\leq CE$. $D\sim E$
means $D\lesssim E$ and $E\lesssim D$. ${\bf 1}_A$ denotes the characteristic function on the set $A$.

\item
 We denote $\| \cdot \|_2 $ to be the norm of  $L^{2}((-1,1)
\times \R^{3})$ or $L^{2}(-1,1)$ or $L^{2}(\R^3)$, and use $|\cdot|_2$ to stand for the norm of  $L^2(\R^3_\Fv)$.
Sometimes, without causing confusion, we use $\|\cdot \|_{\infty }$ to denote either the $L^{\infty }([-1,1]
\times \R^{3})-$norm or $L^{2}(-1,1)-$norm or $L^{\infty }(\R^3)-$norm.
Furthermore,
$(\cdot,\cdot)$ denotes the $L^{2}$ inner product in
$(-1,1)\times {\R}^{3}\times\Z^2$  and $\langle\cdot\rangle$ denotes the $L^{2}$ inner product in $\R^3_\Fv\times\Z^2$.
In addition, $(\cdot|\cdot)$ stands for the pure complex inner product in the complex number field $\C$.
For the sake of brevity, sometimes, we also use $L^1_{\bar{k}}$ to denote the discrete summation $\sum_{\bar{k}\in\Z^2}$.
Besides, $H^k_z$ represents the standard Sobolev space concerning the variable $z$, while $C([z_1,z_2])$ denotes the space of continuous functions defined on the interval $[z_1,z_2]$.

\item
We denote the outgoing set by 
\begin{equation*}
\ga_+=\{(1,\Fv)|\Fv\in\R^3,v_z>0\}\cup\{(-1,\Fv)|\Fv\in\R^3,v_z<0\},
\end{equation*}%
we denote the incoming set by 
\begin{equation*}
\ga_-=\{(1,\Fv)|\Fv\in\R^3,v_z<0\}\cup\{(-1,\Fv)|\Fv\in\R^3,v_z>0\},
\end{equation*}
and we also denote the grazing set by 
\begin{equation*}
\ga_0=\{(\pm1,\Fv)|\Fv\in\R^3,v_z=0\}.
\end{equation*}
Moreover $|f|_{2,\pm}= |f \mathbf{1}_{\gamma_{\pm}}|_2$
represents the $L^2$ norm of $f(z,\Fv)$ at the boundary $z=\pm1$, e.g. 
\begin{equation*}
|f|_{2,\pm}^2:= \int_{\pm v_z>0} |f(1, \Fv)|^2 |v_z| d\Fv+\int_{\pm v_z<0} |f(-1, \Fv)|^2 |v_z|d\Fv,
\end{equation*}
and we further define 
$$
\|\cdot\|^2_{L^2_TL^2_{\ga_{\pm}}}=\int_0^T|\cdot|^2_{2,\pm}dt.
$$
\item
 Finally, we define
\begin{equation*}
P_{\gamma }f(\pm1,\Fv)=\sqrt{\mu (\Fv)}\int_{n(\pm1)\cdot \FTv>0}f(\pm1,\FTv)\sqrt{\mu (\FTv)}(n(\pm1)\cdot \FTv)d\FTv,
\end{equation*}%
where the unit normal vector is $n(\pm1)=(0,0,\pm1).$
\end{itemize}

\subsection{Organization of the paper} 
The rest of this paper is arranged as follows.
Section \ref{sec-sape} is devoted to establishing the existence of the steady solution to \eqref{f-rbe} with \eqref{st-bd},  thereby proving  Theorem \ref{st-sol-th};  in this section, we provide only the {\it a priori} estimates (\ref{fr-es-lem}) and omit the detailed iteration process for brevity. The stability of the steady solution obtained in Theorem \ref{sta-th} is then demonstrated in Section \ref{sec-usp}; similar to Section \ref{sec-sape}, we refrain from presenting the rigorous derivation and focus on clarifying the {\it a priori} estimate \eqref{rm-de}. Finally, in Appendix \ref{app-sec}, we present some important estimates to complement the main results discussed in the preceding sections.

\section{Steady problem}\label{sec-sape}
In this section, we aim to verify Theorem \ref{st-sol-th}.
Actually, the existence of steady solution of \eqref{f-rbe} and \eqref{st-bd} can be constructed via the iteration method as \cite{DLY-2021}. In this article, we skip this tedious procedure and only derive the uniform estimates on both $f_{R,1}$ and $f_{R,2}$
which satisfy \eqref{f1-eq}, \eqref{f1-bd}, \eqref{f2-eq} and \eqref{f2-bd}.
The proof of Theorem \ref{st-sol-th} is proceeded in the following four subsections.

The first step is concerned with the macroscopic estimates.

\subsection{Macroscopic estimates}
To control the nonlinear collision operator, i.e. $\eps^{\frac{1}{2}}\Ga(f_{R,2},f_{R,2})$, it will be necessary for us to deduce the $L^2$ estimates on the macroscopic part of $f_{R,2}$ denoted as 
$$
\FP f_{R,2}=[a_{s}^{(2)}+\Fv\cdot \Fb_{s}^{(2)}+\frac{1}{2}(|\Fv|^2-3)c_{s}^{(2)}]\sqrt{\mu}.
$$
However, due to the non-conservation of mass in $f_{R,2}$, we need to restore to $ f_{R}$ which satisfies \eqref{f-rbe} and \eqref{st-bd}.
It is worth noting that the steady problem described by equations \eqref{f-rbe} and \eqref{st-bd} does not explicitly guarantee the conservation of mass. To address this, a penalty term like $\lambda f_{R}$ can be added to enforce conservation of mass, as discussed in \cite{EGKM-13} and therein.
We will omit this argument and focus solely on deriving {\it a priori} estimates for the subsequent original steady problem
\begin{multline}\label{p-f-rbe}
\Fv\cdot\na_\Fx f_{R}+\eps^2\Phi\cdot\na_\Fv f_{R}+\al\eps z\pa_{x}f_{R}
-\al\eps v_z\pa_{v_x}f_{R}
+\frac{\al\eps v_xv_z}{2}f_{R}-\frac{\eps^2}{2}\Phi\cdot \Fv f_{R}
+\frac{1}{\eps} Lf_{R}
\\
=
\eps^{\frac{1}{2}}\Ga(f_{R},f_{R})+\Ga(f_{R},f_1+\eps f_2)+\Ga(f_1+\eps f_2,f_{R})+\eps^{\frac{3}{2}}\Ga(f_2,f_2)
\\
-\eps^{\frac{3}{2}}\mu^{-1/2}\Phi\cdot \na_\Fv(\sqrt{\mu}f_{1})-\eps^{\frac{5}{2}}\mu^{-1/2}\Phi\cdot \na_\Fv(\sqrt{\mu}f_{2})
\\
-\al\eps^{\frac{3}{2}}z\pa_xf_2-\al\eps^{\frac{3}{2}} \mu^{-1/2}v_z\pa_{v_x}\{\sqrt{\mu}f_2\},
\end{multline}
and
\begin{align}
f_R(x,y,\pm1,\Fv)|_{v_z\lessgtr0}=\sqrt{2\pi\mu}\int_{\tilde{v}_z\gtrless0}f_R(x,y,\pm1,\tilde{\Fv})\sqrt{\mu}|\tilde{v}_z|d\tilde{\Fv}+\eps^{\frac{1}{2}}r,\notag
\end{align}
with
$$
\int_{\Om\times\R^3}f_R\sqrt{\mu}d\Fv d\Fx=0.
$$
Next, we define the corresponding macroscopic components of the solutions as follows
\begin{align}
\bar{\FP} f_{R,1}=[a_s^{(1)}+\Fv\cdot \Fb_s^{(1)}+\frac{1}{2}(|\Fv|^2-3)c_s^{(1)}]\mu,\notag
\end{align}
and
\begin{align}
\FP f_R=[a_s+\Fv\cdot \Fb_s+\frac{1}{2}(|\Fv|^2-3)c_s]\sqrt{\mu},\notag
\end{align}
with the given definitions and using
\eqref{fst-dec},
we see that
\begin{align}
a_s=a_s^{(1)}+a_s^{(2)},\ \Fb=\Fb_s^{(1)}+\Fb_s^{(2)},\ c=c_s^{(1)}+c_s^{(2)}.\notag
\end{align}
The point is therefore to derive the $L_{\bar{k}}^1L_z^2$ estimates on $[a_s,\Fb_s,c_s]$.
For results in this direction, we have
\begin{lemma}\label{abc-lem}
It holds that
\begin{align}
&\sum\limits_{\bar{k}\in\Z^2}\|[\hat{a}_s,\hat{\Fb}_s,\hat{c}_s\|_2
\notag\\
\lesssim& 
\sqrt{\eps}\sum\limits_{\bar{k}\in\Z^2}\|w^{l_\infty}\hat{f}_{R,2}\|_\infty
\sum\limits_{\bar{k}\in\Z^2}\|\hat{f}_{R,2}\|_2
+\sum\limits_{\bar{k}\in\Z^2}\|w^{l_\infty}\hat{f}_{R,1}\|_\infty
\sum\limits_{\bar{k}\in\Z^2}\|\hat{f}_{R,2}\|_2
\notag\\&
+\sum\limits_{\bar{k}\in\Z^2}\|w^{l_\infty}\hat{f}_{R,1}\|_\infty
+\sum\limits_{\bar{k}\in\Z^2}|\{\FI-P_\ga\}\hat{f}_{R,2}(\pm1)|_{2,+}
+\left(\sum\limits_{\bar{k}\in\Z^2}\|w^{l_\infty}\hat{f}_{R,1}\|_\infty\right)^2
\notag\\&
+\eps^{-1}\sum\limits_{\bar{k}\in\Z^2}
\|w^{l_2}\hat{f}_{R,1}\|_2+
\eps^{-1}\sum\limits_{\bar{k}\in\Z^2}\|\{\FI-\FP\}\hat{f}_{R,2}\|_2
+\vps_{\Phi,\al}\sum\limits_{\bar{k}\in\Z^2}\|\FP\hat{f}_{R,2}\|_2\notag\\
&+\vps_{\Phi,\al}\eps^{\frac{1}{2}}.\notag
\end{align}
\end{lemma}
The next corollary is a direct consequence of Lemma \ref{abc-lem}.
\begin{corollary}It holds that
\begin{align}\label{abc-2-es}
&\sum\limits_{\bar{k}\in\Z^2}\|[\hat{a}^{(2)}_s,\hat{\Fb}^{(2)}_s,\hat{c}^{(2)}_s\|_2
\notag\\
\lesssim& \sqrt{\eps}\sum\limits_{\bar{k}\in\Z^2}\|w^{l_\infty}\hat{f}_{R,2}\|_\infty
\sum\limits_{\bar{k}\in\Z^2}\|\hat{f}_{R,2}\|_2+\sum\limits_{\bar{k}\in\Z^2}\|w^{l_\infty}\hat{f}_{R,1}\|_\infty
\sum\limits_{\bar{k}\in\Z^2}\|\hat{f}_{R,2}\|_2
\notag\\
&+\sum\limits_{\bar{k}\in\Z^2}\|w^{l_\infty}\hat{f}_{R,1}\|_\infty
+\sum\limits_{\bar{k}\in\Z^2}|\{\FI-P_\ga\}\hat{f}_{R,2}(\pm1)|_{2,+}
+\left(\sum\limits_{\bar{k}\in\Z^2}\|w^{l_\infty}\hat{f}_{R,1}\|_\infty\right)^2
\notag\\&+\eps^{-1}\sum\limits_{\bar{k}\in\Z^2}
\|w^{l_2}\hat{f}_{R,1}\|_2+
\eps^{-1}\sum\limits_{\bar{k}\in\Z^2}\|\{\FI-\FP\}\hat{f}_{R,2}\|_2+\vps_{\Phi,\al}\eps^{\frac{1}{2}}.
\end{align}
\end{corollary}

To achieve this, let's begin by taking the Fourier transform of equation \eqref{p-f-rbe} with respect to $\bar{x}$ to obtain
\begin{align}\label{fst-ft}
i\bar{k}\cdot\bar{v} \hat{f}_R&+v_z\pa_z\hat{f}_R+\eps^2\Phi\cdot\na_\Fv \hat{f}_R+\al\eps ik_xz\hat{f}_R-\al\eps v_z\pa_{v_x}\hat{f}_R
+\frac{\al\eps v_xv_z}{2}\hat{f}_R
\notag\\&
-\frac{\eps^2}{2}\Phi\cdot \Fv\hat{f}_R
+\frac{1}{\eps} L\hat{f}_R
+\eps^{\frac{3}{2}}\mu^{-1/2}\Phi\cdot \na_\Fv(\sqrt{\mu}\hat{f}_{1})
\notag\\&
+\eps^{\frac{5}{2}}\mu^{-1/2}\Phi\cdot \na_\Fv(\sqrt{\mu}\hat{f}_{2})
+i\al\eps^{\frac{3}{2}}zk_x\hat{f}_2-\al\eps^{\frac{3}{2}} \mu^{-1/2}v_z\pa_{v_x}\{\sqrt{\mu}\hat{f}_2\}\notag\\
=&
\eps^{\frac{1}{2}}\hat{\Ga}(\hat{f}_{R},\hat{f}_{R})+\{\hat{\Ga}(\hat{f}_{R},\hat{f}_1+\eps \hat{f}_2)+\hat{\Ga}(\hat{f}_1+\eps \hat{f}_2,\hat{f}_{R})\}+\eps^{\frac{3}{2}}\hat{\Ga}(\hat{f}_2,\hat{f}_2),
\end{align}
where
\begin{align}
\hat{\Ga}(\hat{f},\hat{g})=\CF_{\bar{x}}\{\Ga(f,g)\}.\label{Ga-ft}
\end{align}
Then letting $\hat{\Psi}(z,\bar{k},\Fv)\in C^1((-1,1)\times\Z^2\times\R^3)$ with $\bar{k}=(k_x,k_y)\in \Z^2$
and taking the inner product of \eqref{fst-ft} with $\hat{\Psi}$ over $(-1,1)\times\R^3$, we get
\begin{align}\label{fst-ip}
(\hat{\Psi}(1),& v_z \hat{f}_R(1))-(\hat{\Psi}(-1), v_z \hat{f}_R(-1))-(\Fv\cdot\widehat{\na_\Fx\Psi}, \hat{f}_R)-\eps^2(\Phi\cdot\na_\Fv \hat{\Psi},\hat{f}_R)
\notag\\&
-\al\eps (z\hat{f}_R,\widehat{\pa_{x}\Psi})
+\al\eps (v_z\hat{f}_R,\pa_{v_x}\hat{\Psi})
+\frac{\al\eps }{2}(v_xv_z\hat{f}_R,\hat{\Psi})
-\frac{\eps^2}{2}(\Phi\cdot \Fv\hat{f}_R,\hat{\Psi})
\notag\\&
+\frac{1}{\eps} (L\hat{f}_R,\hat{\Psi})
+(\eps^{\frac{3}{2}}\mu^{-1/2}\Phi\cdot \na_\Fv(\sqrt{\mu}\hat{f}_{1}),\hat{\Psi})
+(\eps^{\frac{5}{2}}\mu^{-1/2}\Phi\cdot \na_\Fv(\sqrt{\mu}\hat{f}_{2})
\notag\\&
+i\al\eps^{\frac{3}{2}}zk_x\hat{f}_2
-\al\eps^{\frac{3}{2}} \mu^{-1/2}v_z\pa_{v_x}
\{\sqrt{\mu}\hat{f}_2\},\hat{\Psi})=(\hat{H},\hat{\Psi}),
\end{align}
where $\hat{H}$ denotes the right hand side of \eqref{fst-ft}.

From \eqref{fst-ip} and by $f_R=\FP f_R+\{\FI-\FP\}f_R$, it follows
\begin{align}
-(\Fv\cdot\widehat{\na_\Fx\Psi}, \FP \hat{f}_{R})=\sum\limits_{i=1}^7S_i,\label{abc-test}
\end{align}
where
\begin{align}
S_1=-\lag\hat{\Psi}(1), v_z \hat{f}_{R}(1)\rag+\lag\hat{\Psi}(-1), v_z \hat{f}_{R}(-1)\rag,\label{s1}
\end{align}
\begin{align}
S_2=(\Fv\cdot\widehat{\na_\Fx\Psi}, \{\FI-\FP\}\hat{f}_{R}),\label{s2}
\end{align}
\begin{align}
S_3=\eps^2(\Phi\cdot\na_\Fv \hat{\Psi},\hat{f}_{R})+\frac{\eps^2}{2}(\Phi\cdot \Fv\hat{f}_{R},\hat{\Psi}),\label{s3}
\end{align}
\begin{align}
S_4=
\al\eps (z\hat{f}_{R},\pa_{x}\hat{\Psi})-\al\eps (v_z\hat{f}_{R},\pa_{v_x}\hat{\Psi})-\frac{\al\eps }{2}(v_xv_z\hat{f}_{R},\hat{\Psi}),\label{s4}
\end{align}
\begin{align}
S_5=
-(S_{5,1},\hat{\Psi}),\label{s5}
\end{align}
with
\begin{align}\notag
S_{5,1}=&
~\eps^{\frac{3}{2}}\mu^{-1/2}\Phi\cdot \na_\Fv(\sqrt{\mu}\hat{f}_{1})-\eps^{\frac{5}{2}}\mu^{-1/2}\Phi\cdot \na_\Fv(\sqrt{\mu}\hat{f}_{2})-i\al\eps^{\frac{3}{2}}zk_x\hat{f}_2
\\ \notag &
+\al\eps^{\frac{3}{2}} \mu^{-1/2}v_z\pa_{v_x}\{\sqrt{\mu}\hat{f}_2\},
\end{align}
\begin{align}
S_6=-\frac{1}{\eps} (L\hat{f}_{R},\hat{\Psi}),\ \
S_7=(\hat{H},\hat{\Psi}).\label{s6}
\end{align}
Now, using \eqref{abc-test}, we can proceed to estimate $[a_s,\Fb_s,c_s]$ separately.

\medskip
\noindent\underline{{\it Estimates on $a_s$.}} Set
\begin{align}\label{test-a}
\hat{\Psi}=\hat{\Psi}_{a_s}=\sqrt{\mu}(|\Fv|^2-10)\Fv\cdot \widehat{\na_\Fx\phi_{a_s}},
\end{align}
where $\phi_{a_s}$ satisfies
\begin{align}\label{ep-a}
\left\{\begin{array}{rll}
&(|\bar{k}|^2-\pa_z^2) \hat{\phi}_{a_s}=\hat{a}_s,\\[2mm]
&\pa_{z}\hat{\phi}_{a_s}(\bar{k},\pm 1)=0,\ \int_{-1}^1\hat{\phi}_{a_s}(0,z)dz=0.
\end{array}\right.
\end{align}
From \eqref{ep-a}, one has
\begin{align}\label{epes-ap}
\||\bar{k}|^2\hat{\phi}_{a_s}(\bar{k},z)\|_{2}+\||\bar{k}|\hat{\phi}_{a_s}(\bar{k},z)\|_{H^{1}_z}
+\|\hat{\phi}_{a_s}\|_{H^2_z}\lesssim \|\hat{a}_s\|_{2},\ \forall\bar{k}\in\Z^2,
\end{align}
in particular, we have
\begin{align}\label{epes-a2}
(1+|\bar{k}|)\left\{\||\bar{k}|\hat{\phi}_{a_s}(\bar{k},z)\|_{2}+\|\hat{\phi}_{a_s}(\bar{k},z)\|_{H^1_z}\right\}
\lesssim \|\hat{a}_s\|_{2},\ \forall\bar{k}\in\Z^2.
\end{align}
Furthermore one also has the trace inequalities
\begin{align}\label{trace}
(1+|\bar{k}|)\|\hat{\phi}_{a_s}(\bar{k},\pm1)\|_2\lesssim\|\hat{a}\|_{2},\ \forall\bar{k}\in\Z^2.
\end{align}
The proof of \eqref{epes-ap} and \eqref{trace} will be given in Appendix \ref{app-sec}.

According to the choice specified in \eqref{test-a}, it can be seen that the left hand side of \eqref{abc-test} is equal to $\|\hat{a}_s\|_2^2$. We now turn to estimate $S_i$ $(1\leq i\leq 7)$ individually.
For $S_1: \eqref{s1}$, note that
\begin{align}\label{fst-bdp}
\hat{f}_R|_\ga=P_\ga \hat{f}_R+\{\FI-P_\ga\}\hat{f}_R{\bf 1}_{\ga_+}+\eps^{\frac{1}{2}}\hat{r}{\bf 1}_{\ga_-}.
\end{align}
We thus have by using \eqref{fst-dec}, \eqref{ep-a}, \eqref{trace} and  \eqref{fst-bdp} that
\begin{align} \notag 
|S_1|\leq&\left|( (|\Fv|^2-10)\sqrt{\mu}\bar{v} \cdot \bar{k}\hat{\phi}_{a_s}(1), v_z \hat{f}_{R}(1))\right|
\\ & +\left|( (|\Fv|^2-10)\sqrt{\mu}\bar{v}\cdot \bar{k}\hat{\phi}_{a_s}(-1), v_z \hat{f}_{R}(-1))\right|\notag\\
\leq&\left|( (|\Fv|^2-10)\sqrt{\mu}\bar{v} \cdot \bar{k}\hat{\phi}_{a_s}(1), v_z( \{\FI-P_\ga\}\hat{f}_{R}(1){\bf 1}_{\ga_+}+\eps^{\frac{1}{2}}\hat{r}{\bf 1}_{\ga_-})\right|
\notag\\&+\left|( (|\Fv|^2-10)\sqrt{\mu}\bar{v} \cdot \bar{k}\hat{\phi}_{a_s}(-1), v_z (\{\FI-P_\ga\}\hat{f}_{R}(-1){\bf 1}_{\ga_+}+\eps^{\frac{1}{2}}\hat{r}{\bf 1}_{\ga_-})\right|\notag\\
\lesssim& \|\hat{a}_s\|_2\{|w^{l_2}\sqrt{\mu}\{\FI-P_\ga\}\hat{f}_{R}(\pm1)|_{2,+}+\eps^{\frac{1}{2}}|\hat{r}|_{2,-}\}
\notag\\
\lesssim& \eta\|\hat{a}_s\|_2^{2}
+C_\eta|w^{l_2}\hat{f}_{R,1}(\pm1)|^2_{2,+}+C_\eta|\{\FI-P_\ga\}\hat{f}_{R,2}(\pm1)|^2_{2,+}+\eps|\hat{r}|_{2,-}^2\notag\\
\lesssim& \eta\|\hat{a}_s\|_2^2
+C_\eta\|w^{l_\infty}\hat{f}_{R,1}\|^2_{\infty}+C_\eta|\{\FI-P_\ga\}\hat{f}_{R,2}(\pm1)|^2_{2,+}+\eps|\hat{r}|_{2,-}^2.\notag
\end{align}

As to $S_2: \eqref{s2}$, Cauchy-Schwarz's inequality together with \eqref{epes-ap} gives
\begin{align}
|S_2|\leq& \eta\|\hat{a}_s\|_2^2+C_\eta\|w^{l_\infty}\hat{f}_{R,1}\|_\infty^{2}
+C_\eta\|\{\FI-\FP\}\hat{f}_{R,2}\|_2^2.\notag
\end{align}
For $S_3: \eqref{s3}$, one has from \eqref{epes-ap} that
\begin{align}
|S_3|\leq& C\eps^2\|\Phi\|_2\|\widehat{\na_\Fx\phi_{a_s}}\|_2\{\|w^{l_\infty}\hat{f}_{R,1}\|_\infty+\|\hat{f}_{R,2}\|_{\infty}\}\notag\\
\leq& \vps_{\Phi,\al}\|\hat{a}_s\|_2^2+
C\eps^{4}\|w^{l_\infty}\hat{f}_{R,1}\|_\infty^2
+C\eps^{4}\|\hat{f}_{R,2}\|_{\infty}^2.\notag
\end{align}
For $S_4: \eqref{s4}$, in view of H\"older's inequality and \eqref{epes-ap}, it follows
\begin{align}
|S_4|\leq& C\al\eps\|\widehat{\na_\Fx\phi_{a_s}}\|_{H_z^{1}}\{\|w^{l_2}\hat{f}_{R,1}\|_2+\|\hat{f}_{R,2}\|_{2}\}\notag\\
\leq& \eta\|\hat{a}_s\|_2^2+
C_\eta(\al\eps)^{2}\|w^{l_\infty}\hat{f}_{R,1}\|_\infty^2
+C_\eta(\al\eps)^{2}\|\hat{f}_{R,2}\|_{2}^2.\notag
\end{align}
For $S_5: \eqref{s5}$, applying Cauchy-Schwarz's inequality, we get
\begin{align}
|S_5|\leq& C\vps_{\Phi,\al}\eps^{\frac{3}{2}}\|\widehat{\na_\Fx\phi_{a_s}}\|_{H_z^{1}}\{\|\hat{f}_{1}\|_\infty+\|\bar{k}\hat{f}_{2}\|_{\infty}\}\notag\\
\leq& \eta\|\hat{a}_s\|_2^2+
C_\eta\vps_{\Phi,\al}^2\eps^{3}\{\|\hat{f}_{1}\|_\infty+\|\bar{k}\hat{f}_{2}\|_{\infty}\}^2.\notag
\end{align}
For $S_6: \eqref{s6}$, using Cauchy-Schwarz's inequality again, we get from Lemma \ref{es-tri} and \eqref{epes-ap} that for $l_2\gg2$
\begin{align}
|S_6|\leq& \frac{1}{\eps}\left(\hat{Q}(\mu,\sqrt{\mu}\{\FI-\FP\}\hat{f}_{R})+\hat{Q}(\sqrt{\mu}\{\FI-\FP\}\hat{f}_{R},\mu),(|\Fv|^2-10)\Fv\cdot \widehat{\na_\Fx\phi_{a_s}}\right)
\notag\\
\lesssim& \frac{1}{\eps}\int_{-1}^1\left|w^{l_2}\sqrt{\mu}\{\FI-\FP\}\hat{f}_{R}\right|_{2} |\widehat{\na_\Fx\phi_{a_s}}|dz
\notag\\
\lesssim&
\frac{1}{\eps}\left\{\|\{\FI-\FP\}\hat{f}_{R,2}\|_2+\|w^{l_2}\hat{f}_{R,1}\|_2\right\}
\|\widehat{\na_\Fx\phi_{a_s}}\|_2
\notag\\
\leq& \eta\|\hat{a}_s\|_2^2+
C_\eta\eps^{-2}
\|w^{l_2}\hat{f}_{R,1}\|_2^2
+C_\eta\eps^{-2}\|\{\FI-\FP\}\hat{f}_{R,2}\|_2^2,\notag
\end{align}
where we have used the following notation
\begin{align}\label{Q-ft-def}
\hat{Q}(\hat{f},\hat{g})=\CF_{\bar{x}}\{Q(f,g)\}.
\end{align}
Finally, for $S_7: \eqref{s6}$, utilizing Lemma \ref{es-tri} and \eqref{epes-ap}, we have
\begin{align}
|S_7|
\leq& \eps^{\frac{1}{2}}\Big|\Big(\hat{Q}(\hat{f}_{R,1}+\sqrt{\mu}\hat{f}_{R,2},\hat{f}_{R,1}+\sqrt{\mu}\hat{f}_{R,2}),(|\Fv|^2-10)\Fv\cdot \widehat{\na_\Fx\phi_{a_s}}\Big)\Big|\notag\\
&+\left|(\hat{Q}(\sqrt{\mu}(\hat{f}_{1}+\eps\hat{f}_{2}),\sqrt{\mu}\hat{f}_{R})\right.\notag\\
&\qquad \qquad \qquad \qquad \left.
+\hat{Q}(\sqrt{\mu}\hat{f}_{R},\sqrt{\mu}(\hat{f}_{1}+\eps\hat{f}_{2})),(|\Fv|^2-10)\Fv\cdot \widehat{\na_\Fx\phi_{a_s}})\right|\notag\\
&+\eps^{\frac{3}{2}}\left|(\hat{Q}(\sqrt{\mu}\hat{f}_{2},\sqrt{\mu}\hat{f}_{2}),(|\Fv|^2-10)\Fv\cdot \widehat{\na_\Fx\phi_{a_s}})\right|\notag\\
\lesssim &\eps^{\frac{1}{2}}\int_{-1}^1\sum\limits_{\bar{l}}|w^{l_2}\hat{f}_{R,1}(\bar{k}-\bar{l})|_2|w^{l_2}\sqrt{\mu}\hat{f}_{R,2}(\bar{l})|_2
|\widehat{\na_\Fx\phi_{a_s}}(\bar{k})|dz
\notag\\&+\eps^{\frac{1}{2}}\int_{-1}^1\sum\limits_{\bar{l}}|w^{l_2}\hat{f}_{R,1}(\bar{k}-\bar{l})|_2|w^{l_2}\hat{f}_{R,1}(\bar{l})|_2
|\widehat{\na_\Fx\phi_{a_s}}(\bar{k})|dz\notag\\
&+\eps^{\frac{1}{2}}\int_{-1}^1\sum\limits_{\bar{l}}|w^{l_2}\sqrt{\mu}\hat{f}_{R,2}(\bar{k}-\bar{l})|_2|w^{l_2}\sqrt{\mu}\hat{f}_{R,2}(\bar{l})|_2
|\widehat{\na_\Fx\phi_{a_s}}(\bar{k})|dz\notag\\
&+\int_{-1}^1\sum\limits_{\bar{l}}\left\{|w^{l_2}\sqrt{\mu}\hat{f}_{1}(\bar{k}-\bar{l})|_2
+|w^{l_2}\sqrt{\mu}\hat{f}_{2}(\bar{k}-\bar{l})|_2\right\}\notag\\
&\qquad\qquad\times\left\{|w^{l_2}\hat{f}_{R,1}(\bar{l})|_2+
|w^{l_2}\sqrt{\mu}\hat{f}_{R,2}(\bar{l})|_2\right\}|\widehat{\na_\Fx\phi_{a_s}}(\bar{k})|dz\notag\\
&+\eps^{\frac{3}{2}}\int_{-1}^1\sum\limits_{\bar{l}}|w^{l_2}\sqrt{\mu}\hat{f}_{2}(\bar{k}-\bar{l})|_2
|w^{l_2}\sqrt{\mu}\hat{f}_{2}(\bar{l})|_2|\widehat{\na_\Fx\phi_{a_s}}(\bar{k})|dz.
\notag
\end{align}
Then we get
\begin{align}
|S_7|\lesssim&\eps^{\frac{1}{2}}\|\widehat{\na_\Fx\phi_{a_s}}(\bar{k},z)\|_2
\left\|\sum\limits_{\bar{l}}|w^{l_2}\hat{f}_{R,1}(\bar{k}-\bar{l})|_2|w^{l_2}\sqrt{\mu}\hat{f}_{R,2}(\bar{l})|_2\right\|_2
\notag\\&+\eps^{\frac{1}{2}}\|\widehat{\na_\Fx\phi_{a_s}}(\bar{k},z)\|_2
\left\|\sum\limits_{\bar{l}}|w^{l_2}\hat{f}_{R,1}(\bar{k}-\bar{l})|_2|w^{l_2}\hat{f}_{R,1}(\bar{l})|_2\right\|_2
\notag\\&+\eps^{\frac{1}{2}}\|\widehat{\na_\Fx\phi_{a_s}}(\bar{k},z)\|_2
\left\|\sum\limits_{\bar{l}}|w^{l_2}\sqrt{\mu}\hat{f}_{R,2}(\bar{k}-\bar{l})|_2|w^{l_2}\sqrt{\mu}\hat{f}_{R,2}(\bar{l})|_2\right\|_2
\notag\\
&+\|\widehat{\na_\Fx\phi_{a_s}}(\bar{k},z)\|_2\Bigg\|\sum\limits_{\bar{l}}\left\{|w^{l_2}\sqrt{\mu}\hat{f}_{1}(\bar{k}-\bar{l})|_2
+|w^{l_2}\sqrt{\mu}\hat{f}_{2}(\bar{k}-\bar{l})|_2\right\}\notag\\& \qquad\qquad\qquad\qquad\times\left\{|w^{l_2}\hat{f}_{R,1}(\bar{l})|_2+
|w^{l_2}\sqrt{\mu}\hat{f}_{R,2}(\bar{l})|_2\right\}\Bigg\|_2\notag\\&
+\eps^{\frac{3}{2}}\|\widehat{\na_\Fx\phi_{a_s}}(\bar{k},z)\|_2
\left\|\sum\limits_{\bar{l}}|w^{l_2}\sqrt{\mu}\hat{f}_{2}(\bar{k}-\bar{l})|_2
|w^{l_2}\sqrt{\mu}\hat{f}_{2}(\bar{l})|_2\right\|_2.\label{s7-as}
\end{align}
Furthermore, using generalized Minkowski's inequality
\begin{align}\label{gm-ine}
\left\|\sum\limits_{\bar{k}}|\hat{f}(\bar{k},z)|\right\|_{2}
\leq \sum\limits_{\bar{k}}\left\|\hat{f}(\bar{k},z)\right\|_{2},
\end{align}
one has
\begin{align}
|S_7|\leq&\eta\|\hat{a}_s\|^2_2
+C_\eta\eps\left[\sum\limits_{\bar{l}}
\left(\int_{-1}^1\left(|w^{l_2}\hat{f}_{R,1}(\bar{k}-\bar{l})|_2|w^{l_2}\sqrt{\mu}\hat{f}_{R,2}(\bar{l})|_2\right)^2dz\right)^{\frac{1}{2}}\right]^2
\notag\\&+C_\eta\eps\left[\sum\limits_{\bar{l}}
\left(\int_{-1}^1\left(|w^{l_2}\hat{f}_{R,1}(\bar{k}-\bar{l})|_2|w^{l_2}\hat{f}_{R,1}(\bar{l})|_2\right)^2dz\right)^{\frac{1}{2}}\right]^2
\notag\\&+C_\eta\eps\left[\sum\limits_{\bar{l}}
\left(\int_{-1}^1\left(|w^{l_2}\sqrt{\mu}\hat{f}_{R,2}(\bar{k}-\bar{l})|_2|w^{l_2}\sqrt{\mu}\hat{f}_{R,2}(\bar{l})|_2\right)^2dz\right)
^{\frac{1}{2}}\right]^2\notag\\
&+C_\eta\Bigg[\sum\limits_{\bar{l}}
\Big(\int_{-1}^1\Big(\left\{|w^{l_2}\sqrt{\mu}\hat{f}_{1}(\bar{k}-\bar{l})|_2
+|w^{l_2}\sqrt{\mu}\hat{f}_{2}(\bar{k}-\bar{l})|_2\right\}\notag\\&
\qquad\qquad\times\left\{|w^{l_2}\hat{f}_{R,1}(\bar{l})|_2+
|w^{l_2}\sqrt{\mu}\hat{f}_{R,2}(\bar{l})|_2\right\}\Big)^2dz\Big)
^{\frac{1}{2}}\Bigg]^2\notag\\&
+C\eps^{\frac{3}{2}}\left[\sum\limits_{\bar{l}}
\left(\int_{-1}^1\left(|w^{l_2}\sqrt{\mu}\hat{f}_{2}(\bar{k}-\bar{l})|_2
|w^{l_2}\sqrt{\mu}\hat{f}_{2}(\bar{l})|_2\right)^2dz\right)
^{\frac{1}{2}}\right]^2\notag\\
\leq&
\eta\|\hat{a}_s\|^2_2
+C_\eta\eps\left[\sum\limits_{\bar{l}}
\|w^{l_\infty}\hat{f}_{R,1}(\bar{k}-\bar{l})\|_\infty\|\hat{f}_{R,2}(\bar{l})\|_2\right]^2
\notag\\&
+C_\eta\eps\left[\sum\limits_{\bar{l}}
\|w^{l_\infty}\hat{f}_{R,1}(\bar{k}-\bar{l})\|_\infty\|w^{l_\infty}\hat{f}_{R,1}(\bar{l})\|_\infty\right]^2
\notag\\&
+C_\eta\eps\left[\sum\limits_{\bar{l}}
\|w^{l_2}\hat{f}_{R,2}(\bar{k}-\bar{l})\|_\infty\|\hat{f}_{R,2}(\bar{l})\|_2\right]^2
\notag\\&+C\eps^{\frac{3}{2}}\left[\sum\limits_{\bar{l}}
\|\hat{f}_{2}(\bar{k}-\bar{l})\|_\infty\|\hat{f}_{2}(\bar{l})\|_2\right]^2
\notag\\&
+C_\eta\Bigg[\sum\limits_{\bar{l}}\left\{\|\hat{f}_{1}(\bar{k}-\bar{l})\|_\infty+
\|\hat{f}_{2}(\bar{k}-\bar{l})\|_\infty\right\}\notag\\&\qquad\qquad\times
\left\{\|w^{l_\infty}\hat{f}_{R,1}(\bar{l})\|_\infty^2+\|\hat{f}_{R,2}(\bar{l})\|_2\right\}\Bigg]^2.
\notag
\end{align}
Now putting the above estimates together and utilizing Lemma \ref{coef-es}, we conclude
\begin{align}\label{ak-es}
\sum\limits_{\bar{k}\in\Z^2}\|\hat{a}_s\|_2\lesssim& \sqrt{\eps}\sum\limits_{\bar{k}\in\Z^2}\|w^{l_\infty}\hat{f}_{R,2}\|_\infty
\sum\limits_{\bar{k}\in\Z^2}\|\hat{f}_{R,2}\|_2\notag\\
&+\sum\limits_{\bar{k}\in\Z^2}\|w^{l_\infty}\hat{f}_{R,1}\|_\infty
\sum\limits_{\bar{k}\in\Z^2}\|\hat{f}_{R,2}\|_2
+\sum\limits_{\bar{k}\in\Z^2}\|w^{l_\infty}\hat{f}_{R,1}\|_\infty\notag\\
&
+\sum\limits_{\bar{k}\in\Z^2}|\{\FI-P_\ga\}\hat{f}_{R,2}(\pm1)|_{2,+}
+\left(\sum\limits_{\bar{k}\in\Z^2}\|w^{l_\infty}\hat{f}_{R,1}\|_\infty\right)^2
\notag\\&+\eps^{-1}\sum\limits_{\bar{k}\in\Z^2}
\|w^{l_2}\hat{f}_{R,1}\|_2
+\eps^{-1}\sum\limits_{\bar{k}\in\Z^2}\|\{\FI-\FP\}\hat{f}_{R,2}\|_2
\notag\\&+\vps_{\Phi,\al}\|\FP\hat{f}_{R,2}\|_2+\vps_{\Phi,\al}\eps^{\frac{1}{2}}.
\end{align}

\medskip
\noindent
\underline{{\it Estimates on $\Fb_s$.}} We intend to show that $\Fb_s$ shares the same bound as \eqref{ak-es}.
To do so,
let us define
\begin{eqnarray*}
\hat{\Psi}=\hat{\Psi}_{\Fb_s}=\sum\limits_{m=1}^3\hat{\Psi}^{J,m}_{\Fb_s}
,\ J=1,2,3,\ \Fb_s=(b_{s,1},b_{s,2},b_{s,3}),
\end{eqnarray*}
with
\begin{eqnarray*}
\hat{\Psi}^{J,m}_{\Fb_s}
=\left\{\begin{array}{l}
|\Fv|^2v_mv_J\widehat{\pa_m\phi_{J}}-\frac{7}{2}(v_m^2-1)\widehat{\pa_J\phi_{J}}\mu^{\frac{1}{2}},\ \ J\neq m,\\[2mm]
\frac{7}{2}(v_J^2-1)\widehat{\pa_J\phi_{J}}\mu^{\frac{1}{2}},\ \ J=m,
\end{array}\right.
\end{eqnarray*}
where
\begin{equation}
(|\bar{k}|^2-\pa_z^2)\hat{\phi}_{J}=\hat{b}_{s,J},\ \hat{\phi}_{J}(\bar{k},\pm1)=0,\  {\rm and}\ \pa_1=\pa_{x}, \pa_2=\pa_{y}, \pa_3=\pa_{z}.\notag
\end{equation}
As \eqref{epes-ap}, \eqref{epes-a2} and \eqref{trace}, one has the following elliptic estimates
\begin{align}
\||\bar{k}|^2\hat{\phi}_{J}(\bar{k},z)\|_{2}+\||\bar{k}|\hat{\phi}_{J}(\bar{k},z)\|_{H^1_z}
+\|\hat{\phi}_{J}\|_{H^{2}_z}\lesssim \|\hat{b}_{s,J}\|_{2},\ \forall\bar{k}\in\Z^2,\notag
\end{align}
\begin{align}
(1+|\bar{k}|)\left\{\||\bar{k}|\hat{\phi}_{J}(\bar{k},z)\|_{2}+\|\hat{\phi}_{J}(\bar{k},z)\|_{H^1_z}\right\}
\lesssim \|\hat{b}_{s,J}\|_{2},\ \forall\bar{k}\in\Z^2,\notag
\end{align}
and the trace inequality
\begin{align}\label{trace-b}
(1+|\bar{k}|)\|\hat{\phi}_{J}(\bar{k},\pm1)\|_2+\|\pa_z\hat{\phi}_{J}(\bar{k},\pm1)\|_{2}\lesssim\|\hat{b}_{s,J}\|_{2},\ \forall\bar{k}\in\Z^2.
\end{align}
We now turn to compute \eqref{abc-test} with $\Psi$ being replaced by $\Psi_{\Fb_s}$ term by term. By a direct calculation, the left hand side of \eqref{abc-test} gives
\begin{equation*}
\begin{split}
&-\sum\limits_{m=1}^3\left(\FP \hat{f}_{R},\Fv\cdot \widehat{\na \Psi^{J,m}_{\Fb_s}}\right)
\\
&
=-\sum\limits_{m=1,m\neq J}^3\left(v_mv_J\mu^{\frac{1}{2}}\hat{b}_{s,J}
,|\Fv|^2v_mv_J\mu^{\frac{1}{2}}\widehat{\pa^2_m\phi_{J}}\right)
\\
&\qquad
-\sum\limits_{m=1,m\neq J}^3
\left(v_mv_J\mu^{\frac{1}{2}}\hat{b}_{s,m},|\Fv|^2v_mv_J\mu^{\frac{1}{2}}\widehat{\pa_J\pa_m\phi_{J}}\right)
\\
&\qquad
+7\sum\limits_{m=1,m\neq J}^3\left( \hat{b}_{s,m}
,\widehat{\pa_m\pa_J\phi_{J}}\right)-7\left(\hat{b}_{s,J}
,\widehat{\pa^2_J\phi_{J}}\right)=7\|\hat{b}_{s,J}\|_2^2.
\end{split}
\end{equation*}
For $S_1$ on the right hand side of \eqref{abc-test}, by using \eqref{trace-b} and \eqref{fst-bdp}, one has
\begin{align}\label{bs-es}
|S_1|=&\left|\sum\limits_{J=1}^3( \hat{\Psi}^{J,m}_{\Fb_s}(1), v_z \hat{f}_{R}(1))\right|
+\left|\sum\limits_{J=1}^3( \hat{\Psi}^{J,m}_{\Fb_s}(-1), v_z \hat{f}_{R}(-1))\right|\notag\\
\leq&\left|\sum\limits_{J=1}^3( \hat{\Psi}^{J,m}_{\Fb_s}(1), v_z \{\FI-P_\ga\}\hat{f}_{R}(1){\bf 1}_{\ga_+}+\eps^{\frac{1}{2}}v_z\hat{r}{\bf 1}_{\ga_-})\right|
\notag\\&+\left|\sum\limits_{J=1}^3( \hat{\Psi}^{J,m}_{\Fb_s}(-1), v_z \{\FI-P_\ga\}\hat{f}_{R}(-1){\bf 1}_{\ga_+}+\eps^{\frac{1}{2}}v_z\hat{r}{\bf 1}_{\ga_-})\right|\notag\\
\lesssim& \sum\limits_{J=1}^3\|\hat{b}_{s,J}\|_2\{|w^{l_2}\sqrt{\mu}\{\FI-P_\ga\}\hat{f}_{R}(\pm1)|_{2,+}+\eps^{\frac{1}{2}}|\hat{r}|_{2,-}\}
\notag\\
\lesssim&  \eta\|\hat{\Fb}_s\|_{2}^2+C_\eta\|w^{l_\infty}\hat{f}_{R,1}\|_{\infty}^2
+C_\eta|\{\FI-P_\ga\}\hat{f}_{R,2}(\pm1)|_{2,+}^2+\eps|\hat{r}|_{2,-}^2.
\end{align}
The remaining terms on the right hand side of \eqref{abc-test} can be handled in the same way as for deriving \eqref{ak-es}, we omit the details for brevity.

\medskip
\noindent
\underline{{\it Estimates on $c_s$.}} Set
$$
\hat{\Psi}=\hat{\Psi}_{c_s}=(\vert\Fv \vert^2-5)\left\{\Fv\cdot\widehat{ \na_{\Fx}\phi_{c_s}}\right\}\mu^{\frac{1}{2}},
$$
where
\begin{equation}
(|\bar{k}|^2-\pa_z^2)\hat{\phi}_{c_s}=\hat{c}_s,\quad \hat{\phi}_{c_s}(\bar{k},\pm1)=0.\notag
\end{equation}
It can be seen that $\hat{\phi}_{c_s}$ enjoys the same estimates as \eqref{epes-ap}, \eqref{epes-a2} and \eqref{trace}.
To evaluate each term in \eqref{abc-test} with $\Psi$ replaced by $\Psi_{c_s}$, we focus exclusively on handling $S_1$, as the other terms can be computed using the same approach employed to derive \eqref{ak-es}. The key point is that $S_1$ will be also reduced to
\begin{multline*}
    S_1=-( (|\Fv|^2-5)\{i\bar{k}\cdot\bar{v}+v_z\pa_z\}\hat{\phi}_{c_s}(1), v_z \{\FI-P_\ga\}\hat{f}_{R}(1){\bf 1}_{\ga_+}+ v_z\eps^{\frac{1}{2}}\hat{r}{\bf 1}_{\ga_-})
\\+( (|\Fv|^2-5)\{i\bar{k}\cdot\bar{v}+v_z\pa_z\}\hat{\phi}_{c_s}(-1), v_z \{\FI-P_\ga\}\hat{f}_{R}(-1){\bf 1}_{\ga_+}
+v_z\eps^{\frac{1}{2}}\hat{r}{\bf 1}_{\ga_-}),
\end{multline*}
due to the parity \eqref{fst-bdp}. Hence $c_s$ also enjoys the bound as \eqref{ak-es}.


This is the end of the macroscopic estimates. In the next step, we concern the  $L_{\bar{k}}^1L_{z,\Fv}^\infty$ estimates.

\subsection{\texorpdfstring{$L_{\bar{k}}^1L_{z,\Fv}^\infty$}{Lg} estimates}
Taking the Fourier transform of equations \eqref{f1-eq}, \eqref{f1-bd}, \eqref{f2-eq} and \eqref{f2-bd} with respect to $\bar{x}=(x,y)$, one has
\begin{align}\label{f1-ft}
&i\bar{k}\cdot\bar{v}\hat{f}_{R,1}+v_z\pa_{z}\hat{f}_{R,1}+\eps^2\Phi(z)\cdot\na_\Fv \hat{f}_{R,1}\notag\\
&\quad+ik_x\al\eps z\hat{f}_{R,1}-\al\eps v_z\pa_{v_x}\hat{f}_{R,1}+\frac{1}{\eps} \nu \hat{f}_{R,1}
\notag\\
&=\frac{1}{\eps}\chi_M\CK \hat{f}_{R,1}+\frac{\eps^2}{2}\Phi(z)\cdot \Fv\sqrt{\mu}\hat{f}_{R,2}-\frac{\al \eps v_xv_z}{2}\sqrt{\mu}\{\FI-\FP\}\hat{f}_{R,2}
\notag\\&
\quad-\eps^{\frac{5}{2}}\Phi\cdot \na_\Fv(\sqrt{\mu}\hat{f}_{2})-
\al\eps^{\frac{3}{2}} izk_x(\sqrt{\mu}\hat{f}_{2})+\al\eps^{\frac{3}{2}} v_z\pa_{v_x}(\sqrt{\mu}\hat{f}_{2})
\notag\\&
\quad+\eps^{\frac{1}{2}}\hat{Q}(\hat{f}_{R,1},\hat{f}_{R,1})
+\eps^{\frac{1}{2}}\hat{Q}(\hat{f}_{R,1},\sqrt{\mu}\hat{f}_{R,2})+\eps^{\frac{1}{2}}\hat{Q}(\sqrt{\mu}\hat{f}_{R,2},\hat{f}_{R,1})
\notag\\
&
\quad+\hat{Q}(\hat{f}_{R,1},\sqrt{\mu}\{\hat{f}_1+\eps \hat{f}_2\})+\hat{Q}(\sqrt{\mu}\{\hat{f}_1+\eps \hat{f}_2\},\hat{f}_{R,1})\notag\\
&\quad+\eps^{\frac{3}{2}}\hat{Q}(\sqrt{\mu}\hat{f}_2,\sqrt{\mu}\hat{f}_2),
\end{align}
\begin{align}\label{f1-ft-bd}
\hat{f}_{R,1}(\bar{k},\pm1,\Fv)|_{v_{z}\lessgtr0}=0,
\end{align}
and
\begin{align}\label{f2-ft}
&i\bar{k} \cdot\bar{v} \hat{f}_{R,2}+v_{z} \pa_{z}\hat{f}_{R,2}+\eps^2\Phi(z)\cdot\na_\Fv \hat{f}_{R,2}\notag\\
&\quad+\al\eps ik_xz\hat{f}_{R,2}-\al\eps v_{z}\pa_{v_x}\hat{f}_{R,2}
+\frac{1}{\eps} L\hat{f}_{R,2}\notag\\
&=\frac{1}{\eps}(1-\chi_M)\mu^{-1/2}\CK \hat{f}_{R,1}-\frac{\al\eps  v_xv_{z}}{2}\FP \hat{f}_{R,2}
\notag\\
&\quad+\hat{\Ga}(\hat{f}_{R,2},\hat{f}_1+\eps \hat{f}_2)+\hat{\Ga}(\hat{f}_1+\eps\hat{f}_2,\hat{f}_{R,2})+\eps^{\frac{1}{2}}\hat{\Ga}(\hat{f}_{R,2},\hat{f}_{R,2}),
\end{align}
\begin{align}\label{f2-ft-bd}
\hat{f}_{R,2}(\bar{k},\pm1,v)|_{v_{z}\lessgtr0}=&\sqrt{2\pi\mu}\int_{\tilde{v}_{z}\gtrless0}\{\hat{f}_{R,1}(\bar{k},\pm1,\tilde{\Fv})
+\hat{f}_{R,2}(\bar{k},\pm1,\tilde{\Fv})\sqrt{\mu}\}|\tilde{v}_{z}|d\tilde{\Fv}
\notag\\&
+\eps^{\frac{1}{2}} \hat{r}.
\end{align}
Next, let us introduce a parameter $t\in(-\infty,+\infty)$ and regard $(z,\Fv)$
as the functions of $t$.
Denote $[Z(s;t,z,\Fv),$ $\FV(s;t,z,\Fv)]$ by the characteristic line of both the equations \eqref{f1-ft} and \eqref{f2-ft} passing through $(t,z,\Fv)$.
It is straightforward to see that $[Z(s;t,z,\Fv),\FV(s;t,z,\Fv)]$ is determined by the ODE system
\begin{align}\label{chl}
\frac{d Z}{ds}=V_{z},\
\frac{d \FV}{ds}=\eps^2\Phi(Z)-\al\eps V_{z}\Fe_1,\ \FV=(V_x,V_y,V_{z}),
\end{align}
with $[Z(t;t,z,\Fv),\FV(t;t,z,\Fv)]=(z,\Fv)$ and $\Fe_1=(1,0,0)$. It follows from \eqref{chl} that
\begin{align}\label{chl-sol}
\left\{\begin{array}{rll}
&\dis Z(s)=z+(s-t)v_{z}+\eps^2\int_t^s\int_t^\tau\Phi_{z}(Z(\eta))d\eta d\tau,\\[3mm]
&\dis \FV(s)=\Fv+\eps^2\int_t^s\Phi(Z(\tau))d\tau-\al\eps v_{z}(s-t)\Fe_1\\
&\dis\qquad\qquad-\al\eps^3\int_t^s\int_t^\tau\Phi_{z}(Z(\eta))d\eta d\tau \Fe_1.
\end{array}\right.
\end{align}
For any $(t,z,\Fv)\in(-\infty,+\infty)\times[-1,1]\times\R^3$ with $z\neq0$, we define
the {\it backward exit time} as
\begin{align}\label{ex-t}
t_{b}(z,\Fv)=\inf\{\tau\geq0|Z(t-\tau;t,z,\Fv)\notin(-1,1)\}.
\end{align}
This is the first moment when the backward characteristic line  $[Z(s;t,z,\Fv),$ $\FV(s;t,z,\Fv)]$ emerges from $z=\pm1$.
Note that $t_{b}(z,\Fv)$ is well-defined if $(z,\Fv)\in [-1,1]\times\R^3/\ga_0$, and $t_{b}(z,\Fv)=0$ if $(z,\Fv)\in \ga_-$.
We also define the {\it backward exit position}
$$z_{b}(z,\Fv)=Z(t-t_b;t,z,\Fv)\in\{\pm1\}.$$
Similarly, we introduce the {\it forward exit time}
\begin{align}
t_{f}(z,\Fv)=\inf\{\tau\geq0|Z(t+\tau;t,z,\Fv)\notin(-1,1)\},\notag
\end{align}
and define
\begin{align}
z_{f}(z,\Fv)=Z(t+t_f;t,z,\Fv).\notag
\end{align}
For random variable $\Fv_l$ with $l\geq0$, we define the backward time cycles
\begin{equation}\label{cyc}
\left\{\begin{aligned}
(t_0,z_0,\Fv_0)=&~(t,z,\Fv), \\
\Fv_l=&~(v_{l,x},v_{l,y},v_{l,z}),
\\ 
(t_{l+1},z_{l+1},\Fv_{l+1})=&~(t_l-t_{\mathbf{b}}(z_l,\Fv_l),z_{b}(z_l,\Fv_l),\Fv_{l+1}). 
\end{aligned}\right.
\end{equation}
We also set
\begin{eqnarray*}
Z^{l}_{\mathbf{cl}}(s;t,z,\Fv) &=&\mathbf{1}_{[t_{l+1},t_{l})}(s)%
Z(s;t_l,z_l,\Fv_l), \\
\FV^{l}_{\mathbf{cl}}(s;t,z,\Fv) &=&\mathbf{1}%
_{[t_{l+1},t_{l})}(s)\FV(s;t_l,z_l,\Fv_l).
\end{eqnarray*}
Note that $[Z^{0}_{\mathbf{cl}}(s),\FV^{0}_{\mathbf{cl}}(s)]=[Z(s),\FV(s)]$ and $z_{l}\in\{\pm1\}$ for $l\geq1$. Moreover, $t_l$ can be negative.

Next, to handle the diffusive reflection boundary condition \eqref{f2-ft-bd}, we introduce the stochastic cycle integrals as follows.
Denote $
\mathcal{V}_{l}=\{\Fv_l\in \R^{3}\ |\ \Fv_l\cdot
n(z_{l})>0\},$
where $n(z_{l})=(0,0,1)$ if $z_{l}=1$ and $n(z_{l})=(0,0,-1)$ if $z_{l}=-1$. Let the iterated integral for $\mathfrak{L}\geq 2$ be defined as
\begin{equation}\label{def.sil}
\int_{\prod\limits_{l=1}^{\mathfrak{L}-1}
\mathcal{V}_{l}}\prod\limits_{l=1}^{\mathfrak{L}-1}d\sigma_{l}\equiv \int_{\mathcal{V}_{1}}\cdots
\left\{ \int_{\mathcal{V}%
_{\mathfrak{L}-1}}d\sigma_{\mathfrak{L}-1}\right\} d\sigma_{1},
\end{equation}%
where $d\sigma _{l}=\sqrt{2\pi}\mu (\Fv_l)|v_{l,z}|d\Fv_l$ is a probability
measure.

Along the characteristic line \eqref{chl}, for $(z,\Fv)\in[-1,1]\times\R^3\backslash(\ga_-\cup\ga_0)$, we now write the solution of the system \eqref{f1-ft} and \eqref{f1-ft-bd} with polynomial weight $w^{l_\infty}(\Fv)$ as the following mild form
\begin{align}
(w^{l_\infty}&\hat{f}_{R,1})(z(t),\Fv(t))\notag\\=&\frac{1}{\eps}\int_{t_1}^{t}e^{-\int_{s}^t\CA^\eps(\tau,\FV(\tau))d\tau}\left\{\chi_{M}w^{l_\infty}\CK
\hat{f}_{R,1}\right\}(Z(s),\FV(s))\,ds\notag
\\
&-\al\eps\int_{t_1}^{t}e^{-\int_{s}^t\CA^\eps(\tau,\FV(\tau))d\tau}
\left\{w^{l_\infty}
\frac{v_xv_{z}}{2}\sqrt{\mu}\{\FI-\FP\}\hat{f}_{R,2}\right\}(Z(s),\FV(s))\,ds\notag\\
&+\frac{\eps^2}{2}\int_{t_1}^{t}e^{-\int_{s}^t\CA^\eps(\tau,V(\tau))d\tau}\left\{w^{l_\infty}\Phi\cdot \Fv\sqrt{\mu}\hat{f}_{R,2}\right\}(Z(s),V(s))\,ds\notag\\
&+\int_{t_1}^{t}e^{-\int_{s}^t\CA^\eps(\tau,V(\tau))d\tau}
\Big\{w^{l_\infty}\Big(-\eps^{\frac{5}{2}}\Phi\cdot \na_\Fv(\sqrt{\mu}\hat{f}_{2})-
\al\eps^{\frac{3}{2}} izk_1(\sqrt{\mu}\hat{f}_{2})
\notag\\&\qquad\qquad+\al\eps^{\frac{3}{2}} v_z\pa_{v_x}(\sqrt{\mu}\hat{f}_{2})\Big)
\Big\}(Z(s),\FV(s))ds\notag\\
&+\eps^{\frac{1}{2}}\int_{t_1}^{t}e^{-\int_{s}^t\CA^\eps(\tau,V(\tau))d\tau}\Big\{w^{l_\infty}\Big(\hat{Q}(\hat{f}_{R,1},\hat{f}_{R,1})
+\hat{Q}(\hat{f}_{R,1},\sqrt{\mu}\hat{f}_{R,2})\notag\\&\qquad\qquad+\hat{Q}(\sqrt{\mu}\hat{f}_{R,2},\hat{f}_{R,1})\Big)\Big\}(Z(s),\FV(s))\,ds,\notag\\
&+\int_{t_1}^{t}e^{-\int_{s}^t\CA^\eps(\tau,V(\tau))d\tau}
\Big\{w^{l_\infty}\Big(\hat{Q}(\hat{f}_{R,1},\sqrt{\mu}\{\hat{f}_1+\eps \hat{f}_2\})\notag\\&\qquad\qquad+\hat{Q}(\sqrt{\mu}\{\hat{f}_1+\eps \hat{f}_2\},\hat{f}_{R,1})\Big)\Big\}(Z(s),\FV(s))\,ds\notag\\
&+\eps^{\frac{3}{2}}\int_{t_1}^{t}e^{-\int_{s}^t\CA^\eps(\tau,V(\tau))d\tau}
\left\{w^{l_\infty}\left(\hat{Q}(\sqrt{\mu}\hat{f}_2,\sqrt{\mu}\hat{f}_2)\right)\right\}(Z(s),\FV(s))\,ds,\notag
\end{align}
where
\begin{align}
\CA^\eps(\tau,\FV(\tau))=& 
\frac{1}{\eps}\nu(\FV(\tau))-2\eps^2l_\infty\frac{V(\tau)\cdot\Phi(Z)}{{1+|\FV(\tau)|^2}}+2l_\infty \al\eps \frac{V_x(\tau)V_{z}(\tau)}{{1+|\FV(\tau)|^2}}
\notag\\
&+i\bar{k}\cdot \bar{\FV}+ik_xZ\al\eps,
\notag\\
\bar{\FV}=&(V_x,V_y).\notag
\end{align}
Note that
\begin{align}
\left|e^{-\int_{s}^t\CA^\eps(\tau,\FV(\tau))d\tau}\right|\leq e^{-\frac{1}{2\eps}\int_{s}^t\nu(\FV(\tau))d\tau},\label{CAe-lbd}
\end{align}
provided that $\eps>0$ and $\eps l_\infty>0$ are suitably small.
By Lemmas \ref{g-ck-lem} and \ref{lk1-ne}, it is easy to show the following bound
\begin{align}
&\sum\limits_{\bar{k}\in\Z^2}\|w^{l_\infty}\hat{f}_{R,1}\|_{\infty} \notag\\
&\lesssim
\left\{(1+M)^{-\ga}+\varsigma+\eps\right\}\sum\limits_{\bar{k}\in\Z^2}\|w^{l_\infty}\hat{f}_{R,1}\|_{\infty}
\notag\\&
\quad+\eps^{\frac{3}{2}}\sum\limits_{\bar{k}\in\Z^2}\|w^{l_\infty}\hat{f}_{R,1}\|_{\infty}
\left\{\sum\limits_{\bar{k}\in\Z^2}\|w^{l_\infty}\hat{f}_{R,2}\|_{\infty}
+\sum\limits_{\bar{k}\in\Z^2}\|w^{l_\infty}\hat{f}_{R,1}\|_{\infty}\right\}
\notag\\&
\quad +
\eps^2\{\al+\eps\}\sum\limits_{\bar{k}\in\Z^2}\|w^{l_\infty}\hat{f}_{R,2}\|_{\infty}
+\vps_{\Phi,\al}\eps^{\frac{5}{2}}.
\label{H1m.p1}
\end{align}
By taking $M>0$ sufficiently large and $\varsigma>0$ as well as $\eps>0$ small enough, \eqref{H1m.p1} further gives
\begin{align}\label{f1-sum}
&\sum\limits_{\bar{k}\in\Z^2}\|w^{l_\infty}\hat{f}_{R,1}\|_{\infty}\notag\\
&\lesssim
\eps^2\{\al+\eps\}\sum\limits_{\bar{k}\in\Z^2}\|w^{l_\infty}\hat{f}_{R,2}\|_{\infty}
+\vps_{\Phi,\al}\eps^{\frac{5}{2}}\notag
\\&
\quad+\eps^{\frac{3}{2}}\sum\limits_{\bar{k}\in\Z^2}\|w^{l_\infty}\hat{f}_{R,1}\|_{\infty}
\left\{\sum\limits_{\bar{k}\in\Z^2}\|w^{l_\infty}\hat{f}_{R,2}\|_{\infty}
+\sum\limits_{\bar{k}\in\Z^2}\|w^{l_\infty}\hat{f}_{R,1}\|_{\infty}\right\}.
\end{align}
Similarly, one can write the solution of \eqref{f2-ft} and \eqref{f2-ft-bd} with polynomial weight $w^{l_\infty}(\Fv)$  as
\begin{align}\label{H2m}
(w^{l_\infty}&\hat{f}_{R,2})(z(t),\Fv(t))\notag\\
=&\underbrace{\frac{1}{\eps}\int_{t_1}^{t}e^{-\int_{s}^t\CA^\eps(\tau,\FV(\tau))d\tau}\left\{w^{l_\infty}K
\hat{f}_{R,2}\right\}(Z(s),\FV(s))\,ds}_{I_1}\notag
\\
&+\underbrace{\frac{1}{\eps}\int_{t_1}^{t}e^{-\int_{s}^t\CA^\eps(\tau,V(\tau))d\tau}
\left\{w^{l_\infty}((1-\chi_{M})\mu^{-\frac{1}{2}}\CK \hat{f}_{R,1}\right\}(Z(s),\FV(s))\,ds}_{I_2}\notag\\
& \underbrace{-\frac{\al\eps}{2}\int_{t_1}^{t}e^{-\int_{s}^t\CA^\eps(\tau,\FV(\tau))d\tau}\left\{w^{l_\infty}v_xv_{z}\FP \hat{f}_{R,2}\right\}(Z(s),\FV(s))\,ds}_{I_3}
\notag\\
&+ \int_{t_1}^{t}e^{-\int_{s}^t\CA^\eps(\tau,\FV(\tau))d\tau}
\left\{w^{l_\infty}\left(\hat{\Ga}(\hat{f}_{R,2},\hat{f}_1+\eps \hat{f}_2)\right.\right.\notag\\
&\underbrace{\qquad \qquad \left.\left. +\hat{\Ga}(\hat{f}_1+\eps\hat{f}_2,\hat{f}_{R,2})+\eps^{\frac{1}{2}}\hat{\Ga}(\hat{f}_{R,2},\hat{f}_{R,2})
\right)\right\} \left(Z(s),\FV(s)\right)\,ds}_{I_4}\notag\\
&+\sum\limits_{n=5}^{11}I_n,
\end{align}
where $(z,\Fv)\in[-1,1]\times\R^3\backslash(\ga_-\cup\ga_0)$, and for $\RL\geq2$,
\begin{align}
I_5=&\underbrace{\sqrt{2\pi}e^{-\int_{t_1}^t\CA^\eps(\tau,V(\tau))d\tau}
\left[w^{l_\infty}\sqrt{\mu}\right](V(t_1))}_{\CW}\notag\\&\qquad\qquad\times\int_{\prod\limits_{j=1}^{\RL-1}\CV_j}
(w^{l_\infty}\hat{f}_{R,2})(t_\RL,z_\RL,V_{\mathbf{cl}}^{\RL-1}(t_\RL))\,d\Sigma_{\RL-1}(t_\RL),\notag
\end{align}
\begin{align}
I_6=&\sum\limits_{l=1}^{\RL-1}\CW\int_{\prod\limits_{j=1}^{\RL-1}\CV_j}
\int_{t_{l+1}}^{t_l}\left\{w^{l_\infty}\left(\hat{\Ga}(\hat{f}_{R,2},\hat{f}_1+\eps \hat{f}_2)+\hat{\Ga}(\hat{f}_1+\eps\hat{f}_2,\hat{f}_{R,2})\right.\right.
\notag\\
&\qquad \qquad \qquad \qquad\qquad\left. \left.+\eps^{\frac{1}{2}}\hat{\Ga}(\hat{f}_{R,2},\hat{f}_{R,2})\right)
\right\}(Z^{l}_{\mathbf{cl}},\FV^{l}_{\mathbf{cl}})(s)\,d\Sigma_{l}(s)ds,\notag
\end{align}
\begin{align}
I_7=& \frac{1}{\eps}\sum\limits_{l=1}^{\RL-1}\CW\int_{\prod\limits_{j=1}^{\RL-1}\CV_j}
\int_{t_{l+1}}^{t_l}\left\{w^{l_\infty}K\hat{f}_{R,2}\right\}(Z^{l}_{\mathbf{cl}},\FV^{l}_{\mathbf{cl}})(s)\,d\Sigma_{l}(s)ds,\notag
\end{align}
\begin{align}
I_8=& \frac{1}{\eps}\sum\limits_{l=1}^{\RL-1}\CW\int_{\prod\limits_{j=1}^{\RL-1}\CV_j}
\int_{t_{l+1}}^{t_l}\left\{(1-\chi_{M})w^{l_\infty}\mu^{-\frac{1}{2}}\CK
\hat{f}_{R,1}\right\}(Z^{l}_{\mathbf{cl}},\FV^{l}_{\mathbf{cl}})(s)\,d\Sigma_{l}(s)ds,\notag
\end{align}
\begin{align}
I_9=& -\frac{\al}{2}\sum\limits_{l=1}^{\RL-1}\CW\int_{\prod\limits_{j=1}^{\RL-1}\CV_j}
\int_{t_{l+1}}^{t_l}\left\{w^{l_\infty}v_xv_{z}\FP \hat{f}_{R,2}\right\}(Z^{l}_{\mathbf{cl}},\FV^{l}_{\mathbf{cl}})(s)\,d\Sigma_{l}(s)ds,\notag
\end{align}
\begin{align}
I_{10}=&\CW \sum\limits_{l=1}^{\RL-1}\int_{\prod\limits_{j=1}^{\RL-1}\CV_j}\left(\frac{w^{l_\infty}}{\sqrt{\mu}}\hat{f}_{R,1}\right)
(t_l,z_l,\FV_{\mathbf{cl}}^{l}(t_l))d\Sigma_{l}(t_l),\notag
\end{align}
\begin{align}
I_{11}=&\eps^{\frac{1}{2}}e^{-\int_{t_1}^t\CA^\eps(\tau,\FV(\tau))d\tau}\left(w^{l_\infty}\hat{r}\right)(t_1,z_1,\FV(t_1))
\notag \\ &+
\eps^{\frac{1}{2}}\CW \sum\limits_{l=1}^{\RL-1}\int_{\prod\limits_{j=1}^{\RL-1}\CV_j}\left(w^{l_\infty}\hat{r}\right)
(t_{l+1},z_{l+1},\FV_{\mathbf{cl}}^{l}(t_{l+1}))d\Sigma_{l}(t_{l+1}).\notag
\end{align}
Moreover, in the above expressions we have used the following notations
\begin{align}\label{Sigma}
\Sigma_{l}(s)=\prod\limits_{j=l+1}^{\RL-1}d\si_j 
&
e^{-\int_s^{t_l}
\CA^\eps(\tau,\FV_{\mathbf{cl}}^l(\tau))d\tau}{w_2}(\Fv_l)d\si_l 
\notag\\
 &
\times \prod\limits_{j=1}^{l-1}\frac{{w_2}(\Fv_j)}{{w_2}(\FV^{j}_{\mathbf{cl}}(t_{j+1}))}
e^{-\int_{t_{j+1}}^{t_j}
\CA^\eps(\tau,\FV_{\mathbf{cl}}^j(\tau))d\tau}d\si_j,
\end{align}
and
\begin{align}\label{wt-2}
{w_2}(\Fv)=(\sqrt{2\pi}w^{l_\infty}\sqrt{\mu})^{-1}(\Fv).
\end{align}
Note that we take $\RL=C_1T_0^{\frac{5}{4}}$ for $T_0$ defined in Lemma \ref{k-cyc}.

The $L^\infty$ estimates for $\hat{f}_{R,2}$ is more complicated because $K$ has no smallness property. To overcome this difficulty, one available way is to iterate \eqref{H2m} twice. Let us first compute $I_n$ $(1\leq n\leq 11)$ term by term.
Recalling the definition \eqref{kw-def} for ${\bf k}_{w}$, one directly has by \eqref{CAe-lbd}
\begin{align*}
|I_1|\leq \frac{1}{\eps}\int_{t_1}^te^{-\int_s^t\frac{\nu(\FV(\tau))}{2\eps}d\tau}\int_{\R^3}{\bf k}_{w}(V(s),\Fv')|(w^{l_\infty}\hat{f}_{R,2})(s,Z(s),\Fv')|d\Fv'ds.
\end{align*}
By Lemma \ref{op-es-lem}, it follows
\begin{align*}
|I_2|\leq \frac{C}{\eps}\|w^{l_\infty}\hat{f}_{R,1}\|_{\infty}
\int_{t_1}^te^{-\int_s^t\frac{\nu(\FV(\tau))}{2\eps}d\tau}ds\leq C\|w^{l_\infty}\hat{f}_{R,1}\|_{\infty}.
\end{align*}
It is straightforward to see
\begin{align*}
|I_3|\leq C \al\eps^2\|w^{l_\infty}\hat{f}_{R,2}\|_{\infty}.
\end{align*}
For $I_4$, Lemma \ref{lk1-ne} gives
\begin{align*}
|I_4|\leq& C \eps\left\{\|w^{l_\infty}\nu^{-1}\hat{\Ga}(\hat{f}_{1}+\eps\hat{f}_{2},\hat{f}_{R,2})\|_{\infty}
+\|w^{l_\infty}\nu^{-1}\hat{\Ga}(\hat{f}_{R,2},\hat{f}_{1}+\eps\hat{f}_{2})\|_{\infty}\right\}
\notag\\&+C\eps^{\frac{3}{2}} \|w^{l_\infty}\nu^{-1}\hat{\Ga}(\hat{f}_{R,2},\hat{f}_{R,2})\|_{\infty}\notag\\
\leq& C \eps\sum\limits_{\bar{l}\in \Z^2}\{\|w^{l_\infty}\hat{f}_{1}(\bar{k}-\bar{l})\|_{\infty}+\|w^{l_\infty}\hat{f}_{2}(\bar{k}-\bar{l})\|_{\infty}\}
\|w^{l_\infty}\hat{f}_{R,2}(\bar{l})\|_{\infty}
\notag\\
&+C\eps^{\frac{3}{2}} \sum\limits_{\bar{l}\in \Z^2}\|w^{l_\infty}\hat{f}_{R,2}(\bar{k}-\bar{l})\|_{\infty}
\|w^{l_\infty}\hat{f}_{R,2}(\bar{l})\|_{\infty}.
\end{align*}
For $I_5$, motivated by the approach developed in \cite{EGMW-23}, let us first choose a sufficiently large $T_0>0$ and then write
\begin{align}
I_5=&\CW\int_{\prod\limits_{j=1}^{\RL-1}\CV_j}{\bf 1}_{t_\RL\geq t-T_0}
(w^{l_\infty}\hat{f}_{R,2})(t_\RL,z_\RL,V_{\mathbf{cl}}^{\RL-1}(t_\RL))\,d\Sigma_{\RL-1}(t_\RL)\notag\\
&+\CW\int_{\prod\limits_{j=1}^{\RL-1}\CV_j}{\bf 1}_{t_\RL< t-T_0}
(w^{l_\infty}\hat{f}_{R,2})(t_\RL,z_\RL,V_{\mathbf{cl}}^{\RL-1}(t_\RL))\,d\Sigma_{\RL-1}(t_\RL)\notag\\
=:& I_{5,1}+I_{5,2}.\notag
\end{align}
For $I_{5,1}$,
in light of Lemma \ref{k-cyc}, it follows
\begin{align*}
|I_{5,1}|\leq C(l_\infty)2^{-C_2T^{\frac{5}{4}}_0} \|w^{l_\infty}\hat{f}_{R,2}\|_{\infty}.
\end{align*}
Here, for $0<\eta_0\ll1$, we utilized the estimate
\begin{align}\label{diff-w}
\frac{{w_2}(\Fv_j)}{{w_2}(V^{j}_{\mathbf{cl}}(t_{j+1}))}
=&\frac{w^{l_\infty}(\FV^{j}_{\mathbf{cl}}(t_{j+1}))\mu^{\frac{1}{2}}(\FV^{j}_{\mathbf{cl}}(t_{j+1}))}{w^{l_\infty}(\Fv_j)
\mu^{\frac{1}{2}}(\Fv_j)}
\notag \\ 
=& \frac{(1+|\FV^{j}_{\mathbf{cl}}(t_{j+1})|^2)^{l_\infty}}{(1+|\Fv_j|^2)^{l_\infty}} e^{\frac{|\Fv_j|^2-|\FV^{j}_{\mathbf{cl}}(t_{j+1})|^2}{4}}\notag\\
\leq&2^{l_\infty}(1+|\FV^{j}_{\mathbf{cl}}(t_{j+1})-\Fv_j|^2)^{l_\infty} e^{C_{\eta_0}\eps^2}e^{\frac{\eta_0|\Fv_{j}|^2}{4}}\notag\\
\leq & C(l_\infty)e^{\frac{\eta_0|\Fv_j|^2}{4}}. 
\end{align}
As for $I_{5,2}$, since $t-t_\CL>T_0>0$, the exponential term becomes very small, that is, from Lemma \ref{k-cyc} and \eqref{diff-w},
and using
\begin{align}
\left\{\begin{array}{rll}
&\left|e^{-\int_{t_{l+1}}^{t_l}\CA^\eps(\tau,\FV(\tau))d\tau}\right|\geq e^{-\frac{1}{2\eps}\int_{t_{l+1}}^{t_l}\nu(\FV(\tau))d\tau}\geq e^{-\frac{\nu_0(t_l-t_{l+1})}{2\eps}},\\
&e^{-\frac{\nu_0(t-t_{1})}{2\eps}} e^{-\frac{\nu_0(t_1-t_{2})}{2\eps}}\cdots e^{-\frac{\nu_0(t_l-s)}{2\eps}}=e^{-\frac{\nu_0(t-s)}{2\eps}},
\end{array}\right.
\label{CAe-lbd2}
\end{align}
we have for $0<\eta_0\ll1$ that
\begin{align}\label{tdecay-sm}
|I_{5,2}|\leq& 
\|w^{l_\infty}\hat{f}_{R,2}\|_{\infty}
C(l_\infty)e^{-\frac{\nu_0(t-t_1)}{2\eps}}
\notag\\ & \qquad 
\times
\int_{\Pi_{l=1}^{\RL-1}\mathcal{V}_{l}} d\sigma _{l}
\mathbf{1}_{
t_{\RL}(t,z,\Fv,\Fv_{1},\Fv_{2}...,\Fv_{\RL-1})<t-T_0}
\Pi _{l=1}^{\RL-1}e^{-\frac{\nu_0(t_l-t_{l+1})}{2\eps}}e^{\frac{\eta_0}{2}|\Fv_{l}|^2}
\notag\\
\leq& \|w^{l_\infty}\hat{f}_{R,2}\|_{\infty} C(l_\infty)e^{-\frac{\nu_0T_0}{2\eps}}.
\end{align}
Using Lemma \ref{lk1-ne} and applying \eqref{diff-w} and \eqref{CAe-lbd2} again, one gets
\begin{align*}
|I_6|\leq& C(l_\infty)\sum\limits_{l=1}^{\RL-1}\int_{\prod\limits_{j=1}^{\RL-1}\CV_j}
\int_{t_{l+1}}^{t_l}e^{-\int_s^{t_l}\frac{\nu(\FV_{cl}^l(\tau))}{2\eps}d\tau}\nu(\Fv_l)\notag\\
&\quad \times\prod\limits_{j=l+1}^{\RL-1}d\si_j {w_2}(\Fv_l)d\si_l
\prod\limits_{j=1}^{l-1}e^{\frac{\eta_0}{2}|\Fv_{j}|^2}d\si_j ds
\notag\\&\quad \times
\left\{\left\|w^{l_\infty}\nu^{-1}\left(\hat{\Ga}(\hat{f}_{R,2},\hat{f}_1+\eps \hat{f}_2)+\hat{\Ga}(\hat{f}_1+\eps\hat{f}_2,\hat{f}_{R,2})\right)\right\|_{\infty}\right.\notag\\
&\quad \quad \quad\quad \quad \left.+
\eps^{\frac{1}{2}}\|w^{l_\infty}\nu^{-1}\hat{\Ga}(\hat{f}_{R,2},\hat{f}_{R,2})\|_{\infty}\right\}\notag\\
\leq& \eps\RL C(l_\infty)\left\{\sum\limits_{\bar{l}\in \Z^2}\|w^{l_\infty}(\hat{f}_1+\eps\hat{f}_2)(\bar{k}-\bar{l})\|_{\infty}\|w^{l_\infty}\hat{f}_{R,2}\|_{\infty}\right.\notag\\
&\quad \quad \quad \quad \quad \left.+
\eps^{\frac{1}{2}}\sum\limits_{\bar{l}\in \Z^2}\|w^{l_\infty}\hat{f}_{R,2}(\bar{k}-\bar{l})\|_{\infty}
\|w^{l_\infty}\hat{f}_{R,2}(\bar{l})\|_{\infty}\right\}.\notag
\end{align*}
Similarly, it follows
\begin{align*}
|I_8|,\ |I_{10}|\leq \RL C(l_\infty)\|w^{l_\infty}\hat{f}_{R,1}\|_{\infty},\\ 
|I_9|\leq \al\RL C(l_\infty)\|w^{l_\infty}\hat{f}_{R,2}\|_{\infty},
\\ |
I_{11}|\leq \eps^{\frac{1}{2}}\RL C(l_\infty)\|w^{l_\infty}\hat{r}\|_{\infty},
\end{align*}
and
\begin{multline*}
|I_7|\leq \eps^{-1} C(l_\infty)e^{-\int_{t_1}^t\frac{\nu(\FV(\tau))}{2\eps}d\tau}\sum\limits_{l=1}^{\RL-1}\int_{\prod\limits_{j=1}^{\RL-1}\CV_j}
\int_{t_{l+1}}^{t_l}\int_{\R^3}{\bf k}_{w}(\FV^{l}_{\mathbf{cl}}(s),\Fv')\\
|(w^{l_\infty}\hat{f}_{R,2})(s,Z^{l}_{\mathbf{cl}}(s),\Fv')|d\Fv'\,d\Sigma_{l}(s)ds.
\end{multline*}
Up to now, we cannot obtain the desired estimate of the right hand side of the above inequality. However, putting all the estimates above together, we arrive at
\begin{align}\label{H2sum1}
& |w^{l_\infty}\hat{f}_{R,2}|
\notag \\
\leq &  \frac{C}{\eps}\int_{t_1}^te^{-\int_{s}^t\frac{\nu(\FV(\tau))}{2\eps}d\tau}
\int_{\R^3}{\bf k}_{w}(\FV(s),\Fv')|(w^{l_\infty}\hat{f}_{R,2})(s,Z(s;t,z,\Fv),\Fv')|d\Fv'ds\notag\\
&+\frac{C(l_\infty)}{\eps}e^{-\int_{t_1}^t\frac{\nu(\FV(\tau))}{2\eps}d\tau}
\sum\limits_{l=1}^{\RL-1}\int_{\prod\limits_{j=1}^{\RL-1}\CV_j}
\int_{t_{l+1}}^{t_l}\int_{\R^3}{\bf k}_{w}(\FV^{l}_{\mathbf{cl}}(s),v')\notag\\&
\qquad\qquad\qquad\qquad\qquad\qquad\times|(w^{l_\infty}\hat{f}_{R,2})(s,Z^{l}_{\mathbf{cl}}(s;t,z,\Fv),\Fv')|d\Fv'\,d\Sigma_{l}(s)ds
\notag\\&+\CQ(t),
\end{align}
where
\begin{align}
\CQ(t)=&\RL C(l_\infty)\|w^{l_\infty}\hat{f}_{R,1}\|_{\infty}
\notag\\&+
C \eps\sum\limits_{\bar{l}\in \Z^2}\left\{\|w^{l_\infty}\hat{f}_{1}(\bar{k}-\bar{l})\|_{\infty}+\|w^{l_\infty}\hat{f}_{2}(\bar{k}-\bar{l})\|_{\infty}\right\}
\|w^{l_\infty}\hat{f}_{R,2}(\bar{l})\|_{\infty}
\notag\\
&+C\eps^{\frac{3}{2}} \sum\limits_{\bar{l}\in \Z^2}\|w^{l_\infty}\hat{f}_{R,2}(\bar{k}-\bar{l})\|_{\infty}
\|w^{l_\infty}\hat{f}_{R,2}(\bar{l})\|_{\infty}+\eps^{\frac{1}{2}}\RL C(l_\infty)\|w^{l_\infty}\hat{r}\|_{\infty}.\notag
\end{align}
The key point now is to handle the above bound by iteration method. To do this, let us define a new  backward stochastic cycle as
\begin{align}
(t_{\ell+1}',z_{\ell+1}',\Fv_{\ell+1}')=(t_{\ell}'-t_{b}(z_{\ell}',\Fv_\ell'),z_{b}(z_\ell',\Fv_\ell'),\Fv_{\ell+1}'),\notag
\end{align}
and the starting point
\begin{align}\notag
(t_{0}',z_{0}',\Fv_{0}')=(s,z',\Fv'):=(s,Z(s),\Fv')\ \textrm{or}\ (s,Z^l_{\mathbf{cl}}(s),\Fv'),
\end{align}
for some $s\in\R$ and $l\in\Z^+$. Furthermore, for $\ell\in\Z^+$, we also denote
\begin{align}
Z^\ell_{\mathbf{cl}}(s';s,z',\Fv') &=\mathbf{1}_{[t'_{\ell+1},t'_{\ell})}(s')
Z(s';t'_\ell,z'_\ell,\Fv'_\ell),\notag\\
\FV^{\ell}_{\mathbf{cl}}(s';s,z',\Fv') &=\mathbf{1}_{[t'_{\ell+1},t'_{\ell})}(s')\FV(s';t'_\ell,z'_\ell,\Fv'_\ell),\notag
\end{align}
where $Z$ and $\FV$ are defined in \eqref{chl-sol}.
To be consistent, we set 
$$
[Z^{0}_{\mathbf{cl}}(s'),\FV^{0}_{\mathbf{cl}}(s')]:=[Z(s'),\FV(s')].
$$

Iterating \eqref{H2sum1} again, one has
\begin{align}\label{f2-itr}
|w^{l_\infty}&\hat{f}_{R,2}|\leq \frac{C}{\eps^2}\int_{t_1}^te^{-\int_{s}^t\frac{\nu(\FV(\tau))}{2\eps}d\tau}\int_{\R^3}{\bf k}_{w}(\FV(s),\Fv')
\int_{t'_1}^{s}e^{-\int_{s'}^s\frac{\nu(\FV(\tau))}{2\eps}d\tau}\notag\\
&\qquad\quad\times\int_{\R^3}{\bf k}_{w}(\FV(s';s,Z(s),\Fv'),\Fv'')
\notag\\&\qquad\qquad\times|(w^{l_\infty}\hat{f}_{R,2})(s',Z(s';s,Z(s),\Fv'),\Fv'')|~d\Fv''ds'd\Fv'ds\notag\\
&+\frac{C}{\eps^2}\int_{t_1}^te^{-\int_{s}^t\frac{\nu(\FV(\tau))}{2\eps}d\tau}\int_{\R^3}{\bf k}_{w}(\FV(s),\Fv')e^{-\int_{t_1'}^s\frac{\nu(\FV(\tau))}{2\eps}d\tau}
\notag\\
&\qquad\times\sum\limits_{\ell=1}^{\RL-1}\int_{\prod\limits_{j=1}^{\RL-1}\CV'_j}
\int_{t'_{\ell+1}}^{t'_\ell}\int_{\R^3}{\bf k}_{w}(\FV^{\ell}_{\mathbf{cl}}(s';s,Z(s),\Fv'),\Fv'')
\notag
\\&\qquad\quad\times|(w^{l_\infty}\hat{f}_{R,2})(s',Z^{\ell}_{\mathbf{cl}}(s';s,Z(s),\Fv'),\Fv'')|d\Fv''\,d\Sigma_{\ell}(s')ds'd\Fv'ds\notag\\
&+\frac{C}{\eps^2}e^{-\int_{t_1}^t\frac{\nu(\FV(\tau))}{2\eps}d\tau}\sum\limits_{l=1}^{\RL-1}\int_{\prod\limits_{j=1}^{\RL-1}\CV_j}
\int_{t_{l+1}}^{t_l}\int_{\R^3}{\bf k}_{w}(\FV^{l}_{\mathbf{cl}}(s),\Fv')
\notag\\
&\qquad\times\int_{t'_1}^{s}e^{-\int_{s'}^s\frac{\nu(\FV(\tau))}{2\eps}d\tau}\int_{\R^3}{\bf k}_{w}(\FV(s';s,Z^{l}_{\mathbf{cl}}(s),\Fv'),\Fv'')
\notag\\&\qquad\quad\times|(w^{l_\infty}\hat{f}_{R,2})(s',Z(s';s,Z^{l}_{\mathbf{cl}}(s),\Fv'),\Fv'')|~d\Fv''ds'd\Fv'\,d\Sigma_{l}(s)ds
\notag\\&+\frac{C}{\eps^2}e^{-\int_{t_1}^t\frac{\nu(\FV(\tau))}{2\eps}d\tau}\sum\limits_{l=1}^{\RL-1}\int_{\prod\limits_{j=1}^{\RL-1}\CV_j}
\int_{t_{l+1}}^{t_l}\int_{\R^3}{\bf k}_{w}(\FV^{l}_{\mathbf{cl}}(s;\Fv),\Fv')e^{-\int_{t_1'}^s\frac{\nu(\FV(\tau))}{2\eps}d\tau}
\notag\\
&\quad\times\sum\limits_{\ell=1}^{\RL-1}\int_{\prod\limits_{j=1}^{\RL-1}\CV'_j}
\int_{t'_{\ell+1}}^{t'_\ell}\int_{\R^3}  {\bf k}_{w}(\FV^{\ell}_{\mathbf{cl}}(s';s,Z^{l}_{\mathbf{cl}}(s),\Fv'),\Fv'')
\notag\\
&\quad\times|(w^{l_\infty}\hat{f}_{R,2})(s',Z^{\ell}_{\mathbf{cl}}(s';s,Z^{l}_{\mathbf{cl}}(s),\Fv'),\Fv'')|
d\Fv''\,d\Sigma_{\ell}(s')ds'd\Fv'\,d\Sigma_{l}(s)ds
\notag\\
&+ \frac{C}{\eps}\int_{t_1}^te^{-\int_{s}^t\frac{\nu(\FV(\tau))}{2\eps}d\tau}\int_{\R^3}{\bf k}_{w}(\FV(s),\Fv')\CQ(s)d\Fv'ds
\notag\\&+\frac{C(l_\infty)}{\eps}e^{-\int_{t_1}^t\frac{\nu(\FV(\tau))}{2\eps}d\tau}\int_{\prod\limits_{j=1}^{\RL-1}\CV_j}
\int_{t_{l+1}}^{t_l}\int_{\R^3}{\bf k}_{w}(\FV^{l}_{\mathbf{cl}}(s),\Fv')\CQ(s)d\Fv'\,d\Sigma_{l}(s)ds,
\end{align}
where
\begin{multline}
\Sigma_{\ell}(s')=\prod\limits_{j=\ell+1}^{\RL-1}d\si'_j e^{-\int_{s'}^{t'_\ell}
\CA^\eps(\tau,\FV_{\mathbf{cl}}^\ell(\tau))d\tau}{w_2}(\Fv'_\ell)d\si'_\ell \notag\\
\times\prod\limits_{j=1}^{\ell-1}\frac{{w_2}(\Fv'_j)}{{w_2}(\FV^{j}_{\mathbf{cl}}(t_{j+1}))}
e^{-\int_{t'_{j+1}}^{t'_j}
\CA^\eps(\tau,\FV_{\mathbf{cl}}^j(\tau))d\tau}d\si'_j,\notag
\end{multline}
with
$$
d\si'_j=\sqrt{2\pi}\mu (\Fv'_j)|v'_{j,z}|d\Fv'_j,\ \ \Fv'_j=(v'_{j,x},v'_{j,y},v'_{j,z}),\ 1\leq j\leq \RL-1,
$$
and
$$
\mathcal{V}'_{\ell}=\{\Fv'_\ell\in \R^{3}\ |\ \Fv'_\ell\cdot
n(z'_{\ell})>0\}.$$
\eqref{f2-itr} further provides
\begin{align}
\|w^{l_\infty}\hat{f}_{R,2}\|_{\infty}\leq \eta\|w^{l_\infty}\hat{f}_{R,2}\|_{\infty}+\eps^{-\frac{1}{2}}\|\hat{f}_{R,2}\|_2+\CQ.\label{f2-lif-sum}
\end{align}
To prove \eqref{f2-lif-sum}, we only compute the fourth term on the right hand side of \eqref{f2-itr}, because the other terms can be estimated similarly. For any sufficiently small $\ka_0>0$, we first divide $[t_{\ell+1}',t_{\ell}']$ as $[t_{\ell+1}',t_{\ell}'-\ka_0\eps]\cup(t_{\ell}'-\ka_0\eps,t_{\ell}']$, then rewrite the fourth term on the right hand side of \eqref{f2-itr} as
\begin{align}
\CJ:=& \frac{C}{\eps^2}e^{-\int_{t_1}^t\frac{\nu(\FV(\tau))}{2\eps}d\tau}\sum\limits_{l=1}^{\RL-1}\int_{\prod\limits_{j=1}^{\RL-1}\CV_j}
\int_{t_{l+1}}^{t_l}\int_{\R^3}{\bf k}_{w}(\FV^{l}_{\mathbf{cl}}(s;\Fv),\Fv')e^{-\int_{t_1'}^s\frac{\nu(\FV(\tau))}{2\eps}d\tau}
\notag\\
&\quad\times \sum\limits_{\ell=1}^{\RL-1}\int_{\prod\limits_{j=1}^{\RL-1}\CV_j}
\left(\int_{t'_{\ell+1}}^{t'_\ell-\ka_0\eps}+\int_{t'_\ell-\ka_0\eps}^{t'_\ell}\right) \int_{\R^3} {\bf k}_{w}(\FV^{\ell}_{\mathbf{cl}}(s';s,Z^{l}_{\mathbf{cl}}(s),\Fv'),\Fv'')\notag\\
&\quad\times
|(w^{l_\infty}\hat{f}_{R,2})(s',Z^{\ell}_{\mathbf{cl}}(s';s,Z^{l}_{\mathbf{cl}}(s),\Fv'),\Fv'')|
d\Fv''\,d\Sigma_{\ell}(s')ds'd\Fv'\,d\Sigma_{l}(s)ds\notag\\
:=& \CJ_1+\CJ_2.\notag
\end{align}
It should be pointed out that we always assume $t-t_{\RL}\leq T_0$ and $t'-t'_{\RL}\leq T_0$ in the following, because if $t-t_{\RL}> T_0$ or $t'-t_{\RL}'> T_0$ , the small contribution can be obtained as the way of showing \eqref{tdecay-sm}.

For $\CJ_2$, by Lemma \ref{Kop}, one has
\begin{align}
\CJ_2\leq& \frac{\ka_0C\RL}{\eps}e^{-\int_{t_1}^t\frac{\nu(\FV(\tau))}{2\eps}d\tau}\sum\limits_{l=1}^{\RL-1}\int_{\prod\limits_{j=1}^{\RL-1}\CV_j}
\int_{t_{l+1}}^{t_l}d\Sigma_{l}(s)ds
\notag\\
&\quad \times\int_{\Pi _{j=1}^{\RL-1}\mathcal{V}_{j}}
\Pi _{\ell=1}^{\RL-1}e^{\frac{\eta_0}{2}|\Fv'_{\ell}|^2}d\sigma _{\ell}\|w^{l_\infty}\hat{f}_{R,2}\|_{\infty}
\notag
\\
\leq& 
\frac{\ka_0C\RL^2}{\eps}
\int_{t_{l+1}}^{t_l}e^{-\frac{\nu_0(t-s)}{2\eps}}ds\int_{\Pi _{j=1}^{\RL-1}\mathcal{V}_{j}}
\Pi _{l=1}^{\RL-1}e^{\frac{\eta_0}{2}|\Fv_{l}|^2}d\sigma _{l}\|w^{l_\infty}\hat{f}_{R,2}\|_{\infty}
\notag
\\
\leq& 
 \ka_0C\RL^2\|w^{l_\infty}\hat{f}_{R,2}\|_{\infty}.\label{J2}
\end{align}
Note that $\ka_0\eps$ is chosen so that one can eliminate $\eps-$singularity of order one.

For  $\CJ_1$, the computation is divided into the following three cases.

\medskip
\noindent\underline{{\it Case 1.}}\\
\underline{{\it $|\FV^{l}_{\mathbf{cl}}(s;\Fv)|>M$ or $|\FV^{\ell}_{\mathbf{cl}}(s';Z^{l}_{\mathbf{cl}}(s),\Fv')|> M$.}} \\
In this case, by  Lemma \ref{Kop}, it follows that
\begin{align}
\int_{\R^3}|{\bf k}_{w}(\FV^{l}_{\mathbf{cl}}(s;\Fv),\Fv')|d\Fv',\ \textrm{or}\ \int_{\R^3}|{\bf k}_{w}(\FV^{\ell}_{\mathbf{cl}}(s';Z^{l}_{\mathbf{cl}}(s),\Fv'),\Fv'')|d\Fv''\leq\frac{C(l^\infty)}{1+M}.\notag
\end{align}
Therefore, one has by performing the similar calculations as for obtaining \eqref{J2} that
\begin{align}
|\CJ_1|\leq\frac{C(l^\infty)}{1+M}\|w^{l_\infty}\hat{f}_{R,2}\|_{\infty}.\notag
\end{align}
\noindent\underline{{\it Case 2.}}\\
\underline{{\it $|\FV^{l}_{\mathbf{cl}}(s;\Fv)|\leq M$ and $|\Fv'|> 2M$,}}\\
\underline{{\it or $|\FV^{\ell}_{\mathbf{cl}}(s';Z^{l}_{\mathbf{cl}}(s),\Fv')|\leq M$ and $|\Fv''|> 2M$.}} \\
In this regime, we have either $|\FV^{l}_{\mathbf{cl}}(s;\Fv)-\Fv'|>M$ or $|\FV^{\ell}_{\mathbf{cl}}(s';Z^{l}_{\mathbf{cl}}(s),\Fv')-\Fv''|>M$. Then either of the following two estimates holds correspondingly
\begin{equation*}
\begin{split}
&{\bf k}_{w}(\FV^{l}_{\mathbf{cl}}(s;\Fv),\Fv')
\leq Ce^{-\frac{\vps M^2}{16}}{\bf k}_{w}(\FV^{l}_{\mathbf{cl}}(s;\Fv),\Fv')e^{\frac{\vps |\FV^{l}_{\mathbf{cl}}-\Fv'|^2}{16}},\\
&{\bf k}_{w}(\FV^{\ell}_{\mathbf{cl}}(s';Z^{l}_{\mathbf{cl}}(s),\Fv'),\Fv'')
\leq Ce^{-\frac{\vps M^2}{16}}{\bf k}_{w}(\FV^{\ell}_{\mathbf{cl}}(s';Z^{l}_{\mathbf{cl}}(s),\Fv'),\Fv'')e^{\frac{\vps |\FV^{\ell}_{\mathbf{cl}}-\Fv''|^2}{16}}.
\end{split}
\end{equation*}
This together with Lemma \ref{Kop} implies
\begin{align}
|\CJ_1|\leq C(l^\infty)e^{-\frac{\vps M^2}{16}}\|w^{l_\infty}\hat{f}_{R,2}\|_{\infty}.\notag
\end{align}
\noindent\underline{{\it Case 3.}}\\
\noindent\underline{{\it $|\FV^{l}_{\mathbf{cl}}(s;\Fv)|\leq M$, $|\Fv'|\leq2M$, $|\FV^{\ell}_{\mathbf{cl}}(s';Z^{l}_{\mathbf{cl}}(s),\Fv')|\leq M$ and $|\Fv''|\leq 2M$.}} \\The key point here is to convert  the $L^1$ integral with respect to the double $\Fv$ variables into the $L^2$ norm with respect to the variables $z$ and $\Fv$. To do so, for any large $N>0$,
we choose a number $M(N)$ to define ${\bf k}_{w,M}(\Fv,\Fv')$ as
\begin{align}
\label{kw-M}
{\bf k}_{w,M}(\Fv,\Fv')
&\equiv \mathbf{1}
_{|\Fv-\Fv^{\prime }|\geq \frac{1}{M},|\Fv^{\prime}|\leq 2M}{\bf k}_{w}(\Fv,\Fv'), {\bf k}_{w,M}(\Fv',\Fv'')
 \notag\\
&\equiv \mathbf{1}
_{|\Fv'-\Fv''|\geq \frac{1}{M},|\Fv''|\leq 2M}{\bf k}_{w}(\Fv',\Fv''),
\end{align}
and then decompose
\begin{align*}
{\bf k}_{w}(\FV^{l}_{\mathbf{cl}}(s),\Fv')&
{\bf k}_{w}(\FV^{\ell}_{\mathbf{cl}}(s'),\Fv^{\prime \prime}) 
\\&
=\{{\bf k}_{w}(\FV^{l}_{\mathbf{cl}}(s),\Fv^{\prime})-{\bf k}_{w,M}(\FV^{l}_{\mathbf{cl}}(s),\Fv^{\prime})\}{\bf k}_{w}(\FV^{\ell}_{\mathbf{cl}},\Fv^{\prime \prime})
\\&
+\{{\bf k}_{w}(\FV^{\ell}_{\mathbf{cl}}(s'),\Fv^{\prime \prime})
-{\bf k}_{w,M}(\FV^{\ell}_{\mathbf{cl}}(s'),\Fv^{\prime \prime})\}{\bf k}_{w,M}(\FV^{l}_{\mathbf{cl}},\Fv')\notag\\&+{\bf k}_{w,M}(\FV^{l}_{\mathbf{cl}}(s),\Fv^{\prime}){\bf k}_{w,M}(\FV^{\ell}_{\mathbf{cl}}(s'),\Fv^{\prime \prime}).
\end{align*}
As for deducing \eqref{J2}, the first two difference terms lead to a small
contribution in $\CJ_1$ as
\begin{equation}
\frac{C(l^\infty)}{M}\|w^{l_\infty}\hat{f}_{R,2}\|_{\infty}.  \notag
\end{equation}
For the remaining main contribution of the bounded product corresponding to the last term,
we denote
\begin{align}\label{z-ti}
\tilde{z}
&=Z^{\ell}_{\mathbf{cl}}(s';s,Z^{l}_{\mathbf{cl}}(s),\Fv')=\mathbf{1}_{[t'_{\ell+1},t'_{\ell})}(s')
Z(s';t'_\ell,z'_\ell,\Fv'_\ell)
\notag\\ 
&=z_\ell'+(s'-t_\ell')v'_{\ell,z}
+\eps^2\int_{t_\ell'}^{s'}\int_{t_\ell'}^{\tau'}\Phi_{z}(Z(\eta'))d\eta' d\tau',
\end{align}
where $Z(\eta')=Z(\eta';t'_\ell,z'_\ell,\Fv'_\ell)$.
Applying a change of variable $v'_{\ell,z}\rightarrow\tilde{z}$, one gets
\begin{align}
\left|\frac{\pa \tilde{z}}{\pa v'_{\ell,z}}\right|=\left|\frac{\pa(z_\ell'-(t_\ell'-s')v'_{\ell,z})}{\pa v'_{\ell,z}}+\eps^2\int_{t_\ell'}^{s'}\int_{t_\ell'}^{\tau'}\frac{\pa[\Phi_{z}(Z(\eta'))]}{\pa v'_{\ell,z} }d\eta' d\tau'\right|.
\label{Jco1}
\end{align}
On the other hand, from \eqref{z-ti}, we have
\begin{align}
\frac{\pa Z(\eta')}{\pa v'_{\ell,z}}=(\eta'-t_\ell')
+\eps^2\int_{t_\ell'}^{\eta'}\int_{t_\ell'}^{\tau'}\frac{d\Phi_{z}(Z(\eta''))}{dZ}\frac{\pa Z(\eta'')}{\pa v'_{\ell,z}}d\eta'' d\tau',\notag
\end{align}
which together with Gronwall's inequality gives
\begin{align}
\left|\frac{\pa Z(\eta')}{\pa v'_{\ell,z}}\right|\leq |\eta'-t_\ell'|e^{\eps^2C(\Phi)T_0^2}.
\notag
\end{align}
Plugging this into \eqref{Jco1} and noticing that $\eps>0$ can be small enough, we further have
\begin{align}
\left|\frac{\pa \tilde{z}}{\pa v'_{\ell,z}}\right|\gtrsim |t_\ell'-s'|\geqslant \ka_0\eps,\label{Jco2}
\end{align}
due to $s'\in(t_{\ell+1}',t_{\ell}'-\ka_0\eps).$ We now compute the remaining delicate term as follows
\begin{align}
\frac{C}{\eps^2}&e^{-\int_{t_1}^t\frac{\nu(\FV(\tau))}{2\eps}d\tau}\notag\\&\qquad \times
\sum\limits_{l=1}^{\RL-1}\int_{\prod\limits_{j=1}^{\RL-1}\CV_j}
\int_{t_{l+1}}^{t_l}\int_{|\Fv'|\leqslant2M}{\bf k}_{w,M}(\FV^{l}_{\mathbf{cl}}(s;\Fv),\Fv')e^{-\int_{t_1'}^s\frac{\nu(\FV(\tau))}{2\eps}d\tau}
\notag\\
&\qquad\times\sum\limits_{\ell=1}^{\RL-1}\int_{\prod\limits_{j=1}^{\RL-1}\CV'_j}
\int_{t'_{\ell+1}}^{t'_\ell-\ka_0\eps} \int_{|\Fv''|\leqslant2M} {\bf k}_{w,M}(\FV^{\ell}_{\mathbf{cl}}(s';s,Z^{l}_{\mathbf{cl}}(s),\Fv'),\Fv'')
\notag\\&\qquad \times
|(w^{l_\infty}\hat{f}_{R,2})(s',Z^{\ell}_{\mathbf{cl}}(s';s,Z^{l}_{\mathbf{cl}}(s),\Fv'),\Fv'')|
d\Fv''\,d\Sigma'_{\ell}(s')ds'd\Fv'\,d\Sigma_{l}(s)ds\notag\\
\lesssim&\frac{C(l^\infty,M)}{\eps^2}e^{-\int_{t_1}^t\frac{\nu(\FV(\tau))}{2\eps}d\tau}\sum\limits_{l=1}^{\RL-1}
\int_{\prod\limits_{j=1}^{\RL-1}\CV_j}
\int_{t_{l+1}}^{t_l}\int_{|\Fv'|\leqslant2M}e^{-\int_{t_1'}^s\frac{\nu(\FV(\tau))}{2\eps}d\tau}
\notag\\& \qquad \times \sum\limits_{\ell=1}^{\RL-1}\int_{\prod\limits_{j=1,j\neq\ell}^{\RL-1}\CV'_j}
\int_{t'_{\ell+1}}^{t'_\ell-\ka_0\eps}\prod\limits_{j=\ell+1}^{\RL-1}d\si'_j\int_{\CV_\ell}\int_{|\Fv''|\leqslant2M} \notag\\
&\qquad\times|\hat{f}_{R,2}(s',Z^{\ell}_{\mathbf{cl}}(s';s,Z^{l}_{\mathbf{cl}}(s),\Fv'),\Fv'')|
d\Fv'' e^{-\int_{s'}^{t'_\ell}
\CA^\eps(\tau,\FV_{\mathbf{cl}}^\ell(\tau))d\tau}{w_2}(\Fv'_\ell)d\si'_\ell\notag\\ &\qquad \times\prod\limits_{j=1}^{\ell-1}e^{\frac{\eta_0}{2}|\Fv'_{j}|^2}
e^{-\int_{t'_{j+1}}^{t'_j}
\CA^\eps(\tau,\FV_{\mathbf{cl}}^j(\tau))d\tau}d\si'_j
ds'd\Fv'\,d\Sigma_{l}(s)ds
\notag\\
\lesssim&\frac{C(l^\infty,M)}{\eps^2}
\int_{t_{l+1}}^{t_l}e^{-\frac{\nu_0(t-s)}{2\eps}}ds
\int_{\CV'_\ell}\int_{t'_{\ell+1}}^{t'_\ell-\ka_0\eps}e^{-\frac{\nu_0(s-s')}{2\eps}}\notag\\
&\qquad\times\int_{|\Fv''|\leqslant2M} |\hat{f}_{R,2}(s',Z^{\ell}_{\mathbf{cl}}(s';s,Z^{l}_{\mathbf{cl}}(s),\Fv'),\Fv'')|
d\Fv''{w_2}(\Fv'_\ell)d\si'_\ell ds'\notag\\
\lesssim&\frac{C(l^\infty,M)}{\eps^2}
\int_{t_{l+1}}^{t_l}e^{-\frac{\nu_0(t-s)}{2\eps}}ds
\int_{v'_{\ell,z}>0,\ |\Fv'_\ell|\leq M}\int_{t'_{\ell+1}}^{t'_\ell-\ka_0\eps}e^{-\frac{\nu_0(s-s')}{2\eps}}\notag\\&\qquad\times\int_{|\Fv''|\leqslant2M} |\hat{f}_{R,2}(s',Z^{\ell}_{\mathbf{cl}}(s';s,Z^{l}_{\mathbf{cl}}(s),\Fv'),\Fv'')|
d\Fv''d{w_2}(\Fv'_\ell)\si'_\ell ds'\notag\\
&+\frac{C(l^\infty,M)}{\eps^2}
\int_{t_{l+1}}^{t_l}e^{-\frac{\nu_0(t-s)}{2\eps}}ds
\int_{v'_{\ell,z}>0,\ |\Fv'_\ell|\geq M}\int_{t'_{\ell+1}}^{t'_\ell-\ka_0\eps}e^{-\frac{\nu_0(s-s')}{2\eps}}\notag\\&\qquad\times\int_{|\Fv''|\leqslant2M} |\hat{f}_{R,2}(s',Z^{\ell}_{\mathbf{cl}}(s';s,Z^{l}_{\mathbf{cl}}(s),\Fv'),\Fv'')|
d\Fv''d{w_2}(\Fv'_\ell)\si'_\ell ds'\notag\\
\lesssim&\frac{C(l^\infty,M)}{\eps^2}
\int_{t_{l+1}}^{t_l}e^{-\frac{\nu_0(t-s)}{2\eps}}ds
\int_{v'_{\ell,z}>0,\ |\Fv'_\ell|\leq M'}\int_{t'_{\ell+1}}^{t'_\ell-\ka_0\eps}e^{-\frac{\nu_0(s-s')}{2\eps}}\notag\\&\qquad\times\int_{|\Fv''|\leqslant2M} |\hat{f}_{R,2}(s',Z^{\ell}_{\mathbf{cl}}(s';s,Z^{l}_{\mathbf{cl}}(s),\Fv'),\Fv'')|
d\Fv''{w_2}(\Fv'_\ell)d\si'_\ell ds'\notag\\
&+\frac{C(l^\infty,M)}{M'}\|\hat{f}_{R,2}\|_\infty\notag\\
\lesssim& C(l^\infty,M)
\int_{v'_{\ell,z}>0,\ |\Fv'_\ell|\leq M'}\int_{|\Fv''|\leqslant2M} |\hat{f}_2(s',Z^{\ell}_{\mathbf{cl}}(s';s,Z^{l}_{\mathbf{cl}}(s),\Fv'),\Fv'')|
d\Fv''\notag\\
&\qquad\times\sqrt{2\pi\mu(\Fv_\ell')}v'_{\ell,z}d\Fv'_{\ell} \notag\\
&+\frac{C(l^\infty,M)}{M'}\|\hat{f}_{R,2}\|_\infty.
\notag
\end{align}
Utilizing H\"older's inequality and \eqref{Jco2}, we see that the above bound is further dominated by
\begin{align}
C(l^\infty,M)&
\left\{\int_{v'_{\ell,z}>0,\ |\Fv'_\ell|\leq M'}\int_{|\Fv''|\leqslant2M} |\hat{f}_{R,2}(s',Z^{\ell}_{\mathbf{cl}}(s'),\Fv'')|^2
d\Fv''d\Fv'_{\ell}\right\}^{\frac{1}{2}}
\notag \\ \notag
& 
+\frac{C(l^\infty,M)}{M'}\|\hat{f}_{R,2}\|_\infty\\
\lesssim &
\frac{C(l^\infty,M)}{\eps^{\frac{1}{2}}}
\left\{\int_{-1}^1\int_{|\Fv''|\leqslant2M} |\hat{f}_{R,2}(s',\tilde{z},\Fv'')|^2
d\Fv''d\tilde{z} \right\}^{\frac{1}{2}}
\notag\\&+
\frac{C(l^\infty,M)}{M'}\|\hat{f}_{R,2}\|_\infty,\label{de-j2}
\end{align}
where $M'$ is a large positive constant.
Putting all the estimates above together, we see that \eqref{f2-lif-sum} is valid. Furthermore, \eqref{f2-lif-sum} together with Lemma \ref{coef-es} gives
\begin{align}\label{f2-sum}
\sum\limits_{\bar{k}\in\Z^2}\|w^{l_\infty}\hat{f}_{R,2}\|_{\infty}
\leq&
\eps^{-\frac{1}{2}}\sum\limits_{\bar{k}\in\Z^2}\|\hat{f}_{R,2}\|_2
+C\sum\limits_{\bar{k}\in\Z^2}\|w^{l_\infty}\hat{f}_{R,1}\|_{\infty}
\notag\\
& +C\eps^{\frac{3}{2}}\left(\sum\limits_{\bar{k}\in\Z^2}\|w^{l_\infty}\hat{f}_{R,2}\|_{\infty}\right)^2+C\vps_{\Phi,\al}\eps^{\frac{1}{2}}.
\end{align}
Consequently, we get from \eqref{f1-sum} and \eqref{f2-sum} that
\begin{align}\label{f1f2-sum}
\eps^{-\frac{3}{2}}&\sum\limits_{\bar{k}\in\Z^2}\|w^{l_\infty}\hat{f}_{R,1}\|_{\infty}
+\sqrt{\eps}\sum\limits_{\bar{k}\in\Z^2}\|w^{l_\infty}\hat{f}_{R,2}\|_{\infty}
\notag\\ \notag
\leq& C\sum\limits_{\bar{k}\in\Z^2}\|\hat{f}_{R,2}\|_2
+C\eps^2\left(\sum\limits_{\bar{k}\in\Z^2}\|w^{l_\infty}\hat{f}_{R,2}\|_{\infty}\right)^2
\\ 
&+C\sum\limits_{\bar{k}\in\Z^2}\|w^{l_\infty}\hat{f}_{R,1}\|_{\infty}
\left\{\sum\limits_{\bar{k}\in\Z^2}\|w^{l_\infty}\hat{f}_{R,2}\|_{\infty}
+\sum\limits_{\bar{k}\in\Z^2}\|w^{l_\infty}\hat{f}_{R,1}\|_{\infty}\right\}\notag\\
&+C\vps_{\Phi,\al}\eps.
\end{align}
This is the end of the second step. Next, we need to close our final estimates, that's to obtain the $L^1_{\bar{k}}L^2_{z,\Fv}$ estimates on $f_{R,1}$ and $f_{R,2}$.

\subsection{\texorpdfstring{$L^1_{\bar{k}}L^2_{z,\Fv}$}{Lg} estimate on \texorpdfstring{$f_{R,2}$}{Lg}}
In this step, we intend to obtain the $L^1_{\bar{k}}L^2_{z,\Fv}$ control of $f_{R,2}$. Taking the inner product of \eqref{f2-ft} with $\hat{f}_{R,2}$ over $(\bar{k},z,\Fv)\in\Z^2\times(-1,1)\times\R^3$ and taking the real part of the identity, one has
\begin{align}\label{f2-ip}
\eps^{-1}Re(v_{z} &\pa_{z}\hat{f}_{R,2},\hat{f}_{R,2})
+\frac{1}{\eps^2}Re (L\hat{f}_{R,2},\hat{f}_{R,2})\notag\\
=&Re\left(\eps^{-2}(1-\chi_M)\mu^{-1/2}\CK \hat{f}_{R,1}-\frac{\al  v_xv_{z}}{2}\FP \hat{f}_{R,2},\hat{f}_{R,2}\right)
\notag\\&+\eps^{-1}Re\left(\{\hat{\Ga}(\hat{f}_{R,2},\hat{f}_1+\eps \hat{f}_2)+\Ga(\hat{f}_1+\eps \hat{f}_2,\hat{f}_{R,2})\},\hat{f}_{R,2}\right)
\notag\\&
+\eps^{-\frac{1}{2}}Re(\hat{\Ga}(\hat{f}_{R,2},\hat{f}_{R,2}),\hat{f}_{R,2}).
\end{align}
Next, using Lemma \ref{es-tri} and
the generalized Minkowski's inequality \eqref{gm-ine}, one has
\begin{align}\label{tri-es-ga}
&\eps^{-\frac{1}{2}}|(\hat{\Ga}(\hat{f}_{R,2},\hat{f}_{R,2}),\hat{f}_{R,2})|\notag\\
\lesssim&
\frac{1}{\eps^2}\|\{\FI-\FP\}\hat{f}_{R,2}\|^2_{\nu}
+
\eps\int_{-1}^1\left(\sum_{\bar{l}}\|\hat{f}_{R,2}(\bar{k}-\bar{l})\|_2
\|\hat{f}_{R,2}(\bar{l})\|_\nu\right)^2dz \notag\\
\lesssim
&
\frac{1}{\eps^2}\|\{\FI-\FP\}\hat{f}_{R,2}\|^2_{\nu}
+
\eps\left\{\sum_{\bar{l}}\|w^{l_\infty}\hat{f}_{R,2}(\bar{k}-\bar{l})\|_{\infty}\|\hat{f}_{R,2}(\bar{l})\|_{\nu}\right\}^2.
\end{align}
Thus, \eqref{tri-es-ga}, \eqref{f2-ip} and \eqref{f2-ft-bd} give
\begin{align}
\frac{\de_0}{\eps^2}\|\{\FI-\FP\}&\hat{f}_{R,2}\|^2_\nu+\frac{1}{2\eps}|\{\FI-P_\ga\}\hat{f}_{R,2}|_{2,+}^2
\notag\\
\lesssim& \frac{C_\eta}{\eps^4}\|\hat{f}_{R,1}\|_2^2+(\eta+C\al)\|\hat{f}_{R,2}\|_2^2
\notag\\&
+C\eps^{-1}Re\left(\{\hat{\Ga}(\hat{f}_{R,2},\hat{f}_1+\eps \hat{f}_2)+\Ga(\hat{f}_1+\eps \hat{f}_2,\hat{f}_{R,2})\},\hat{f}_{R,2}\right)
\notag\\&
+\eps^{-\frac{1}{2}}|Re(\hat{\Ga}(\hat{f}_{R,2},\hat{f}_{R,2}),\hat{f}_{R,2})|+\eta|P_\ga\hat{f}_{R,2}|_{2,+}^2
\notag\\&
+\frac{C_\eta}{\eps^2}\|w^{l_\infty}\hat{f}_{R,1}\|_\infty^2+C|\hat{r}|_{2,-}^2
\notag\\
\lesssim& \frac{C_\eta}{\eps^4}\|\hat{f}_{R,1}\|_2^2+(\eta+C\al)\|\hat{f}_{R,2}\|_2^2
+\frac{\eta}{\eps^2}\|\{\FI-\FP\}\hat{f}_{R,2}\|^2_\nu
\notag\\&
+C_\eta\left\{\sum\limits_{\bar{l}\in Z^2}\|\hat{f}_{R,2}(\bar{k}-\bar{l})\|_\nu\left\{\|w^{l_\infty}\hat{f}_{1}(\bar{l})\|_{\infty}
+\|w^{l_\infty}\hat{f}_{2}(\bar{l})\|_{\infty}\right\}\right\}^2
\notag\\
&+\eps\left\{\sum_{\bar{l}\in\Z^2}\|w^{l_\infty}\hat{f}_{R,2}(\bar{k}-\bar{l})\|_{\infty}\|\hat{f}_{R,2}(\bar{l})\|_{\nu}\right\}^2
+\eta|P_\ga\hat{f}_{R,2}|_{2,+}^2
\notag\\&
+\frac{C_\eta}{\eps^2}\|w^{l_\infty}\hat{f}_{R,1}\|_\infty^2+C|\hat{r}|_{2,-}^2,\label{f2-ip-2}
\end{align}
where the following estimate on the boundary term has been used
\begin{align}
\int_{\R^3}&v_{z}\hat{f}_{R,2}^2(1)d\Fv-\int_{\R^3}v_{z}\hat{f}_{R,2}^2(-1)d\Fv\notag\\=&\int_{v_{z}>0}v_{z}\hat{f}_{R,2}^2(1)d\Fv
-\int_{v_{z}<0}v_{z}( P_\ga\hat{f}_{R,2}+\bar{P}_{\ga}\hat{f}_{R,1}+\eps^{\frac{1}{2}}\hat{r})^2(1)d\Fv
\notag\\&-\int_{v_{z}<0}v_{z}\hat{f}_2^2(-1)d\Fv
-\int_{v_{z}>0}v_{z}( P_\ga\hat{f}_{R,2}+\bar{P}_{\ga}\hat{f}_{R,1}+\eps^{\frac{1}{2}}\hat{r})^2(-1)d\Fv\notag\\
\geq& \int_{v_{z}>0}v_{z}(\{\FI-P_\ga\}\hat{f}_{R,2})^2(1)d\Fv
+\int_{v_{z}<0}|v_{z}|(\{\FI-P_\ga\}\hat{f}_{R,2})^2(-1)d\Fv
\notag\\&-\eta\eps\int_{v_{z}<0}|v_{z}|(P_\ga\hat{f}_{R,2})^2(1)d\Fv
-\eta\eps\int_{v_{z}>0}|v_{z}|(P_\ga\hat{f}_{R,2})^2(-1)d\Fv
\notag\\
&-\frac{C_\eta}{\eps}\int_{v_{z}<0}|v_{z}||\bar{P}_{\ga}\hat{f}_{R,1}(1)|^2d\Fv
-\frac{C_\eta}{\eps}\int_{v_{z}>0}|v_{z}||\bar{P}_{\ga}\hat{f}_{R,1}(-1)|^2d\Fv\notag\\
&-C\eps|\hat{r}|^2_{2,-}\notag\\
\geq& |\{\FI-P_\ga\}\hat{f}_{R,2}|^2_{2,+}-\eta\eps|P_\ga\hat{f}_{R,2}|^2_{2,+}
-\frac{C_\eta}{\eps}\|w^{l_\infty}\hat{f}_{R,1}\|^2_{\infty}-C\eps|\hat{r}|^2_{2,-}.
\notag
\end{align}
Moreover, we have also used the notation
$$
\bar{P}_{\ga}\hat{f}_{R,1}(\pm1)=\sqrt{2\pi\mu}\int_{v_{z}\gtrless0}\hat{f}_{R,1}(\pm1)|v_{z}|d\Fv,
$$
and the orthogonality relation \eqref{r-bd} as well as the estimate
\begin{align*}
\left|\int_{v_{z}\gtrless0}\hat{f}_{R,1}(\pm1)|v_{z}|d\Fv\right|\leq C\|w^{l_\infty}\hat{f}_{R,1}\|_{\infty},\ \textrm{for}\ l_\infty>5/2.
\end{align*}
One the other hand, note that
\begin{align}\label{pga-f2}
|P_{\gamma }\hat{f}_{R,2}(\pm1)|_{2,\pm }^{2} =&\int_{v_{z}\gtrless0}\left[ \int_{\{v_{z}\gtrless0\}}\hat{f}_{R,2}(\pm1)\sqrt{\mu }|v_{z}|d\Fv\right] ^{2}2\pi\mu |v_{z}| d\Fv
 \notag\\ 
=&\left[ \int_{\{v_{z}\gtrless0\}}\hat{f}_{R,2}(\pm1)\sqrt{\mu }|v_{z}|d\Fv\right] ^{2}.
\end{align}
Then, by dividing the integral domain as
\begin{multline}\nonumber
\{\Fv\in\R^{3}:v_{z}>0\} \\
=\underbrace{\{\Fv\in \R^{3}:0<v_{z}<\varepsilon \ \text{or}\ v_{z}>1/\varepsilon\}}_{V^\vps}
\cup \{\Fv\in \R^{3}:\varepsilon \leq v_{z}\leq 1/\varepsilon \},
\end{multline}%
one sees that
the grazing part of
$|P_{\gamma}\hat{f}_{R,2}(1)|_{2,+}^{2}$ is bounded by the H\"{o}lder inequality as
\begin{equation}
\begin{split}
 \Bigg(\int_{V^\vps}\mu (\Fv)|v_{z}|d\Fv\Bigg)&
\int_{v_{z}>0}|\hat{f}_{R,2}(1)|^{2}v_{z}d\Fv  \lesssim \vps\int_{v_{z}>0}|\hat{f}_{R,2}(1)|^{2}v_{z}d\Fv.
\end{split}
\label{trace1}
\end{equation}
For non-grazing region, we have by using Lemmas \ref{ukai} and \ref{es-tri} that
\begin{align}\label{trace2}
&\int_{\{\Fv\in \R^{3}:v_{z}>0\}\backslash V^\vps}|\hat{f}_{R,2}(1)|^{2}v_{z}d\Fv
\notag\\
\leq& 
C \|\hat{f}_{R,2}^2\|_{L^1} +C\|Re(\{v_{z}\pa _{z}+\eps^2\Phi(z)\cdot\na_\Fv-\al\eps v_{z}\pa_{v_x}\}\hat{f}_{R,2}|\hat{f}_{R,2})\| _{L^1}
\notag\\
\leq& 
C\|\hat{f}_{R,2}\|_2^{2}+\frac{C}{\eps}|(L\hat{f}_{R,2},\hat{f}_{R,2})|
\notag\\&
+\left|Re\left(\eps^{-1}(1-\chi_M)\mu^{-1/2}\CK \hat{f}_{R,1}-\frac{\al\eps  v_xv_{z}}{2}\FP \hat{f}_{R,2},\hat{f}_{R,2}\right)\right|
\notag\\&
+\left|Re\left(\{\hat{\Ga}(\hat{f}_{R,2},\hat{f}_1+\eps \hat{f}_2)+\hat{\Ga}(\hat{f}_1+\eps \hat{f}_2,\hat{f}_{R,2})\},\hat{f}_{R,2}\right)\right|
\notag\\&
+\eps^{\frac{1}{2}}\left|Re(\hat{\Ga}(\hat{f}_{R,2},\hat{f}_{R,2}),\hat{f}_{R,2})\right|\notag\\
\leq& C\|\hat{f}_{R,2}\|_2^{2}+\frac{C}{\eps}\|\{\FI-\FP\}\hat{f}_{R,2}\|^2_\nu+\frac{C}{\eps^2}\|w^{l_2}\hat{f}_{R,1}\|_{2}^{2}
\notag\\&+\eps^2\left\{\sum_{\bar{l}}\|w^{l_\infty}\hat{f}_{R,2}(\bar{k}-\bar{l})\|_{\infty}\|\hat{f}_{R,2}(\bar{l})\|_{\nu}\right\}^2.
\end{align}
As a consequence, from \eqref{pga-f2}, \eqref{trace1} and \eqref{trace2}, we arrive at
\begin{align}\label{pga-f2-2}
|P_{\gamma }\hat{f}_{R,2}(\pm1)|_{2,\pm }^{2}\leq& \vps\int_{v_{z}\gtrless0}|\hat{f}_{R,2}(\pm1)|^{2}|v_{z}|d\Fv
+C\|\hat{f}_{R,2}\|_2^{2}\notag\\&+\frac{C}{\eps}\|\{\FI-\FP\}\hat{f}_{R,2}\|^2_\nu
+\frac{C}{\eps^2}\| w^{l_2}\hat{f}_{R,1}\|_{2}^{2}\notag\\&+\eps^2\left\{\sum_{\bar{l}}\|w^{l_\infty}\hat{f}_{R,2}(\bar{k}-\bar{l})\|_{\infty}\|\hat{f}_{R,2}(\bar{l})\|_{\nu}\right\}^2.
\end{align}

Next, letting $1\gg\ka_0\gg \eta>0$, we get from the summation of $\ka_0\times\eqref{pga-f2-2}$ and \eqref{f2-ip-2} that
\begin{align}
\frac{1}{\eps^2}\|&\{\FI-\FP\}\hat{f}_{R,2}\|^2_\nu+\frac{1}{\eps}|\{\FI-P_\ga\}\hat{f}_{R,2}|_{2,+}^2
+\frac{\ka_0}{2}|P_{\gamma }\hat{f}_{R,2}(\pm1)|_{2,\pm }^{2}
\notag\\
\lesssim& 
\frac{C_\eta}{\eps^4}\|w^{l_2}\hat{f}_{R,1}\|_2^2+(\eta+C\al+\ka_0)\|\hat{f}_{R,2}\|_2^2
\notag\\&
+C_\eta\left\{\sum\limits_{\bar{l}\in Z^2}\|\hat{f}_{R,2}(\bar{k}-\bar{l})\|_\nu\{\|w^{l_\infty}\hat{f}_{1}(\bar{l})\|_{\infty}
+\|w^{l_\infty}\hat{f}_{2}(\bar{l})\|_{\infty}\}\right\}^2
\notag\\&
+\eps\left\{\sum_{\bar{l}\in\Z^2}\|w^{l_\infty}\hat{f}_{R,2}(\bar{k}-\bar{l})\|_{\infty}\|\hat{f}_{R,2}(\bar{l})\|_{\nu}\right\}^2+
\frac{C_\eta}{\eps^2}\|w^{l_\infty}\hat{f}_{R,1}\|_\infty^2+C|\hat{r}|^2_{2,-},\notag
\end{align}
which further implies
\begin{align}
\frac{1}{\eps}\sum\limits_{\bar{k}\in\Z^2}&\|\{\FI-\FP\}\hat{f}_{R,2}\|_\nu+\frac{1}{\sqrt{\eps}}\sum\limits_{\bar{k}\in\Z^2}|\{\FI-P_\ga\}\hat{f}_{R,2}|_{2,+}
\notag\\&+\sqrt{\frac{\ka_0}{2}}\sum\limits_{\bar{k}\in\Z^2}|P_{\gamma }\hat{f}_{R,2}(\pm1)|_{2,\pm }\notag\\
\lesssim& \frac{C_\eta}{\eps^2}\sum\limits_{\bar{k}\in\Z^2}\|w^{l_2}\hat{f}_{R,1}\|_2
+\sqrt{(\eta+C\al+\ka_0)}\sum\limits_{\bar{k}\in\Z^2}\|\hat{f}_{R,2}\|_2
\notag\\&+C_\eta\vps_{\Phi,\al}\sum\limits_{\bar{k}\in Z^2}\|\hat{f}_{R,2}\|_\nu
+\sqrt{\eps}\sum_{\bar{k}\in\Z^2}\|w^{l_\infty}\hat{f}_{R,2}\|_{\infty}\sum_{\bar{k}\in\Z^2}\|\hat{f}_{R,2}\|_{\nu}\notag\\&
+\frac{C_\eta}{\eps}\sum_{\bar{k}\in\Z^2}\|w^{l_\infty}\hat{f}_{R,1}\|_\infty
+C\sum_{\bar{k}\in\Z^2}|\hat{r}|^2_{2,-}.\label{f2-ip-4}
\end{align}
Finally, combing \eqref{abc-2-es} and \eqref{f2-ip-4} and utilizing Lemma \ref{coef-es}, we conclude that
\begin{align}\label{f2-l2-sum}
&\eps^{-1}\sum\limits_{\bar{k}\in\Z^2}\|\{\FI-\FP\}\hat{f}_{R,2}\|_\nu
+\sum\limits_{\bar{k}\in\Z^2}\|[\hat{a}^{(2)}_s,\hat{\Fb}^{(2)}_s,\hat{c}^{(2)}_s\|_2
\notag\\
&
\quad+\frac{1}{\sqrt{\eps}}\sum\limits_{\bar{k}\in\Z^2}|\{\FI-P_\ga\}\hat{f}_{R,2}|_{2,+}
+\sum\limits_{\bar{k}\in\Z^2}|P_{\gamma }\hat{f}_{R,2}(\pm1)|_{2,\pm }\notag\\
&\lesssim \eps^{-2}\sum\limits_{\bar{k}\in\Z^2}\|w^{l_2}\hat{f}_{R,1}\|_2+
\eps^{-1}\sum_{\bar{k}\in\Z^2}\|w^{l_\infty}\hat{f}_{R,1}\|_\infty
\notag\\
&\quad
+\sum\limits_{\bar{k}\in\Z^2}\|w^{l_\infty}\hat{f}_{R,1}\|_\infty
\sum\limits_{\bar{k}\in\Z^2}\|\hat{f}_{R,2}\|_2\notag\\
&
\quad+\sqrt{\eps}\sum_{\bar{k}\in\Z^2}\|w^{l_\infty}\hat{f}_{R,2}\|_{\infty}\sum_{\bar{k}\in\Z^2}\|\hat{f}_{R,2}\|_{\nu}
+\left(\sum\limits_{\bar{k}\in\Z^2}\|w^{l_\infty}\hat{f}_{R,1}\|_\infty\right)^2+\vps_{\Phi,\al}.
\end{align}
We then end  the $L^1_{\bar{k}}L^2_{z,\Fv}$ control of $f_{R,2}$.  The rest is to obtain the $L^1_{\bar{k}}L^2_{z,\Fv}$ estimate on $f_{R,1}$.

\subsection{\texorpdfstring{$L^1_{\bar{k}}L^2_{z,\Fv}$}{Lg} estimate on \texorpdfstring{$f_{R,1}$}{Lg}}
To close the final $L^1_{\bar{k}}L^2_{z,\Fv}$ estimate for $f_{R,2}$, it is necessary to deduce the $L^1_{\bar{k}}L^2_{z,\Fv}$ estimate on $f_{R,1}$. To do this, we get from the inner product of \eqref{f1-ft} and $w^{2l_2}\hat{f}_{R,1}$ over $\Z^2\times\R^3$ and take the real part of the identity part of the resulting equality to deduce that
\begin{align}
\frac{1}{\eps}&\|w^{l_2}\hat{f}_{R,1}\|_\nu^2+
\int_{v_{z}>0}v_{z}w^{2l_2}\hat{f}_{R,1}^2(1)d\Fv-\int_{v_z<0}v_{z}w^{2l_2}\hat{f}_{R,1}^2(-1)d\Fv\notag\\
=&Re\left(\frac{1}{\eps}\chi_M\CK \hat{f}_{R,1},w^{2l_2}\hat{f}_{R,1}\right)
+Re\left(\frac{\eps^2}{2}\Phi(z)\cdot \Fv\sqrt{\mu}\hat{f}_{R,2},w^{2l_2}\hat{f}_{R,1}\right)
\notag\\&
-
Re\left(\frac{\al \eps v_xv_z}{2}\sqrt{\mu}\{\FI-\FP\}\hat{f}_{R,2},w^{2l_2}\hat{f}_{R,1}\right)
\notag\\&
-\eps^{\frac{5}{2}}Re\left(\Phi\cdot \na_\Fv(\sqrt{\mu}\hat{f}_{2}),w^{2l_2}\hat{f}_{R,1}\right)
\notag\\&
-
\al\eps^{\frac{3}{2}} Re\left(izk_1(\sqrt{\mu}\hat{f}_{2}),w^{2l_2}\hat{f}_{R,1}\right)+\al\eps^{\frac{3}{2}} Re\left(v_z\pa_{v_x}(\sqrt{\mu}\hat{f}_{2}),w^{2l_2}\hat{f}_{R,1}\right)
\notag\\&
+\eps^{\frac{1}{2}}Re\left(\hat{Q}(\hat{f}_{R,1},\hat{f}_{R,1}),w^{2l_2}\hat{f}_{R,1}\right)
\notag\\&
+\eps^{\frac{1}{2}}Re\left(\hat{Q}(\hat{f}_{R,1},\sqrt{\mu}\hat{f}_{R,2})
+\hat{Q}(\sqrt{\mu}\hat{f}_{R,2},\hat{f}_{R,1}),w^{2l_2}\hat{f}_{R,1}\right)
\notag\\&
+Re\left(\hat{Q}(\hat{f}_{R,1},\sqrt{\mu}\{\hat{f}_1+\eps \hat{f}_2\})
+\hat{Q}(\sqrt{\mu}\{\hat{f}_1+\eps \hat{f}_2\},\hat{f}_{R,1}),w^{2l_2}\hat{f}_{R,1}\right)
\notag\\&+\eps^{\frac{3}{2}}Re\left(\hat{Q}(\sqrt{\mu}\hat{f}_2,\sqrt{\mu}\hat{f}_2),w^{2l_2}\hat{f}_{R,1}\right).\notag
\end{align}
The above estimate together with Lemma \ref{CK-l2-lem} gives
\begin{align}
\frac{1}{\eps}\|w^{l_2}\hat{f}_{R,1}\|_\nu^2
\lesssim &\vps_{\Phi,\al}^2\eps^5\|\hat{f}_{R,2}\|^2_2
+\al^2\eps^3 \|\{\FI-\FP\}\hat{f}_{R,2}\|_2^2
+\eps^4\|(1+|\bar{k}|)\hat{f}_2\|_2^2
\notag\\&+\eps^2\|w^{l_\infty}\nu^{-1}\hat{Q}(\hat{f}_{R,1},\hat{f}_{R,1})\|^2_{\infty}
+\eps^2\|w^{l_\infty}\nu^{-1}\hat{Q}(\hat{f}_{R,2},\hat{f}_{R,1})\|^2_{\infty}
\notag\\&+\eps^2\|w^{l_\infty}\nu^{-1}\hat{Q}(\hat{f}_{R,1},\hat{f}_{R,2})\|^2_{\infty}\notag\\
&+\eps\|w^{l_\infty}\nu^{-1}\hat{Q}(\hat{f}_{R,1},\sqrt{\mu}\{\hat{f}_1+\eps \hat{f}_2\})\|^2_{\infty}\notag\\
&+\eps\|w^{l_\infty}\nu^{-1}\hat{Q}(\sqrt{\mu}\{\hat{f}_1+\eps \hat{f}_2\},\hat{f}_{R,1})\|^2_{\infty}
\notag\\&
+C_\eta\eps^{4}\|\nu^{-1}w^{l_\infty}\hat{Q}(\sqrt{\mu}\hat{f}_2,\sqrt{\mu}\hat{f}_2)\|_\infty^2,\notag
\end{align}
which further yields that
\begin{align}
\frac{1}{\eps}\|w^{l_2}\hat{f}_{R,1}\|_\nu^2
\lesssim &\vps_{\Phi,\al}^2\eps^5\|w^{l_\infty}\hat{f}_{R,2}\|^2_\infty
+\al^2\eps^3 \|\{\FI-\FP\}\hat{f}_{R,2}\|_2^2
+\eps^4\|(1+|\bar{k}|)\hat{f}_2\|_2^2
\notag\\&+\eps^2\left\{\sum_{\bar{l}\in\Z^2}\|w^{l_\infty}\hat{f}_{R,1}(\bar{k}-\bar{l})\|_{\infty}
\|w^{l_\infty}\hat{f}_{R,1}(\bar{l})\|_{\infty}\right\}^2
\notag\\&+\eps^2\left\{\sum_{\bar{l}\in\Z^2}\|w^{l_\infty}\hat{f}_{R,2}(\bar{k}-\bar{l})\|_\infty
\|w^{l_\infty}\hat{f}_{R,1}(\bar{l})\|_{\infty}\right\}^2
\notag\\
&+\eps\left\{\sum_{\bar{l}\in\Z^2}\{\|\hat{f}_{1}(\bar{k}-\bar{l})\|_\infty+\|\hat{f}_{2}(\bar{k}-\bar{l})\|_\infty\}
\|w^{l_\infty}\hat{f}_{R,1}(\bar{l})\|_{\infty}\right\}^2
\notag\\&+\eps^{4}\left\{\sum_{\bar{l}\in\Z^2}\|\hat{f}_{2}(\bar{k}-\bar{l})\|_\infty\|\hat{f}_{2}(\bar{l})\|_\infty\right\}^2.\notag
\end{align}
Therefore, it follows
\begin{align}\label{f1-l2-sum}
\eps^{-\frac{5}{2}}\sum\limits_{\bar{k}\in\Z^2}\|w^{l_2}\hat{f}_{R,1}\|_\nu
\lesssim &\vps_{\Phi,\al}\sqrt{\eps}\sum\limits_{\bar{k}\in\Z^2}\|w^{l_\infty}\hat{f}_{R,2}\|_\infty
+\al\eps^{-1}\sum\limits_{\bar{k}\in\Z^2}\|\{\FI-\FP\}\hat{f}_{R,2}\|_2
\notag\\&
+\vps_{\Phi,\al}
+\eps^{-1}\left\{\sum_{\bar{k}\in\Z^2}\|w^{l_\infty}\hat{f}_{R,1}\|_{\infty}\right\}^2
\notag\\&
+\eps^{-1}\sum_{Z\in\Z^2}\|w^{l_\infty}\hat{f}_{R,2}(\bar{k})\|_{\infty}\sum_{\bar{k}\in\Z^2}\|w^{l_\infty}\hat{f}_{R,1}\|_{\infty}
\notag\\
&+\vps_{\Phi,\al}\eps^{-\frac{3}{2}}\sum_{\bar{k}\in\Z^2}\|w^{l_\infty}\hat{f}_{R,1}(\bar{k})\|_{\infty}.
\end{align}
Finally, by putting \eqref{f1f2-sum}, \eqref{f2-l2-sum} and \eqref{f1-l2-sum} together, we conclude that
\begin{align}
\eps^{-\frac{3}{2}}\sum\limits_{\bar{k}\in\Z^2}&\|w^{l_\infty}\hat{f}_{R,1}\|_{\infty}+\eps^{-\frac{5}{2}}\sum\limits_{\bar{k}\in\Z^2}\|w^{l_2}\hat{f}_{R,1}\|_\nu
+\eps^{\frac{1}{2}}\sum\limits_{\hat{k}\in\Z^2}\|w^{l_\infty}\hat{f}_{R,2}\|_{\infty}
\notag\\&+\sum\limits_{\bar{k}\in\Z^2}\|\FP\hat{f}_{R,2}\|_2
+\eps^{-1}\sum\limits_{\bar{k}\in\Z^2}\|\{\FI-\FP\}\hat{f}_{R,2}\|_\nu\lesssim \vps_{\Phi,\al}.\notag
\end{align}
Hence, the desired a priori estimate \eqref{fr-es-lem} has been proved.
\qed
 


\section{Unsteady problem}\label{sec-usp}
In this section, we are concerned with the following time-evolutionary problem
\begin{align}\label{us-rbe}
\eps\pa_tF^\eps+\Fv\cdot\na_\Fx F^\eps+\eps^2\Phi\cdot\na_\Fv F^\eps+\al\eps z\pa_{x}F^\eps-\al\eps v_z\pa_{v_x}F^\eps=\frac{1}{\eps}Q(F^\eps,F^\eps),
\end{align}
for $t>0$, $\Fx\in\Om$, $\Fv\in\R^3$, where
\begin{align}\label{us-id}
F^\eps(0,\Fx,\Fv)=F_0^\eps(\Fx,\Fv),
\end{align}
and 
\begin{align}\label{us-obd}
F^\eps(t,x,y,\pm1,\Fv)|_{v_z\lessgtr0}=\sqrt{2\pi}\mu\int_{\tv_z\gtrless0}F^\eps(t,x,y,\pm1,\FTv)|\tv_z|d\FTv.
\end{align}

To solve \eqref{us-rbe}, \eqref{us-id} and \eqref{us-obd} around the obtained steady solution $F^\eps_{st}$, we set
\begin{align}
F^\eps &=F^\eps_{st}+\eps\sqrt{\mu}\{g_1+\eps g_2+\eps^{\frac{1}{2}}g_R\}
\notag \\
& =\mu+\eps\sqrt{\mu}\{f_1+\eps f_2+\eps^{\frac{1}{2}}f_R\}+\eps\sqrt{\mu}\{g_1+\eps g_2+\eps^{\frac{1}{2}}g_R\},
\notag
\end{align}
where the expansion \eqref{st-exp} for $F^\eps_{st}$ has been used.
Substituting the above into \eqref{us-rbe} and comparing the orders of $\eps$, one has the following equations for coefficients $g_1$ and $g_2$:
\begin{align}\label{g1}
Lg_1=0,
\end{align}
\begin{align}\label{mi-g2}
\Fv\cdot\na_\Fx g_1+Lg_2=\Ga(g_1,g_1)+\Ga(f_1,g_1)+\Ga(g_1,f_1),
\end{align}
and
\begin{align}\label{ev-g1}
\pa_tg_1&+\Fv\cdot\na_\Fx g_2+\al z\pa_{x}g_1-\al\mu^{-\frac{1}{2}}v_z\pa_{v_x}\{\sqrt{\mu}g_1\}\notag\\
=&\Ga(g_1,g_2)+\Ga(g_2,g_1)+\Ga(f_1,g_2)+\Ga(g_2,f_1)+\Ga(g_1,f_2)+\Ga(f_2,g_1),
\end{align}
as well as the equation for the remainder $g_R$:
\begin{align}
\eps\pa_tg_R&+\Fv\cdot\na_\Fx g_R+\eps^2\Phi\cdot\na_\Fv g_R+\al\eps z\pa_{x}g_R-\al\eps v_z\pa_{v_x}g_R
+\frac{1}{\eps}Lg_R
\notag\\
=&
\eps^{\frac{3}{2}}\pa_tg_2-\eps^{\frac{3}{2}}\al z\pa_xg_2+\eps^{\frac{3}{2}}\al\mu^{-\frac{1}{2}}v_z\pa_{v_x}\{\sqrt{\mu}g_2\}
\notag\\&-\eps^{\frac{3}{2}}\mu^{-\frac{1}{2}}\Phi\cdot\na_\Fv[\sqrt{\mu}(g_1+\eps g_2)]
+\frac{\eps^2}{2}\Phi\cdot\Fv g_R-\frac{\al\eps v_xv_z}{2}g_R\notag\\&+\eps^{\frac{1}{2}}\Ga(g_R,g_R)
+\Ga(g_R,g_1+\eps g_2)
+\Ga(g_1+\eps g_2,g_R)+\eps^{\frac{3}{2}}\Ga(g_2,g_2)\notag\\&+\Ga(g_R,f_1+\eps f_2)
+\Ga(f_1+\eps f_2,g_R)
+\Ga(f_R,g_1+\eps g_2)
\notag\\&+\Ga(g_1+\eps g_2,f_R)+\eps^{\frac{1}{2}}\{\Ga(g_R,f_R)+\Ga(f_R,g_R)\},\notag
\end{align}
supplemented with initial data
\begin{align}
\sqrt{\mu}g_R(0,\Fx,\Fv)=\sqrt{\mu}g_{R,0}(\Fx,\Fv),\notag
\end{align}
and boundary data
\begin{align}\label{gr-bd}
g_R(t,\pm1,\Fv)|_{v_z\lessgtr0}=P_\ga g_R+\eps^{\frac{1}{2}}\tilde{r},
\end{align}
where $\tilde{r}=-g_2+P_\ga g_2.$ It is easy to see that the conservation of mass is valid, i.e.
$$
\int_{\Om\times\R^3}\sqrt{\mu}g_Rd\Fx d\Fv=0,\quad \forall\,t\geq 0.
$$

As in the steady case, we first determine $g_1$ and $g_2$. In fact, from \eqref{g1}, it follows
\begin{align}\label{g1-ex}
g_1=\left\{\rho(t,\Fx)+\Fu(t,\Fx)\cdot \Fv+\frac{|\Fv|^2-3}{2}\ta(t,\Fx)\right\}\sqrt{\mu}.
\end{align}
By \eqref{mi-g2}, one has
\begin{align}\label{bq}
\na_{\Fx}\cdot\Fu=0,\ \na_\Fx(\rho+\ta)=0,
\end{align}
and
\begin{align}\label{g2-def}
g_2=&L^{-1}\left\{-\Fv\cdot\na_\Fx g_1+\Ga(g_1,g_1)+\Ga(f_1,g_1)+\Ga(g_1,f_1)\right\}
\notag\\&
+\left\{\Fu_2\cdot \Fv+\frac{|\Fv|^2-3}{2}\ta_2\right\}\sqrt{\mu}
\notag\\
=&
-\sum\limits_{i,j=1}^3\bar{A}_{ij}\pa_iu_{j}-\sum\limits_{i=1}^3\bar{B}_i\pa_i\ta
+\{\FI-\FP\}\left\{\frac{(\Fv\cdot\Fu)^2}{2}\sqrt{\mu}\right\}
\notag\\&+(|\Fv|^2-5)(\Fv\cdot \Fu)\ta\sqrt{\mu} 
+\{\FI-\FP\}\left\{\frac{(|\Fv|^2-5)^2\ta^2}{8}\sqrt{\mu}\right\}
\notag\\&+\{\FI-\FP\}\left\{(\Fv\cdot\Fu)(\Fv\cdot\Fu_s)\sqrt{\mu}\right\}
+\{\FI-\FP\}\left\{\frac{(|\Fv|^2-5)^2\ta \ta_s}{4}\sqrt{\mu}\right\}\notag\\&+\left\{\Fu_2\cdot \Fv+\frac{|\Fv|^2-5}{2}\ta_2\right\}\sqrt{\mu},
\end{align}
where we also have used the fact that
$$
\Ga(\FP f,\FP g)+\Ga(\FP g,\FP f)=L\left\{\frac{\FP f \cdot\FP g}{\sqrt{\mu}}\right\},
$$
and $\Fu_2$ is chosen as
\begin{align}\label{uu2}
\Fu_2=\na_\Fx\Delta^{-1}\{\pa_t\rho-\al z\pa_x\rho\},\ u_{2,z}(x,y,\pm1)=0,\ \Fu_2=(u_{2,x},u_{2,y},u_{2,z}).
\end{align}
Moreover, the macroscopic temperature of $g_2$, denoted by $\ta_2$, is given by \eqref{up}. Now, by \eqref{g1-ex}, \eqref{bq}, \eqref{g2-def} and \eqref{ev-g1}, we can derive the following unsteady Navier-Stokes equations around the Couette flow $(\al z,0,0)$:
\begin{eqnarray}
\left\{\begin{array}{rll}
&\na_\Fx\cdot \Fu=0,\ \ \rho=-\ta,\\[2mm]
&\pa_t\Fu+\Fu\cdot\na_x \Fu+\na_\Fx \tilde{P}+\al z\pa_{x}\Fu\\[2mm]
&\qquad+\Fu\cdot\na_x \Fu_s+\Fu_s\cdot\na_x \Fu+\al(u_{z},0,0)^T=\eta\Delta_\Fx \Fu,\\[2mm]
&\pa_t\ta+\al z\pa_{x}\ta+\na_\Fx\ta\cdot \Fu+\na_\Fx\ta\cdot \Fu_s+\na_\Fx\ta_s\cdot \Fu=\frac{2}{5}\ka\Delta_\Fx \ta,\\[2mm]
&\Fu(x,y,\pm1)=0,\ \ta(x,y,\pm1)=0,\\
&\Fu(0,\Fx)=\Fu_0(\Fx),\ \ta(0,\Fx)=\ta_0(\Fx),\notag
\end{array}\right.
\end{eqnarray}
where
\begin{align}\label{up}
\na_\Fx\tilde{P}=\na_\Fx\left\{\ta_2-\frac{1}{3}|\Fu|^2-\frac{1}{3}\Fu\cdot\Fu_s\right\}.
\end{align}
Note that we may choose
\begin{align}\label{ta2}
\ta_2=\tilde{P}-|\Om|^{-1}\int_{\Om}\tilde{P}d\Fx+\frac{1}{3}|\Fu|^2+\frac{1}{3}\Fu\cdot\Fu_s.
\end{align}

To proceed solving the remainder, let us now set
\begin{align}
\sqrt{\mu}g_R=g_{R,1}+\sqrt{\mu}g_{R,2},\label{gr-def}
\end{align}
where
\begin{align}\label{gr1}
\eps\pa_tg_{R,1}&+\Fv\cdot\na_\Fx g_{R,1}+\eps^2\Phi\cdot\na_\Fv g_{R,1}+\al\eps z\pa_{x}g_{R,1}-\al\eps v_z\pa_{v_x}g_{R,1}
+\frac{1}{\eps}\nu g_{R,1}\notag\\
=&\eps^{\frac{3}{2}}\sqrt{\mu}\pa_tg_2-\eps^{\frac{3}{2}}\al z\sqrt{\mu}\pa_xg_2+\eps^{\frac{3}{2}}\al v_z\pa_{v_x}\{\sqrt{\mu}g_2\}
\notag\\&-\eps^{\frac{3}{2}}\Phi\cdot\na_\Fv[\sqrt{\mu}(g_1+\eps g_2)]
+\frac{1}{\eps}\chi_M\CK g_{R,1}+\frac{\eps^2}{2}\sqrt{\mu}\Phi\cdot\Fv g_{R,2}
\notag\\&-\frac{\al\eps v_xv_z}{2}\sqrt{\mu}\{\FI-\FP\}g_{R,2}
+\eps^{\frac{1}{2}}Q(g_{R,1},g_{R,1})+\eps^{\frac{1}{2}}Q(g_{R,1},\sqrt{\mu}g_{R,2})\notag\\&+\eps^{\frac{1}{2}}Q(\sqrt{\mu}g_{R,2},g_{R,1})
+Q(\sqrt{\mu}g_R,\sqrt{\mu}(g_1+\eps g_2))
\notag\\&+Q(\sqrt{\mu}(g_1+\eps g_2),\sqrt{\mu}g_R)+\eps^{\frac{3}{2}}Q(\sqrt{\mu}g_2,\sqrt{\mu}g_2)\notag\\&+Q(g_{R,1},\sqrt{\mu}(f_1+\eps f_2))
+Q(\sqrt{\mu}(f_1+\eps f_2),g_{R,1})\notag\\&+Q(\sqrt{\mu}f_R,\sqrt{\mu}(g_1+\eps g_2))
+Q(\sqrt{\mu}(g_1+\eps g_2),\sqrt{\mu}f_R)
\notag\\&+\eps^{\frac{1}{2}}\{Q(\sqrt{\mu}g_{R},f_{R,1})+Q(f_{R,1},\sqrt{\mu}g_{R})\}
+\eps^{\frac{1}{2}}\{Q(g_{R,1},\sqrt{\mu}f_{R,2})\notag\\& +Q(\sqrt{\mu}f_{R,2},g_{R,1})\},
\end{align}
\begin{align}\label{gr1-id}
g_{R,1}(0,\Fx,\Fv)=\eps^{\frac{3}{2}} g^{(1)}_{R,0}(\Fx,\Fv),
\end{align}
\begin{align}\label{gr1-bd}
g_{R,1}(t,\pm1,\Fv)|_{v_z\lessgtr0}=0,
\end{align}
and
\begin{align}\label{gr2}
\eps\pa_tg_{R,2}&+\Fv\cdot\na_\Fx g_{R,2}+\eps^2\Phi\cdot\na_\Fv g_{R,2}+\al\eps z\pa_{x}g_{R,2}-\al\eps v_z\pa_{v_x}g_{R,2}
+\frac{1}{\eps}Lg_{R,2}\notag\\
=&\frac{1}{\eps}\mu^{-\frac{1}{2}}(1-\chi_M)\CK g_{R,1}-\frac{\al\eps v_xv_z}{2}\FP g_{R,2}
+\eps^{\frac{1}{2}}\Ga(g_{R,2},g_{R,2})\notag\\
&+\Ga(g_{R,2},g_1+\eps g_2)
+\Ga(g_1+\eps g_2,g_{R,2})+\Ga(g_{R,2},f_1+\eps f_2)
\notag\\&+\Ga(f_1+\eps f_2,g_{R,2})+\eps^{\frac{1}{2}}\{\Ga(g_{R,2},f_{R,2})+\Ga(f_{R,2},g_{R,2})\},
\end{align}
\begin{align}\label{gr2-id}
g_{R,2}(0,\Fx,\Fv)=g^{(2)}_{R,0}(\Fx,\Fv),
\end{align}
\begin{align}\label{gr2-bd}
g_{R,2}(t,\pm1,\Fv)|_{v_z\lessgtr0}=P_\ga g_{R,2}+\bar{P}_\ga g_{R,1}+\eps^{\frac{1}{2}}\tilde{r}.
\end{align}
Note that $\sqrt{\mu}g^{(2)}_{R,0}=\eps^{\frac{3}{2}} g^{(1)}_{R,0}+\sqrt{\mu}g^{(2)}_{R,0}$.
To analyze the long time behavior of the solutions, we may set
\begin{align}
[h_{R,1},h_{R,2}]=e^{\la_0 t}[g_{R,1},g_{R,2}](t),\notag
\end{align}
where $\la_0>0$ is a positive small constant to be given in Lemma \ref{ust-sh-lem}. Consequently, the system \eqref{gr1}, \eqref{gr1-id}, \eqref{gr1-bd}, \eqref{gr2}, \eqref{gr2-id} and \eqref{gr2-bd} can be rewritten as
\begin{align}\label{hr1}
\eps\pa_th_{R,1}&+\Fv\cdot\na_\Fx h_{R,1}+\eps^2\Phi\cdot\na_\Fv h_{R,1}+\al\eps z\pa_{x}h_{R,1}
\notag\\
&-\al\eps v_z\pa_{v_x}h_{R,1}
+\frac{1}{\eps}\nu h_{R,1}-\eps\la_0 h_{R,1}\notag\\
=&\eps^{\frac{3}{2}}\sqrt{\mu}\pa_th_2-\eps^{\frac{3}{2}}\al z\sqrt{\mu}\pa_xh_2+\eps^{\frac{3}{2}}\al v_z\pa_{v_x}\{\sqrt{\mu}h_2\}
\notag\\
&-\eps^{\frac{3}{2}}\Phi\cdot\na_\Fv[\sqrt{\mu}(h_1+\eps h_2)]
+
\frac{1}{\eps}\chi_M\CK h_{R,1}+\frac{\eps^2}{2}\sqrt{\mu}\Phi\cdot\Fv h_{R,2}
\notag\\
&-\frac{\al\eps v_xv_z}{2}\sqrt{\mu}\{\FI-\FP\}h_{R,2}
+\eps^{\frac{1}{2}}e^{-\la_0 t}Q(h_{R,1},h_{R,1})\notag\\
&+\eps^{\frac{1}{2}}e^{-\la_0 t}Q(h_{R,1},\sqrt{\mu}h_{R,2})
+\eps^{\frac{1}{2}}e^{-\la_0 t}Q(\sqrt{\mu}h_{R,2},h_{R,1})
\notag\\&+e^{-\la_0 t}Q(\sqrt{\mu}h_R,\sqrt{\mu}(h_1+\eps h_2))+e^{-\la_0t}Q(\sqrt{\mu}(h_1+\eps h_2),\sqrt{\mu}h_R)\notag\\
&+\eps^{\frac{3}{2}}e^{-\la_0t}Q(\sqrt{\mu}h_2,\sqrt{\mu}h_2)
+Q(h_{R,1},\sqrt{\mu}(f_1+\eps f_2))\notag\\
&+Q(\sqrt{\mu}(f_1+\eps f_2),h_{R,1})+Q(f_{R,1},\sqrt{\mu}(h_1+\eps h_2))
\notag\\&+Q(\sqrt{\mu}(h_1+\eps h_2),f_{R,1})+\eps^{\frac{1}{2}}\{Q(\sqrt{\mu}h_{R},f_{R,1})+Q(f_{R,1},\sqrt{\mu}h_{R})\}
\notag\\&+\eps^{\frac{1}{2}}\{Q(h_{R,1},\sqrt{\mu}f_{R,2})+Q(\sqrt{\mu}f_{R,2},h_{R,1})\},
\end{align}
\begin{align}\label{hr1-id}
h_{R,1}(0,\Fx,\Fv)=\eps^{\frac{3}{2}} g^{(1)}_{R,0},
\end{align}
\begin{align}\label{hr1-bd}
h_{R,1}(t,\pm1,\Fv)|_{v_z\lessgtr0}=0,
\end{align}
and
\begin{align}\label{hr2}
\eps\pa_th_{R,2}&+\Fv\cdot\na_{\Fx}h_{R,2}+\eps^2\Phi\cdot\na_\Fv h_{R,2}+\al \eps z\pa_{x}h_{R,2}\notag\\
&-\al\eps v_z\pa_{v_x}h_{R,2}
+\frac{1}{\eps}Lh_{R,2}-\eps\la_0 h_{R,2}\notag\\
=&\frac{1}{\eps}\mu^{-\frac{1}{2}}(1-\chi_M)\CK h_{R,1}-\frac{\al \eps v_xv_z}{2}\FP h_{R,2}
+\eps^{\frac{1}{2}}e^{-\la_0t}\Ga(h_{R,2},h_{R,2})\notag\\
&+e^{-\la_0t}\Ga(h_{R,2},h_1+\eps h_2)
+e^{-\la_0t}\Ga(h_1+\eps h_2,h_{R,2})\notag\\
&+\Ga(h_{R,2},f_1+\eps f_2)
+\Ga(f_1+\eps f_2,h_{R,2})+\Ga(f_{R,2},h_1+\eps h_2)
\notag\\&+\Ga(h_1+\eps h_2,f_{R,2})+\eps^{\frac{1}{2}}\{\Ga(h_{R,2},f_{R,2})+\Ga(f_{R,2},h_{R,2})\},
\end{align}
\begin{align}\label{hr2-id}
h_{R,2}(0,\Fx,\Fv)=g^{(2)}_{R,0}(\Fx,\Fv),
\end{align}
\begin{align}\label{hr2-bd}
h_{R,2}(t,\pm1,\Fv)|_{v_z\lessgtr0}=P_\ga h_{R,2}+\bar{P}_\ga h_{R,1}+\eps^{\frac{1}{2}}e^{\la_0t}\tilde{r},
\end{align}
where $h_R=e^{\la_0t}g_R$, $h_i=e^{\la_0t}g_i$.

We are now ready to complete the proof of Theorem \ref{sta-th}.

\begin{proof}[The proof of Theorem \ref{sta-th}]The proof of Theorem \ref{sta-th} is established on the basis of the local existence of the system
\eqref{hr1}, \eqref{hr1-id}, \eqref{hr1-bd}, \eqref{hr2}, \eqref{hr2-id} and \eqref{hr2-bd} and the {\it a priori} estimate as well as a continuity argument. Here we will focus solely on deducing the {\it a priori} estimate as \eqref{rm-de} for the sake of brevity.
We now turn to establish the $L^\infty$ estimates for $h_{R,1}$ and $h_{R,2}$ based on the {\it a priori} assumption that
\begin{align}\label{ust-apa}
\eps^{-\frac{3}{2}}\sum\limits_{\bar{k}\in\Z^2}\sup\limits_{0\leq \tau\leq t}\|w^{l_\infty}\hat{h}_{R,1}(\tau)\|_{\infty}
+\sqrt{\eps}\sum\limits_{\bar{k}\in\Z^2}\sup\limits_{0\leq \tau\leq t}\|w^{l_\infty}\hat{h}_{R,2}(\tau)\|_{\infty}\leq \sqrt{\vps_0},
\end{align}
with $t\in[0,+\infty).$

We divide the proof in the  following four steps.

\medskip
\noindent\underline{{\bf Step 1. $L_{\bar{k}}^1L_{z,\Fv}^\infty$ estimates.}} Taking the Fourier transform of \eqref{hr1}, \eqref{hr1-id}, \eqref{hr1-bd}, \eqref{hr2}, \eqref{hr2-id} and \eqref{hr2-bd}, one has
\begin{align}\label{f-hr1}
\pa_t\hat{h}_{R,1}&+\eps^{-1}i\bar{k}\cdot\bar{v}\hat{h}_{R,1}+\eps^{-1}v_z\pa_z\hat{h}_{R,1}+\eps\Phi\cdot\na_\Fv \hat{h}_{R,1}\notag\\
&\qquad\qquad \qquad  +i\al zk_x\hat{h}_{R,1}-\al v_z\pa_{v_x}\hat{h}_{R,1}
+\frac{1}{\eps^2}\nu \hat{h}_{R,1}-\la_0 \hat{h}_{R,1}\notag\\
=&-\eps^{\frac{1}{2}}\sqrt{\mu}\pa_t\hat{h}_2+\eps^{\frac{1}{2}}\al z\sqrt{\mu}i k_x\hat{h}_2+\eps^{\frac{1}{2}}\al v_z\pa_{v_x}\{\sqrt{\mu}\hat{h}_2\}
\notag\\
&-\eps^{\frac{1}{2}}\Phi\cdot\na_\Fv[\sqrt{\mu}(h_1+\eps h_2)]
+\frac{1}{\eps^2}\chi_M\CK \hat{h}_{R,1}+\frac{\eps}{2}\sqrt{\mu}\Phi\cdot\Fv \hat{h}_{R,2}
\notag\\
&-\frac{\al v_xv_z}{2}\sqrt{\mu}\{\FI-\FP\}\hat{h}_{R,2}
+\SH_1,
\end{align}
with
\begin{align}
\SH_1=&\eps^{-\frac{1}{2}}e^{-\la_0t}\hat{Q}(\hat{h}_{R,1},\hat{h}_{R,1})
+\eps^{-\frac{1}{2}}e^{-\la_0t}\hat{Q}(\hat{h}_{R,1},\sqrt{\mu}\hat{h}_{R,2})\notag\\
&+\eps^{-\frac{1}{2}}e^{-\la_0t}\hat{Q}(\sqrt{\mu}\hat{h}_{R,2},\hat{h}_{R,1})
+\eps^{-1}e^{-\la_0t}\hat{Q}(\hat{h}_{R,1},\sqrt{\mu}(\hat{h}_1+\eps \hat{h}_2))\notag\\
&
+\eps^{-1}e^{-\la_0t}\hat{Q}(\sqrt{\mu}(\hat{h}_1+\eps \hat{h}_2),\hat{h}_{R,1})+\eps^{\frac{1}{2}}e^{-\la_0t}\hat{Q}(\sqrt{\mu}\hat{h}_2,\sqrt{\mu}\hat{h}_2)
\notag\\&
+\eps^{-1}\hat{Q}(\hat{h}_{R,1},\sqrt{\mu}(\hat{f}_1+\hat{f}_2))
+\eps^{-1}\hat{Q}(\sqrt{\mu}(\hat{f}_1+\eps \hat{f}_2),\hat{h}_{R,1})\notag\\
&+\hat{Q}(\sqrt{\mu}\hat{f}_{R,1}, \sqrt{\mu}\hat{h}_2)
+\hat{Q}(\sqrt{\mu}\hat{h}_2,\sqrt{\mu}\hat{f}_{R,1})
+\eps^{-1}\hat{Q}(\hat{f}_{R,1},\sqrt{\mu}\hat{h}_1)
\notag\\
&+\eps^{-1}\hat{Q}(\sqrt{\mu}\hat{h}_1,\hat{f}_{R,1})
+\eps^{-\frac{1}{2}}\left\{\hat{Q}(\sqrt{\mu}\hat{h}_{R},\hat{f}_{R,1})+\hat{Q}(\hat{f}_{R,1},\sqrt{\mu}\hat{h}_{R})\right\}
\notag\\
&+\eps^{-\frac{1}{2}}\left\{\hat{Q}(\hat{h}_{R,1},\sqrt{\mu}\hat{f}_{R,2})+\hat{Q}(\sqrt{\mu}\hat{f}_{R,2},\hat{h}_{R,1})\right\},\notag
\end{align}
\begin{align}\label{f-hr1-id}
\hat{h}_{R,1}(0,z,\Fv)=\eps^{\frac{3}{2}}\hat{g}^{(1)}_{R,0}(\bar{k},z,\Fv),
\end{align}
\begin{align}\label{f-hr1-bd}
\hat{h}_{R,1}(t,\pm1,\Fv)|_{v_z\lessgtr0}=0,
\end{align}
and
\begin{align}\label{f-hr2}
\pa_t\hat{h}_{R,2}&+\eps^{-1}i\bar{k}\cdot\bar{v}\hat{h}_{R,2}+\eps^{-1}v_z\pa_z\hat{h}_{R,2}+\eps\Phi\cdot\na_\Fv \hat{h}_{R,2}\notag\\
&\qquad \qquad +i\al zk_x\hat{h}_{R,2}-\al v_z\pa_{v_x}\hat{h}_{R,2}
+\frac{1}{\eps^2}L\hat{h}_{R,2}-\la_0 \hat{h}_{R,2}\notag\\
=&\frac{1}{\eps^2}\mu^{-\frac{1}{2}}(1-\chi_M)\CK h_{R,1}-\frac{\al v_xv_z}{2}\FP \hat{h}_{R,2}+\SH_2
\end{align}
with
\begin{align}
\SH_2=&\eps^{-\frac{1}{2}}e^{-\la_0t}\hat{\Ga}(\hat{h}_{R,2},\hat{h}_{R,2})+\eps^{-1}e^{-\la_0t}\hat{\Ga}(\hat{h}_{R,2},\hat{h}_1+\eps \hat{h}_2)
\notag\\&+\eps^{-1}e^{-\la_0t}\hat{\Ga}(\hat{h}_1+\eps \hat{h}_2,\hat{h}_{R,2})+\eps^{-1}\hat{\Ga}(\hat{h}_{R,2},\hat{f}_1+\eps \hat{f}_2)\notag\\
&+\eps^{-1}\hat{\Ga}(\hat{f}_1+\eps \hat{f}_2,\hat{h}_{R,2})
+\eps^{-\frac{1}{2}}\{\hat{\Ga}(\hat{h}_{R,2},\hat{f}_{R,2})+\hat{\Ga}(\hat{f}_{R,2},\hat{h}_{R,2})\}\notag\\
&+\eps^{-1}\hat{\Ga}(\hat{f}_{R,2},\hat{h}_1+\eps\hat{h}_2)
+\eps^{-1}\hat{\Ga}(\hat{h}_1+\eps\hat{h}_2,\hat{f}_{R,2}),\notag
\end{align}
\begin{align}\label{f-hr2-id}
\hat{h}_{R,2}(0,\bar{k},z,\Fv)=\hat{g}^{(2)}_{R,0}(\bar{k},z,\Fv),
\end{align}
\begin{align}\label{f-hr2-bd}
\hat{h}_{R,2}(t,\bar{k},\pm1,\Fv)|_{v_z\lessgtr0}=P_\ga \hat{h}_{R,2}+\bar{P}_\ga \hat{h}_{R,1}+\SR,
\end{align}
where $\SR=\eps^{\frac{1}{2}}e^{\la_0t}\hat{\tilde{r}}.$
Similar to \eqref{chl}, we
define the following characteristic line
\begin{align}\label{chl-u}
\frac{d \tilde{Z}}{ds}=\eps^{-1}\tilde{V}_{z},\
\frac{d \tilde{\FV}}{ds}=\eps\Phi(\tilde{Z})-\al \tilde{V}_{z}\Fe_1,\ \tilde{\FV}=(\tilde{V}_x,\tilde{V}_y,\tilde{V}_{z}),
\end{align}
with $[\tilde{Z}(t;t,z,\Fv),\tilde{\FV}(t;t,z,\Fv)]=(z,\Fv)$ and $\Fe_1=(1,0,0)$. It follows from \eqref{chl-u} that
\begin{align}\label{chl-sol-u}
\left\{\begin{array}{rll}
\dis\tilde{Z}(s)=&\dis z+\eps^{-1}(s-t)v_{z}+\eps\int_t^s\int_t^\tau\Phi_z(\tilde{Z}(\eta))d\eta d\tau,\\[3mm]
\dis\tilde{\FV}(s)=&\dis \Fv+\eps\int_t^s\Phi(\tilde{Z}(\tau))d\tau-\al v_{z}(s-t)\Fe_1\\[2mm]
&\dis \quad -\al\eps^2\int_t^s\int_t^\tau\Phi_z(\tilde{Z}(\eta))d\eta d\tau \Fe_1.
\end{array}\right.
\end{align}
Next, for any $(t,z,\Fv)\in(0,+\infty)\times[-1,1]\times\R^3$ with $z\neq0$, we define
the {\it backward exit time} as
\begin{align}\label{u-ex-t}
\tilde{t}_{b}(z,\Fv)=\inf\{\tau\geq0|\tilde{Z}(t-\tau;t,z,\Fv)\notin(-1,1)\},
\end{align}
and the {\it backward exit position} as
$$\tilde{z}_{b}(z,\Fv)=\tilde{Z}(t-\tilde{t}_b;t,z,\Fv)\in\{\pm1\}.$$
Furthermore, let $\Fv_l\in\R^3$ be a random variable, we define the following stochastic backward time cycles
\begin{equation}\label{cyc-u}
\left\{\begin{aligned}
(\tilde{t}_0,\tilde{z}_0,\Fv_0)&=(t,z,\Fv),\\
(\tilde{t}_{l+1},\tilde{z}_{l+1},\Fv_{l+1})
&=(\tilde{t}_l-\tilde{t}_{\mathbf{b}}(\tilde{z}_l,\Fv_l),\tilde{z}_{b}(\tilde{z}_l,\Fv_l),\Fv_{l+1}),\\
\ \Fv_l&=(v_{l,1},v_{l,2},v_{l,3}),
\end{aligned}\right.
\end{equation}
and
\begin{equation}
\label{cyc-u2}
\left\{\begin{aligned}
\tilde{Z}^{l}_{\mathbf{cl}}(s;t,z,\Fv) &=\mathbf{1}_{[\tilde{t}_{l+1},\tilde{t}_{l})}(s)%
\tilde{Z}(s;\tilde{t}_l,\tilde{z}_l,\Fv_l),\\ 
\tilde{\FV}^{l}_{\mathbf{cl}}(s;t,z,\Fv) &=\mathbf{1}%
_{[\tilde{t}_{l+1},\tilde{t}_{l})}(s)\tilde{\FV}(s;\tilde{t}_l,\tilde{z}_l,\Fv_l).
\end{aligned}\right.
\end{equation}
Note that $[\tilde{Z}^{0}_{\mathbf{cl}}(s),\tilde{\FV}^{0}_{\mathbf{cl}}(s)]
=[\tilde{Z}(s),\tilde{\FV}(s)]$ and $\tilde{z}_{l}\in\{\pm1\}$ for $l\geq1$.

With those definitions, we now write the solution of \eqref{f-hr1}, \eqref{f-hr1-id}, \eqref{f-hr1-bd}, \eqref{f-hr2}, \eqref{f-hr2-id} and \eqref{f-hr2-bd} as the following mild form
\begin{align}\label{hr1-mf}
(&w^{l_\infty}\hat{h}_{R,1})(t,z,\Fv)\notag\\
=&\underbrace{\frac{1}{\eps^2}\int_{\max\{0,\tilde{t}_1\}}^{t}
e^{-\int_{s}^t\tilde{\CA}^\eps(\tau,\tilde{\FV}(\tau))d\tau}\left\{\chi_{M}w^{l_\infty}\CK
\hat{h}_{R,1}\right\}(\tilde{Z}(s),\tilde{\FV}(s))\,ds}_{\CH_1}\notag
\\
  &+\eps^{\frac{1}{2}}\int_{\max\{0,\tilde{t}_1\}}^{t}e^{-\int_{s}^t\tilde{\CA}^\eps(\tau,\tilde{\FV}(\tau))d\tau}
\left\{w^{l_\infty}[-
\sqrt{\mu}\pa_t\hat{h}_2-\al z\sqrt{\mu}ik_x\hat{h}_2\right.\notag\\
&\underbrace{\quad \quad\quad \quad\quad \quad\quad \quad \left.+\al v_z\pa_{v_x}\{\sqrt{\mu}\hat{h}_2\}-\Phi\cdot\na_\Fv[\sqrt{\mu}(\hat{h}_1+\eps \hat{h}_2)]]\right\}\,ds}_{\CH_2}\notag\\
&\underbrace{-\al\int_{\max\{0,\tilde{t}_1\}}^{t}e^{-\int_{s}^t\tilde{\CA}^\eps(\tau,\tilde{\FV}(\tau))d\tau}
\left\{w^{l_\infty}
\frac{v_xv_{z}}{2}\sqrt{\mu}\{\FI-\FP\}\hat{h}_{R,2}\right\}(\tilde{Z}(s),\tilde{\FV}(s))\,ds}_{\CH_3}\notag\\
&\underbrace{+\frac{\eps}{2}\int_{\max\{0,\tilde{t}_1\}}^{t}e^{-\int_{s}^t\tilde{\CA}^\eps(\tau,\tilde{\FV}(\tau))d\tau}
\left\{w^{l_\infty}\Phi\cdot \Fv\sqrt{\mu}\hat{h}_{R,2}\right\}(\tilde{Z}(s),\tilde{\FV}(s))\,ds}_{\CH_4}\notag\\
&+\underbrace{\int_{\max\{0,\tilde{t}_1\}}^{t}e^{-\int_{s}^t\tilde{\CA}^\eps(\tau,\tilde{\FV}(\tau))d\tau}
\left\{w^{l_\infty}\SH_1\right\}(\tilde{Z}(s),\tilde{\FV}(s))\,ds}_{\CH_5}
\notag\\&+\underbrace{{\bf 1}_{\tilde{t}_1\leq0}\eps^{\frac{3}{2}}e^{-\int_{0}^t\tilde{\CA}^\eps(\tau,\tilde{\FV}(\tau))d\tau}\left\{ w^{l_\infty}\hat{g}^{(1)}_{R,0}\right\}(\tilde{Z}(0),\tilde{\FV}(0))}_{\CH_6},
\end{align}
and
\begin{align}\label{hr2-mf}
(&w^{l_\infty}\hat{h}_{R,2})(t,z,\Fv)\notag\\
=&\underbrace{{\bf 1}_{\tilde{t}_1<0}e^{-\int_{0}^t\tilde{\CA}^\eps(\tau,\tilde{\FV}(\tau))d\tau}
\left\{w^{l_\infty}\hat{g}^{(2)}_{R,0}\right\}(\tilde{Z}(0),\tilde{\FV}(0))}_{\CI_0}
\notag\\
&+\frac{1}{\eps^2}\underbrace{\int_{\max\{0,\tilde{t}_1\}}^{t}e^{-\int_{s}^t\tilde{\CA}^\eps(\tau,\tilde{\FV}(\tau))d\tau}\left\{w^{l_\infty}K
\hat{h}_{R,2}\right\}(\tilde{Z}(s),\tilde{\FV}(s))\,ds}_{\CI_1}\notag
\\
&+\frac{1}{\eps^2}\underbrace{\int_{\max\{0,\tilde{t}_1\}}^{t}e^{-\int_{s}^t\tilde{\CA}^\eps(\tau,\tilde{\FV}(\tau))d\tau}
\left\{w^{l_\infty}(1-\chi_{M})\mu^{-\frac{1}{2}}\CK \hat{h}_{R,1}\right\}(\tilde{Z}(s),\tilde{\FV}(s))\,ds}_{\CI_2}\notag\\
&-\frac{\al}{2} \underbrace{\int_{\max\{0,\tilde{t}_1\}}^{t}e^{-\int_{s}^t\tilde{\CA}^\eps(\tau,\tilde{\FV}(\tau))d\tau}\left\{w^{l_\infty}v_xv_{z}\FP \hat{h}_{R,2}\right\}(\tilde{Z}(s),\tilde{\FV}(s))\,ds}_{\CI_3}
\notag\\
&+\underbrace{\int_{\max\{0,\tilde{t}_1\}}^{t}e^{-\int_{s}^t\tilde{\CA}^\eps(\tau,\tilde{\FV}(\tau))d\tau}
\left\{w^{l_\infty}\SH_2\right\}(\tilde{Z}(s),\tilde{\FV}(s))\,ds}_{\CI_4}+\sum\limits_{n=1}^7\CJ_n,
\end{align}
where
\begin{align}
\CJ_1=&\underbrace{e^{-\int_{\tilde{t}_1}^t\tilde{\CA}^\eps(\tau,\tilde{V}(\tau))d\tau}
{w_2}^{-1}(\tilde{V}(\tilde{t}_1))}_{\CW_1}\notag\\
&\times \int_{\prod\limits_{j=1}^{\RL-1}\CV_j}{\bf 1}_{\tilde{t}_k>0}
\left\{w^{l_\infty}\hat{h}_{R,2}\right\}(\tilde{t}_k,\tilde{z}_k,\tilde{V}_{\mathbf{cl}}^{\RL-1}(\tilde{t}_k))\,d\tilde{\Sigma}_{\RL-1}(\tilde{t}_k),\notag
\end{align}
\begin{align}
\CJ_2=&\CW_1\sum\limits_{l=1}^{k-1}\int_{\prod\limits_{j=1}^{\RL-1}\CV_j}{\bf 1}_{\tilde{t}_{l+1}\leq 0<\tilde{t}_l}
\left\{w^{l_\infty}\hat{g}_{R,0}\right\}(\tilde{Z}^{l}_{\mathbf{cl}}(0),\tilde{\FV}^{l}_{\mathbf{cl}}(0))\,d\tilde{\Sigma}_{l}(0),\notag
\end{align}
\begin{align}
\CJ_3=&\frac{\CW_1}{\eps^2}\sum\limits_{l=1}^{\RL-1}\int_{\prod\limits_{j=1}^{\RL-1}\CV_j}
\int_{\max\{0,\tilde{t}_{l+1}\}}^{\tilde{t}_l}
\left\{w^{l_\infty} K\hat{h}_{R,2}\right\}(\tilde{Z}^{l}_{\mathbf{cl}},\tilde{V}^{l}_{\mathbf{cl}})(s)\,
d\tilde{\Sigma}_{l}(s)ds,\notag
\end{align}
\begin{align}
\CJ_4=&\frac{\CW_1}{\eps^2}\sum\limits_{l=1}^{\RL-1}\int_{\prod\limits_{j=1}^{\RL-1}\CV_j}
\int_{\max\{0,\tilde{t}_{l+1}\}}^{\tilde{t}_l}
\left\{w^{l_\infty}(1-\chi_{M})\mu^{-\frac{1}{2}}\CK \hat{h}_{R,1}\right\}\notag\\
&\times(\tilde{Z}^{l}_{\mathbf{cl}},\tilde{V}^{l}_{\mathbf{cl}})(s)\,
d\tilde{\Sigma}_{l}(s)ds,\notag
\end{align}
\begin{align}
\CJ_5=&-\frac{\al}{2}\CW_1\sum\limits_{l=1}^{\RL-1}\int_{\prod\limits_{j=1}^{\RL-1}\CV_j}
\int_{\max\{0,\tilde{t}_{l+1}\}}^{\tilde{t}_l}
\left\{w^{l_\infty}v_xv_{z}\FP \hat{h}_{R,2}\right\}(\tilde{Z}^{l}_{\mathbf{cl}},\tilde{V}^{l}_{\mathbf{cl}})(s)\,
d\tilde{\Sigma}_{l}(s)ds,\notag
\end{align}
\begin{align}
\CJ_6=&\CW_1\sum\limits_{l=1}^{\RL-1}\int_{\prod\limits_{j=1}^{\RL-1}\CV_j}
\int_{\max\{0,\tilde{t}_{l+1}\}}^{\tilde{t}_l}
\left\{w^{l_\infty}\SH_2\right\}(\tilde{Z}^{l}_{\mathbf{cl}},\tilde{V}^{l}_{\mathbf{cl}})(s)\,
d\tilde{\Sigma}_{l}(s)ds,\notag
\end{align}
\begin{align}
\CJ_7=&{\bf1}_{\tilde{t}_1>0}e^{-\int_{\tilde{t}_1}^t\tilde{\CA}^\eps(\tau,\tilde{\FV}(\tau))d\tau}\left(w^{l_\infty}\SR\right)
(\tilde{t}_1,\tilde{z}_1,\tilde{\FV}(\tilde{t}_1))\notag\\&+\CW_1 \sum\limits_{l=1}^{\RL-1}{\bf1}_{\tilde{t}_l>0}\int_{\prod\limits_{j=1}^{\RL-1}\CV_j}\left(w^{l_\infty}\SR\right)
(\tilde{t}_{l+1},\tilde{z}_{l+1},\tilde{\FV}_{\mathbf{cl}}^{l}(\tilde{t}_{l+1}))d\tilde{\Sigma}_{l}(\tilde{t}_{l+1}),\notag
\end{align}
\begin{align}
\CJ_8=\CW_1\sum\limits_{l=1}^{\RL-1}\int_{\prod\limits_{j=1}^{\RL-1}\CV_j}\left(\frac{w^{l_\infty}}{\sqrt{\mu}}\hat{h}_{R,1}\right)
(t_l,z_l,\FV_{\mathbf{cl}}^{l}(t_l))d\tilde{\Sigma}_{l}(\tilde{t}_l).\notag
\end{align}
Moreover, the generator of the above semi-group is given by
\begin{align}
\tilde{\CA}^\eps(\tau,\tilde{\FV}(\tau))=&
\frac{1}{\eps^2}\nu(\tilde{\FV}(\tau))-2\eps l_\infty\frac{\tilde{V}(\tau)\cdot\Phi(\tilde{Z})}{{1+|\tilde{\FV}(\tau)|^2}}+2l_\infty \al \frac{\tilde{V}_x(\tau)\tilde{V}_{z}(\tau)}{{1+|\tilde{\FV}(\tau)|^2}}
\notag\\&
+\eps^{-1}i\bar{k}\cdot \bar{\tilde{\FV}}
+ik_x\tilde{Z}\al,
\notag\\
 \bar{\tilde{\FV}}=&(\tilde{V}_{x},\tilde{V}_{y}),\notag
\end{align}
the weighted measure $\tilde{\Sigma}_{l}^{(i)}(s)$ is defined as
\begin{multline}\label{Sigma-i}
\tilde{\Sigma}_{l}(s)=\prod\limits_{j=l+1}^{\RL-1}d\si_j e^{-\int_s^{\tilde{t}_l}
\tilde{\CA}^\eps(\tau,\tilde{\FV}_{\mathbf{cl}}^l(\tau))d\tau}{w_2}(\Fv_l)d\si_l \\
\prod\limits_{j=1}^{l-1}\frac{{w_2}(\Fv_j)}{{w_2}(\tilde{\FV}^{j}_{\mathbf{cl}}(\tilde{t}_{j+1}))}
e^{-\int_{\tilde{t}_{j+1}}^{\tilde{t}_j}
\tilde{\CA}^\eps(\tau,\tilde{\FV}_{\mathbf{cl}}^j(\tau))d\tau}d\si_j,
\end{multline}
and ${w_2}$ is already defined as \eqref{wt-2}.

As \eqref{CAe-lbd}, it can be shown that
\begin{align}
\left|e^{-\int_{s}^t\tilde{\CA}^\eps(\tau,\tilde{\FV}(\tau))d\tau}\right|\leq e^{-\frac{1}{2\eps^2}\int_{s}^t\nu(\tilde{\FV}(\tau))d\tau}.\label{t-CAe-lbd}
\end{align}

We now turn to estimate the right hand side of both \eqref{hr1-mf} and  \eqref{hr2-mf}.
By \eqref{t-CAe-lbd}, it is direct to see
\begin{align}
|\CH_6|\lesssim \eps^{\frac{3}{2}} e^{-\frac{\nu_0t}{2\eps^2}}\|w^{l_\infty}\hat{g}^{(1)}_{R,0}\|_\infty,\
|\CH_2|\lesssim \eps^{\frac{5}{2}}\vps_{\Phi,\al}\sup\limits_{0\leq s\leq t}\|[\hat{h}_1,\hat{h}_2](s)\|_\infty,
\notag
\end{align}
and
\begin{align}
|\CH_3|\lesssim \al\eps^2\sup\limits_{0\leq s\leq t}\left\|w^{l_\infty}\hat{h}_{R,2}(s)\right\|_{\infty},\ |\CH_4|\lesssim \vps_{\Phi,\al}\eps^2\sup\limits_{0\leq s\leq t}\left\|w^{l_\infty}\hat{h}_{R,2}(s)\right\|_{\infty}.
\notag
\end{align}
Next, in light of Lemma \ref{g-ck-lem}, it follows
\begin{align}
|\CH_1|\lesssim \{M^{-\ga}+\zeta\}\sup\limits_{0\leq s\leq t}\left\|w^{l_\infty}\hat{h}_{R,1}(s)\right\|_{\infty}.\notag
\end{align}
Using Lemma \ref{es-tri}, one has, for any $t>0$
\begin{align}
|\CH_5|\lesssim& \eps^{\frac{3}{2}}\sup\limits_{0\leq s\leq t}\sum\limits_{\bar{l}}\|w^{l_\infty}\hat{h}_{R,1}(\bar{k}-\bar{l})\|_\infty
\left\{\|w^{l_\infty}\hat{h}_{R,1}(\bar{l})\|_\infty+\|w^{l_\infty}\hat{h}_{R,2}(\bar{l})\|_\infty\right\}
\notag\\&
+\eps\sup\limits_{0\leq s\leq t}\sum\limits_{\bar{l}}\|w^{l_\infty}\hat{h}_{R,1}(\bar{k}-\bar{l})\|_\infty
\|[\hat{h}_{1},\hat{h}_{2}](\bar{l})\|_\infty
\notag\\&
+\eps^{\frac{5}{2}}\sup\limits_{0\leq s\leq t}\sum\limits_{\bar{l}}\|\hat{h}_{2}(\bar{k}-\bar{l})\|_\infty
\|\hat{h}_{2}(\bar{l})\|_\infty
\notag\\&
+\eps\sup\limits_{0\leq s\leq t}\sum\limits_{\bar{l}}\|w^{l_\infty}\hat{h}_{R,1}(\bar{k}-\bar{l})\|_\infty
\|[\hat{f}_{1},\hat{f}_{2}](\bar{l})\|_\infty
\notag\\
&+\eps\sup\limits_{0\leq s\leq t}\sum\limits_{\bar{l}}\|w^{l_\infty}\hat{f}_{R,1}(\bar{k}-\bar{l})\|_\infty
\|[\hat{h}_{1},\hat{h}_{2}](\bar{l})\|_\infty
\notag\\&+\eps^{\frac{3}{2}}\sup\limits_{0\leq s\leq t}\sum\limits_{\bar{l}}\|w^{l_\infty}\hat{f}_{R,1}(\bar{k}-\bar{l})\|_\infty
\|w^{l_\infty}[\hat{h}_{R,1},\hat{h}_{R,2}](\bar{l})\|_\infty
\notag\\&+\eps^{\frac{3}{2}}\sup\limits_{0\leq s\leq t}\sum\limits_{\bar{l}}\|w^{l_\infty}\hat{f}_{R,2}(\bar{k}-\bar{l})\|_\infty
\|w^{l_\infty}\hat{h}_{R,1}(\bar{l})\|_\infty=:\CP_0(t).\notag
\end{align}
Putting the estimates above together and using the {\it a priori} assumption, we get the following estimate for $\hat{h}_{R,1}$:
\begin{align}
&\eps^{-\frac{3}{2}} \|w^{l_\infty}\hat{h}_{R,1}(t)\|_\infty \notag\\
&\lesssim
e^{-\frac{\nu_0t}{2\eps^2}}\|w^{l_\infty}\hat{g}^{(1)}_{R,0}\|_\infty+
\{M^{-1}+\zeta\}\eps^{-\frac{3}{2}}\sup\limits_{0\leq s\leq t}\left\|w^{l_\infty}\hat{h}_{R,1}(s)\right\|_{\infty}\notag\\
&\quad+
(\al+\vps_{\Phi,\al})\eps^{\frac{1}{2}}\sup\limits_{0\leq s\leq t}\left\|w^{l_\infty}\hat{h}_{R,2}(s)\right\|_{\infty}+\CP_1(t),\label{lif-hr1}
\end{align}
where
\begin{align}
\CP_1(t)=\eps^{-\frac{3}{2}}\CP_0(t)+\eps\vps_{\Phi,\al}\sup\limits_{0\leq s\leq t}\|[\hat{h}_1,\hat{h}_2](s)\|_\infty.\label{cp1}
\end{align}
As for the corresponding estimate for $\hat{h}_{R,2}$,  we first get from \eqref{hr2-mf} that
\begin{align}\label{hr2-sum1}
&|w^{l_\infty}\hat{h}_{R,2}(t)|\notag\\
&\leq  \frac{C}{\eps^2}\int_{\max\{0,\tilde{t}_1\}}^te^{-\int_{s}^t\frac{\nu(\tilde{\FV}(\tau))}{2\eps^2}d\tau}
\int_{\R^3}{\bf k}_{w}(\tilde{\FV}(s),\Fv')\notag\\
&\qquad\qquad\qquad \times|(w^{l_\infty}\hat{h}_{R,2})(s,\tilde{Z}(s;t,z,\Fv),\Fv')|d\Fv'ds\notag\\
&\quad+\frac{C(l_\infty)}{\eps^2}e^{-\int_{\tilde{t}_1}^t\frac{\nu(\tilde{\FV}(\tau))}{2\eps^2}d\tau}
\sum\limits_{l=1}^{\RL-1}\int_{\prod\limits_{j=1}^{\RL-1}\CV_j}
\int_{\max\{0,\tilde{t}_{l+1}\}}^{\tilde{t}_l}\int_{\R^3}{\bf k}_{w}(\tilde{\FV}^{l}_{\mathbf{cl}}(s),\Fv')\notag\\
&
\qquad\qquad\qquad\qquad\times|(w^{l_\infty}\hat{h}_{R,2})
(s,\tilde{Z}^{l}_{\mathbf{cl}}(s;t,z,\Fv),\Fv')|d\Fv'\,d\tilde{\Sigma}_{l}(s)ds
\notag\\
&\quad+\left\|e^{-\frac{\nu_0t}{2\eps^2}}w^{l_\infty}\hat{g}^{(2)}_{R,0}\right\|_{\infty}+\tilde{\CQ}(t),
\end{align}
where
\begin{align}
\tilde{\CQ}(t)=&C(l_\infty)\sup\limits_{0\leq s\leq t}\|w^{l_\infty}\hat{h}_{R,1}(s)\|_{\infty}+C(T_0)\eta\sup\limits_{0\leq s\leq t}\|w^{l_\infty}\hat{h}_{R,2}(s)\|_{\infty}
\notag\\&+
C \eps\sup\limits_{0\leq s\leq t}\sum\limits_{\bar{l}\in \Z^2}\left\{\|w^{l_\infty}\hat{f}_{1}(\bar{k}-\bar{l})(s)\|_{\infty}\right.
\notag\\&\left.\qquad \qquad \qquad\qquad  \qquad \qquad +\|w^{l_\infty}\hat{f}_{2}(\bar{k}-\bar{l})(s)\|_{\infty}\right\}
\|w^{l_\infty}\hat{f}_{R,2}(\bar{l})\|_{\infty}\notag\\
&+C \sup\limits_{0\leq s\leq t}\sum\limits_{\bar{l}\in \Z^2}\eps\left\{ e^{-\la_0 s}\{\|w^{l_\infty}\hat{h}_{1}(\bar{k}-\bar{l})\|_{\infty}\right.
\notag\\&\left.\qquad \qquad \qquad\qquad \qquad\qquad +\|w^{l_\infty}\hat{h}_{2}(\bar{k}-\bar{l})\|_{\infty}\}
\|w^{l_\infty}\hat{h}_{R,2}(\bar{l})\|_{\infty}\right\}
\notag\\
&+C\eps^{\frac{3}{2}}\sup\limits_{0\leq s\leq t}\sum\limits_{\bar{l}\in \Z^2}\left\{ e^{-\la_0 s} \|w^{l_\infty}\hat{h}_{R,2}(\bar{k}-\bar{l})\|_{\infty}
\|w^{l_\infty}\hat{h}_{R,2}(\bar{l})\|_{\infty}\right\}\notag\\&+C\eps^{\frac{3}{2}} \sup\limits_{0\leq s\leq t}\sum\limits_{\bar{l}\in \Z^2}\|w^{l_\infty}\hat{f}_{R,2}(\bar{k}-\bar{l})\|_{\infty}
\|w^{l_\infty}\hat{h}_{R,2}(\bar{l})\|_{\infty}
\notag\\&+C\sup\limits_{0\leq s\leq t}\|w^{l_\infty}\SR(s,\pm1,\bar{k})\|_\infty.
\label{t-cq}
\end{align}
Here, $\eta$ is positive and suitably small.

Next, iterating \eqref{hr2-sum1} again, one further has
\begin{align}\label{hr2-itr}
|&w^{l_\infty}\hat{h}_{R,2}|\leq \frac{C}{\eps^4}\int_{\max\{0,\tilde{t}_1\}}^te^{-\int_{s}^t\frac{\nu(\tilde{\FV}(\tau))}{2\eps^2}d\tau}\int_{\R^3}{\bf k}_{w}(\tilde{\FV}(s),\Fv')
\notag\\&\quad\times\int_{\max\{0,\tilde{t}'_1\}}^{s} e^{-\int_{s'}^s\frac{\nu(\tilde{\FV}(\tau))}{2\eps^2}d\tau}
\int_{\R^3}{\bf k}_{w}(\tilde{\FV}(s';s,\tilde{Z}(s),\Fv'),\Fv'')|\notag\\
&\qquad\quad \times (w^{l_\infty}\hat{h}_{R,2})(s',\tilde{Z}(s';s,\tilde{Z}(s),\Fv'),\Fv'')|~d\Fv''ds'd\Fv'ds\notag\\
&+\frac{C}{\eps^4}\int_{\max\{\tilde{t}_1,0\}}^te^{-\int_{s}^t\frac{\nu(\tilde{\FV}(\tau))}{2\eps^2}d\tau}\int_{\R^3}{\bf k}_{w}(\tilde{\FV}(s),\Fv')e^{-\int_{\tilde{t}_1'}^s\frac{\nu(\tilde{\FV}(\tau))}{2\eps^2}d\tau}\notag\\
&\qquad \times \sum\limits_{\ell=1}^{\RL-1}\int_{\prod\limits_{j=1}^{\RL-1}\CV'_j}
\int_{\max\{0,\tilde{t}'_{\ell+1}\}}^{\tilde{t}'_\ell}
\int_{\R^3}{\bf k}_{w}(\tilde{\FV}^{\ell}_{\mathbf{cl}}(s';s,\tilde{Z}(s),\Fv'),\Fv'')\notag\\
&\qquad \times |(w^{l_\infty}\hat{h}_{R,2})(s',\tilde{Z}^{\ell}_{\mathbf{cl}}(s';s,\tilde{Z}(s),\Fv'),\Fv'')|
d\Fv''\,d\tilde{\Sigma}_{\ell}(s')ds'd\Fv'ds\notag\\
&+\frac{C}{\eps^4}e^{-\int_{\tilde{t}'_1}^t\frac{\nu(\tilde{\FV}(\tau))}{2\eps^2}d\tau}\sum\limits_{l=1}^{\RL-1}
\int_{\prod\limits_{j=1}^{\RL-1}\CV_j}
\int_{\max\{0,\tilde{t}_{l+1}\}}^{\tilde{t}_l}\int_{\R^3}{\bf k}_{w}(\tilde{\FV}^{l}_{\mathbf{cl}}(s),\Fv')\notag\\
&\quad \times
\int_{\max\{0,\tilde{t}'_1\}}^{s}e^{-\int_{s'}^s\frac{\nu(\tilde{\FV}(\tau))}{2\eps^2}d\tau}
\int_{\R^3}{\bf k}_{w}(\tilde{\FV}(s';s,\tilde{Z}^{l}_{\mathbf{cl}}(s),\Fv'),\Fv'')\notag\\
&\quad \times|(w^{l_\infty}\hat{h}_{R,2})
(s',\tilde{Z}(s';s,\tilde{Z}^{l}_{\mathbf{cl}}(s),\Fv'),\Fv'')|~d\Fv''ds'd\Fv'\,d\tilde{\Sigma}_{l}(s)ds
\notag\\&+\frac{C}{\eps^4}e^{-\int_{\tilde{t}_1}^t\frac{\nu(\tilde{\FV}(\tau))}{2\eps^2}d\tau}
\sum\limits_{l=1}^{\RL-1}\int_{\prod\limits_{j=1}^{\RL-1}\CV_j}
\int_{\max\{0,\tilde{t}_{l+1}\}}^{\tilde{t}_l}\int_{\R^3}{\bf k}_{w}(\tilde{\FV}^{l}_{\mathbf{cl}}(s;v),v')\notag\\
&\quad \times e^{-\int_{\tilde{t}_1'}^s\frac{\nu(\tilde{\FV}(\tau))}{2\eps^2}d\tau}
\sum\limits_{\ell=1}^{\RL-1}\int_{\prod\limits_{j=1}^{\RL-1}\CV'_j}
\int_{\max\{0,\tilde{t}'_{\ell+1}\}}^{\tilde{t}'_\ell} 
\int_{\R^3}{\bf k}_{w}(\tilde{\FV}^{\ell}_{\mathbf{cl}}(s';s,\tilde{Z}^{l}_{\mathbf{cl}}(s),\Fv'),\Fv'')
\notag\\&\quad \times
|(w^{l_\infty}\hat{h}_{R,2})(s',\tilde{Z}^{\ell}_{\mathbf{cl}}(s';s,\tilde{Z}^{l}_{\mathbf{cl}}(s),\Fv'),\Fv'')|
d\Fv''\,d\tilde{\Sigma}_{\ell}(s')ds'd\Fv'\,d\tilde{\Sigma}_{l}(s)ds
\notag\\
&+ \frac{C}{\eps^2}\int_{\max\{0,\tilde{t}_1\}}^te^{-\int_{s}^t\frac{\nu(\tilde{\FV}(\tau))}{2\eps^2}d\tau}\int_{\R^3}{\bf k}_{w}(\tilde{\FV}(s),\Fv')\tilde{\CQ}(s)d\Fv'ds
\notag\\&+\frac{C(l_\infty)}{\eps^2}e^{-\int_{\tilde{t}_1}^t\frac{\nu(\tilde{\FV}(\tau))}{2\eps^2}d\tau}\int_{\prod\limits_{j=1}^{\RL-1}\CV_j}
\int_{\max\{0,\tilde{t}_{l+1}\}}^{\tilde{t}_l}\notag\\
&\qquad \qquad\qquad\times\int_{\R^3}{\bf k}_{w}(\tilde{\FV}^{l}_{\mathbf{cl}}(s),\Fv')\tilde{\CQ}(s)d\Fv'\,d\tilde{\Sigma}_{l}(s)ds\notag\\
&+ \frac{C}{\eps^2}\int_{\max\{0,\tilde{t}_1\}}^te^{-\int_{s}^t\frac{\nu(\tilde{\FV}(\tau))}{2\eps^2}d\tau}\int_{\R^3}{\bf k}_{w}(\tilde{\FV}(s),\Fv')\left\|e^{-\frac{\nu_0s}{2\eps^2}}w^{l_\infty}\hat{g}^{(2)}_{R,0}\right\|_{\infty}d\Fv'ds
\notag\\&+\frac{C(l_\infty)}{\eps^2}e^{-\int_{\tilde{t}_1}^t\frac{\nu(\tilde{\FV}(\tau))}{2\eps^2}d\tau}\int_{\prod\limits_{j=1}^{\RL-1}\CV_j}
\int_{\max\{0,\tilde{t}_{l+1}\}}^{\tilde{t}_l} \int_{\R^3}{\bf k}_{w}(\tilde{\FV}^{l}_{\mathbf{cl}}(s),\Fv')\notag\\
&\qquad\qquad\qquad \times \left\|e^{-\frac{\nu_0s}{2\eps^2}}w^{l_\infty}\hat{g}^{(2)}_{R,0}\right\|_{\infty}d\Fv'\,d\tilde{\Sigma}_{l}(s)ds,
\end{align}
where $(\tilde{t}_{\ell+1}',\tilde{z}_{\ell+1}',\Fv_{\ell+1}')$ is defined as \eqref{cyc-u} with the starting points replaced by
\begin{align}\notag
(\tilde{t}_{0}',\tilde{z}_{0}',\Fv_{0}')=(s,z',\Fv'):=(s,\tilde{Z}(s),\Fv')\ \textrm{or}\ (s,\tilde{Z}^l_{\mathbf{cl}}(s),\Fv'),
\end{align}
and similar to \eqref{cyc-u2}, we also have introduced
\begin{align}
\tilde{Z}^\ell_{\mathbf{cl}}(s';s,z',\Fv') &=\mathbf{1}_{[\tilde{t}'_{\ell+1},\tilde{t}'_{\ell})}(s')
\tilde{Z}(s';\tilde{t}'_\ell,\tilde{z}'_\ell,\Fv'_\ell),\notag\\
\tilde{\FV}^{\ell}_{\mathbf{cl}}(s';s,z',\Fv') &=\mathbf{1}_{[\tilde{t}'_{\ell+1},\tilde{t}'_{\ell})}(s')\tilde{\FV}(s';\tilde{t}'_\ell,\tilde{z}'_\ell,\Fv'_\ell),\notag\\
[\tilde{Z}^{0}_{\mathbf{cl}}(s'),\tilde{\FV}^{0}_{\mathbf{cl}}(s')] &:=[\tilde{Z}(s'),\tilde{\FV}(s')].\notag
\end{align}
Here, $[\tilde{Z}^\ell_{\mathbf{cl}}(s'),\tilde{\FV}^{\ell}_{\mathbf{cl}}(s')]$ is defined as \eqref{cyc-u2} and $[\tilde{Z}(s'),\tilde{\FV}(s')]$ is defined as \eqref{chl-sol-u}.
Moreover, $\tilde{\Sigma}_{\ell}(s')$ is given by
\begin{multline}
\tilde{\Sigma}_{\ell}(s')=\prod\limits_{j=\ell+1}^{\RL-1}d\si'_j e^{-\int_{s'}^{\tilde{t}_\ell'}
\tilde{\CA}^\eps(\tau,\tilde{\FV}_{\mathbf{cl}}^\ell(\tau))d\tau}w_2(\Fv'_\ell)d\si'_\ell \\
\prod\limits_{j=1}^{\ell-1}\frac{{w_2}(\Fv'_j)}{w_2(\tilde{\FV}^{j}_{\mathbf{cl}}(\tilde{t}'_{j+1}))}
e^{-\int_{\tilde{t}'_{j+1}}^{\tilde{t}'_j}
\tilde{\CA}^\eps(\tau,\tilde{\FV}_{\mathbf{cl}}^j(\tau))d\tau}d\si'_j.\notag
\end{multline}
Our aim now is to show that for $t\in[0,\eps T_0]$,
\begin{align}\label{hr2-if-es}
\|w^{l_\infty}\hat{h}_{R,2}(t)\|_{\infty}\leq& C\left\|e^{-\frac{\nu_0t}{2\eps^2}}w^{l_\infty}\hat{g}^{(2)}_{R,0}\right\|_{\infty}
+\tilde{\eta}\sup\limits_{t\in[0,\eps T_0]}\|w^{l_\infty}\hat{h}_{R,2}(t)\|_{\infty}
\notag\\&+C(T_0)\eps^{-\frac{1}{2}}\sup\limits_{t\in[0,\eps T_0]}\|\hat{h}_{R,2}(t)\|_2
+C\sup\limits_{t\in[0,\eps T_0]}\tilde{\CQ}(t),
\end{align}
where $\tilde{\eta}>0$ is suitably small.
To establish \eqref{hr2-if-es}, our focus lies on estimating the fourth term within the right hand side of \eqref{hr2-itr}. The remaining terms can be handled in a similar manner. For any sufficiently small $\ka_0>0$, we initially partition $[t_{\ell+1}',t_{\ell}']$ as $[t_{\ell+1}',t_{\ell}'-\ka_0\eps^2]\cup(t_{\ell}'-\ka_0\eps^2,t_{\ell}']$. Subsequently, we express the fourth term on the right hand side of \eqref{hr2-itr} as follows:
\begin{align}
\tilde{\CJ}:=& \frac{C}{\eps^4}e^{-\int_{\max\{0,\tilde{t}_1\}}^t\frac{\nu(\tilde{\FV}(\tau))}{2\eps^2}d\tau}
\sum\limits_{l=1}^{\RL-1}\int_{\prod\limits_{j=1}^{\RL-1}\CV_j}
\int_{\max\{0,\tilde{t}_{l+1}\}}^{\tilde{t}_l}\int_{\R^3}{\bf k}_{w}(\tilde{\FV}^{l}_{\mathbf{cl}}(s;\Fv),\Fv')\notag\\
&\quad\times e^{-\int_{\tilde{t}_1'}^s\frac{\nu(\tilde{\FV}(\tau))}{2\eps^2}d\tau}
\sum\limits_{\ell=1}^{\RL-1}\int_{\prod\limits_{j=1}^{\RL-1}\CV'_j}
\left(\int_{\max\{0,\tilde{t}'_{\ell+1}\}}^{\tilde{t}'_\ell-\ka_0\eps^2}+\int_{\tilde{t}'_\ell-\ka_0\eps^2}^{\tilde{t}'_\ell}\right)\notag\\
&\quad\times\int_{\R^3} {\bf k}_{w}(\tilde{\FV}^{\ell}_{\mathbf{cl}}(s';s,\tilde{Z}^{l}_{\mathbf{cl}}(s),\Fv'),\Fv'')\notag\\
&\quad\times
|(w^{l_\infty}\hat{h}_{R,2})(s',\tilde{Z}^{\ell}_{\mathbf{cl}}(s';s,\tilde{Z}^{l}_{\mathbf{cl}}(s),\Fv'),\Fv'')|
d\Fv''\,d\tilde{\Sigma}_{\ell}(s')ds'd\Fv'\,d\tilde{\Sigma}_{l}(s)ds\notag\\
:=& \tilde{\CJ}_1+\tilde{\CJ}_2.\notag
\end{align}
In view of Lemma \ref{k-cyc}, as for estimating \eqref{J2}, one has
\begin{align}
\tilde{\CJ}_2\lesssim \CL\ka_0\|w^{l_\infty}\hat{h}_{R,2}\|_{\infty}.\notag
\end{align}
Regarding $\tilde{\CJ}_1$, we can prove that
\begin{align}
|\tilde{\CJ}_1|\leq\left\{\frac{C(l^\infty)}{1+M}+C(l^\infty)e^{-\frac{\vps M^2}{16}}\right\}\|w^{l_\infty}\hat{h}_{R,2}\|_{\infty}
+\eps^{-\frac{1}{2}}\|\hat{h}_{R,2}\|_2.\label{tj2}
\end{align}
The derivation of \eqref{tj2} closely mirrors the treatment of the steady case $\CJ_1$. The calculations are partitioned into three cases, with the only distinctive case being:
\begin{center}
\underline{$\tilde{\CI}=$\{ \it $|\tilde{\FV}^{l}_{\mathbf{cl}}(s;\Fv)|\leq M$,$|\Fv'|\leq2M$,$|\tilde{\FV}^{\ell}_{\mathbf{cl}}(s';\tilde{Z}^{l}_{\mathbf{cl}}(s),\Fv')|\leq M$ and $|\Fv''|\leq 2M$\}}.
\end{center}
At this stage, using \eqref{kw-M}, we have the decomposition
\begin{equation*}
\begin{split}
{\bf k}_{w}&(\tilde{\FV}^{l}_{\mathbf{cl}}(s),\Fv')
{\bf k}_{w}(\tilde{\FV}^{\ell}_{\mathbf{cl}}(s'),\Fv^{\prime \prime})\notag\\
=&\{{\bf k}_{w}(\tilde{\FV}^{l}_{\mathbf{cl}}(s),\Fv^{\prime})-{\bf k}_{w,M}(\tilde{\FV}^{l}_{\mathbf{cl}}(s),\Fv^{\prime})\}{\bf k}_{w}(\tilde{\FV}^{\ell}_{\mathbf{cl}}(s'),\Fv^{\prime \prime})
\\&+\{{\bf k}_{w}(\tilde{\FV}^{\ell}_{\mathbf{cl}}(s'),\Fv^{\prime \prime})
-{\bf k}_{w,M}(\tilde{\FV}^{\ell}_{\mathbf{cl}}(s'),\Fv^{\prime \prime})\}{\bf k}_{w,M}(\tilde{\FV}^{l}_{\mathbf{cl}}(s),\Fv')
\\&
+{\bf k}_{w,M}(\tilde{\FV}^{l}_{\mathbf{cl}}(s),\Fv^{\prime}){\bf k}_{w,M}(\tilde{\FV}^{\ell}_{\mathbf{cl}}(s'),\Fv^{\prime \prime}).
\end{split}
\end{equation*}
Incorporating this decomposition, the contribution from the first two terms is small:
\begin{equation}
\frac{C(l^\infty)}{1+M}\|w^{l_\infty}\hat{h}_{R,2}\|_{\infty}.  \notag
\end{equation}
For the contribution of the last bounded product corresponding to the last term,
we introduce
\begin{align}
\CZ=\tilde{Z}^{\ell}_{\mathbf{cl}}(s';s,\tilde{Z}^{l}_{\mathbf{cl}}(s),\Fv')=&\mathbf{1}_{[\tilde{t}'_{\ell+1},\tilde{t}'_{\ell})}(s')
\tilde{Z}(s';\tilde{t}'_\ell,\tilde{z}'_\ell,\Fv'_\ell)=\tilde{z}_\ell'+\eps^{-1}(s'-\tilde{t}_\ell')v'_{\ell,z}
\notag\\&+\eps\int_{\tilde{t}_\ell'}^{s'}\int_{\tilde{t}_\ell'}^{\tau'}\Phi_{z}(\tilde{Z}(\eta'))d\eta' d\tau',\notag
\end{align}
where $\tilde{Z}(\eta')=\tilde{Z}(\eta';t'_\ell,\tilde{z}'_\ell,\Fv'_\ell)$.
Applying a change of variable $v'_{\ell,z}\rightarrow\CZ$, one gets
\begin{align}
\left|\frac{\pa \CZ}{\pa v'_{\ell,z}}\right|=\left|\frac{\pa(\tilde{z}_\ell'-\eps^{-1}(\tilde{t}_\ell'-s')v'_{\ell,z})}{\pa v'_{\ell,z}}+\eps\int_{\tilde{t}_\ell'}^{s'}\int_{\tilde{t}_\ell'}^{\tau'}\frac{\pa[\Phi_{z}(\tilde{Z}(\eta'))]}{\pa v'_{\ell,z} }d\eta' d\tau'\right|\geq Ck_0\eps,\notag
\end{align}
where we have employed a calculation similar to that used to derive \eqref{Jco2}. With this preparation, we perform a computation analogous to \eqref{de-j2} to obtain
\begin{align}
|\tilde{\CJ}_{1}{\bf1}_{\tilde{\CI}}|\leq C\eps^{-\frac{1}{2}}\|\hat{h}_{R,2}\|_2+\frac{C(l^\infty)}{M}\|w^{l_\infty}\hat{h}_{R,2}\|_{\infty}.\notag
\end{align}
Now combing all the estimates above, we ascertain the validity of \eqref{hr2-if-es}. 

Consequently, we get from \eqref{lif-hr1} and \eqref{hr2-if-es} that for $t\in[0,\eps T_0]$,
\begin{align}
\eps^{-\frac{3}{2}}\|w^{l_\infty}&\hat{h}_{R,1}(t)\|_\infty
+\eps^{\frac{1}{2}}\|w^{l_\infty}\hat{h}_{R,2}(t)\|_{\infty}\notag\\
\lesssim&
C\left\|e^{-\frac{\nu_0t}{2\eps^2}}w^{l_\infty}\hat{g}^{(1)}_{R,0}\right\|_{\infty}
+C(T_{0})\eps^{\frac{1}{2}}\left\|e^{-\frac{\nu_0t}{2\eps^2}}w^{l_\infty}\hat{g}^{(2)}_{R,0}\right\|_{\infty}
\notag\\&+\{M^{-1}+\zeta\}\eps^{-\frac{3}{2}}\sup\limits_{0\leq t\leq \eps T_0}\left\|w^{l_\infty}\hat{h}_{R,1}(t)\right\|_{\infty}\notag\\
&+(\al+\vps_{\Phi,\al})\eps^{\frac{1}{2}}\sup\limits_{0\leq t\leq \eps T_0}\left\|w^{l_\infty}\hat{h}_{R,2}(t)\right\|_{\infty}\notag\\&+C(T_{0})\sup\limits_{t\in[0,\eps T_0]}\left\{\tilde{\eta}\eps^{\frac{1}{2}}\|w^{l_\infty}\hat{h}_{R,2}(t)\|_{\infty}+\CP_1(t)+\eps^{\frac{1}{2}}\tilde{\CQ}(t)\right\}\notag\\
&+C(T_{0})\sup\limits_{t\in[0,\eps T_0]}\|\hat{g}_{R,2}(t)\|_2\notag\\
=&C(T_{0})\left\|e^{-\frac{\nu_0t}{2\eps^2}}w^{l_\infty}[\hat{g}^{(1)}_{R,0},\eps^{\frac{1}{2}}\hat{h}^{(2)}_{R,0}]\right\|_{\infty}
+\sup\limits_{0\leq t\leq \eps T_0}C(T_{0})\CD(t),\label{hr1-hr2}
\end{align}
where
\begin{align}
\CD(t)=&\eps^{-\frac{3}{2}}\{M^{-1}+\zeta\}\left\|w^{l_\infty}\hat{h}_{R,1}(t)\right\|_{\infty}
+(\al+\vps_{\Phi,\al})\eps^{\frac{1}{2}}\sup\limits_{0\leq t\leq \eps T_0}\left\|w^{l_\infty}\hat{h}_{R,2}(t)\right\|_{\infty}
\notag\\&+\tilde{\eta}\eps^{\frac{1}{2}}\|w^{l_\infty}\hat{h}_{R,2}(t)\|_{\infty}+\CP_1(t)+\eps^{\frac{1}{2}}\tilde{\CQ}(t)
+\|\hat{h}_{R,2}(t)\|_2.\notag
\end{align}
In particular, we have
\begin{align}\label{g1g2decay}
\|w^{l_\infty}[\eps^{-\frac{3}{2}}\hat{h}_{R,1},\eps^{\frac{1}{2}}\hat{h}_{R,2}](\eps T_0)\|_{\infty}
\leq& Ce^{-\frac{\nu_0T_0}{4\eps}}\left\|w^{l_\infty}[\hat{g}^{(1)}_{R,0},\eps^{\frac{1}{2}}\hat{g}^{(2)}_{R,0}]\right\|_{\infty}\notag\\
&+\sup\limits_{0\leq t\leq \eps T_0}C(T_{0})\CD(s).
\end{align}
Moreover, \eqref{hr1-hr2} can be extended to
\begin{align}\label{g1g2decay.p1}
\|w^{l_\infty}[\eps^{-\frac{3}{2}}\hat{h}_{R,1},\eps^{\frac{1}{2}}\hat{h}_{R,2}](t)\|_{\infty}\leq &Ce^{-\frac{\nu_0(t-s)}{2\eps^2}}\|w^{l_\infty}[\eps^{-\frac{3}{2}}\hat{h}_{R,1},\eps^{\frac{1}{2}}\hat{h}_{R,2}](s)\|_{\infty}\notag\\
&+\sup\limits_{s\leq t\leq s+\eps T_0}C(T_{0})\CD(s),
\end{align}
for any $t\in[s,s+\eps T_0]$ with $s\geq0.$

Next, for any integer $m\geq 1,$ we can repeat the estimate \eqref{g1g2decay}  in finite times so that
the functions $[\eps^{-\frac{3}{2}}\hat{h}_{R,1},\eps^{\frac{1}{2}}\hat{h}_{R,2}](\eps lT_{0})$ for $l=m-1,m-2,...,0$  satisfy
\begin{align}\label{lif-lift}
\|w^{l_\infty}&[\eps^{-\frac{3}{2}}\hat{h}_{R,1},\eps^{\frac{1}{2}}\hat{h}_{R,2}](\eps mT_{0})\|_{\infty}\notag\\
\leq&Ce^{-\frac{\nu_0T_0}{4\eps}}\left\|w^{l_\infty}[\eps^{-\frac{3}{2}}\hat{h}_{R,1},\eps^{\frac{1}{2}}\hat{h}_{R,2}](\eps\{m-1\}T_0)\right\|_{\infty}\notag\\
&
+\sup\limits_{\eps (m-1)T_0\leq t\leq\eps mT_{0}}C(T_{0})\CD(t)\notag\\
\leq&Ce^{-\frac{\nu_0T_0}{2\eps}}\left\|w^{l_\infty}[\eps^{-\frac{3}{2}}\hat{h}_{R,1},\eps^{\frac{1}{2}}\hat{h}_{R,2}](\eps\{m-2\}T_0)\right\|_{\infty}\notag\\
&+e^{-\frac{\nu_0T_0}{4\eps}}\sup\limits_{\eps (m-2)T_0\leq t\leq\eps (m-1)T_{0}}C(T_{0})\CD(t)
+\sup\limits_{\eps (m-1)T_0\leq t\leq\eps mT_{0}}C(T_{0})\CD(t)\notag \\
\leq&\cdots\notag\\
\leq&Ce^{-\frac{m\nu_0T_0}{4\eps}}\left\|w^{l_\infty}[\hat{g}^{(1)}_{R,0},\eps^{\frac{1}{2}}\hat{g}^{(2)}_{R,0}]\right\|_{\infty}\notag\\
&
+\sum\limits_{j=0}^{m-1}e^{-\frac{j\nu_0T_0}{4\eps}}\sup\limits_{\eps (m-1-j)T_0\leq t\leq\eps (m-j)T_{0}}C(T_{0})\CD(t)
\notag\\
\leq&C\left\|w^{l_\infty}[\hat{g}^{(1)}_{R,0},\eps^{\frac{1}{2}}\hat{g}^{(2)}_{R,0}]\right\|_{\infty}
+\frac{1}{1-e^{-\frac{\nu_0T_0}{4\eps}}}\sup\limits_{0\leq t\leq\eps mT_{0}}C(T_{0})\CD(t).
\end{align}
Furthermore, for any $t\geq T_0 $ and fixed $\eps>0$, we can find an integer $m\geq0$ such that $t=\eps m T_0+s$ with $0\leq s\leq \eps T_0$. Then we have, on the one hand,
by \eqref{lif-lift}, that
\begin{align}\label{hr12-com-p1}
&\|w^{l_\infty}[\eps^{-\frac{3}{2}}\hat{h}_{R,1},\eps^{\frac{1}{2}}\hat{h}_{R,2}](\eps m T_0)\|_{\infty}\notag\\
&\leq C\left\|w^{l_\infty}[\hat{g}^{(1)}_{R,0},\eps^{\frac{1}{2}}\hat{g}^{(2)}_{R,0}]\right\|_{\infty}
+C(T_0)\sup\limits_{0\leq s\leq \eps m T_0}\left\{\CP_1(s)+\eps^{\frac{1}{2}}\tilde{\CQ}(s)\right\}
\notag\\
&\quad+C(T_0)\sup\limits_{0\leq s\leq \eps m T_0}\|\hat{h}_{R,2}(s)\|_2.
\end{align}
On the other hand,  \eqref{g1g2decay.p1} implies that
\begin{align}\label{hr12-com-p2}
&\left\|w^{l_\infty}[\eps^{-\frac{3}{2}}\hat{g}_{R,1},\eps^{\frac{1}{2}}\hat{g}_{R,2}](t)\right\|_{\infty }=\left\|w^{l_\infty}[\eps^{-\frac{3}{2}}\hat{g}_{R,1},\eps^{\frac{1}{2}}\hat{g}_{R,2}](\eps m T_0+s)\right\|_{\infty}\notag\\
&\leq Ce^{-\frac{\nu_0 s}{4\eps^2}}\left\|w^{l_\infty}[\eps^{-\frac{3}{2}}\hat{g}_{R,1},\eps^{\frac{1}{2}}\hat{g}_{R,2}](\eps m T_0)\right\|_{\infty }\notag\\
&\quad+C(T_0)\sup\limits_{\eps m T_0\leq \tau\leq \eps
m T_0+s}\|\hat{h}_{R,2}(\tau)\|_2
\notag\\
&\quad+C(T_0)\sup\limits_{\eps m T_0\leq \tau\leq \eps m T_0+s}\left\{\CP_1(\tau)+\eps^{\frac{1}{2}}\tilde{\CQ}(\tau)\right\}.
\end{align}
Now, \eqref{hr12-com-p1} and \eqref{hr12-com-p2} give that
\begin{align}\label{hr12-com-p3}
&\sum\limits_{\bar{k}\in\Z^2}\sup_{0\leq s\leq t}\|w^{l_\infty}[\eps^{-\frac{3}{2}}\hat{h}_{R,1},\eps^{\frac{1}{2}}\hat{h}_{R,2}](t)\|_{\infty}\notag\\
&
\leq C\left\|w^{l_\infty}[\hat{g}^{(1)}_{R,0},\eps^{\frac{1}{2}}\hat{g}^{(2)}_{R,0}]\right\|_{\infty}
+C(T_0)\sum\limits_{\bar{k}\in\Z^2}\sup\limits_{0\leq s\leq t}\left\{\CP_1(s)+\eps^{\frac{1}{2}}\tilde{\CQ}(s)\right\}
\notag\\
&\quad+C(T_0)\sum\limits_{\bar{k}\in\Z^2}\sup\limits_{0\leq s\leq t}\|\hat{h}_{R,2}(s)\|_2,
\end{align}
for any $t\in[0,\infty).$

Furthermore, recalling \eqref{cp1} and \eqref{t-cq}, one has
\begin{align}\label{s-pq}
&\sum\limits_{\bar{k}\in\Z^2}\sup\limits_{0\leq s\leq t}\left\{\CP_1(s)+\eps^{\frac{1}{2}}\tilde{\CQ}(s)\right\} \notag\\
&
\lesssim \eps^{-\frac{1}{2}}\sum\limits_{\bar{k}\in\Z^2}\sup_{0\leq s\leq t}\|w^{l_\infty}\hat{h}_{R,1}(t)\|_{\infty}
+\eps\sum\limits_{\bar{k}\in\Z^2}\sup_{0\leq s\leq t}\|w^{l_\infty}\hat{h}_{R,2}(t)\|_{\infty}
\notag\\
&\quad+\eps\vps_0\vps_{\Phi,\al}+\eps^{\frac{1}{2}}\sum\limits_{\bar{k}\in\Z^2}\|w^{l_\infty} \hat{g}_{R,0}\|_{\infty}
+\eps \vps_0.
\end{align}
Therefore, we get from \eqref{hr12-com-p3} and \eqref{s-pq} that
\begin{align}\label{hr12-if-sum}
\sum\limits_{\bar{k}\in\Z^2}\sup_{0\leq s\leq t}\|w^{l_\infty}[\eps^{-\frac{3}{2}}\hat{h}_{R,1},\eps^{\frac{1}{2}}\hat{h}_{R,2}](t)\|_{\infty}
\leq& C\left\|w^{l_\infty}[\hat{g}^{(1)}_{R,0},\eps^{\frac{1}{2}}\hat{g}^{(2)}_{R,0}]\right\|_{\infty}
+\eps\vps_0\notag\\&+C(T_0)\sum\limits_{\bar{k}\in\Z^2}\sup\limits_{0\leq s\leq t}\|\hat{h}_{R,2}(s)\|_2,
\end{align}
for any $t\in[0,\infty)$. This ends the first step for the $L_{\bar{k}}^1L_{z,\Fv}^\infty$ estimates.

\medskip
\noindent
\underline{{\bf Step 2. $L^1_{\bar{k}}L^2_{T,z,\Fv}$ estimates for $h_{R,2}$.}} The $L^1_{\bar{k}}L^2_{T,z,\Fv}$ estimates for $h_{R,2}$ are obtained from two perspectives: one is the $L^1_{\bar{k}}L^2_{T,z,\Fv}$ estimate for the macroscopic component, and the other is the microscopic part. In line with Section \ref{sec-sape}, we
begin by defining the macroscopic components of the solutions as follows:
\begin{align}
\bar{\FP} \hat{h}_{R,1}=[\hat{a}^{(1)}+\Fv\cdot \hat{\Fb}^{(1)}+\frac{1}{2}(|\Fv|^2-3)\hat{c}^{(1)}]\mu,\notag
\end{align}
\begin{align}
\FP \hat{h}_{R,2}=[\hat{a}^{(2)}+\Fv\cdot \hat{\Fb}^{(2)}+\frac{1}{2}(|\Fv|^2-3)\hat{c}^{(2)}]\sqrt{\mu},\notag
\end{align}
and
\begin{align}
\FP \hat{h}_R=[\hat{a}+\Fv\cdot \hat{\Fb}+\frac{1}{2}(|\Fv|^2-3)\hat{c}]\sqrt{\mu}.
\notag
\end{align}
With these definitions, it follows that
\begin{align}
\hat{a}=\hat{a}^{(1)}+\hat{a}^{(2)},\ \hat{\Fb}=\hat{\Fb}^{(1)}+\hat{\Fb}^{(2)},\ \hat{c}=\hat{c}^{(1)}+\hat{c}^{(2)}.\notag
\end{align}
Note that
\begin{align}\label{mac-hr1}
\|[\hat{a}^{(1)},\hat{\Fb}^{(1)},\hat{c}^{(1)}](\bar{k})\|_2\leq C\|w^{l_2}\hat{h}_{R,1}(\bar{k})\|_2,
\end{align}
for $l_2\geq2$. To establish the $L^1_{\bar{k}}L^2_{T,z}$ estimate for $[a^{(2)}, \Fb^{(2)}, c^{(2)}]$, it suffices to derive the $L^1_{\bar{k}}L^2_{T,z}$ estimate for $[a, \Fb, c]$. For this, similar to the definition of $g_R$ in \eqref{gr-def}, we first define $h_R$ as 
$$
\sqrt{\mu}h_R=h_{R,1}+\sqrt{\mu}h_{R,2},
$$
and then we consider the following equations that the Fourier transform of $h_R$, denoted as $\hat{h}_R$, satisfies
\begin{align}\label{hr}
v_z\pa_z\hat{h}_{R}&+i\bar{k}\cdot\bar{v}\hat{h}_{R}=-\eps\pa_t\hat{h}_{R}+\SH_3,
\end{align}
with
\begin{align}
\SH_3=&-\eps^2\Phi\cdot\na_\Fv \hat{h}_{R}
-i\al\eps zk_x\hat{h}_{R}+\al\eps v_z\pa_{v_x}\hat{h}_{R}
-\frac{1}{\eps}L\hat{h}_{R}-\eps\la_0 \hat{h}_{R}\notag\\
&-\eps^{\frac{3}{2}}\pa_t\hat{h}_2-\eps^{\frac{3}{2}}\al zik_x\hat{h}_2+\eps^{\frac{3}{2}}\al \mu^{-\frac{1}{2}}v_z\pa_{v_x}\{\sqrt{\mu}\hat{h}_2\}\notag\\
&
-\eps^{\frac{3}{2}}\mu^{-\frac{1}{2}}\Phi\cdot\na_\Fv[\sqrt{\mu}(\hat{h}_1+\eps \hat{h}_2)]+\frac{\eps^2}{2}\Phi\cdot\Fv \hat{h}_{R}-\frac{\al\eps v_xv_z}{2}\hat{h}_R\notag\\
&+\eps^{\frac{1}{2}}e^{-\la_0t}\hat{\Ga}(\hat{h}_R,\hat{h}_R)+e^{-\la_0t}\hat{\Ga}(\hat{h}_R,\hat{h}_1+\eps\hat{h}_2)
+e^{-\la_0t}\hat{\Ga}(\hat{h}_1+\eps \hat{h}_2,\hat{h}_R)\notag\\
&+\eps^{\frac{3}{2}}e^{-\la_0t}\hat{\Ga}(\hat{h}_2,\hat{h}_2)
+\Ga(\hat{h}_R,\hat{f}_1+\eps \hat{f}_2)
+\hat{\Ga}(\hat{f}_1+\eps \hat{f}_2,\hat{h}_R)\notag\\
&+\hat{\Ga}(\hat{f}_R,\hat{h}_1+\eps \hat{h}_2)
+\hat{\Ga}(\hat{h}_1+\eps \hat{h}_2,\hat{f}_R)+\eps^{\frac{1}{2}}\{\hat{\Ga}(\hat{h}_R,\hat{f}_R)+\hat{\Ga}(\hat{f}_R,\hat{h}_R)\},\notag
\end{align}
\begin{align}
\sqrt{\mu}\hat{h}_R(0,\Fx,\Fv)=\sqrt{\mu}\hat{g}_{R,0}(\Fx,\Fv),\notag
\end{align}
and 
\begin{align}\notag 
\hat{h}_R(t,\pm1,\Fv)|_{v_z\lessgtr0}=P_\ga \hat{h}_R+\SR.
\end{align}
Note that the mass of $\hat{h}_R$ at zero frequency is conserved, i.e. 
$$
\int_{-1}^1\int_{\R^3}\hat{h}_R(t,z,0,\Fv)\sqrt{\mu}d\Fv dz=0,
$$
and such conservation is of significant importance in the estimating of $a$.

In the following, we intend to show that
\begin{align}\label{mac-es}
&\eps\CE^{int}_{mac}(t)+\sum\limits_{\bar{k}}\|[\hat{a},\hat{\Fb},\hat{c}](t)\|_{L^2_{T,z}}\notag\\
&\lesssim \eps\CE^{int}_{mac}(0)\notag\\
&\quad
+\eps^{\frac{1}{2}}\sum\limits_{\bar{k}\in\Z^2}\left\{\|w^{l_\infty}\hat{h}_{R,1}\|_{L_T^\infty L^2_{z,\Fv}}+\|w^{l_\infty}\hat{h}_{R,2}\|_{L_T^\infty L^2_{z,\Fv}}\right\}
\sum\limits_{\bar{k}\in\Z^2}\|\hat{h}_{R,2}\|_{L^2_{T,z,\Fv}}
\notag\\
&\quad+\sum\limits_{\bar{k}\in\Z^2}|\{\FI-\bar{P}_\ga\}\hat{f}_{R,1}(\pm1)|_{L^2_TL^2_{\ga_+}}
+\sum\limits_{\bar{k}\in\Z^2}|\{\FI-P_\ga\}\hat{f}_{R,2}(\pm1)|_{L^2_TL^2_{\ga_+}}
\notag\\
&\quad+\eps^{\frac{1}{2}}\sum\limits_{\bar{k}\in\Z^2}\|w^{l_\infty}\hat{h}_{R,1}\|_{L_T^\infty L^2_{z,\Fv}}
\sum\limits_{\bar{k}\in\Z^2}\|\hat{h}_{R,1}\|_{L^2_{T,z,\Fv}}
+\eps^{-1}\sum\limits_{\bar{k}\in\Z^2}
\|w^{l_2}\hat{h}_{R,1}\|_{L^2_{T,z,\Fv}}\notag\\
&\quad+
\eps^{-1}\sum\limits_{\bar{k}\in\Z^2}\|\{\FI-\FP\}\hat{h}_{R,2}\|_{L^2_{T,z,\Fv}}+\vps_0,
\end{align}
where
\begin{align}
|\CE^{int}_{mac}(t)|\lesssim\|w^{l_2}\hat{h}_{R,1}\|_{L_T^\infty L^2_{z,\Fv}}+\|\hat{h}_{R,2}\|_{L_T^\infty L^2_{z,\Fv}}.\notag
\end{align}

To prove \eqref{mac-es},
we only present the $L^1_{\bar{k}}L^2_{T,z}$ estimate for $\Fb$ for brevity, as the corresponding estimates for $a$ and $c$ can be obtained similarly. To achieve this, we introduce
\begin{eqnarray*}
\hat{\Psi}_{\Fb}=\sum\limits_{m=1}^3\hat{\Psi}^{J,m}_{\Fb}
,\ J=1,2,3,\ \Fb=(b_{1},b_{2},b_{3}),
\end{eqnarray*}
with
\begin{eqnarray*}
\hat{\Psi}^{J,m}_{\Fb}
=\left\{\begin{array}{l}
\dis |\Fv|^2v_mv_J\widehat{\pa_m\phi_{J}}-\frac{7}{2}(v_m^2-1)\widehat{\pa_J\phi_{J}}\mu^{\frac{1}{2}},\ \ J\neq m,\\[2mm]
\dis \frac{7}{2}(v_J^2-1)\widehat{\pa_J\phi_{J}}\mu^{\frac{1}{2}},\ \ J=m,
\end{array}\right.
\end{eqnarray*}
where
\begin{equation}
(|\bar{k}|^2-\pa_z^2)\hat{\phi}_{J}=\hat{b}_{s,J},\ \hat{\phi}_{J}(\bar{k},\pm1)=0,\
  {\rm and}\ \pa_1=\pa_{x}, \pa_2=\pa_{y}, \pa_3=\pa_{z}.\notag
\end{equation}
Moreover, by elliptic estimates and trace theorem, it follows
\begin{align}
\||\bar{k}|^2\hat{\phi}_{J}(\bar{k},z)\|_{2}+\||\bar{k}|\hat{\phi}_{J}(\bar{k},z)\|_{H^1_z}
+\|\hat{\phi}_{J}\|_{H^{2}_z}\lesssim \|\hat{b}_{s,J}\|_{2},\ \forall\bar{k}\in\Z^2,\notag
\end{align}
\begin{align}
(1+|\bar{k}|)\left\{\||\bar{k}|\hat{\phi}_{J}(\bar{k},z)\|_{2}+\|\hat{\phi}_{J}(\bar{k},z)\|_{H^1_z}\right\}
\lesssim \|\hat{b}_{s,J}\|_{2},\ \forall\bar{k}\in\Z^2,\notag
\end{align}
and
\begin{align}
(1+|\bar{k}|)\|\hat{\phi}_{J}(\bar{k},\pm1)\|_2+\|\pa_z\hat{\phi}_{J}(\bar{k},\pm1)\|_2\lesssim\|\hat{b}_{s,J}\|_{2},\ \forall\bar{k}\in\Z^2.\notag
\end{align}

Now taking the inner product of \eqref{hr} and $\hat{\Psi}^{J,m}_{\Fb}$ over $(\bar{k},v)\in\Z^2\times\R^3$, one has
\begin{align}
&-\sum\limits_{m=1}^3(\Fv\cdot\widehat{\na_\Fx\Psi^{J,m}_{\Fb}}, \FP \hat{h}_{R})\notag\\
&=-\sum\limits_{J=1}^3( \hat{\Psi}^{J,m}_{\Fb}(1), v_z \hat{h}_{R}(1))
+\sum\limits_{J=1}^3( \hat{\Psi}^{J,m}_{\Fb_s}(-1), v_z \hat{h}_{R}(-1))\notag\\
&\quad-\eps(\pa_t\hat{h}_{R},\hat{\Psi}^{J,m}_{\Fb})
-\sum\limits_{m=1}^3(\Fv\cdot\widehat{\na_\Fx\Psi^{J,m}_{\Fb}}, \{\FI-\FP\} \hat{h}_{R})+(\SH_3,\hat{\Psi}^{J,m}_{\Fb}).\label{b-ust}
\end{align}
In accordance with the computation in \eqref{bs-es}, the left hand side of \eqref{b-ust} gives
$7\|\hat{\Fb}\|_2^2.$ In order to control the third term on the right hand side of \eqref{b-ust}, by taking the moment of \eqref{hr}
with $\Fv\sqrt{\mu}$, we obtain
\begin{align}
\eps \pa_t\hat{\Fb}+\widehat{\na_{\Fx}(a+2c)}=-\lag\widehat{\Fv\cdot\na_\Fx\{\FI-\FP\}h_R},\Fv\sqrt{\mu}\rag
+\lag\SH_3,\Fv\sqrt{\mu}\rag,\notag
\end{align}
with this, we further have
\begin{align}
-\eps(\pa_t\hat{h}_{R},\hat{\Psi}^{J,m}_{\Fb})=&-\eps\frac{d}{dt}(\hat{h}_{R},\hat{\Psi}^{J,m}_{\Fb})
+\eps(\hat{h}_{R},\pa_t\hat{\Psi}^{J,m}_{\Fb})\notag\\
\leq&-\eps\frac{d}{dt}(\hat{h}_{R},\hat{\Psi}^{J,m}_{\Fb})+C\|[\hat{a},\hat{c}]\|_2^2+\eta\|\hat{\Fb}\|_2^2
+C_\eta\|w^{l_2}\hat{h}_{R,1}\|_2^2
\notag\\&+C_\eta\|\{\FI-\FP\}\hat{h}_{R,2}\|_2^2+C_\eta\|\lag\SH_3,\Fv\sqrt{\mu}\rag\|_2^2.\label{t-der}
\end{align}
Plugging \eqref{t-der} into \eqref{b-ust}, we then have
\begin{align}
\eps\sum\limits_{\bar{k}}&(\hat{h}_{R}(t),\hat{\Psi}^{J,m}_{\Fb}(t))+\sum\limits_{\bar{k}}\|\hat{\Fb}(t)\|_{L^2_{T,z}}\notag\\
\lesssim& \eps\sum\limits_{\bar{k}}(\hat{h}_{R}(0),\hat{\Psi}^{J,m}_{\Fb}(0))\notag\\
&+\eps^{\frac{1}{2}}\sum\limits_{\bar{k}\in\Z^2}\left\{\|w^{l_\infty}\hat{h}_{R,1}\|_{L_T^\infty L^2_{z,\Fv}}+\|w^{l_\infty}\hat{h}_{R,2}\|_{L_T^\infty L^2_{z,\Fv}}\right\}
\sum\limits_{\bar{k}\in\Z^2}\|\hat{h}_{R,2}\|_{L^2_{T,z,\Fv}}
\notag\\
&+\sum\limits_{\bar{k}\in\Z^2}|\{\FI-\bar{P}_\ga\}\hat{f}_{R,1}(\pm1)|_{L^2_TL^2_{\ga_+}}
+\sum\limits_{\bar{k}\in\Z^2}|\{\FI-P_\ga\}\hat{f}_{R,2}(\pm1)|_{L^2_TL^2_{\ga_+}}
\notag\\&+\eps^{\frac{1}{2}}\|w^{l_\infty}\hat{h}_{R,1}\|_{L_T^\infty L^2_{z,\Fv}}
\sum\limits_{\bar{k}\in\Z^2}\|\hat{h}_{R,1}\|_{L^2_{T,z,\Fv}}
+\eps^{-1}\sum\limits_{\bar{k}\in\Z^2}
\|w^{l_2}\hat{h}_{R,1}\|_{L^2_{T,z,\Fv}}\notag\\&+
\eps^{-1}\sum\limits_{\bar{k}\in\Z^2}\|\{\FI-\FP\}\hat{h}_{R,2}\|_{L^2_{T,z,\Fv}}
+\sum\limits_{\bar{k}}\|[\hat{a},\hat{c}](t)\|_{L^2_{T,z}}+\vps_0.\notag
\end{align}
Here, we also employed a calculation similar to that in \eqref{s7-as} to manage the nonlinear estimates in $\SH_3$. Consequently, \eqref{mac-es} holds true. Recalling \eqref{mac-hr1}, \eqref{mac-es} further gives
\begin{align}\label{mac-hr2-es}
\eps\CE^{int}_{mac}(t)&+\sum\limits_{\bar{k}}\|[\hat{a}^{(2)},\hat{\Fb}^{(2)},\hat{c}^{(2)}](t)\|_{L^2_{T,z}}\notag\\
\lesssim& \eps\CE^{int}_{mac}(0)\notag\\
&+\eps^{\frac{1}{2}}\sum\limits_{\bar{k}\in\Z^2}\left\{\|w^{l_\infty}\hat{h}_{R,1}\|_{L_T^\infty L^\infty_{z,\Fv}}+\|w^{l_\infty}\hat{h}_{R,2}\|_{L_T^\infty L^\infty_{z,\Fv}}\right\}
\sum\limits_{\bar{k}\in\Z^2}\|\hat{h}_{R,2}\|_{L^2_{T,z,\Fv}}
\notag\\
&+\sum\limits_{\bar{k}\in\Z^2}|\{\FI-\bar{P}_\ga\}\hat{f}_{R,1}(\pm1)|_{L^2_TL^2_{\ga_+}}
+\sum\limits_{\bar{k}\in\Z^2}|\{\FI-P_\ga\}\hat{f}_{R,2}(\pm1)|_{L^2_TL^2_{\ga_+}}
\notag\\&+\eps^{\frac{1}{2}}\|w^{l_\infty}\hat{h}_{R,1}\|_{L_T^\infty L^\infty_{z,\Fv}}
\sum\limits_{\bar{k}\in\Z^2}\|\hat{h}_{R,1}\|_{L^2_{T,z,\Fv}}
+\eps^{-1}\sum\limits_{\bar{k}\in\Z^2}
\|w^{l_2}\hat{h}_{R,1}\|_{L^2_{T,z,\Fv}}\notag\\&+
\eps^{-1}\sum\limits_{\bar{k}\in\Z^2}\|\{\FI-\FP\}\hat{h}_{R,2}\|_{L^2_{T,z,\Fv}}+\vps_0.
\end{align}
We now turn to obtain the estimate for $\{\FI-\FP\}\hat{h}_{R,2}$, which represents the microscopic component of $\hat{h}_{R,2}$.
To achieve this, we take the inner product of \eqref{f-hr2} with $\hat{h}_{R,2}$ over $(z,\Fv)\in(-1,1)\times\R^3$. Further, by considering the real part of the resulting identity, we obtain
\begin{align}\label{hr2-ip}
&\frac{1}{2}\frac{d}{dt}\|\hat{h}_{R,2}\|_2^2+\eps^{-1}|\{\FI-P_\ga\}\hat{h}_{R,2}|^2_{2,+}
+\frac{\de_0}{\eps^2}\|\{\FI-\FP\}\hat{h}_{R,2}\|_\nu^2\notag\\
&\qquad-\{\la_0+\eta+\al\} \|\hat{h}_{R,2}\|_2^2\notag\\
&\leq \frac{1}{\eps^4}\| h_{R,1}\|_2^2+|(\SH_2,\hat{h}_{R,2})|+\eta|P_\ga\hat{h}_{R,2}|^2_{2,-}
+\frac{C_\eta}{\eps^2}|\bar{P}_\ga\hat{h}_{R,1}|^2_{2,-}+C\eps^{-1}\|\SR\|_2^2.
\end{align}
On the other hand, trace Lemma \ref{ukai} leads us to
\begin{align}\label{pga-h2-2}
|P_{\gamma }\hat{h}_{R,2}(\pm1)|_{2,\pm }^{2}\leq& \vps\int_{v_{z}\gtrless0}|\hat{h}_{R,2}(\pm1)|^{2}|v_{z}|d\Fv
+C\|\hat{h}_{R,2}\|_2^{2}+\frac{C}{\eps^2}\|\{\FI-\FP\}\hat{h}_{R,2}\|^2_\nu\notag\\&+\frac{C}{\eps^4}\| w^{l_2}\hat{h}_{R,1}\|_{2}^{2}
+|(\SH_2,\hat{h}_{R,2})|.
\end{align}
Next, by utilizing Lemma \eqref{es-L} together with \eqref{hr2-ip} and \eqref{pga-h2-2}, we can deduce that
\begin{align}\label{hr2-ip2}
&\sum\limits_{\bar{k}\in\Z^2}\sup\limits_{0\leq t\leq T}\|\hat{h}_{R,2}\|_2+\eps^{-\frac{1}{2}}\sum\limits_{\bar{k}\in\Z^2}\left(\int_0^T|\hat{h}_{R,2}(\pm1)|^2dt\right)^{\frac{1}{2}}\notag\\
&\qquad\qquad+\frac{\de_0}{2\eps}\sum\limits_{\bar{k}\in\Z^2}\left(\int_0^T\|\{\FI-\FP\}\hat{h}_{R,2}\|_\nu^2dt\right)^{\frac{1}{2}}\notag\\
&\lesssim\sum\limits_{\bar{k}\in\Z^2}\left\|\hat{g}^{(2)}_{R,0}\right\|_{2}
+\eps^{-2}\sum\limits_{\bar{k}\in\Z^2}\left(\int_0^T\| \hat{h}_{R,1}\|_2^2dt\right)^{\frac{1}{2}}
+\sum\limits_{\bar{k}\in\Z^2}\left(\int_0^T|(\SH_2,\hat{h}_{R,2})|dt\right)^{\frac{1}{2}}\notag\\
&\quad+\eps^{-1}\sum\limits_{\bar{k}\in\Z^2}\left(\int_0^T|\bar{P}_\ga\hat{h}_{R,1}|^2_{2,-}dt\right)^{\frac{1}{2}}
+\eps^{-\frac{1}{2}}\sum\limits_{\bar{k}\in\Z^2}\left(\int_0^T\|\SR\|_2^2dt\right)^{\frac{1}{2}}.
\end{align}
Next, by applying Lemma \ref{es-tri}, one has
\begin{align}
\sum\limits_{\bar{k}\in\Z^2}&\left(\int_0^T|(\SH_2,\hat{h}_{R,2})|dt\right)^{\frac{1}{2}}\notag\\
\leq&\frac{\eta}{\eps}\sum\limits_{\bar{k}\in\Z^2}\left(\int_0^T\|\{\FI-\FP\}\hat{h}_{R,2}\|_\nu^2dt\right)^{\frac{1}{2}}\notag\\\
&\qquad \qquad \qquad \qquad +\eps^{\frac{1}{2}}C_\eta\sum\limits_{\bar{k}\in\Z^2}\|w^{l_\infty}\hat{h}_{R,2}\|_{L_T^\infty L^\infty_{z,\Fv}}
\sum\limits_{\bar{k}\in\Z^2}\|\hat{h}_{R,2}\|_{L^2_{T,z,\Fv}}\notag\\
&+C_\eta\left(\sum\limits_{\bar{k}\in\Z^2}\|w^{l_\infty}[\hat{h}_{1},\hat{h}_{2}]\|_{L_T^\infty L^\infty_{z,\Fv}}\right.\notag\\
&\qquad \qquad \qquad \qquad \left.+\sum\limits_{\bar{k}\in\Z^2}\|w^{l_\infty}[\hat{f}_{1},\hat{f}_{2},\eps^{\frac{1}{2}}\hat{f}_{R,2}]\|_{L^\infty_{z,\Fv}}\right)
\sum\limits_{\bar{k}\in\Z^2}\|\hat{h}_{R,2}\|_{L^2_{T,z,\Fv}}\notag\\
&+C_\eta\sum\limits_{\bar{k}\in\Z^2}\|w^{l_\infty}\hat{f}_{R,2}\|_{L^2_{z,\Fv}}
\sum\limits_{\bar{k}\in\Z^2}\|w^{l_\infty}[\hat{h}_{1},\hat{h}_{2}]\|_{L_T^\infty L^\infty_{z,\Fv}}.
\label{h2-nn-es}
\end{align}
Consequently, \eqref{hr2-ip2}, \eqref{h2-nn-es}, \eqref{ust-apa} and Lemma \ref{coef-es} lead to
\begin{align}\label{hr2-mic-es}
&\sum\limits_{\bar{k}\in\Z^2}\|\hat{h}_{R,2}\|_{L_T^\infty L^2_{z,\Fv}}+\eps^{-\frac{1}{2}}\sum\limits_{\bar{k}\in\Z^2}|\{\FI-P_\ga\}\hat{h}_{R,2}(\pm1)|_{L^2_TL^2_{\ga_+}}\notag\\
&\quad
+\eps^{-1}\sum\limits_{\bar{k}\in\Z^2}\|\{\FI-\FP\}\hat{h}_{R,2}\|_{L_T^2 L^2_{z}L^2_{\nu}}\notag\\
&\lesssim\sum\limits_{\bar{k}\in\Z^2}\left\|\hat{g}^{(2)}_{R,0}\right\|_{2}+
(\la_0+\eta+\al)\sum\limits_{\bar{k}}\|[\hat{a}^{(2)},\hat{\Fb}^{(2)},\hat{c}^{(2)}](t)\|_{L^2_{T}L^2_{z}}\notag\\
&
\quad+\eps^{-2}\sum\limits_{\bar{k}\in\Z^2}\|\hat{h}_{R,1}\|_{L^2_{T,z,\Fv}}+\eps^{-1}\sum\limits_{\bar{k}\in\Z^2}|w^{l_2}\hat{h}_{R,1}(\pm1)|_{L^2_TL^2_{\ga_+}}+\vps_0.
\end{align}
This ends the second step for the $L^1_{\bar{k}}L^2_{T,z,\Fv}$ estimates of  $h_{R,2}$.

\medskip
\noindent
\underline{{\bf Step 3. $L^1_{\bar{k}}L^2_{T,z,\Fv}$ estimates for $\hat{h}_{R,1}$.}} In this step, we turn to derive the $L^1_{\bar{k}}L^2_{T,z,\Fv}$ estimates for $\hat{h}_{R,1}$. For this, taking the inner product of \eqref{f-hr1} with $w^{l_2}\hat{h}_{R,1}$
over $(z,\Fv)\in(-1,1)\times\R^3$, one has
\begin{align}\label{f-hr1-ip}
(\pa_t\hat{h}_{R,1},&w^{2l_2}\hat{h}_{R,1})+\eps^{-1}(i\bar{k}\cdot\bar{v}\hat{h}_{R,1},w^{2l_2}\hat{h}_{R,1})
+\eps^{-1}(v_z\pa_z\hat{h}_{R,1},w^{2l_2}\hat{h}_{R,1})\notag\\
&+\eps(\Phi\cdot\na_\Fv \hat{h}_{R,1},w^{2l_2}\hat{h}_{R,1})
+i\al (zk_x\hat{h}_{R,1},w^{2l_2}\hat{h}_{R,1})\notag\\
&-\al (v_z\pa_{v_x}\hat{h}_{R,1},w^{2l_2}\hat{h}_{R,1})
+\frac{1}{\eps^2}(\nu \hat{h}_{R,1},w^{2l_2}\bar{h}_{R,1})-\la_0 (\hat{h}_{R,1},w^{2l_2}\hat{h}_{R,1})\notag\\
=&\frac{1}{\eps^2}(\chi_M\CK \hat{h}_{R,1},w^{2l_2}\bar{h}_{R,1})+\frac{\eps}{2}(\sqrt{\mu}\Phi\cdot\Fv \hat{h}_{R,2},w^{2l_2}\hat{h}_{R,1})
\notag\\&+\eps^{\frac{1}{2}}(-
\sqrt{\mu}\pa_t\hat{h}_2-\al z\sqrt{\mu}ik_x\hat{h}_2+\al v_z\pa_{v_x}\{\sqrt{\mu}\hat{h}_2\}\notag\\
&-\Phi\cdot\na_\Fv[\sqrt{\mu}(h_1+\eps h_2)]],w^{2l_2}\hat{h}_{R,1})
\notag\\&
-\frac{\al}{2} (v_xv_z\sqrt{\mu}\{\FI-\FP\}\hat{h}_{R,2},w^{2l_2}\hat{h}_{R,1})
\notag\\
&+(\SH_1,w^{2l_2}\hat{h}_{R,1}).
\end{align}
Taking the real part of \eqref{f-hr1-ip} and using Lemma \ref{es-tri}, we then have
\begin{align}
\frac{d}{dt}\|&w^{l_2}\hat{h}_{R,1}\|_2^2
+\eps^{-1}\int_{v_z>0}v_zw^{2l_2}|\hat{h}_{R,1}(1)|^2d\Fv
\notag\\
&\qquad \qquad \qquad -\eps^{-1}\int_{v_z<0}v_zw^{2l_2}|\hat{h}_{R,1}(-1)|^2d\Fv
+\frac{1}{2\eps^2}\|w^{l_2}\hat{h}_{R,1}\|_\nu^2\notag\\
\lesssim& \eps^4\|\hat{h}_{R,2}\|_2^2+\al^2\eps^{2}\|\{\FI-\FP\}\hat{h}_{R,2}\|_2^2
+\eps^{3}\|\pa_{t}\hat{h}_2\|_2^2+\eps^{3}\|\lag\bar{k}\rag[\hat{h}_1,\hat{h}_2]\|_2^2
\notag\\
&+\eps^{-\frac{1}{2}}\int_{-1}^1\|w^{l_2}\hat{h}_{R,1}(\bar{k})\|_\nu
\sum\limits_{\bar{l}}
\left\{\|w^{l_2}\hat{h}_{R,1}(\bar{k}-\bar{l})\|_2\|w^{l_2}\hat{h}_{R,1}(\bar{l})\|_\nu \right. \notag\\
& \qquad \qquad\qquad\qquad\qquad\qquad\left.+\|w^{l_2}\hat{h}_{R,1}(\bar{k}-\bar{l})\|_\nu\|w^{l_2}\hat{h}_{R,1}(\bar{l})\|_2\right\}dz
\notag\\&
+\eps^{-\frac{1}{2}}\int_{-1}^1\|w^{l_2}\hat{h}_{R,1}(\bar{k})\|_\nu\sum\limits_{\bar{l}}
\left\{\|\hat{h}_{R,2}(\bar{k}-\bar{l})\|_2\|w^{l_2}\hat{h}_{R,1}(\bar{l})\|_\nu\right.
\notag\\
&\qquad \quad\qquad\qquad\qquad\qquad\qquad+\left. \|\hat{h}_{R,2}(\bar{k}-\bar{l})\|_\nu\|w^{l_2}\hat{h}_{R,1}(\bar{l})\|_2\right\}dz
\notag\\
&+\eps^{-1}\int_{-1}^1\|w^{l_2}\hat{h}_{R,1}(\bar{k})\|_\nu\sum\limits_{\bar{l}}
\left\{\|[\hat{h}_{1},\hat{h}_{2}](\bar{k}-\bar{l})\|_2\|w^{l_2}\hat{h}_{R,1}(\bar{l})\|_\nu
\right.\notag\\
&\qquad \qquad\qquad\qquad\qquad\qquad\left.+\|[\hat{h}_{1},\hat{h}_{2}](\bar{k}-\bar{l})\|_\nu\|w^{l_2}\hat{h}_{R,1}(\bar{l})\|_2\right\}dz
\notag\\&+\eps^{-1}\int_{-1}^1\|w^{l_2}\hat{h}_{R,1}(\bar{k})\|_\nu\sum\limits_{\bar{l}}
\left\{\|[\hat{f}_{1},\hat{f}_{2}](\bar{k}-\bar{l})\|_\nu\|w^{l_2}\hat{h}_{R,1}(\bar{l})\|_2
\right. \notag \\
&\qquad\qquad \qquad\qquad\qquad\qquad\left.+\|[\hat{f}_{1},\hat{f}_{2}](\bar{k}-\bar{l})\|_\nu\|w^{l_2}\hat{h}_{R,1}(\bar{l})\|_\nu\right\}dz
\notag\\&+\eps^{-\frac{1}{2}}\int_{-1}^1\|w^{l_2}\hat{h}_{R,1}(\bar{k})\|_\nu\sum\limits_{\bar{l}}
\Big\{\|w^{l_2}[\hat{f}_{R,1},\sqrt{\mu}\hat{f}_{R,2}](\bar{k}-\bar{l})\|_\nu\|w^{l_2}\hat{h}_{R,1}(\bar{l})\|_2
\notag\\&\qquad\qquad\qquad\qquad+\|w^{l_2}[\hat{f}_{R,1},\sqrt{\mu}\hat{f}_{R,2}](\bar{k}-\bar{l})\|_2
\|w^{l_2}\hat{h}_{R,1}(\bar{l})\|_\nu\Big\}dz
\notag\\
&+\eps^{-\frac{1}{2}}\int_{-1}^1\|w^{l_2}\hat{h}_{R,1}(\bar{k})\|_\nu\sum\limits_{\bar{l}}
\left\{\|w^{l_2}\hat{f}_{R,1}(\bar{k}-\bar{l})\|_\nu\|\hat{h}_{R,2}(\bar{l})\|_2\right. \notag \\
&\qquad\qquad \qquad\qquad\qquad\qquad\qquad\left.+\|w^{l_2}\hat{f}_{R,1}(\bar{k}-\bar{l})\|_2
\|\hat{h}_{R,2}(\bar{l})\|_\nu\right\}dz\notag\\
&+\eps^{-1}\int_{-1}^1\|w^{l_2}\hat{h}_{R,1}(\bar{k})\|_\nu\sum\limits_{\bar{l}}
\left\{\|w^{l_2}\hat{f}_{R,1}(\bar{k}-\bar{l})\|_\nu \|[\hat{h}_{1},\hat{h}_{2}](\bar{l})\|_2\right. \notag \\
& \qquad\qquad\qquad\qquad\qquad\qquad+\left. \|w^{l_2}\hat{f}_{R,1}(\bar{k}-\bar{l})\|_2 \|[\hat{h}_{1},\hat{h}_{2}](\bar{l})\|_\nu
\right\}dz.\notag
\end{align}
Consequently, we have
\begin{align}
\|&w^{l_2}\hat{h}_{R,1}(t)\|_2
+\eps^{-\frac{1}{2}}\left(\int_0^t\int_{v_z>0}v_zw^{2l_2}|\hat{h}_{R,1}(1)|^2d\Fv ds\right)^{\frac{1}{2}}
\notag\\
&\quad  +\eps^{-\frac{1}{2}}\left(\int_0^t\int_{v_z<0}|v_z|w^{2l_2}|\hat{h}_{R,1}(-1)|^2d\Fv ds\right)^{\frac{1}{2}}
+\frac{1}{2\eps}\left(\int_0^t\|w^{l_2}\hat{h}_{R,1}\|_\nu^2ds\right)^{\frac{1}{2}}\notag\\
\lesssim&
\sum\limits_{\bar{k}\in\Z^2}\eps^{\frac{3}{2}}\left\|w^{l_2}\hat{g}^{(1)}_{R,0}\right\|_{2}
+\eps^2\left(\int_0^t\|\hat{h}_{R,2}\|_2^2ds\right)^{\frac{1}{2}}
+\al\eps\left(\int_0^t\|\{\FI-\FP\}\hat{h}_{R,2}\|_2^2ds\right)^{\frac{1}{2}}
\notag\\&+\eps^{\frac{3}{2}}\left(\int_0^t\|\pa_{t}\hat{h}_2\|_2^2ds\right)^{\frac{1}{2}}
+\eps^{\frac{3}{2}}\left(\int_0^t\|\lag\bar{k}\rag[\hat{h}_1,\hat{h}_2]\|_2^2ds\right)^{\frac{1}{2}}
\notag\\&+\frac{\eta}{\eps}\left(\int_0^t\|w^{l_2}\hat{h}_{R,1}\|_\nu^2ds\right)^{\frac{1}{2}}\notag\\
&+C_\eta\sqrt{\eps}\sum\limits_{\bar{l}}\left(\int_0^t\notag
\|w^{l_\infty}\hat{h}_{R,1}(\bar{k}-\bar{l})\|_\infty^2\|w^{l_2}\hat{h}_{R,1}(\bar{l})\|_\nu^2ds\right)^{\frac{1}{2}}
\\&+C_\eta\sqrt{\eps}\sum\limits_{\bar{l}}\left(\int_0^t\|w^{l_\infty}\hat{h}_{R,2}(\bar{k}-\bar{l})\|_\infty^2
\|w^{l_2}\hat{h}_{R,1}(\bar{l})\|_\nu^2ds\right)^{\frac{1}{2}}
\notag\\
&+C_\eta\sum\limits_{\bar{l}}\left(\int_0^t\|w^{l_\infty}[\hat{f}_{1},\hat{f}_{2},\hat{h}_{1},\hat{h}_{2}]
(\bar{k}-\bar{l})\|_\infty^2\|w^{l_2}\hat{h}_{R,1}(\bar{l})\|_\nu^2ds\right)^{\frac{1}{2}}
\notag\\&+C_\eta\sqrt{\eps}\sum\limits_{\bar{l}}\left(\int_0^t\|w^{l_\infty}[\hat{f}_{R,1},\hat{f}_{R,2}](\bar{k}-\bar{l})\|_\infty^2
\|w^{l_2}\hat{h}_{R,1}(\bar{l})\|_\nu^2ds\right)^{\frac{1}{2}}
\notag\\
&+C_\eta\sqrt{\eps}\sum\limits_{\bar{l}}\left(\int_0^t\|w^{l_\infty}\hat{f}_{R,1}(\bar{k}-\bar{l})\|_\infty^2
\|\hat{h}_{R,2}(\bar{l})\|_\nu^2ds\right)^{\frac{1}{2}}\notag\\&
+C_\eta\sum\limits_{\bar{l}}\left(\int_0^t\|w^{l_\infty}\hat{f}_{R,1}(\bar{k}-\bar{l})\|_\infty^2
\|[\hat{h}_{1},\hat{h}_{2}](\bar{l})\|_\nu^2ds\right)^{\frac{1}{2}}.\notag
\end{align} 
The above estimate further yields
\begin{align}\label{hr1-l2-es}
&\eps^{-\frac{3}{2}}\sum\limits_{\bar{k}\in\Z^2}\|w^{l_2}\hat{h}_{R,1}\|_{L_T^\infty L^2_{z,\Fv}}+\eps^{-2}\sum\limits_{\bar{k}\in\Z^2}|w^{l_2}\hat{h}_{R,1}(\pm1)|_{L^2_TL^2_{\ga_+}}\notag\\
&\quad 
+\eps^{-\frac{5}{2}}\sum\limits_{\bar{k}\in\Z^2}\|w^{l_2}\hat{h}_{R,1}\|_{L^2_{T,z,\nu}}\notag\\
&
\lesssim \sum\limits_{\bar{k}\in\Z^2}\left\|w^{l_2}\hat{g}^{(1)}_{R,0}\right\|_{\infty}+
(\eps^{\frac{1}{2}}+\sqrt{\vps_0}+\al+\vps_{\Phi,\al})\sum\limits_{\bar{k}\in\Z^2}\|\hat{h}_{R,2}\|_{L^2_{T,z,\nu}}
+\vps_0.
\end{align}
This ends the $L^1_{\bar{k}}L^2_{T,z,\Fv}$ estimates for $\hat{h}_{R,1}$.

Finally, we get from \eqref{hr12-if-sum}, \eqref{mac-hr2-es}, \eqref{hr2-mic-es} and \eqref{hr1-l2-es} that
\begin{align}\label{rm-es-fin}
&\eps^{-\frac{3}{2}}\sum\limits_{\bar{k}\in\Z^2}\|w^{l_\infty}\hat{h}_{R,1}\|_{L_T^\infty L^\infty_{z,\Fv}}+
\eps^{-\frac{3}{2}}\sum\limits_{\bar{k}\in\Z^2}\|w^{l_2}\hat{h}_{R,1}\|_{L_T^\infty L^2_{z,\Fv}}\notag\\
&
\quad+\eps^{-\frac{5}{2}}\sum\limits_{\bar{k}\in\Z^2}\|w^{l_2}\hat{h}_{R,1}\|_{L^2_{T,z,\nu}}
+\eps^{\frac{1}{2}}\sum\limits_{\bar{k}\in\Z^2}\|w^{l_\infty}\hat{h}_{R,2}\|_{L_T^\infty L^\infty_{z,\Fv}}\notag\\
&\quad
+\sum\limits_{\bar{k}\in\Z^2}\|\hat{h}_{R,2}\|_{L_T^\infty L^2_{z,\Fv}}+\sum\limits_{\bar{k}}\|[\hat{a}^{(2)},\hat{\Fb}^{(2)},\hat{c}^{(2)}](t)\|_{L^2_{T,z}}
\notag\\
&\quad+\eps^{-1}\sum\limits_{\bar{k}\in\Z^2}\|\{\FI-\FP\}\hat{h}_{R,2}\|_{L^2_{T,z,\nu}}
\leq C\sum\limits_{\bar{k}\in\Z^2}\left\|w^{l_\infty}[\hat{g}^{(1)}_{R,0},\hat{g}^{(2)}_{R,0}]\right\|_{\infty}+C\vps_0.
\end{align}
Then, the desired estimate \eqref{rm-de}  follows from \eqref{rm-es-fin}.

\medskip
\noindent
\underline{{\bf Step 4. Non-negativity of the solution.}}
To complete the proof of Theorem \ref{sta-th},
we now turn to prove that the unique global solution constructed above is non-negative, i.e.
$$
F^\eps(t,\Fx,\Fv)=F_{st}^\eps(\Fx,\Fv)+\eps\sqrt{\mu}\{[g_1+\eps g_2+\eps^{\frac{1}{2}}g_R](t,\Fx,\Fv)\}\geq0,
$$ 
under the condition that 
$$
F^\eps_0(\Fx,\Fv)=F_{st}^\eps(\Fx,\Fv)+\eps\sqrt{\mu}\{g_1(0,\Fx,\Fv)+\eps g_2(0,\Fx,\Fv)+\eps^{\frac{1}{2}}g_{R,0}(\Fx,\Fv)\}\geq0,
$$
which also indicates the non-negativity of the steady solution $F_{st}^\eps(\Fx,\Fv)$ obtained in Theorem \ref{st-sol-th} due to the large time asymptotic behavior \eqref{rm-de}. To do so, let us start from the following approximation system to the equation \eqref{s-rbe}:
\begin{eqnarray}\label{ap.bef}
\left\{
\begin{array}{l}
\dis \pa_tF^{n+1}+\eps^{-1}\Fv\cdot\na_\Fx F^{n+1}+\eps\Phi\cdot\na_\Fv F^{n+1}+\al z\pa_{x}F^{n+1}-\al v_z\pa_{v_x}F^{n+1}\\[2mm]
\dis \qquad\qquad+\eps^{-2}F^{n+1}\SV(F^{n})=\eps^{-2}Q_{\rm{gain}}(F^{n},F^{n}),\ t>0,\ \Fx\in\Om,\ \Fv\in\R^3,\\[2mm]
\dis F^{n+1}(t,x,y,\pm1,\Fv)|_{v_{z}\lessgtr0}=\sqrt{2\pi}\mu\int_{v_{z}\gtrless0}F^{n}(t,x,y,\pm1,\Fv)|v_{z}|d\Fv,\\[2mm]
\dis F^{n+1}(0,\Fx,\Fv)=F_0(\Fx,\Fv),
\end{array}\right.
\end{eqnarray}
where
$$
\SV(F^{n})=\int_{\R^3\times\S^2_+}B_0(\Fv-\Fv_\ast,\om)F^{n}(\Fv_\ast)d\Fv_\ast d\om.
$$
One can see that if $F_0(\Fx,\Fv)\geq0$ and $F^{n}(t,\Fx,\Fv)\geq0$, then any solution of \eqref{ap.bef} should be non-negative.
Let $$F^{n+1}=F_{st}^\eps+\eps\sqrt{\mu}\{[g_1+\eps g_2+\eps^{\frac{1}{2}}g^{n+1}_R](t,\Fx,\Fv)\},$$
where $g_1$ and $g_2$ are given by \eqref{g1-ex} and \eqref{g2-def}, respectively.
We further decompose $g^{n+1}_R$ as
\begin{align}\notag
\sqrt{\mu}g^{n+1}_R=g^{n+1}_{R,1}+\sqrt{\mu}g^{n+1}_{R,2}.
\end{align}
We shall show that the approximation sequence $[g^{n+1}_{R,1},g^{n+1}_{R,2}]|_{n=1}^{+\infty}$ converges to $[g_{R,1},g_{R,2}]$ which satisfies \eqref{gr1},
\eqref{gr1-bd}, \eqref{gr1-id}, \eqref{gr2},
\eqref{gr2-bd}, \eqref{gr2-id}
in function space $X_{\eps^2T_1}\times X_{\eps^2T_1}$ for some $T_1>0$.

Next, we consider the following coupled equations for $g_{R,1}$ and $g_{R,2}$:
\begin{align}\label{pg-hr1}
&\pa_t\widehat{g^{n+1}}_{R,1}+\eps^{-1}i\bar{k}\cdot\bar{v}\widehat{g^{n+1}}_{R,1}
+\eps^{-1}v_z\pa_z\widehat{g^{n+1}}_{R,1}+\eps\Phi\cdot\na_\Fv \widehat{g^{n+1}}_{R,1}\notag\\&+i\al zk_x\widehat{g^{n+1}}_{R,1}-\al v_z\pa_{v_x}\widehat{g^{n+1}}_{R,1}
+\frac{1}{\eps^2}\nu \widehat{g^{n+1}}_{R,1}\notag\\
&=-\eps^{-\frac{1}{2}}\widehat{g^{n+1}}_{R}\SV(\widehat{g^{n}}_{R})
-\eps^{-1}\widehat{g^{n+1}}_{R}\SV\left[\sqrt{\mu}(\hat{f}_1+\eps \hat{f}_2)+\sqrt{\mu}(\hat{g}_1+\eps \hat{g}_2)\right]
\notag\\
&-\eps^{\frac{1}{2}}\sqrt{\mu}\pa_t\hat{g}_2-\eps^{\frac{1}{2}}\al z\sqrt{\mu}ik_x\hat{g}_2+\eps^{\frac{1}{2}}\al v_z\pa_{v_x}\{\sqrt{\mu}\hat{g}_2\}
-\eps^{\frac{1}{2}}\Phi\cdot\na_\Fv[\sqrt{\mu}(\hat{g}_1+\eps \hat{g}_2)]
\notag\\
&+\frac{1}{\eps^2}\chi_M\CK \widehat{g^{n}}_{R,1}+\frac{\eps}{2}\sqrt{\mu}\Phi\cdot\Fv \widehat{g^{n+1}}_{R,2}
-\frac{\al v_xv_z}{2}\sqrt{\mu}\{\FI-\FP\}\widehat{g^{n+1}}_{R,2}+\tilde{\SH}_1,
\end{align}
with
\begin{align}
\tilde{\SH}_1=&\eps^{-\frac{1}{2}}\hat{Q}_{\rm{gain}}(\widehat{g^{n}}_{R,1},\widehat{g^{n}}_{R,1})\notag\\
&+\eps^{-1}\hat{Q}(\widehat{g^{n}}_{R,1},\sqrt{\mu}(\hat{g}_1+\eps \hat{g}_2))
+\eps^{-1}\hat{Q}(\sqrt{\mu}(\hat{g}_1+\eps \hat{g}_2),\widehat{g^{n}}_{R,1})\notag\\
&+\eps^{\frac{1}{2}}\hat{Q}(\sqrt{\mu}\hat{g}_2,\sqrt{\mu}\hat{g}_2)\notag\\
&+\eps^{-1}\hat{Q}(\widehat{g^{n}}_{R,1},\sqrt{\mu}(\hat{f}_1+\hat{f}_2))
+\eps^{-1}\hat{Q}(\sqrt{\mu}(\hat{f}_1+\eps \hat{f}_2),\widehat{g^{n}}_{R,1})\notag\\
&+\hat{Q}(\sqrt{\mu}\hat{f}_{R,1}, \sqrt{\mu}\hat{g}_2)
+\hat{Q}(\sqrt{\mu}\hat{g}_2,\sqrt{\mu}\hat{f}_{R,1})\notag\\
&+\eps^{-1}\hat{Q}(\hat{f}_{R,1},\sqrt{\mu}\hat{g}_1)
+\eps^{-1}\hat{Q}(\sqrt{\mu}\hat{g}_1,\hat{f}_{R,1})
\notag\\
&+\eps^{-\frac{1}{2}}\left\{\hat{Q}(\sqrt{\mu}\widehat{g^n}_{R},\hat{f}_{R,1})+\hat{Q}(\hat{f}_{R,1},\sqrt{\mu}\widehat{g^n}_{R})\right\}\notag\\
&+\eps^{-\frac{1}{2}}\left\{\hat{Q}(\widehat{g^{n}}_{R,1},\sqrt{\mu}\hat{f}_{R,2})
+\hat{Q}(\sqrt{\mu}\hat{f}_{R,2},\widehat{g^{n}}_{R,1})\right\},\notag
\end{align}
\begin{align}\label{pg-hr1-id}
\widehat{g^{n+1}}_{R,1}(0,z,\Fv)=\eps^{\frac{3}{2}}\hat{g}^{(1)}_{R,0}(z,\Fv),
\end{align}
\begin{align}\label{pg-hr1-bd}
\widehat{g^{n+1}}_{R,1}(t,\pm1,\Fv)|_{v_z\lessgtr0}=0,
\end{align}
and
\begin{align}\label{pg-hr2}
\pa_t\widehat{g^{n+1}}_{R,2}&+\eps^{-1}i\bar{k}\cdot\bar{v}\widehat{g^{n+1}}_{R,2}+\eps^{-1}v_z\pa_z\widehat{g^{n+1}}_{R,2}
+\eps\Phi\cdot\na_\Fv \widehat{g^{n+1}}_{R,2}
\notag\\&+i\al zk_x\widehat{g^{n+1}}_{R,2}-\al v_z\pa_{v_x}\widehat{g^{n+1}}_{R,2}
+\frac{1}{\eps^2}\nu \widehat{g^{n+2}}_{R,2}\notag\\
=&\frac{1}{\eps^2}K\widehat{g^{n}}_{R,2}+\frac{1}{\eps^2}\mu^{-\frac{1}{2}}(1-\chi_M)\CK \widehat{g^{n}}_{R,1}-\frac{\al v_xv_z}{2}\FP \widehat{g^{n}}_{R,2}+\tilde{\SH}_2
\end{align}
with
\begin{align}
\tilde{\SH}_2=&\eps^{-\frac{1}{2}}\hat{\Ga}_{\rm{gain}}(\widehat{g^{n}}_{R,2},\widehat{g^{n}}_{R,2})
+\eps^{-1}\hat{\Ga}(\widehat{g^{n}}_{R,2},\hat{g}_1+\eps \hat{g}_2)
+\eps^{-1}\hat{\Ga}(\hat{g}_1+\eps \hat{g}_2,\widehat{g^{n}}_{R,2})\notag\\
&+\eps^{-1}\hat{\Ga}(\widehat{g^{n}}_{R,2},\hat{f}_1+\eps \hat{f}_2)
+\eps^{-1}\hat{\Ga}(\hat{f}_1+\eps \hat{f}_2,\widehat{g^{n}}_{R,2})\notag\\
&+\eps^{-\frac{1}{2}}\{\hat{\Ga}(\widehat{g^{n}}_{R,2},\hat{f}_{R,2})
+\hat{\Ga}(\hat{f}_{R,2},\widehat{g^{n}}_{R,2})\}\notag\\
&+\eps^{-1}\hat{\Ga}(\hat{f}_{R,2},\hat{g}_1+\eps\hat{g}_2)
+\eps^{-1}\hat{\Ga}(\hat{g}_1+\eps\hat{g}_2,\hat{f}_{R,2}),\notag
\end{align}
\begin{align}\label{pg-hr2-id}
\widehat{g^{n+1}}_{R,2}(0,\bar{k},z,\Fv)=\hat{g}^{(2)}_{R,0}(\bar{k},\Fv),
\end{align}
\begin{align}\label{pg-hr2-bd}
\widehat{g^{n+1}}_{R,2}(t,\bar{k},\pm1,\Fv)|_{v_z\lessgtr0}=P_\ga \widehat{g^n}_{R,2}+\bar{P}_\ga \widehat{g^n}_{R,1}+\SR.
\end{align}
We now intend to show inductively that there exists a finite $T_{1}>0$ and a constant $\tilde{C}_0>0$ such that
\begin{align}
\sum\limits_{\bar{k}\in\Z^2}&\sup_{0\leq t\leq\eps^2 T_1}\|w^{l_\infty}[\eps^{-\frac{3}{2}}\widehat{g^m}_{R,1},\eps^{\frac{1}{2}}\widehat{g^m}_{R,2}](t)\|_{\infty}\leq \tilde{C}_0\vps_0,\label{fn1.bd}
\end{align}
for any $m\geq0$, on the condition that
\begin{align}
\sum\limits_{\bar{k}\in\Z^2}&\sup_{0\leq t\leq \eps^2T_1}\|w^{l_\infty}[\widehat{g^0}_{R,1},\widehat{g^0}_{R,2}](t)\|_{\infty}
+\sum\limits_{m\leq8}\|\pa_i^m\Fu_0(\Fx)\|_{H^4_z}
\notag\\=&\sum\limits_{\bar{k}\in\Z^2}\|w^{l_\infty}[\hat{g}^{(1)}_{R,0},\hat{g}^{(2)}_{R,0}](\bar{k},\Fv)\|_{\infty}
+\sum\limits_{m\leq8}\|\pa_i^m\Fu_0(\Fx)\|_{H^4_z}\leq \vps_0.\notag
\end{align}
Note that, by choosing $\tilde{C}_0>$ to be suitably large, one can easily deduce
\begin{align}
\sup\limits_{0\leq l\leq \CL}&\sum\limits_{\bar{k}\in\Z^2}\sup_{0\leq t\leq \eps^2T_1}\|w^{l_\infty}[\eps^{-\frac{3}{2}}\widehat{g^l}_{R,1},\eps^{\frac{1}{2}}\widehat{g^l}_{R,2}](t)\|_{\infty}\notag\\
\leq& C(\CL)\sum\limits_{\bar{k}\in\Z^2}\|w^{l_\infty}[\hat{g}^{(1)}_{R,0},\hat{g}^{(2)}_{R,0}](\bar{k},\Fv)\|_{\infty}
+\sum\limits_{m\leq8}\|\pa_i^m\Fu_0(\Fx)\|_{H^4_z}\leq \frac{1}{2}\tilde{C}_0\vps_0.
\notag
\end{align}
In fact, the above estimate is achieved by utilizing \eqref{pg-hr1},  \eqref{pg-hr1-id}, \eqref{pg-hr1-bd}, \eqref{pg-hr2}, \eqref{pg-hr2-id} and \eqref{pg-hr2-bd} recursively, since $\CL$ is finite.

In the following, we prove \eqref{fn1.bd} for $m=n+1$ under the assumption that it holds for $m\leq n.$
Performing the similar calculations as for obtaining \eqref{lif-hr1} and \eqref{hr2-if-es}, one has
\begin{align}
\sum\limits_{\bar{k}\in\Z^2}&\sup_{0\leq t\leq \eps^2 T_1}\|w^{l_\infty}\widehat{g^{n+1}}_{R,1}(t)\|_{\infty}\notag\\
\leq&\left[(1+M)^{-\ga}+\vps_{\Phi,\al}+\vps_0+\eps+\varsigma+T_1+\bar{\vps}\right]\notag\\
&\quad\times\sup\limits_{1\leq l\leq\CL}\sum\limits_{\bar{k}\in\Z^2}\sup_{0\leq t\leq \eps^2T_1}\|w^{l_\infty}\widehat{g^{n+1-l}}_{R,1}(t)\|_{\infty}
\notag\\
&+\eps^2\sup\limits_{1\leq l\leq\CL}\sum\limits_{\bar{k}\in\Z^2}\sup_{0\leq t\leq \eps^2T_1}\|\widehat{g^{n+1-l}}_{R,2}(t)\|_{\infty}
+C\eps^{\frac{3}{2}}\sum\limits_{\bar{k}\in\Z^2}\|w^{l_\infty} \hat{g}^{(1)}_{R,0}\|_{\infty}\notag \\
&+\eps^{\frac{3}{2}}\{\vps_{\Phi,\al}+\vps_0\},\notag
\end{align}
and
\begin{align}
\sum\limits_{\bar{k}\in\Z^2}&\sup_{0\leq t\leq \eps^2 T_1}\|w^{l_\infty}\widehat{g^{n+1}}_{R,2}(t)\|_{\infty}\notag\\
\leq&C\sum\limits_{\bar{k}\in\Z^2}\|w^{l_\infty} \hat{g}^{(2)}_{R,0}\|_{\infty}+CT_1\sup\limits_{1\leq l\leq\CL}\sum\limits_{\bar{k}\in\Z^2}\sup_{0\leq t\leq \eps^2T_1}\|w^{l_\infty}\widehat{g^{n+1-l}}_{R,1}(t)\|_{\infty}\notag\\&+
(\vps_{\Phi,\al}+\vps_0+\eps+T_1+\bar{\vps})\sup\limits_{1\leq l\leq\CL}\sum\limits_{\bar{k}\in\Z^2}\sup_{0\leq t\leq \eps^2T_1}\|w^{l_\infty}\widehat{g^{n+1-l}}_{R,2}(t)\|_{\infty}.\notag
\end{align}
Thus by taking $\vps_{\Phi,\al}$, $\vps_0$, $\eps$, $T_1$ and $\bar{\vps}$ to be suitably small, we see that
\eqref{fn1.bd} holds true for $m=n+1$. Furthermore, one can also prove that $[g^{n+1}_{R,1},g^{n+1}_{R,2}]|_{n=1}^{+\infty}$
is a Cauchy sequence in $X_{\eps^2T_1}\times X_{\eps^2T_1}$ and thus is convergent.
Therefore,
$$
F(t,\Fx,\Fv)=\lim\limits_{n\rightarrow\infty}\left\{F_{st}^\eps+\eps\sqrt{\mu}\{[g_1+\eps g_2+\eps^{\frac{1}{2}}g^n_R](t,\Fx,\Fv)\}\right\}\geq0,
$$ 
for $t\in[0,\eps^2T_1]$.
 This completes the proof of Theorem \ref{sta-th}.
\end{proof}

\appendix 
\section{Estimates }\label{app-sec}
In this appendix, we will give some necessary results and estimates which have been used in previous sections.
The first one is concerned with the {\it a priori} estimates for the fluid equations around the Couette flow.

\subsection{Fluid equations with shear force}
In this subsection, we will derive the energy estimates in the function space $L^1_{\bar{k}}$ for both the steady and unsteady Navier-Stokes equations with shear force. We start with the steady problem \eqref{ins-s}.  Then we have

\begin{lemma}\label{st-sh-lem}
Let $[\Fu_s,\ta_s]$ be a classical solution of \eqref{ins-s}, then for any integer $m>0$, it holds that
\begin{align}\label{st-sh-es}
\sum\limits_{\bar{k}\in\Z^2}\|(1+|\bar{k}|^m)[\hat{\Fu}_s,\hat{\ta}_s]\|_{H^2_z}\leq C\vps_{\Phi,\al}.
\end{align}
\end{lemma}
\begin{proof}
We take $\ta_s(t,x)=0$.
We then derive the estimate for the pressure $P$. For this, we first get from $\eqref{ins-s}_2$ and $\eqref{ins-s}_1$ that
\begin{align}
\pa_zP=-(\Fu_s\cdot\na_\Fx\Fu_s)_z-\al zu_{s,z}+\Phi_{z}+\eta\Delta_{\bar{x}} u_{s,z}-\eta\pa_z(\pa_x u_{s,x}+\pa_y u_{s,y}),\notag
\end{align}
integrating with respect to $z$, we find
\begin{multline*}
    P(z)-P(-1)=-\int_{-1}^z(\Fu_s\cdot\na_\Fx\Fu_s)_z(\tau)d\tau+
\int_{-1}^{z}\Phi_{z}(\tau)d\tau
\\-\al\int_{-1}^{z}z\pa_xu_{s,z}(\tau)d\tau+\eta\int_{-1}^{z}\Delta_{\bar{x}} u_{s,z}(\tau)d\tau-\eta(\pa_x u_{s,x}+\pa_y u_{s,y})(z),
\end{multline*}
which further leads to
\begin{align}
\int_{\T^2}P(x,y,1)dxdy=\int_{\T^2}P(x,y,-1)dxdy,\notag
\end{align}
under the assumption that $\int_{-1}^{1}\Phi_{z}(\tau)d\tau=0$. Actually, we can further assume
\begin{align}\label{P-z}
\int_{\T^2}P(x,y,1)dxdy=\int_{\T^2}P(x,y,-1)dxdy=0.
\end{align}
Next, 
by $\eqref{ins-s}_2$, one has
\begin{align}
\Delta_\Fx \bar{P}=-\na_{\Fx}\cdot(\Fu_s\cdot\na_\Fx\Fu_s)-2\al \pa_xu_{s,z}.\notag
\end{align}
Denote $\pa_i$ by $\pa_x$ or $\pa_y$. By $L^2$ estimate and Sobolev's inequality, we obtain
\begin{align}
\sum\limits_{m'\leq m+2}\|\na_\Fx \pa^{m'}_iP\|_2\leq&\|\pa^{m'}_i\na_{\Fx}\cdot[\Fu_s\cdot\na_\Fx \Fu_s]\|_2+\al\|\pa^{m'}_i\pa_xu_{s,z}\|_2+
\|\Phi_{z}\|_{2}\notag\\
\lesssim& \sum\limits_{\substack{m_1+m_2\leq m+1\\m_2\leq2 }}\|\pa_i^{m_1}\pa^{m_2}_z\Fu_s\|_2^{2}
+\al\|\pa^{m'}_i\pa_xu_{s,z}\|_2+
\|\Phi_{z}\|_{2}.\notag
\end{align}
Next, we act $\pa_i^m$ with $m\geq 4$ to $\eqref{ins-s}_2$ to obtain
\begin{multline}
\pa^m_i(\Fu_s\cdot\na_\Fx \Fu_s)+\na_\Fx \pa^m_iP
+\al z\pa^m_i\pa_{x}\Fu_s+\al\pa^m_i(u_{s,z},0,0)^T-{\bf 1}_{m=0}\Phi\\
=\eta\Delta_\Fx\pa_i^m \Fu_s.\notag
\end{multline}
On the other hand, since $\pa_i^m \Fu_s(x,y,\pm1)=0$,  one has by $L^2$ energy estimate that 
\begin{align}
&\sum\limits_{\substack{m_1+m_2\leq m+2\\ m_2\leq2} }\|\pa_i^{m_1}\pa^{m_2}_z\Fu_s\|_2\notag\\
&\lesssim \|\na_\Fx \pa^m_iP\|_2
+\sum\limits_{\substack{m_1+m_2\leq m+2\\ m_2\leq 2}}\|\pa_i^{m_1}\pa^{m_2}_z\Fu_s\|_2^{2}
+\al\sum\limits_{\substack{m_1+m_2\leq m+2\\ m_2\leq 2}}\|\pa_i^{m_1}\pa^{m_2}_z\Fu_s\|_2\notag\\
&\quad+\|\Phi\|_{2}\notag\\
&\lesssim
\sum\limits_{\substack{m_1+m_2\leq m+2\\ m_2\leq2}}\|\pa_i^{m_1}\pa^{m_2}_z\Fu_s\|_2^{2}
+\al\sum\limits_{\substack{m_1+m_2\leq m+2\\ m_2\leq2}}\|\pa_i^{m_1}\pa^{m_2}_z\Fu_s\|_2+\|\Phi\|_{2}.\notag
\end{align}
Consequently, it follows
\begin{align}
\sum\limits_{\substack{m_1+m_2\leq m+2\\ m_2\leq2}}\|\pa_i^{m_1}\pa^{m_2}_z\Fu_s\|_2\lesssim\|\Phi\|_{2},\notag
\end{align}
which further implies
\begin{align}\label{us-sh-l2es}
\sum\limits_{\bar{k}\in\Z^2}\left(\lag\bar{k}\rag^{2m+4}\|\hat{\Fu}_s\|^2_{H^2_z}\right)^{\frac{1}{2}}\leq C\vps_{\Phi,\al}.
\end{align}
Finally, \eqref{us-sh-l2es} and the interpolation inequality give the desired estimate \eqref{st-sh-es}, and this then ends the proof of Lemma \ref {st-sh-lem}.
\end{proof}

We now consider the unsteady problem \eqref{ns-ust}.  For this problem, we have

\begin{lemma}\label{ust-sh-lem}
Let $[\Fu,\ta]$ be a classical solution of \eqref{ns-ust}.  Then for $\la_0>0$, it holds that
\begin{align}\label{us-sh-es}
\sum\limits_{\bar{k}\in\Z^2}\sup\limits_{m_0\leq1}\|e^{\la_0t}\lag\bar{k}\rag^2\pa_t^{m_0}[\hat{\Fu},\hat{\ta}]\|_{L^\infty_tH^2_z}
&
+\sum\limits_{\bar{k}\in\Z^2}\sup\limits_{m_0\leq1}\|e^{\la_0t}\lag\bar{k}\rag^2\pa_t^{m_0}[\hat{\Fu},\hat{\ta}]\|_{L^2_tH^2_z}
\notag\\
&\leq
C\sum\limits_{m\leq8}\|\pa_i^m\Fu_0(\Fx)\|_{H^4_z}.
\end{align}
\end{lemma}
\begin{proof}
Similar to the proof of Lemma \ref{st-sh-lem}, we can assume $\ta(t,x) = 0$.
Therefore, our task now is to establish the estimate for $\Fu(t,x)$. Setting $\SU=e^{\la_0t}\Fu$.
The estimate \eqref{us-sh-es} is based on the following {\it a priori} assumption
\begin{align}\label{aps-sh}
\sup_t\sum\limits_{\substack{m_1+m_2\leq6\\ m_2\leq2}}&\sum\limits_{m_0\leq1}\|\pa_t^{m_0}\pa_i^{m_1}\pa^{m_2}_z\SU\|_2
\leq \sqrt{\vps_0}.
\end{align}
First of all, one has the equations for
$\SU=(\SU_x,\SU_y,\SU_z)$ as follows
\begin{align}\label{su}
\pa_t\SU-\la_0\SU & +e^{-\la_0t}\SU\cdot\na_x \SU+\na_\Fx \SP+\al z\pa_{x}\SU
\notag \\
&+\SU\cdot\na_x \Fu_s+\Fu_s\cdot\na_x \SU+\al(\SU_{z},0,0)^T=\eta\Delta_\Fx \SU,
\end{align}
and
\begin{align}
\SU(x,y,\pm1)=0,\ \SU(0,\Fx)=\Fu_0(\Fx),\notag
\end{align}
where $\SP=e^{\la_0t}\tilde{P}.$
First of all, like \eqref{P-z}, we also have
we also have
\begin{align}
\int_{\T^2}\SP(x,y,1)dxdy=\int_{\T^2}\SP(x,y,-1)dxdy=0.\notag
\end{align}
Next,  from $\eqref{ns-ust}_2$, it follows
\begin{align*}
\Delta_\Fx\SP=&-\na_\Fx\cdot\left\{e^{-\la_0t}\SU\cdot\na_x \SU+\al z\pa_{x}\SU+\SU\cdot\na_x \Fu_s
+\Fu_s\cdot\na_x \SU\right\}
\\
&
-2\al\pa_x\SU_z,\notag
\end{align*}
then by $L^2$ estimates, we obtain
\begin{align}\label{su-div}
    \|\na_\Fx \pa_t^{m_0}\pa_{i}^{m}\SP\|_2\leq&
\left\|\pa_t^{m_0}\pa_{i}^{m}\left\{e^{-\la_0t}\SU\cdot\na_x \SU\right\}\right\|_2+
\al \left\|z\pa_{x}\pa_t^{m_0}\pa_{i}^{m}\SU\right\|_2 \notag
\\
&+\left\|\pa_t^{m_0}\pa_{i}^{m}(\SU\cdot\na_x \Fu_s)\right\|_2+\|\pa_t^{m_0}\pa_{i}^{m}(\Fu_s\cdot\na_x \SU)\|_2 \notag
\\
&+\al\|\pa_t^{m_0}\pa_{i}^{m}(\SU_{z},0,0)^T\|_2
 \notag\\
\leq&
\sum\limits_{\substack{m_1+m_2\leq4\\ m_2\leq2}}\sum\limits_{m_0\leq1}\|\pa_t^{m_0}\pa_i^{m_1}\pa^{m_2}_z\SU\|_2^{2} \notag
\\
&+(\vps_{\Phi,\al}+\al)\sum\limits_{\substack{m_1+m_2\leq4\\ m_2\leq2}}\sum\limits_{m_0\leq1}\|\pa_t^{m_0}\pa_i^{m_1}\pa^{m_2}_z\SU\|_2.
\end{align}

Next, on one hand, $L^2$ energy estimate gives

\begin{align}\label{su-ep}
\sum\limits_{\substack{m_1+m_2\leq6\\ m_2\leq2}}&\sum\limits_{m_0\leq1}\|\pa_t^{m_0}\pa_i^{m_1}\pa^{m_2}_z\SU\|_2
\lesssim
\sum\limits_{m_0\leq1,m\leq4}\|\pa_t\pa_t^{m_0}\pa_{i}^{m}\SU\|_2
\notag\\
&
+
\sum\limits_{m_0\leq1,m\leq4}\|\na_\Fx \pa_t^{m_0}\pa_{i}^{m}\SP\|_2
+\sum\limits_{\substack{m_1+m_2\leq6\\ m_2\leq2}}\sum\limits_{m_0\leq1}\|\pa_t^{m_0}\pa_i^{m_1}\pa^{m_2}_z\SU\|_2^{2}
\notag\\
&
+(\vps_{\Phi,\al}+\al+\la_0)\sum\limits_{\substack{m_1+m_2\leq6\\ m_2\leq2}}\sum\limits_{m_0\leq1}\|\pa_t^{m_0}\pa_i^{m_1}\pa^{m_2}_z\SU\|_2
\notag\\
\lesssim&
\sum\limits_{m_0\leq1,m\leq2}\|\pa_t\pa_t^{m_0}\pa_{i}^{m}\SU\|_2
+\sum\limits_{\substack{m_1+m_2\leq6\\ m_2\leq2}}\sum\limits_{m_0\leq1}\|\pa_t^{m_0}\pa_i^{m_1}\pa^{m_2}_z\SU\|_2^{2}
\notag\\&+(\vps_{\Phi,\al}+\al+\la_0)\sum\limits_{\substack{m_1+m_2\leq6\\ m_2\leq2}}\sum\limits_{m_0\leq1}\|\pa_t^{m_0}\pa_i^{m_1}\pa^{m_2}_z\SU\|_2.
\end{align}
On the other hand, acting $\pa_t^{m_0}\pa_{i}^{m}$ with $m_0\leq2$ and $m\leq4$ to \eqref{su}, and taking the inner product of the resulting equation with $\pa_t^{m_0}\pa_{i}^{m}\SU$, we have
\begin{align}
\frac{d}{dt}&\|\pa_t^{m_0}\pa_{i}^{m}\SU\|_2^2
+\frac{\eta}{2}\|\pa_t^{m_0}\pa_{i}^{m}\na_\Fx\SU\|_2^2
\lesssim
\sum\limits_{\substack{m_1+m_2\leq6\\ m_2\leq2}}\sum\limits_{m_0\leq1}\|\pa_t^{m_0}\pa_i^{m_1}\pa^{m_2}_z\SU\|_2^{2}
\notag\\
&
+(\vps_{\Phi,\al}+\al+\la_0)\sum\limits_{\substack{m_1+m_2\leq6\\ m_2\leq2}}\sum\limits_{m_0\leq1}\|\pa_t^{m_0}\pa_i^{m_1}\pa^{m_2}_z\SU\|^2_2,\notag
\end{align}
which further yields
\begin{align}\label{su-de-ip2}
\sup_t\sum\limits_{m_0\leq2,m\leq4}&\|\pa_t^{m_0}\pa_{i}^{m}\SU\|_2^2
+\frac{\eta}{2}\sum\limits_{m_0\leq2,m\leq4}\int_0^\infty\|\pa_t^{m_0}\pa_{i}^{m}\SU\|_2^2dt\notag\\
\lesssim& \sum\limits_{m\leq8}\|\pa_i^m\Fu_0(\Fx)\|^2_{H^4_z}+
\sum\limits_{\substack{m_1+m_2\leq6\\ m_2\leq2}}\sum\limits_{m_0\leq1}\int_0^\infty\|\pa_t^{m_0}\pa_i^{m_1}\pa^{m_2}_z\SU\|_2^{2}dt
\notag\\&
+(\vps_{\Phi,\al}+\al+\la_0)\sum\limits_{\substack{m_1+m_2\leq6\\ m_2\leq2}}\sum\limits_{m_0\leq1}
\int_0^\infty\|\pa_t^{m_0}\pa_i^{m_1}\pa^{m_2}_z\SU\|^2_2dt.
\end{align}
Here we have also used the following Poincar\'e inequality
\begin{align}
\|\pa_t^{m_0}\pa_{i}^{m}\SU\|_2\leq C\|\pa_t^{m_0}\pa_{i}^{m}\na_\Fx\SU\|_2.\notag
\end{align}
Consequently, we get from \eqref{su-ep}, \eqref{su-de-ip2} and \eqref{aps-sh} that
\begin{align*}
\sup_t & \sum\limits_{\substack{m_1+m_2\leq6\\ m_2\leq2}}
\sum\limits_{m_0\leq1}\|\pa_t^{m_0}\pa_i^{m_1}\pa^{m_2}_z\SU\|^2_2
\\&
+\int_0^\infty\sum\limits_{\substack{m_1+m_2\leq6 \\ m_2\leq2}}\sum\limits_{m_0\leq1}\|\pa_t^{m_0}\pa_i^{m_1}\pa^{m_2}_z\SU\|^2_2dt
\lesssim
 \sum\limits_{m\leq8}\|\pa_i^m\Fu_0(\Fx)\|^2_{H^4_z}.
\end{align*}
Note that
\begin{align}
&\sum\limits_{\bar{k}\in\Z^2}\sup\limits_{m_0\leq1}\lag\bar{k}\rag^2\|e^{\la_0t}\pa_t^{m_0}\hat{\Fu}\|_{L^\infty_tH^2_z}\notag\\
&\leq C\left(\sum\limits_{\bar{k}\in\Z^2}\left(\sup\limits_{m_0\leq1}\sup_t\|\lag \bar{k}\rag ^4e^{\la_0t}\pa_t^{m_0}\hat{\Fu}\|_{H^2_z}\right)^2\right)^{\frac{1}{2}}\notag\\
&\leq C\sum\limits_{\substack{m_1+m_2\leq8\\ m_2\leq2}}\sum\limits_{m_0\leq1}\|\pa_t^{m_0}\pa_i^{m_1}\pa^{m_2}_z\SU\|_2,\notag
\end{align}
and
\begin{align}
\sum\limits_{\bar{k}\in\Z^2}\sup\limits_{m_0\leq1}\lag\bar{k}\rag^2\|e^{\la_0t}\pa_t^{m_0}\hat{\Fu}\|_{L^2_tH^2_z}
\leq& C\left(\sum\limits_{\bar{k}\in\Z^2}\left(\sup\limits_{m_0\leq1}\|\lag \bar{k}\rag ^4e^{\la_0t}\pa_t^{m_0}\hat{\Fu}\|_{L^2_tH^2_z}\right)^2\right)^{\frac{1}{2}}\notag\\
\leq& C\sum\limits_{\substack{m_1+m_2\leq8\\ m_2\leq2}}\sum\limits_{m_0\leq1}\|\pa_t^{m_0}\pa_i^{m_1}\pa^{m_2}_z\SU\|_2.\notag
\end{align}
Therefore, \eqref{us-sh-es} is valid, and this then finishes the proof of Lemma \ref{ust-sh-lem}.
\end{proof}

The following result is based on Lemmas  \ref{st-sh-lem} and \ref{ust-sh-lem}.

\begin{lemma}[Estimates on coefficients]\label{coef-es}
Under the same conditions as in Theorems \ref{st-sol-th} and \ref{sta-th}, for any $m\geq0$, the following estimate holds:
\begin{align}\label{fi-es}
\sum\limits_{\bar{k}\in\Z^2} \|w^{l_\infty}(1+|\bar{k}|^m)[\hat{f}_1,\hat{f}_2](\bar{k},z,\Fv)\|_{\infty}\lesssim \vps_{\Phi,\al}.
\end{align}
Moreover, for $\la_0>0$ given by Lemma \ref{ust-sh-lem} and for $m_0\leq1$ and $m_1\leq1$, the subsequent estimate is satisfied:
\begin{align}\label{gi-es2}
\sum\limits_{\bar{k}\in\Z^2}&
\|e^{\la_0t}w^{l_\infty}\widehat{\pa^{m_0}_t\pa^{m_1}_x[g_1,g_2]}(t,\bar{k},z,\Fv)\|_{L^\infty_{T,z,\Fv}}
\notag\\&+\sum\limits_{\bar{k}\in\Z^2}
\|e^{\la_0t}w^{l_\infty}\widehat{\pa^{m_0}_t\pa^{m_1}_x[g_1,g_2]}(t,\bar{k},z,\Fv)\|_{L^2_{T,z,\Fv}}\lesssim \vps_0,
\end{align}
for any $0\leq T<+\infty.$

Furthermore, for estimates for the error terms $r$ and $\tilde{r}$ at the boundary given by \eqref{st-bd} and \eqref{gr-bd}, respectively, it holds that
\begin{align}\label{r-es}
\sum\limits_{\bar{k}\in\Z^2} \|w^{l_\infty}\hat{r}(\bar{k},\pm1,\Fv)\|_{L^\infty_{\Fv}}\lesssim \vps_{\Phi,\al},
\end{align}
and
\begin{align}\label{tir-es}
\sum\limits_{\bar{k}\in\Z^2}
\|e^{\la_0t}w^{l_\infty}\widehat{\tilde{r}}(\bar{k},\pm1,\Fv)\|_{L^\infty_{T,\Fv}}
+\sum\limits_{\bar{k}\in\Z^2}
\|e^{\la_0t}w^{l_\infty}\widehat{\tilde{r}}(\bar{k},\pm1,\Fv)\|_{L^2_{T,\Fv}}\lesssim \vps_0,
\end{align}
for any $0\leq T<+\infty.$

\end{lemma}
\begin{proof}
We begin by establishing \eqref{gi-es2}. From \eqref{g2-def}, we get
\begin{align}\label{g2-es-p1}
&\sum\limits_{\bar{k}\in\Z^2}\|e^{\la_0t}w^{l_\infty}\widehat{\pa^{m_0}_t\pa^{m_1}_xg}_2(t,\bar{k},z,\Fv)\|_{L^\infty_{T,z,\Fv}}\notag\\
&\quad+\sum\limits_{\bar{k}\in\Z^2}
\|e^{\la_0t}w^{l_\infty}\widehat{\pa^{m_0}_t\pa^{m_1}_xg}_2(t,\bar{k},z,\Fv)\|_{L^2_{T,z,\Fv}}\notag\\
&\lesssim \sum\limits_{\bar{k}\in\Z^2}
\Bigg\{\|e^{\la_0t}[\widehat{\na_\Fx\pa^{m_0}_t\pa^{m_1}_x\Fu},\widehat{\na_\Fx\pa^{m_0}_t\pa^{m_1}_x\ta}]\|_{L^\infty_{T,z}}\notag\\
&\qquad\qquad+\|e^{\la_0t}[\widehat{\pa^{m_0}_t\pa^{m_1}_x\Fu},\widehat{\pa^{m_0}_t\pa^{m_1}_x\ta}]\|_{L^\infty_{T,z}}
\notag\\&\qquad\qquad+\|e^{\la_0t}[\widehat{\pa^{m_0}_t\pa^{m_1}_x\Fu}_2,\widehat{\pa^{m_0}_t\pa^{m_1}_x\ta}_2]\|_{L^\infty_{T,z}}\Bigg\}\notag\\
&+\sum\limits_{\bar{k}\in\Z^2}
\Bigg\{\|e^{\la_0t}[\widehat{\na_\Fx\pa^{m_0}_t\pa^{m_1}_x\Fu},\widehat{\na_\Fx\pa^{m_0}_t\pa^{m_1}_x\ta}]\|_{L^2_{T,z}}
\notag\\
&\qquad\qquad
+\|e^{\la_0t}[\widehat{\pa^{m_0}_t\pa^{m_1}_x\Fu},\widehat{\pa^{m_0}_t\pa^{m_1}_x\ta}]\|_{L^2_{T,z}}\notag\\
&\qquad\qquad
+\|e^{\la_0t}[\widehat{\pa^{m_0}_t\pa^{m_1}_x\Fu}_2,\widehat{\pa^{m_0}_t\pa^{m_1}_x\ta}_2]\|_{L^2_{T,z}}\Bigg\}.
\end{align}
On the other hand, in view of \eqref{uu2}, we may take $\Fu_2=0,$ and by \eqref{up}, \eqref{su-div} and \eqref{ta2} and utilizing Sobolev's inequality, we also have, 
for $m_0\leq1$ and $m\leq4$
\begin{align}
\sum\limits_{m_0\leq1,m\leq4}\|\pa_t^{m_0}\pa_{i}^{m}\ta_2\|_2
\lesssim&
\sum\limits_{m_0\leq1,m\leq4}\|\na_\Fx \pa_t^{m_0}\pa_{i}^{m}\SP\|_2
\notag\\
&
+(\vps_0+\vps_{\Phi,\al})\sum\limits_{m_0\leq1,m\leq4}\|\na_\Fx \pa_t^{m_0}\pa_{i}^{m}\SU\|_2
\notag\\
\lesssim&(\vps_0+\vps_{\Phi,\al})\sum\limits_{\substack{m_1+m_2\leq4\\ m_2\leq2}}\sum\limits_{m_0\leq1}\|\pa_t^{m_0}\pa_i^{m_1}\pa^{m_2}_z\SU\|_2,\notag
\end{align}
which further implies
\begin{align}\label{g2-es-p2}
\sum\limits_{\bar{k}\in\Z^2}\sup\limits_{m_0\leq1}\|e^{\la_0t}\lag\bar{k}\rag^2\pa_t^{m_0}\hat{\ta}_2\|_{L^\infty_tH^2_z}
+\sum\limits_{\bar{k}\in\Z^2}\sup\limits_{m_0\leq1}\|e^{\la_0t}\lag\bar{k}\rag^2\pa_t^{m_0}\hat{\ta}_2\|_{L^2_tH^2_z}
\lesssim\vps_0.
\end{align}
Thus \eqref{gi-es2} for the component $g_2$  follows from \eqref{g2-es-p1}, \eqref{g2-es-p2} and \eqref{us-sh-es}, and the corresponding estimates for $g_1$ directly follow from Lemma \ref{ust-sh-lem}.

We now turn to prove \eqref{tir-es} and the estimate \eqref{r-es} can be obtained similarly. Similar to \eqref{g2-es-p1}, we have
by using \eqref{uu2} again
\begin{align}
\sum\limits_{\bar{k}\in\Z^2}&\|e^{\la_0t}w^{l_\infty}\widehat{\tilde{r}}(\bar{k},\pm1,\Fv)\|_{L^\infty_{T,\Fv}}
+\sum\limits_{\bar{k}\in\Z^2}
\|e^{\la_0t}w^{l_\infty}\widehat{\tilde{r}}(\bar{k},\pm1,\Fv)\|_{L^2_{T,\Fv}}\notag\\
\lesssim& \sum\limits_{\bar{k}\in\Z^2}
\Bigg\{\|e^{\la_0t}[\widehat{\na_\Fx\Fu}(\pm1),\widehat{\na_\Fx\ta}(\pm1)]\|_{L^\infty_{T}}
+\|e^{\la_0t}\hat{\ta}_2(\pm1)]\|_{L^\infty_{T}}\Bigg\}\notag\\
&+\sum\limits_{\bar{k}\in\Z^2}
\Bigg\{\|e^{\la_0t}[\widehat{\na_\Fx\Fu}(\pm1),\widehat{\na_\Fx\ta}(\pm1)]\|_{L^2_{T}}
+\|e^{\la_0t}\hat{\ta}_2(\pm1)]\|_{L^2_{T}}\Bigg\}\notag\\
\lesssim&\sum\limits_{\bar{k}\in\Z^2}\|e^{\la_0t}\lag\bar{k}\rag^2[\hat{\Fu},\hat{\ta}]\|_{L^\infty_tH^2_z}
+\sum\limits_{\bar{k}\in\Z^2}\|e^{\la_0t}\lag\bar{k}\rag^2[\hat{\Fu},\hat{\ta}]\|_{L^2_tH^2_z}
\notag\\&+\sum\limits_{\bar{k}\in\Z^2}\|e^{\la_0t}\lag\bar{k}\rag^2\hat{\ta}_2\|_{L^\infty_tH^2_z}
+\sum\limits_{\bar{k}\in\Z^2}\|e^{\la_0t}\lag\bar{k}\rag^2\hat{\ta}_2\|_{L^2_tH^2_z}\lesssim\vps_0,\notag
\end{align}
where we have used the trace inequality given as \eqref{tr-as}. Note that the proof for \eqref{fi-es} follows a similar approach, and we omit the details for brevity.
This ends the proof of Lemma \ref{coef-es}.
\end{proof}

\subsection{Other key estimates}
In this subsection, we will collect some important estimates which have been used in the previous sections.
Recall the backward time cycle starting at $(t_0,z_0,\Fv_0)=(t,z,\Fv)$ in \eqref{cyc}, the boundary probability measure $d\si_l$ on $\CV_l$ in \eqref{def.sil} and the product measure $d\Sigma_{l}(s)$ and $d\tilde{\Sigma}_{l}(s)$ over $\prod_{j=1}^{\RL-1}\CV_j$ as follows; see also \eqref{Sigma} and \eqref{Sigma-i}:
\begin{multline}
\Sigma_{l}(s)=\prod\limits_{j=l+1}^{\RL-1}d\si_j e^{-\int_s^{t_l}
\CA^\eps(\tau,V_{\mathbf{cl}}^l(\tau))d\tau}{w_2}(\Fv_l)d\si_l \\
\prod\limits_{j=1}^{l-1}\frac{{w_2}(\Fv_j)}{{w_2}(V^{j}_{\mathbf{cl}}(t_{j+1}))}
\prod\limits_{j=1}^{l-1}e^{-\int_{t_{j+1}}^{t_j}
\CA^\eps(\tau,V_{\mathbf{cl}}^l(\tau))d\tau}d\si_j,\notag
\end{multline}
and
\begin{multline}
\tilde{\Sigma}_{l}(s)=\prod\limits_{j=l+1}^{\RL-1}d\si_j e^{-\int_s^{\tilde{t}_l}
\tilde{\CA}^\eps(\tau,\tilde{\FV}_{\mathbf{cl}}^l(\tau))d\tau}{w_2}(\Fv_l)d\si_l \\
\prod\limits_{j=1}^{l-1}\frac{{w_2}(\Fv_j)}{{w_2}(\tilde{\FV}^{j}_{\mathbf{cl}}(\tilde{t}_{j+1}))}
e^{-\int_{\tilde{t}_{j+1}}^{\tilde{t}_j}
\tilde{\CA}^\eps(\tau,\tilde{\FV}_{\mathbf{cl}}^j(\tau))d\tau}d\si_j,\notag
\end{multline}
where
$$
w_2(\Fv)=(\sqrt{2\pi}w^{l_\infty} \mu^{\frac{1}{2}})^{-1}.
$$

The following two lemmas provide estimates on the weighted measure of the phase space $\Pi _{j=1}^{\RL-1}\mathcal{V}_{j}$ when there are $\RL$ times bounce for two distinct backward exit times $t_{b}(z,\Fv)$ and $\tilde{t}_{b}(z,\Fv)$ defined by \eqref{ex-t} and \eqref{u-ex-t}, respectively.

\begin{lemma}\label{k-cyc}
For any $\bar{\vps}>0$ and any $T_0>0$, there exists an integer $\RL_{0}=\RL_0(\bar{\vps}
,T_{0})$ such that for any integer $\RL\geq \RL_{0}$ and any $1\gg\eta_0\geq0$ and for all $(t,z,\Fv)\in[0, +\infty)\times [-1,1]\times\R^{3}$, it holds
\begin{equation}\label{cy-1}
\int_{\Pi_{l=1}^{\RL-1}\mathcal{V}_{l}}\mathbf{1}_{\mathcal{\{}%
t_{\RL}(t,z,\Fv,\Fv_{1},\Fv_{2}...,\Fv_{\RL-1})>t-T_0\}}\Pi _{l=1}^{\RL-1}e^{\frac{\eta_0}{2}|\Fv_{l}|^2}d\sigma _{l}\leq
\bar{\vps}.
\end{equation}%
In particular, let $T_{0}>0$ be large enough, there exist constants $
C_{1}$ and $C_{2}>0$ independent of $T_{0}$ such that for $\RL=C_{1}T_{0}^{\frac{5}{4}}$ with a suitable choice of $C_1$ such that $\RL$ is an integer and for all $(t,z,\Fv)\in \lbrack 0,\infty)\times [-1,1]\times \R^{3}$, it holds
\begin{equation}\label{cy-2}
\int_{\Pi _{j=1}^{\RL-1}\mathcal{V}_{j}}
\mathbf{1}_{\{t_{\RL}(t,z,\Fv,\Fv_{1},\Fv_{2},\cdots ,\Fv_{\RL-1})>t-T_0\}}\Pi _{l=1}^{\RL-1}e^{\frac{\eta_0}{2}|\Fv_{l}|^2}d\sigma _{l}\leq
\left\{ \frac{1}{2}\right\} ^{C_{2}T_{0}^{5/4}}.
\end{equation}%
Furthermore, for any $q>0$ in the weight function $w^q(v)$, there exist constants $C_3$ and $C_4>0$ independent of $\RL$ and $T_0$ such that
\begin{equation}\label{cy-3}
\begin{split}
\int_{\Pi _{j=1}^{\RL-1}\mathcal{V}_{j}}\sum_{l=1}^{\RL-1}\mathbf{1}_{\{t_{l+1}\leq t-T_0<t_{l}\}}
\int_{t_{l}-T_0}^{t_l} d\Sigma_l(s)ds\leq C_3,
\end{split}
\end{equation}
and
\begin{equation}\label{cy-4}
\begin{split}
\int_{\Pi _{j=1}^{\RL-1}\mathcal{V}_{j}}\sum_{l=1}^{\RL-1}\mathbf{1}_{\{t_{l+1}>t-T_0\}}\int_{t_{l+1}}^{t_l} d\Sigma_l(s)ds\leq
C_4.
\end{split}
\end{equation}
\end{lemma}

\begin{lemma}\label{t-k-cyc}
For any $\bar{\vps}>0$ and any $T_0>0$, there exists an integer $\RL_{0}=\RL_0(\bar{\vps}
,T_{0})$ such that for any integer $\RL\geq \RL_{0}$ and any $1\gg\eta_0\geq0$ and for all $(t,z,\Fv)\in[0, \eps T_0)\times [-1,1]\times\R^{3}$, it holds
\begin{equation}\label{t-cy-1}
\int_{\Pi_{l=1}^{\RL-1}\mathcal{V}_{l}}\mathbf{1}_{\mathcal{\{}%
\tilde{t}_{\RL}(t,z,\Fv,\Fv_{1},\Fv_{2}...,\Fv_{\RL-1})>0\}}\Pi _{l=1}^{\RL-1}e^{\frac{\eta_0}{2}|\Fv_{l}|^2}d\sigma _{l}\leq
\bar{\vps}.
\end{equation}%
In particular, let $T_{0}>0$ be large enough, there exist constants $
C_{1}$ and $C_{2}>0$ independent of $T_{0}$ such that for $\RL=C_{1}T_{0}^{\frac{5}{4}}$ with a suitable choice of $C_1$ such that $\RL$ is an integer and for all $(t,z,\Fv)\in \lbrack 0,\eps T_0)\times [-1,1]\times \R^{3}$, it holds
\begin{equation}\label{t-cy-2}
\int_{\Pi _{j=1}^{\RL-1}\mathcal{V}_{j}}
\mathbf{1}_{\{\tilde{t}_{\RL}(t,z,\Fv,\Fv_{1},\Fv_{2},\cdots ,\Fv_{\RL-1})>0\}}\Pi _{l=1}^{\RL-1}e^{\frac{\eta_0}{2}|\Fv_{l}|^2}d\sigma _{l}\leq
\left\{ \frac{1}{2}\right\} ^{C_{2}T_{0}^{5/4}}.
\end{equation}%
Furthermore, for any $q>0$ in the weight function $w^q(\Fv)$, there exist constants $C_3$ and $C_4>0$ independent of $\RL$ and $T_0$ such that
\begin{equation}\label{t-cy-3}
\begin{split}
\int_{\Pi _{j=1}^{\RL-1}\mathcal{V}_{j}}\sum_{l=1}^{\RL-1}\mathbf{1}_{\{\tilde{t}_{l+1}\leq 0<\tilde{t}_{l}\}}
\int_{0}^{\tilde{t}_l} d\Sigma_l(s)ds\leq C_3,
\end{split}
\end{equation}
and
\begin{equation}\label{t-cy-4}
\begin{split}
\int_{\Pi _{j=1}^{\RL-1}\mathcal{V}_{j}}\sum_{l=1}^{\RL-1}\mathbf{1}_{\{\tilde{t}_{l+1}>0\}}\int_{\tilde{t}_{l+1}}^{\tilde{t}_l} d\Sigma_l(s)ds\leq
C_4.
\end{split}
\end{equation}
\end{lemma}

We are ready to complete
\begin{proof}[The proof of Lemmas \ref{k-cyc} and \ref{t-k-cyc}]
We only prove \eqref{cy-2} and \eqref{t-cy-2}, since \eqref{cy-1} \eqref{cy-3}, \eqref{cy-4} \eqref{t-cy-1}, \eqref{t-cy-2} \eqref{t-cy-3} and \eqref{t-cy-4} can be verified
similarly. It should be pointed out that the main strategies used here follow from Lemma 23 in \cite[pp. 781]{Guo-2010}. The new difficulty is the velocity growth caused by the difference of the velocity weights at two different time.
Let us first define, for $0<\de\ll1$, 
$$
\mathcal{V}_{l}^\de=\left\{\Fv_l\in\mathcal{V}_{l}\big||v_{l,z}|>\de,\ \textrm{and}\ \de<|\Fv_l|<\frac{1}{\de}\right\},
$$
and
$$
\ga_{\pm}^\de=\left\{(z,\Fv)\in\ga_{\pm}\big||v_{z}|>\de,\ \textrm{and}\ \de<|\Fv|<\frac{1}{\de}\right\}.
$$
Obviously, it follows that for $0<\eta_0\ll1$,
$$
\int_{\mathcal{V}_{l}\backslash \mathcal{V}_{l}^\de}e^{\frac{\eta_0}{2}|\Fv_{l}|^2}d\si_l\leq \tilde{C}_0\de,
$$
where $\tilde{C}_0>0$ is independent of $l.$ Next, for $(z_l,\Fv_l)\in\ga_+^\de $, we {\it claim} that
\begin{align}\label{tb-lbde}
|t_{l+1}-t_l|=t_b(z_l,\Fv_l)\geq C_0\de.
\end{align}
To show this, it suffices to prove for $(z_l,\Fv_l)\in\ga_+^\de$ and $0<\eps\ll1$ that
\begin{align}\label{tb-lb}
t_b(z_l,\Fv_l)\gtrsim\frac{1}{|v_{l,z}|}.
\end{align}
Note that $\frac{1}{|v_{l,z}|}\leq \frac{1}{\de}$. It therefore remains to prove the above inequality in the case of $t_b(z_l,\Fv_l)\leq \frac{1}{\de}.$ By $\eqref{chl-sol}_1$ and \eqref{ex-t}, we have
\begin{align}\label{zb1}
z_b(z_l,\Fv_l)=z_l-t_b(z_l,\Fv_l)v_{l,z}+O(1)\vps_{\Phi,\al}\eps^2t_b^2(z_l,\Fv_l),
\end{align}
which gives
\begin{align}
t_b(z_l,\Fv_l)\geq |z_l-z_b(z_l,\Fv_l)||v_{l,z}|^{-1},\notag
\end{align}
provided that $0<\eps<\eps_0\ll \de.$ Consequently, \eqref{tb-lb} is valid by the fact that $|z_l-z_b(z_l,\Fv_l)|\leq2$.
Therefore \eqref{tb-lbde} is true. Furthermore, one can see that there are at most $[\frac{T_0}{C_0\de}]+1$ number of $\Fv_l\in\mathcal{V}_{l}^\de$ for $1\leq l\leq \RL-1$. Hence we get the following bound for $0<\tilde{C}_0\de<1$ and $0<\eta_0<\frac{1}{2}$:
\begin{align}\label{k-cyc-ub}
\int_{\Pi _{l=1}^{\RL-1}\mathcal{V}_{l}}&
\mathbf{1}_{\{t_{\RL}(t,z,\Fv,\Fv_{1},\Fv_{2},\cdots ,\Fv_{\RL-1})>t-T_0\}}\Pi _{l=1}^{\RL-1}e^{\frac{\eta_0}{2}|\Fv_{l}|^2}d\sigma _{l}
\notag
\\
\leq&
\sum\limits_{j=0}^{[\frac{T_0}{C_0\de}]+1}C_{\RL-1}^{j}
\left(\sup\limits_{l\in\{l_1,\cdots,l_{j}\}}\int_{\mathcal{V}_{l}^\de}e^{\frac{\eta_0}{2}|\Fv_{l}|^2}d\sigma _{l}\right)^{j}
\notag \\
& \quad \quad\times
\left(\sup\limits_{l \in \{1,\cdots,\RL-1\}\backslash\{l_1,\cdots,l_{j}\}}\int_{\mathcal{V}_{l}\backslash \mathcal{V}_{l}^\de}e^{\frac{\eta_0}{2}|\Fv_{l}|^2}d\sigma _{l}\right)^{\RL-1-j}
\notag
\\
\leq&
\sum\limits_{j=0}^{[\frac{T_0}{C_0\de}]+1}C_{\RL-1}^{j}(1-\eta_0)^{-2j}(\tilde{C}_0\de)^{\RL-1-j}
\notag\\
\leq&
\sum\limits_{j=0}^{[\frac{T_0}{C_0\de}]+1}(\RL-1)^{[\frac{T_0}{C_0\de}]+1}(1-\eta_0)^{-2[\frac{T_0}{C_0\de}]-2}
(\tilde{C}_0\de)^{\RL-2-[\frac{T_0}{C_0\de}]}
\notag\\
\leq&
\left([\frac{T_0}{C_0\de}]+1\right)\left[4(\RL-1)\right]^{[\frac{T_0}{C_0\de}]+1}(\tilde{C}_0\de)^{\RL-2-[\frac{T_0}{C_0\de}]}.
\end{align}
Next, by taking $\RL-1=N\left([\frac{T_0}{C_0\de}]+1\right)$, there exists $C_N>0$ such that the right hand side of \eqref{k-cyc-ub} can be bounded by
\begin{align}
\int_{\Pi _{l=1}^{\RL-1}\mathcal{V}_{l}}&
\mathbf{1}_{\{t_{\RL}(t,z,\Fv,\Fv_{1},\Fv_{2},\cdots ,\Fv_{\RL-1})>t-T_0\}}\Pi _{l=1}^{\RL-1}e^{\frac{\eta_0}{2}|\Fv_{l}|^2}d\sigma _{l}\notag
\\
\leq&
\left[4N\left([\frac{T_0}{C_0\de}]+1\right)^{1+\frac{1}{[\frac{T_0}{C_0\de}]+1}}(\tilde{C}_0\de)^{N-2}\right]^{[\frac{T_0}{C_0\de}]+1}
\notag
\\
\leq&
\left[4N\left([\frac{T_0}{C_0\de}]+1\right)^{2}(\tilde{C}_0\de)^{N-2}\right]^{[\frac{T_0}{C_0\de}]+1}\leq (C_NT_0^2\de^{N-4})^{[\frac{T_0}{C_0\de}]+1}.\notag
\end{align}
Then we let $C_NT_0^2\de^{N-4}=\frac{1}{2}$ and equivalently we choose 
$$
\de= (\sqrt{2C_N}T_0)^{-\frac{1}{2(N-4)}}.
$$
Finally, let $N=6$, then for $T_0$ sufficiently large, one sees that 
$$
\RL\backsim T_0^{1+\frac{1}{2(N-4)}}=T_0^{\frac{5}{4}}.
$$
Thus \eqref{cy-2} is valid.
We now turn to prove \eqref{t-cy-2}. Compared with \eqref{cy-2}, the only distinction should be the lower bound of the backward exist time $\tilde{t}_b(\Fx,\Fv)$. At this stage, corresponding to \eqref{tb-lb},  we have
\begin{align}
|\tilde{t}_{l+1}-\tilde{t}_l|=\tilde{t}_b(z_l,\Fv_l)\geq C_0\eps\de.\notag
\end{align}
Thus $\eps T_0$ is chosen so that the quantity $\RL=N\left([\frac{\eps T_0}{C_0\eps\de}]+1\right)+1$ keeps unchanged. The remaining computations for the proof of \eqref{t-cy-2} are the same as above.
 This completes the proof of Lemmas \ref{k-cyc} and \ref{t-k-cyc}.
\end{proof}

The following lemma is devoted to the trace theorem of the transport equation with shear force as well as external field in the complex variable space.

\begin{lemma}\label{ukai}
Let $\vps>0$ and $z\in[-h,h]$ with $0<h<+\infty$, and denote the near-grazing set of $\gamma_+$ or $\gamma_-$ as
\begin{equation*}
\gamma^{\varepsilon}_{\pm}\ \equiv \
\left\{(z,\Fv)\in \gamma_{\pm}: |v_{z}|\leq \varepsilon \ \text{or} \ |v_{z}|\geq \frac{1}{\varepsilon },\ \Fv=(v_x,v_y,v_{z})\right\}.
\end{equation*}%
Then, there exists a constant $C_{\varepsilon,h}>0$ depending on $\vps$ and $h$ such that
\begin{equation}\label{utrace}
\begin{split}
\sum\limits_{\bar{k}\in\Z^2}|\hat{f}^2\mathbf{1}_{\gamma_\pm\backslash\gamma^{\varepsilon }_\pm}|_{L^1}
\leq & C_{\varepsilon,h}\sum\limits_{\bar{k}\in\Z^2}\|Re(\{v_{z}\pa _{z}+\eps^2\Phi(z)\cdot\na_\Fv-\al\eps v_{z}\pa_{v_x}\}\hat{f}|\hat{f})\| _{L^1}
\\
&
+C_{\varepsilon,h}
\sum\limits_{\bar{k}\in\Z^2}\|\hat{f}^2\|_{L^1}. 
\end{split}
\end{equation}
Moreover, it also holds
\begin{align}\label{ttrace}
\sum\limits_{\bar{k}\in\Z^2}&\int_{0}^{T}|\hat{f}^2\mathbf{1}_{\gamma_+\backslash\gamma^{\varepsilon}_+}(t)|_{L^1}dt
\leq  C_{\varepsilon,h}\sum\limits_{\bar{k}\in\Z^2}\| \hat{f}^2(0)\|_{L^1}
+C_{\varepsilon,h}\sum\limits_{\bar{k}\in\Z^2}\int_{0}^{T}
\|\hat{f}^2(t)\|_{L^1} dt,
\notag\\
&
+C_{\varepsilon,h}\sum\limits_{\bar{k}\in\Z^2}\int_{0}^{T}
\|Re(\{\pa_t+v_{z}\pa _{z}+\eps^2\Phi(z)\cdot\na_\Fv-\al\eps v_{z}\pa_{v_x}\}\hat{f}|\hat{f})\|_{L^1}dt,
\end{align}
for any $T\geq0.$
\end{lemma}
\begin{proof}
To prove \eqref{utrace}, we only consider the case that the boundary phase is outgoing, because the incoming case can be treated similarly. We introduce a parameter $t\in\R$ and treat $(z,\Fv)$ as functions of $t$.
Define the characteristic line $[s,Z(s;t,z,\Fv),\FV(s;t,z,\Fv)]$ passing through $(z,\Fv)=(t,z(t),\Fv(t))$ such that
\begin{align}
\frac{d Z}{ds}=V_{Z},\
\frac{d \FV}{ds}=\eps^2\Phi(Z)-\al\eps V_{Z}\Fe_1,\ \FV=(V_x,V_y,V_{Z}).\notag
\end{align}
Then it follows
\begin{align}
\left\{\begin{array}{rll}
&\dis Z(s)=z+(s-t)v_{z}+\eps^2\int_t^s\int_t^\tau\Phi_z(Z(\eta))d\eta d\tau,\\[3mm]
&\dis \FV(s)=\Fv+\eps^2\int_t^s\Phi(Z(\tau))d\tau-\al\eps v_{z}(s-t)\Fe_1\\
&\dis \qquad\qquad-\al\eps^3\int_t^s\int_t^\tau\Phi_z(Z(\eta))d\eta d\tau \Fe_1,\notag
\end{array}\right.
\end{align}
for $(z,\Fv)\in\gamma_+\backslash\gamma^{\varepsilon}_+$. Along this trajectory, one has the identity
\begin{align}\label{f-ex-ch}
\hat{f}^2(z,\Fv)=\hat{f}^2((Z,\FV)(s;t,z,\Fv))+\int_{s}^t\frac{d}{d\tau}\hat{f}^2((Z,\FV)(\tau;t,z,\Fv))d\tau.
\end{align}
Next, by taking $s\in[t-t_{\Fb}(z,\Fv),t]$, we get from \eqref{f-ex-ch} that 
\begin{align}\label{f-ex-ch-p1}
\int_{t-t_{b}(z,\Fv)}^t&\int_{\gamma_+\backslash\gamma^{\varepsilon}_+}|\hat{f}(z,\Fv)|^2|v_{z}|d\Fv ds
\\=& \int_{\gamma_+\backslash\gamma^{\varepsilon}_+}\int_{t-t_{b}(z,\Fv)}^t \hat{f}^2(Z(s;t,z,\Fv),\FV(s;t,z,\Fv))|v_{z}|dsd\Fv\notag\\
&+\int_{\gamma_+\backslash\gamma^{\varepsilon}_+}\int_{t-t_{b}(z,\Fv)}^t\int_{s}^t\frac{d}{d\tau}\hat{f}^2(Z(\tau;t,z,\Fv),\FV(\tau;t,z,\Fv))||v_{z}|d\tau ds d\Fv\notag\\
=& \int_{\gamma_+\backslash\gamma^{\varepsilon}_+}\int_{t-t_{b}(z,\Fv))}^t \hat{f}^2(Z(s;t,z,\Fv),\FV(s;t,z,\Fv))|v_{z}|dsd\Fv
\notag\\
&+\int_{\gamma_+\backslash\gamma^{\varepsilon}_+}\int_{t-t_{b}(z,\Fv)}^t\int_{s}^t \mathcal{T}_{0,\Phi}\hat{f}^2((Z,\FV)(s;t,z,\Fv))|v_{z}|d\tau ds d\Fv
\notag\\
=& \int_{\gamma_+\backslash\gamma^{\varepsilon}_+}\int_{t-t_{b}(z,\Fv)}^t \hat{f}^2(Z(s;t,z,\Fv),\FV(s;t,z,\Fv))|v_{z}|dsd\Fv
\notag\\
&+2\int_{\gamma_+\backslash\gamma^{\varepsilon}_+}\int_{t-t_{b}(z,\Fv)}^t\int_{s}^t Re\left( \mathcal{T}_{0,\Phi}\hat{f}|\hat{f}\right)|v_{z}|d\tau ds d\Fv, \notag
\end{align}
where $Re \mathcal {Z}$ denotes the real part of $\CZ$ and we have denoted the differential operator
\begin{equation}\label{def.drj.do}
    \mathcal{T}_{0,\Phi}=V_{Z}\pa_Z+\eps^2\Phi(Z)\cdot\na_\FV-\al\eps V_{Z}\pa_{V_x}.
\end{equation}

On the other hand, for $(z,\Fv)\in\gamma_+\backslash\gamma^{\varepsilon}_+$, by \eqref{tb-lbde} and \eqref{zb1}, one sees that
\begin{align}\label{tb-lub}
\vps\lesssim t_{\Fb}(z,\Fv)\lesssim\frac{1}{\vps}.
\end{align}
Thus \eqref{f-ex-ch-p1} further implies
\begin{align}\label{f-ex-ch-p2}
\int_{\gamma_+\backslash\gamma^{\varepsilon}_+}  |\hat{f}(z,\Fv)|^2|v_{z}|d\Fv 
&\lesssim \int_{\gamma_+\backslash\gamma^{\varepsilon}_+}\int_{t-t_{b}(z,\Fv))}^t \hat{f}^2((Z,\FV)(s;t,z,\Fv))|v_{z}|dsd\Fv
\notag\\ 
&\qquad+\int_{\gamma_+\backslash\gamma^{\varepsilon}_+}\int_{t-t_{b}(z,\Fv)}^t\left|Re\left( \mathcal{T}_{0,\Phi}\hat{f}|\hat{f}\right)\right||v_{z}|ds d\Fv.
\end{align}
Next, we compute the Jacobian as 
\begin{eqnarray}\label{def.matrM}
\begin{aligned} \frac{\pa(Z(s),\FV(s))}{\pa(s,\Fv)}=&\left|\begin{array} {cccc}
v_{z}+\eps^2\int_t^s\Phi_z(Z(\eta))d\eta \  \  \  & 0 \ \ \ & 0 \ \ \ &s-t\\[3mm]
\eps^2\Phi_x-\al\eps^3\int_t^s\Phi_z(Z(\eta))d\eta-\al\eps v_{z}  \  \  \  & 1 \ \ \ & 0 \ \ \ &-\al\eps(s-t)\\[3mm]
\eps^2\Phi_y \  \  \  & 0 \ \ \ & 1 \ \ \ &0\\[3mm]
\eps^2\Phi_z \  \  \  & 0 \ \ \ & 0 \ \ \ &1
\end{array} \right|\notag\\=&v_{z}+\eps^2\int_t^s\Phi_z(Z(\eta))d\eta-(s-t)\eps^2\Phi_z\\
=&v_{z}+\eps^2\int_t^s[\Phi_z(Z(\eta)-\Phi_z(Z(s))]d\eta.
\end{aligned}
\end{eqnarray}
On the other hand, using $\vps\lesssim t_{\Fb}(z,\Fv)\lesssim\frac{1}{\vps}$ again, one has for $0<\eps\ll \eps_0<\vps$,
\begin{align}
\left|\eps^2\int_t^s[\Phi_z(Z(\eta)-\Phi_z(Z(s))]d\eta\right|\leq C_\Phi\eps^2t_{\Fb}(z,\Fv)\leq C_\Phi\frac{\eps^2}{\vps}\leq \frac{\vps}{2},
\end{align}
Thus, if $(z,\Fv)\in\gamma_+\backslash\gamma^{\varepsilon}_+$ then it holds
\begin{align}
\left|\frac{\pa(Z(s),\FV(s))}{\pa(s,\Fv)}\right|\backsim |v_{z}|.\notag
\end{align}
Note that if $\Phi_{z}$ is a constant,
\begin{align}
\left|\frac{\pa(Z(s),\FV(s))}{\pa(s,\Fv)}\right|= |v_{z}|.\notag
\end{align}
By a change of variable
$$[\tilde{z},\Fu]=[Z(s;t,z,\Fv),\FV(s;t,y,\Fv)]\rightarrow(s,\Fv),$$
one gets
\begin{align}\label{bb-p1}
\int_{\gamma_+\backslash\gamma^{\varepsilon}_+}\int_{t-t_b(z,\Fv)}^t|g((Z,\FV)(s;t,z,\Fv))||v_{z}|dsd\Fv
\leq \int_{\R^3}\int_{-h}^h|g(\tilde{z},\Fu)|d\tilde{z}d\Fu.
\end{align}
for any $g\in L^1((-h,h)\times\R^3).$ Hence, \eqref{utrace} follows from \eqref{bb-p1} and \eqref{f-ex-ch-p2}.

We now turn to prove \eqref{ttrace}. For $\hat{f}^2\in L^1([T_1,T]\times[-h,h]\times\R^3)$, we first show that
\begin{align}\label{axch}
\int_{T_1}^T&\int_{\Fu\cdot n(z_f)>0}\int_{\max\{-t_{b}(z_f,\Fu),T_1-\tilde{t}\}}^0
\hat{f}^2(\tilde{t}+s,(Z,\FV)(\tilde{t}+s;\tilde{t},z_f,\Fu))|u_z|dsd\Fu d\tilde{t}\notag\\
\leq&\int_{T_1}^T\int_{-h}^h\int_{\R^3}\hat{f}^2(t,z,\Fv)dzd\Fv dt,
\end{align}
where $z_f=\pm h$, $T\geq T_1\geq0$,
\begin{align}
Z(\tilde{t}+s;\tilde{t},z_f,\Fu)=z_f+su_z+\eps^2\int_{\tilde{t}}^{\tilde{t}+s}\int_{\tilde{t}}^\tau\Phi_z(Z(\eta))d\eta d\tau,\notag
\end{align}
and
\begin{align*}
\FV(\tilde{t}+s;\tilde{t},z_f,\Fu)
=
&
\Fu+\eps^2\int_{\tilde{t}}^{\tilde{t}+s}\Phi(Z(\tau))d\tau-\al\eps v_{z}s\Fe_1
\\
& 
-\al\eps^3\int_{\tilde{t}}^{\tilde{t}+s}\int_{\tilde{t}}^\tau\Phi_z(Z(\eta))d\eta d\tau \Fe_1,
\end{align*}
with
\begin{align}
Z(\tilde{t};\tilde{t},z_f,\Fu)=z_f,\ \ \FV(\tilde{t};\tilde{t},z_f,\Fu)=\Fu=(u_x,u_y,u_z).\notag
\end{align}
Actually, given $(t,z,\Fu)\in[T_1,T]\times[-h,h]\times\R^3$, let us
denote
\begin{align}
z=Z(t;t-s,z_f,\Fu),\ \ \Fv=\FV(t;t-s,z_f,\Fu),\notag
\end{align}
for $\Fu\cdot n(z_f)>0$. It is easy to see that $0\geq s\geq -t_{b}(z_f,\Fu),$ and it is natural to require that $t-s\leq T.$
By a change of variable $(z,\Fv)\rightarrow(s,\Fu)$ and using \eqref{def.matrM},
one has
\begin{align}
\int_{T_1}^T&\int_{\Fu\cdot n(z_f)>0}\int_{\max\{-t_{b}(z_f,\Fu),-(T-t)\}}^0
|\hat{f}^2(t,(Z,V)(t;t-s,z_f,\Fu))||u_z|dsd\Fu dt\notag\\
\leq&\int_{T_1}^T\int_{-h}^h\int_{\R^3}|\hat{f}^2(t,z,\Fv)|dzd\Fv dt.\label{chv-v}
\end{align}
On the other hand, if we denote $\tilde{t}=t-s$, then it follows
$s\geq T_1-\tilde{t}$ due to $t\geq T_1.$  In summary, one has
$$
s\geq\max\{-t_{b}(z_f,\Fu),T_1-\tilde{t}\}, \ T_1\leq\tilde{t}\leq T.
$$
Therefore, we have by  change of variable  $t\rightarrow\tilde{t}$ that
\begin{multline}
\int_{T_1}^T\int_{\Fu\cdot n(z_f)>0}\int_{\max\{-t_{b}(z_f,\Fu),-(T-t)\}}^0
\hat{f}^2(t,(Z,V)(t;t-s,z_f,\Fu))|u_z|dsd\Fu dt\\
=\int_{T_1}^T\int_{\Fu\cdot n(z_f)>0}\int_{\max\{-t_{b}(z_f,\Fu),T_1-\tilde{t}\}}^0\\
\hat{f}^2(\tilde{t}+s,(Z,\FV)(\tilde{t}+s;\tilde{t},z_f,\Fu))|u_z|dsd\Fu d\tilde{t}.\label{cht-t}
\end{multline}
Consequently,
\eqref{chv-v} and \eqref{cht-t} imply \eqref{axch}.
In addition, it follows that
\begin{align}\label{difff}
&\hat{f}^2(t,z_f,\Fu)\notag\\
&=\hat{f}^2(t+s,Z(t+s;t,z_f,\Fu),\FV(t+s;t,z_f,\Fu))\notag
\\
&\quad+\int_{s}^0Re\left\{\frac{d}{d\tau}\hat{f}^2(t+\tau,Z(t+\tau;t,z_f,\Fu),\FV(t+\tau;t,z_f,\Fu))\right\}d\tau\notag\\
&=\hat{f}^2(t+s,Z(t+s;t,z_f,\Fu),\FV(t+s;t,z_f,\Fu))\notag\\
&\quad
+\int_{s}^0Re\left(\mathcal{T}_{1,\Phi}\hat{f}|\hat{f}\right)(t+\tau)d\tau.
\end{align}
where similar to \eqref{def.drj.do}, we have denoted the differential operator
\begin{equation}\notag
    \mathcal{T}_{1,\Phi}=\pa_\tau+V_{Z}\pa_Z+\eps^2\Phi(Z)\cdot\na_\FV-\al\eps V_{Z}\pa_{V_x}.
\end{equation}

For any $(t,z_f,\Fu)\in[\vps_1,T]\times\gamma_+\backslash\gamma^{\varepsilon}_+$ with $\vps_1>0$ to be determined later and for $0\geq s\geq\max\{-t_{b}(z_f,\Fu),\vps_1-t\}$,
we then get from \eqref{difff} and \eqref{axch} that
\begin{align}\label{bd.es1}
&\min\{h\vps,\vps_1\}\int_{\vps_1}^T\int_{\Fu\cdot n(z_f)>0}\hat{f}^2(t,z_f,\Fu)|u_z|d\Fu dt\notag\\
&\leq\int_{\vps_1}^T\int_{\Fu\cdot n(z_f)>0}\int_{\max\{-t_{b}(z_f,\Fu),-t\}}^0\notag\\
&\qquad\qquad\hat{f}^2(t+s,Z(t+s;t,z_f,\Fu),\FV(t+s;t,z_f,\Fu))|u_z|dtdsd\Fu\notag\\&
\quad+\int_{\vps_1}^T\int_{\max\{-t_{b}(z_f,\Fu),-t\}}^0\int_{\Fu\cdot n(z_f)>0}
\int_{s}^0\left|Re\left(\mathcal{T}_{1,\Phi}\hat{f}|\hat{f}\right)(t+\tau)\right||u_z|d\tau d\Fu dt\notag\\
&\leq\int_{0}^T\int_{\Fu\cdot n(z_f)>0}\int_{\max\{-t_{b}(z_f,\Fu),-t\}}^0\notag\\
&\qquad\qquad|f(t+s,Z(t+s;t,z_f,\Fu),\FV(t+s;t,z_f,\Fu))||u_3|dtdsd\Fu\notag\\&
\quad+\int_{0}^T\int_{\max\{-t_{b}(z_f,\Fu),-t\}}^0\int_{\Fu\cdot n(z_f)>0}
\int_{s}^0\left|Re\left(\mathcal{T}_{1,\Phi}\hat{f}|\hat{f}\right)(t+\tau)\right||u_z|d\tau dudt
\notag\\
&\leq\int_{0}^T\int_{-h}^h\int_{\R^3}|f(t,z,\Fu))|dtdzd\Fu\notag\\&
\quad+\int_{0}^T\int_{\max\{-t_{b}(z_f,\Fu),-t\}}^0\int_{\Fu\cdot n(z_f)>0}
\int_{s}^0\left|Re\left(\mathcal{T}_{1,\Phi}\hat{f}|\hat{f}\right)(t+\tau)\right||u_z|d\tau d\Fu dt,
\end{align}
where we have used \eqref{tb-lub} again.

Next, applying Fubini's Theorem and  using \eqref{axch} once more, one also has
\begin{align}\label{bd.es2}
\int_{0}^T&\int_{\Fu\cdot n(z_f)>0}\int_{\max\{-t_{b}(z_f,\Fu),-t\}}^0\int_{s}^0\left|Re\left(\mathcal{T}_{1,\Phi}\hat{f}|\hat{f}\right)(t+\tau)\right||u_z|d\tau d\Fu dtds\notag\\
=& \int_{0}^Tdt\int_{\Fu\cdot n(z_f)>0}d\Fu\int^{\tau}_{\max\{-t_{b}(z_f,\Fu),-t\}}ds\int_{\max\{-t_{\Fb}(z_f,\Fu),-t\}}^0 d\tau\notag\\
&\times\left|Re\left([\pa_\tau+V_{Z}\pa_Z+\eps^2\Phi(Z)\cdot\na_\FV-\al\eps V_{Z}\pa_{V_x}]\hat{f}|\hat{f}\right)(t+\tau)\right||u_z|
\notag\\
\leq&\int_{0}^Tdt\int_{\Fu\cdot n(y_f)>0}d\Fu\int_{\max\{-t_{b}(y_f,u),-t\}}^0 d\tau|\max\{-t_{\Fb}(y_f,u),-t\}|
\notag\\
&\times\left|Re\left([\pa_\tau+V_{Z}\pa_Z+\eps^2\Phi(Z)\cdot\na_\FV-\al\eps V_{Z}\pa_{V_x}]\hat{f}|\hat{f}\right)(t+\tau)\right||u_z|
\notag\\
\leq&\max\{h\vps,\vps_1\}\int_{0}^Tdt\int_{u\cdot n(y_f)>0}d\Fu\int_{\max\{-t_{\Fb}(z_f,\Fu),-t\}}^0 d\tau
\notag\\
&\times\left|Re\left([\pa_\tau+V_{Z}\pa_Z+\eps^2\Phi(Z)\cdot\na_\FV-\al\eps V_{Z}\pa_{V_x}]\hat{f}|\hat{f}\right)(t+\tau)\right||u_z|
\notag\\
\leq&\max\{h\vps,\vps_1\}\int_{0}^Tdt\int_{-h}^hdz\int_{\R^3}d\Fu \notag\\
&\qquad|Re([\pa_t+u_z\pa_z+\eps^2\Phi(z)\cdot\na_\Fu-\al u_z\pa_{u_x}]\hat{f}|\hat{f})(t,z,\Fu)|.
\end{align}
Once \eqref{bd.es1} and \eqref{bd.es2} are obtained, it remains now to compute
\begin{align*}
\int^{\vps_1}_0\int_{\Fu\cdot n(z_f)>0}\hat{f}^2(t,z_f,\Fu)|u_z|d\Fu dt.
\end{align*}
In fact, if we choose $\vps_1$ to be small enough so that $\vps_1\leq h\vps$, then the backward trajectory hits the initial plane first. Therefore, for $(t,z_f,\Fu)\in[0,\vps_1]\times\gamma_+\backslash\gamma^{\varepsilon}_+$, by directly using \eqref{def.matrM} and applying \eqref{axch} once again, it follows
\begin{align}
\int_0^{\vps_1}&\int_{\Fu\cdot n(z_f)>0}\hat{f}^2(t,z_f,\Fu)|u_z|d\Fu dt\notag\\
\leq&\int_0^{\vps_1}\int_{\Fu\cdot n(z_f)>0}\hat{f}^2(0,Z(0;t,z_f,\Fu),\FV(0;t,z_f,\Fu))|u_z|d\Fu dt\notag\\
&+\int_0^{\vps_1}\int_{\Fu\cdot n(z_f)>0}\int_{-t}^0\left|Re\left(\mathcal{T}_{1,\Phi}\hat{f}|\hat{f}\right)(t+\tau)\right||u_z|d\tau d\Fu dt
\notag\\
\leq&C\int_{-h}^h\int_{\R^3}\hat{f}^2(0,z,\Fu)dzd\Fu+C\int_0^{\vps_1}\int_{-h}^h\int_{\R^3}
|I(t,z,\Fu)|dzd\Fu dt,\notag
\end{align}
with
$$
I(t,z,\Fu)=Re([\pa_t+u_z\pa_z+\eps^2\Phi(z)\cdot\na_\Fu-\al u_z\pa_{u_x}]\hat{f}|\hat{f})(t,z,\Fu).
$$
The proof of Lemma \ref{ukai} is then completed.
\end{proof}

The following lemma is concerned with the integral operator $K$ given by \eqref{sp.L}, and its proof in case of the hard sphere model $(\ga=1)$ has been given by \cite[Lemma 3, pp.727]{Guo-2010}.
\begin{lemma}\label{Kop}
Let $K$ be defined as \eqref{sp.L}, then it holds that
\begin{align}
Kf(\Fv)=K_2f(\Fv)-K_1f(\Fv)=\int_{\R^3}(\Fk_2(\Fv,\Fv_\ast)-\Fk_1(\Fv,\Fv_\ast))f(\Fv_\ast)\,d\Fv_\ast\notag
\end{align}
with
\begin{equation*}
\Fk_1(\Fv,\Fv_\ast)=\tilde{C}_1|\Fv-\Fv_\ast|^\ga e^{-\frac{|\Fv|^2+|\Fv_\ast|^{2}}{4}}, 
\end{equation*}
and
\begin{equation*}
\Fk_2(\Fv,\Fv_\ast)= \tilde{C}_2|\Fv-\Fv_\ast|^{-2+\ga}
e^{-\frac{1}{8}|\Fv-\Fv_\ast|^{2}-\frac{1}{8}\frac{\left||\Fv|^{2}-|\Fv_\ast|^{2}\right|^{2}}{|\Fv-\Fv_\ast|^{2}}}. 
\end{equation*}
Here both $\tilde{C}_1$ and $\tilde{C}_2$ are positive constants.
In addition, let
\begin{align}\label{kw-def}
\Fk(\Fv,\Fv_\ast)=\Fk_2(\Fv,\Fv_\ast)-\Fk_1(\Fv,\Fv_\ast),\
\Fk_w(\Fv,\Fv_\ast)=w^{\ell}(\Fv)\Fk(\Fv,\Fv_\ast)w^{-\ell}(\Fv_\ast)
\end{align}
with  $\ell\geq0$,
then it also holds that
\begin{equation}
\int_{\R^3} \Fk_w(\Fv,\Fv_\ast)e^{\frac{\varepsilon|\Fv-\Fv_\ast|^2}{16}}dv_\ast\leq \frac{C}{1+|\Fv|},\notag
\end{equation}
for $\varepsilon=0$ or any $\varepsilon> 0$ small enough.

Moreover, for any $\ell\geq0$ it holds that
\begin{align}
|w^\ell Kf|\leq C\|w^\ell f\|_{\infty}.\notag
\end{align}

\end{lemma}

For the velocity weighted derivative estimates on the nonlinear operator $\Ga$, one has
\begin{lemma}\label{Ga}
Let $0\leq \ga\leq 1$ and $\ta\in[0,1]$. For any $p\in[1,+\infty]$ and any $\ell\geq0$, it holds that
\begin{align}\label{es1.Ga}
\|w^\ell\nu^{-\ta}\Ga(f,g)\|_{L_\Fv^p}\leq C
\left\{\|w^\ell\nu^{1-\ta}f\|_{L_\Fv^p}\|g\|_{L_\Fv^p}+\|f\|_{L_\Fv^p}
\|w^\ell\nu^{1-\ta} g\|_{L_\Fv^p}\right\}.
\end{align}

\end{lemma}

The following lemma is concerned with  coercivity estimates for the linear collision operator $L$.

\begin{lemma}\label{es-L}
Let $0\leq \ga\leq1$, then there is a constant $\de_0>0$ such that
\begin{align}
\lag Lf,f\rag=\lag L\FP_1f,\FP_1f\rag\geq\de_0\|\FP_1f\|_\nu^2,\notag
\end{align}
where $\|\cdot\|_\nu=\|\nu^{\frac{1}{2}}\cdot\|.$
\end{lemma}

In the case of $0\leq\ga\leq1$, the following lemma with $\vth=0$ which can be found in
\cite[Proposition 3.1, pp.397]{AEP-87} enables us to gain the smallness property of $\CK$ defined as \eqref{CK-def} at large velocity.
\begin{lemma}\label{g-ck}
Let $0\leq\ga\leq1$, $\ell>4$, then there exists a function $\varsigma(\ell)$ which satisfies $\varsigma(\ell)\rightarrow0$ as $\ell\rightarrow+\infty$
such that
\begin{align}
w^\ell\{| &Q_{\rm{loss}}(f,g)|+|Q_{\rm{gain}}(f,g)|+|Q_{\rm{gain}}(g,f)|\}\notag\\
\leq& C
\|w^\ell f\|_{\infty}\{C(\ell)\|w^{\ell+\ga/2}g\|_{\infty}
+\varsigma(\ell)\|w^{3}g\|_{\infty}(1+|\Fv|)^\ga\}.\notag
\end{align}
\end{lemma}

The following result, which has been proved in \cite{DL-2022}  is a direct consequence of Lemma \ref{g-ck}.

\begin{lemma}\label{g-ck-lem}
Let $0\leq\ga\leq1$, then there is a constant $C>0$ such that for any arbitrarily large $\ell>4$, there are sufficiently large $M=M(\ell)>0$ and suitably small $\varsigma=\varsigma(\ell)>0$ such that it holds that
\begin{align}
\chi_M\nu^{-1} w^{\ell}|\CK f|\leq C \{(1+M)^{-\ga}+\varsigma\}\|w^{\ell}  f\|_{\infty}.\notag
\end{align}
\end{lemma}

The following lemma is concerned with the $L^2$ estimate on the operator $\CK$.
\begin{lemma}\label{CK-l2-lem}
Let $0\leq\ga\leq1$, then there is a constant $C>0$ such that for suitably large $l_2>0$, there are sufficiently large $M=M(\ell_2)>0$ and suitably small $\varsigma=\varsigma(l_2)>0$ such that it holds that
\begin{align}\label{CK-l2}
\|\nu^{-1/2} w^{l_2}\chi_M\CK f\|_2\leq C\{(1+M)^{-\ga}+\varsigma\}^{1/2}(1+M)^{-\ga/2}\|\nu^{1/2}w^{l_2}f\|_2.
\end{align}

\end{lemma}
\begin{proof}
The validation of \eqref{CK-l2} for the scenario $0 <\gamma \leq 1$ is presented in \cite[pp.7, Proposition 2.1]{DL-2022}. The proof for the case where $\gamma = 0$ can be established in a similar manner.
\end{proof}

The following Lemma concerning the polynomial weighted estimates on the collision operator $Q$ can be verified by using a parallel argument as for obtaining \cite[Proposition 3.1, pp.397]{AEP-87}.

\begin{lemma}\label{op-es-lem}Let $\ell>4$ and $1\geq\ga\geq0$, then it holds that
\begin{equation}\notag
|w^{\ell} \nu^{-1} Q_{gain}(F_1,F_2)|,\ |w^{\ell} \nu^{-1} Q_{loss}(F_1,F_2)|\leq C\|w^\ell  F_1\|_{\infty}\|w^\ell F_2\|_{\infty}.
\end{equation}
\end{lemma}
Base on Lemmas \ref{Ga} and \ref{op-es-lem}, one further has
\begin{lemma}\label{lk1-ne}Let $1\geq\ga\geq0$, for $\ell>4$, then it holds that
\begin{equation}\label{Q-lk1}
|w^{\ell} \nu^{-1} \hat{Q}(\hat{F}_1,\hat{F}_2)|\leq C\sum\limits_{\bar{l}\in\Z^2}\|w^\ell \hat{F}_1(\bar{k}-\bar{l})\|_{\infty}\sum\limits_{\bar{l}\in\Z^2}\|w^\ell \hat{F}_2(\bar{l})\|_{\infty},
\end{equation}
and for any $\ell\geq0$, it holds that
\begin{equation}\label{Ga-lk1}
|w^{\ell} \nu^{-1} \hat{\Ga}(\hat{F}_1,\hat{F}_2)|\leq C\sum\limits_{\bar{l}\in\Z^2}\|w^\ell \hat{F}_1(\bar{k}-\bar{l})\|_{\infty}\sum\limits_{\bar{l}\in\Z^2}\|w^\ell \hat{F}_2(\bar{l})\|_{\infty},
\end{equation}
where $\hat{Q}(\hat{F}_1,\hat{F}_2)$ and $\hat{\Ga}(\hat{F}_1,\hat{F}_2)$ are defined by \eqref{Ga-ft} and \eqref{Q-ft-def}, respectively.

\end{lemma}
\begin{proof}
We only prove \eqref{Q-lk1}, because \eqref{Ga-lk1} can be verified in the same way.
Recalling \eqref{Q-ft-def} and \eqref{Q-op}, we get
\begin{align}\label{q-pft}
\hat{Q}(\hat{F}_1,\hat{F}_2)=\sum\limits_{\hat{l}\in\Z^2}Q(\hat{F}_1(\bar{k}-\bar{l}),\hat{F}_2(\bar{l})),
\end{align}
therefore one has by using Lemma \ref{op-es-lem} and generalized Minkowski' inequality that
\begin{align}
|w^{\ell} \nu^{-1} \hat{Q}(\hat{F}_1,\hat{F}_2)|\leq & C\sum\limits_{\bar{l}\in\Z^2}\|w^\ell \hat{F}_1(\bar{k}-\bar{l})\|_{\infty}\|w^\ell \hat{F}_2(\bar{l})\|_{\infty}.\notag
\end{align}
This ends the proof of Lemma \ref{lk1-ne}.
\end{proof}

The next lemma is devoted to the weighted mixture estimate on the inner product $\lag \hat{Q}(\hat{f},\hat{g}),\hat{h}\rag$.
\begin{lemma}\label{es-tri}
Let $l_{\infty}> 2l_2\gg 4$, then, it holds that
\begin{equation}\label{es-tri-p1}
|\langle \hat{Q}(\hat{f},\hat{g}),w^{2l_2}  \hat{h}\rangle|\leq
C\|\nu^{\frac{1}{2}}w^{l_2}\hat{h}(\bar{k})\|_{2}\sum\limits_{\bar{l}}\|\nu^{\frac{1}{2}}w^{l_2} \hat{f}(\bar{k}-\bar{l})\|_2\|\nu^{\frac{1}{2}}w^{l_2}\hat{g}(\bar{l})\|_2,
\end{equation}
in particular, it follows that
\begin{equation}\label{es-tri-p2}
|\langle(1-\chi_M) \hat{Q}(\hat{f},\hat{g}), \hat{h}\rangle|\leq C\|\hat{h}\|_2\|\nu^{\frac{1}{2}}w^{l_2}\hat{f}\|_2\|\nu^{\frac{1}{2}}w^{l_2}\hat{g}\|_2.
\end{equation}
Moreover, it holds that for $l_2>2$,
\begin{align}\label{q-l2tk}
\left((\hat{Q}(\hat{f},\hat{g}),w^{\ell_2}\hat{h}\right)
\lesssim &
C_\eta\sum\limits_{\bar{l}}\int_{-1}^1\left(\|w^{l_2}\nu^{\frac{1}{2}} \hat{f}(\bar{k}-\bar{l})\|_2\|w^{l_2}\nu^{\frac{1}{2}}\hat{g}(\bar{l})\|_2\right)^2dz
\notag \\
& + \eta\|w^{l_2}\nu^{\frac{1}{2}}\hat{h}(\bar{k})\|^2_{2}.
\end{align}
and
\begin{align}\label{ga-l2tk}
|(\hat{\Ga}(\hat{f},\hat{g}),\hat{h})\lesssim
\eta\|\nu^{\frac{1}{2}}\hat{h}(\bar{k})\|^2_{2}+C_\eta\sum\limits_{\bar{l}}\int_{-1}^1\left(\|\nu^{\frac{1}{2}} \hat{f}(\bar{k}-\bar{l})\|_2\|\nu^{\frac{1}{2}}\hat{g}(\bar{l})\|_2\right)^2dz,
\end{align}

\end{lemma}
\begin{proof}
\eqref{es-tri-p1} and \eqref{es-tri-p2} is direct consequence of that of \cite[pp.37-38, Lemmas 3.3-3.4]{Cao-jfa}.
\eqref{q-l2tk} and \eqref{ga-l2tk} follows from \eqref{q-pft}, \eqref{es-tri-p1} and \eqref{es1.Ga} and Cauchy-Schwarz's inequality as well as generalized Minkowski's equality.
This ends the proof of Lemma \ref{es-tri}.
\end{proof}

We now conclude by presenting the proof for the elliptic estimate, as expressed in \eqref{epes-ap}, and the trace estimate given in \eqref{trace}.

\begin{proof}[The proof of \eqref{epes-ap} and \eqref{trace}]
We first prove \eqref{epes-ap}. The computation is divided into two cases.

\medskip
\noindent{\it Case 1. $\bar{k}\neq0$.} Taking the inner product of $\eqref{ep-a}_1$ and $\hat{\phi}_{a_s}$ and using $\eqref{ep-a}_2$, one directly has
\begin{align}\notag
\|\bar{k}\hat{\phi}_{a_s}\|_2^2+\|\pa_z\hat{\phi}_{a_s}\|_2^2\lesssim\|\hat{a}_s\|_2^2,
\end{align}
which also implies
\begin{align}
(1+|\bar{k}|)\|\hat{\phi}_{a_s}\|_2^2+\|\pa_z\hat{\phi}_{a_s}\|_2^2\lesssim\|\hat{a}_s\|_2^2,\label{as-ip2}
\end{align}
due to $|\bar{k}|\geq1$.
Similarly, the inner product of $\eqref{ep-a}_1$ and $|\bar{k}|^2\hat{\phi}_{a_s}$ gives
\begin{align}
\||\bar{k}|^2\hat{\phi}_{a_s}\|_2^2+\||\bar{k}|\pa_z\hat{\phi}_{a_s}\|_2^2\lesssim\|\hat{a}\|_2^2.\label{as-ip3}
\end{align}
\noindent{\it Case 2. $\bar{k}=0$.} In this case, $\eqref{ep-a}$ reads
\begin{align}\label{ep-a-k0}
\left\{\begin{array}{rll}
&-\pa_z^2 \hat{\phi}_{a_s}(0,z)=\hat{a}_s,\\[2mm]
&\pa_{z}\hat{\phi}_{a_s}(0,\pm 1)=0,\ \int_{-1}^1\hat{\phi}_{a_s}(0,z)dz=0.
\end{array}\right.
\end{align}
Then \eqref{ep-a-k0}$_1$ directly implies that for suitably small $\eta>0$,
\begin{align}\label{k0-p1}
\|\pa_z \hat{\phi}_{a_s}\|_{H^1_z}\leq C_\eta\|\hat{a}_s\|_2+\eta\|\hat{\phi}_{a_s}\|_{2}.
\end{align}
On the other hand, from $\eqref{ep-a-k0}_2$, we may assume there exits $z_0\in(-1,1)$ such that
\begin{align}
\hat{\phi}_{a_s}(0,z_0)=0.\notag
\end{align}
Therefore,
\begin{align}\label{k0-p2}
\|\hat{\phi}_{a_s}\|_{2}\lesssim \|\pa_z \hat{\phi}_{a_s}\|_2.
\end{align}
Then \eqref{k0-p1} and \eqref{k0-p2} yield
\begin{align}\label{k0-p3}
\|\hat{\phi}_{a_s}(0,z)\|_{H^2_z}\leq C\|\hat{a}_s\|_2.
\end{align}
Consequently, \eqref{as-ip2}, \eqref{as-ip3} and \eqref{k0-p3} gives \eqref{epes-ap}.

Finally, we turn to prove the trace inequality \eqref{trace}. Note that
\begin{align}
\hat{\phi}_{a_s}(\bar{k},\pm1)=\int_z^{\pm1}\pa_\tau\hat{\phi}_{a_s}(\bar{k},\tau)d\tau+\hat{\phi}_{a_s}(\bar{k},z),\ \forall z\in(-1,1).\notag
\end{align}
Integrating the above identity with respect to $z\in(-1,1)$ and applying H\"older's inequality, we further obtain
\begin{align}
|\hat{\phi}_{a_s}(\bar{k},\pm1)|\leq\|\pa_\tau\hat{\phi}_{a_s}(\bar{k},\tau)\|_2+\|\hat{\phi}_{a_s}(\bar{k},z)\|_2.\label{tr-as}
\end{align}
Consequently, \eqref{epes-a2} is a direct consequence of \eqref{epes-ap} and \eqref{tr-as}. This completes the proof of \eqref{epes-ap} and \eqref{epes-a2}.
\end{proof}

\medskip
\noindent {\bf Acknowledgments:}
Renjun Duan's research was partially supported by the General Research Fund (Project No.~14303321) from RGC of Hong Kong and the Direct Grant (4053652) from CUHK. Shuangqian Liu's research was supported by grant from the National Natural Science Foundation of China (contract 12325107).  Robert Strain's (\orcidlink{0000-0002-1107-8570} \href{https://orcid.org/0000-0002-1107-8570}{https://orcid.org/0000-0002-1107-8570}) research was partially supported by the NSF grant's DMS-2055271 and DMS-2408264 (USA).

\providecommand{\bysame}{\leavevmode\hbox to3em{\hrulefill}\thinspace}
\providecommand{\href}[2]{#2}


\end{document}